\documentclass[11pt]{article}

\usepackage[a4paper,margin=1.05in]{geometry}
\usepackage[utf8]{inputenc}
\usepackage[T1]{fontenc}
\usepackage{lmodern}
\usepackage{amsmath,amssymb,amsfonts,amsthm,mathtools}
\usepackage{aliascnt}
\usepackage{mathrsfs}
\usepackage{enumitem}
\usepackage{microtype}
\usepackage[colorlinks=true,linkcolor=blue,citecolor=blue,urlcolor=blue]{hyperref}
\usepackage[capitalize,nameinlink]{cleveref}

\numberwithin{equation}{section}

\newtheorem{theorem}{Theorem}[section]
\newaliascnt{proposition}{theorem}
\newtheorem{proposition}[proposition]{Proposition}
\aliascntresetthe{proposition}
\newaliascnt{lemma}{theorem}
\newtheorem{lemma}[lemma]{Lemma}
\aliascntresetthe{lemma}
\newaliascnt{corollary}{theorem}
\newtheorem{corollary}[corollary]{Corollary}
\aliascntresetthe{corollary}

\theoremstyle{definition}
\newaliascnt{definition}{theorem}
\newtheorem{definition}[definition]{Definition}
\aliascntresetthe{definition}
\newaliascnt{assumption}{theorem}
\newtheorem{assumption}[assumption]{Assumption}
\aliascntresetthe{assumption}
\newaliascnt{convention}{theorem}
\newtheorem{convention}[convention]{Convention}
\aliascntresetthe{convention}
\newaliascnt{example}{theorem}

\aliascntresetthe{example}

\theoremstyle{remark}
\newaliascnt{remark}{theorem}
\newtheorem{remark}[remark]{Remark}
\aliascntresetthe{remark}

\crefname{theorem}{Theorem}{Theorems}
\Crefname{theorem}{Theorem}{Theorems}
\crefname{proposition}{Proposition}{Propositions}
\Crefname{proposition}{Proposition}{Propositions}
\crefname{lemma}{Lemma}{Lemmas}
\Crefname{lemma}{Lemma}{Lemmas}
\crefname{corollary}{Corollary}{Corollaries}
\Crefname{corollary}{Corollary}{Corollaries}
\crefname{definition}{Definition}{Definitions}
\Crefname{definition}{Definition}{Definitions}
\crefname{assumption}{Assumption}{Assumptions}
\Crefname{assumption}{Assumption}{Assumptions}
\crefname{convention}{Convention}{Conventions}
\Crefname{convention}{Convention}{Conventions}
\crefname{example}{Example}{Examples}
\Crefname{example}{Example}{Examples}
\crefname{remark}{Remark}{Remarks}
\Crefname{remark}{Remark}{Remarks}
\crefname{section}{Section}{Sections}
\Crefname{section}{Section}{Sections}

\crefname{subsection}{Subsection}{Subsections}
\Crefname{subsection}{Subsection}{Subsections}
\crefname{subsubsection}{Subsection}{Subsections}
\Crefname{subsubsection}{Subsection}{Subsections}
\crefname{part}{Part}{Parts}
\Crefname{part}{Part}{Parts}
\crefname{equation}{equation}{equations}
\Crefname{equation}{Equation}{Equations}

\DeclareMathOperator{\dist}{dist}
\DeclareMathOperator{\supp}{supp}

\newcommand{\R}{\mathbb{R}}

\newcommand{\NS}{\mathrm{NS}}
\newcommand{\loc}{\mathrm{loc}}
\newcommand{\intg}{\mathrm{int}}
\newcommand{\tail}{\mathrm{tail}}
\newcommand{\spl}{\mathrm{split}}
\newcommand{\pkg}{\mathrm{pkg}}
\newcommand{\adm}{\mathrm{adm}}
\newcommand{\mom}{\mathrm{mom}}
\newcommand{\flux}{\mathrm{flux}}
\newcommand{\prs}{\mathrm{prs}}
\newcommand{\proj}{\mathrm{proj}}
\newcommand{\app}{\mathrm{app}}
\newcommand{\harm}{\mathrm{harm}}
\newcommand{\src}{\mathrm{src}}
\newcommand{\act}{\mathrm{act}}
\newcommand{\modsrc}{\mathrm{mod}}
\newcommand{\cl}{\mathrm{cl}}
\newcommand{\rep}{\mathrm{rep}}
\newcommand{\gs}{\mathrm{gs}}
\newcommand{\locerr}{\mathrm{loc}}
\newcommand{\comp}{\mathrm{comp}}
\newcommand{\chain}{\mathrm{chain}}
\newcommand{\quadg}{\mathrm{quad}}
\newcommand{\detc}{\mathrm{det}}

\newcommand{\CZ}{\mathrm{CZ}}
\newcommand{\quot}{\mathrm{quot}}
\newcommand{\res}{\mathrm{res}}
\newcommand{\sync}{\mathrm{sync}}
\newcommand{\smooth}{\mathrm{smooth}}
\newcommand{\aux}{\mathrm{aux}}

\newcommand{\norm}[2]{\left\|#1\right\|_{#2}}

\newcommand{\Err}{\mathsf{Err}}
\newcommand{\Disc}{\mathsf{Disc}}
\newcommand{\Tax}{\mathsf{Tax}}

\newcommand{\Leak}{\mathsf{Leak}}
\newcommand{\Rep}{\mathsf{Rep}}
\newcommand{\Gate}{\mathsf{Gate}}
\newcommand{\Slack}{\mathsf{Slack}}
\newcommand{\Dist}{\operatorname{dist}}

\newcommand{\calA}{\mathcal A}
\newcommand{\calC}{\mathcal C}
\newcommand{\calD}{\mathcal D}
\newcommand{\calE}{\mathcal E}
\newcommand{\calF}{\mathcal F}
\newcommand{\calG}{\mathcal G}
\newcommand{\calH}{\mathcal H}
\newcommand{\calO}{\mathcal O}

\newcommand{\calR}{\mathcal R}
\newcommand{\calX}{\mathcal X}
\newcommand{\calZ}{\mathcal Z}
\newcommand{\bfD}{\mathbf D}
\newcommand{\bD}{\mathbf D}
\newcommand{\bzeta}{\boldsymbol\zeta}
\newcommand{\Achan}{\mathfrak A}
\newcommand{\Xsrc}{X_{\src}}
\newcommand{\Ymom}{\mathcal Y_{\mom}}
\newcommand{\Yprs}{Y_{\prs}}
\newcommand{\Yharm}{Y_{\harm}^{(3/2)}}
\newcommand{\Qone}{Q_1}
\newcommand{\Gammaint}{\Gamma^{\intg}_{\Lambda}}
\newcommand{\Gammaadm}{\Gamma^{\intg}_{\Lambda,\adm}}
\newcommand{\Gadm}{\Gamma^{\mathrm{int}}_{\Lambda,\adm}}
\newcommand{\Gint}{\Gamma_{\Lambda,\adm}^{\mathrm{int}}}
\newcommand{\Gchain}{\Gamma_{\Lambda,\adm}^{\chain}}
\newcommand{\Gcomp}{\Gamma_{\Lambda,\adm}^{\comp}}

\title{\textbf{Finite-Window Local-to-Clean Transfer and Anti-Phantom Detection\\
for Sharp Navier--Stokes Packages}}

\author{Runlong Yu\\
The University of Alabama, Tuscaloosa, AL, USA\\
\texttt{ryu5@ua.edu}}

\date{}

\begin{document}

\maketitle

\begin{abstract}
We prove a fixed finite-window structural theorem for sharp localized Navier--Stokes packages, formulated as both a local-to-clean detection theorem and an anti-phantom principle.  The result addresses whether a defect visible in the baseline quotient geometry can disappear after pressure-tail enrichment, residual transfer, quotienting, and clean-to-local detector comparison.  Under synchronized representatives, baseline-to-tail visibility, component comparison, residual-ledger closure, detector comparison, chart visibility, and a clean quotient gap, any baseline-visible defect is either detected by the localized detector or charged to an explicit quotient-residual ledger.  Quantitatively,
\[
M_\Lambda^{\loc}(\calD-\bzeta_*)
\ge
c_{\Lambda,0},
\Dist_{\loc,\intg,0}(\calD,\Gamma^{\intg}*{\Lambda,\adm})
-
\mathfrak E^{\quot}*{\Lambda,0}(\calD).
\]
The proof assembles three modules: pressure-tail visibility, componentwise residual-ledger closure, and detector comparison.  The anti-phantom interpretation is that a baseline-visible defect cannot be simultaneously detector-silent and residual-cheap.  We also record provenance for the imported quotient interface, finite-dimensional pressure-tail models, explicit matrix realizations of the structural inputs, NS-generated coordinate realizability, compactness criteria for clean pressure images, and reduced pressure/tax kernel-free criteria.

\end{abstract}



\setcounter{tocdepth}{1}
\tableofcontents

\section{Problem, anti-phantom principle, and organization}\label{sec:intro-problem}

The purpose of this paper is to make precise one fixed finite-window step in a broader audit program for the three-dimensional incompressible Navier--Stokes equations \cite{YuSchur2026,YuInvisible2026,YuCriticalLedgers2026,YuSingularityAuditTransfer2026,YuComputationalAntiPhantom2026}
\[
        \partial_t u-\Delta u+u\cdot\nabla u+\nabla p=0,
        \qquad \nabla\cdot u=0.
\]
The weak-solution background goes back to Leray and Hopf, while the local partial-regularity endpoint used as orientation here comes from Scheffer and Caffarelli--Kohn--Nirenberg, together with later refinements and expositions \cite{Leray1934,Hopf1951,Scheffer1976,Scheffer1977,CKN1982,Lin1998,SereginLectureNotes}.  The pressure splitting and pressure-tail bookkeeping below are aligned with the standard local Calderon--Zygmund and harmonic-pressure viewpoint \cite{SohrWahl1986,SereginSverak2002,SereginLectureNotes}, and the residual/flux terminology is compatible with the coarse-grained local energy-transfer literature \cite{ConstantinETiti1994,Eyink1994,DuchonRobert2000}.

The program starts from a simple obstruction.  Localized data, clean pressure-natural data, pressure-tail coordinates, component residuals, and detector channels do not naturally live in the same representation.  A defect may be visible in an older baseline quotient geometry but appear to disappear after localization, pressure splitting, projection, quotienting, or detector comparison.

The central question is therefore the following.
\[
\boxed{\text{Can a baseline-visible finite-window defect be both detector-silent and residual-cheap?}}
\]
The theorem proved here gives a conditional finite-window answer:
\[
\boxed{\text{No, not once all named pressure-tail, residual, chart, detector, and synchronization costs are paid.}}
\]

Relative to the earlier manuscripts in this sequence, the present paper isolates the fixed-window local-to-clean step.  The defect-cascade formulation identifies the moving-window invisible obstruction \cite{YuInvisible2026}; the supply--tax ledger records the finite-scale payment alternative \cite{YuCriticalLedgers2026}; the singularity-audit transfer manuscript organizes the first local-to-clean residual budget \cite{YuSingularityAuditTransfer2026}; and the computational anti-phantom paper proves the clean finite-window quotient gap and enhanced-tail transfer toolkit \cite{YuComputationalAntiPhantom2026}.  The one-component precursors in this sequence are \cite{YuOneComponent2026,YuStrict2026,YuSchur2026}, with \cite{YuSchur2026} supplying the closest one-component Schur-visibility analogue of the anti-phantom terminology.

Informally, a finite-window phantom defect is a package \(\calD\) for which
\[
        \Dist_{\loc,\intg,0}(\calD,\Gamma^{\intg}_{\Lambda,\adm})>0
\]
but
\[
        M_\Lambda^{\loc}(\calD-\bzeta_*)=0.
\]
The anti-phantom principle says that such silence cannot be free.  The precise theorem gives
\[
        M_\Lambda^{\loc}(\calD-\bzeta_*)
        \ge
        c_{\Lambda,0}\,
        \Dist_{\loc,\intg,0}(\calD,\Gamma^{\intg}_{\Lambda,\adm})
        -
        \mathfrak E^{\quot}_{\Lambda,0}(\calD),
\]
where \(\mathfrak E^{\quot}_{\Lambda,0}\) is the assembled finite-window quotient error ledger.  Hence either the localized detector sees a fixed fraction of the baseline defect, or the quotient-residual ledger pays a comparable cost.

The paper has three layers.
\begin{enumerate}[label=(\roman*)]
\item \textbf{Abstract structural closure.}  The main theorem assembles pressure-tail visibility, residual-ledger closure, detector comparison, chart visibility, component-to-baseline comparison, and a clean quotient gap.
\item \textbf{Model realization.}  A finite-dimensional pressure-tail quotient model shows that the structural assumptions are mutually consistent and non-vacuous in a reduced finite-window setting.
\item \textbf{NS-generated interface.}  Local Navier--Stokes data are shown to generate the basic package coordinates, while compactness and kernel-free inputs are isolated as separate verification problems.
\end{enumerate}

This is not a Navier--Stokes regularity theorem.  It is adjacent to, but distinct from, one-component and anisotropic regularity criteria \cite{KukavicaZiane2006,CaoTiti2011,CheminZhang2016,CheminZhangZhang2017,HanLeiLiZhao2019,KangNguyen2023}, critical-space/backward-uniqueness methods \cite{ESS2003}, and recent quantitative regularity or concentration approaches \cite{BarkerPrange2021,AlbrittonBarkerPrange2023}.  It does not prove that arbitrary suitable weak solutions satisfy the structural hypotheses.  It proves that, once the listed fixed-window structural inputs are verified on a chosen package class, a baseline-visible defect cannot vanish from the localized detector without appearing in the explicit error ledger.

\subsection*{0.1 Theorem status table}

The theorem is a finite-window structural transfer result, not a proof of global Navier--Stokes regularity.  The status of the core inputs is as follows.
\begin{center}
\renewcommand{\arraystretch}{1.18}
\begin{tabular}{|p{0.34\linewidth}|p{0.54\linewidth}|}
\hline
\textbf{Item} & \textbf{Status}\\
\hline
Pressure-tail visibility & Proved here as a finite-window closure theorem under explicit baseline-visibility, finite-amplitude, same-gauge, projection-tail, and harmonic-tail hypotheses; the synchronized quotient interface is imported through the provenance contract.\\
\hline
Residual ledger closure & Proved here; see the main residual-ledger module and the componentwise closure theorem.\\
\hline
Detector comparison & Proved under structural detector-intertwining input and channelwise comparison hypotheses.\\
\hline
Clean anti-phantom gap & Structural assumption / criterion; verified in reduced finite-dimensional models and left as a concrete package-class input in the main theorem.\\
\hline
NS-generated package realization & Constructed interface: local Navier--Stokes data generate the package coordinates, but this does not verify every structural detector or clean-gap hypothesis.\\
\hline
Scale-uniform propagation & Not proved.  All constants and closures are fixed finite-window constants.\\
\hline
Singularity extraction & Future work.  No singularity exclusion, infinite-chain theorem, or Clay-problem conclusion is claimed.\\
\hline
\end{tabular}
\end{center}

\subsection*{0.2 Dependency graph}

The main theorem should be read through the following theorem-level chain:
\[
\boxed{
\begin{gathered}
\text{NS-generated coordinates}
\Rightarrow
\text{residual ledger}
\Rightarrow
\text{pressure-tail visibility}\\
\Rightarrow
\text{detector comparison}
\Rightarrow
\text{quotient anti-phantom}
\Rightarrow
\text{finite-window lower bound}.
\end{gathered}
}
\]
Each arrow is traceable to a named module.
\begin{enumerate}[label=\textup{(\arabic*)},leftmargin=*]
\item \emph{NS-generated coordinates $\Rightarrow$ residual ledger.}
The package-realizability interface is constructed in \Cref{thm:package-realizability}.  Once the same-chain representative and component coordinates are fixed, residual closure is supplied by \Cref{thm:main-ledger-module}, with the detailed componentwise proof in \Cref{comp:thm:componentwise-closure-target}.
\item \emph{Residual ledger $\Rightarrow$ pressure-tail visibility.}
The residual ledger is evaluated in the same synchronized representative used by the pressure-tail geometry.  The pressure-tail return to baseline geometry is \Cref{thm:main-pressure-tail-module}, whose proof invokes the finite-window closure statements \Cref{thm:baseline-pressure-tail-closure,thm:baseline-closure-uniform-projection-tail} and the harmonic-tail estimate \Cref{thm:L32-harmonic-tail}.
\item \emph{Pressure-tail visibility $\Rightarrow$ detector comparison.}
After pressure-tail and component errors are expressed in baseline/component distances, detector comparison applies through \Cref{ass:main-detector-comparison} and \Cref{thm:main-detector-module}; the detailed detector theorem is \Cref{thm:detector-comparison}.
\item \emph{Detector comparison $\Rightarrow$ quotient anti-phantom.}
Detector comparison is combined with the clean gap, chart visibility, component-to-baseline comparison, and positive-coefficient condition in \Cref{ass:main-clean-gap,ass:main-chart-visibility,ass:main-component-baseline,ass:main-positive-coefficient}.  The local-to-clean transfer is proved in \Cref{thm:local-to-clean-transfer}.
\item \emph{Quotient anti-phantom $\Rightarrow$ finite-window lower bound.}
The assembled result is \Cref{thm:journal-main}.  Its contrapositive forms are \Cref{cor:final-anti-phantom-alternative,cor:finite-model-anti-phantom-final}: a baseline-visible finite-window defect cannot be simultaneously detector-silent and residual-cheap, except through the explicit finite-window error ledger.
\end{enumerate}

The main text states the modules and proves the assembly.  The appendices contain the detailed pressure-tail, residual-ledger, detector-comparison, package-realizability, compactness, and reduced kernel-free proofs.
\section{Finite-window setup and quotient geometries}\label{sec:setup}

We work in a fixed normalized local geometry.  Space-time cylinders, cutoff functions, harmonic-tail spaces, pressure-source spaces, and detector channels are fixed once and for all.  The constants in this paper may depend on this finite list of choices.  This convention is not cosmetic: it is what separates the finite-window theorem from any scale-uniform or infinite-chain claim.

A finite-window package is denoted by \(\calD\).  It collects the coordinates needed by the local-to-clean audit: localized velocity and pressure data, clean pressure-source coordinates, harmonic pressure coordinates, component residual channels, gate/slack variables, and detector coordinates.  The admissible or model class is denoted by \(\Gamma\), with decorations indicating the geometry in which it is viewed.  A representative of the admissible class is denoted by \(\zeta\).  A finite chain of representatives is denoted by
\[
        \bzeta=(\zeta_0,\ldots,\zeta_K).
\]
The same-chain convention means that all residual channels are evaluated against one compatible choice of representatives.  This prevents a common error in quotient estimates: optimizing pressure, localization, reproduction, and gate/slack errors against different gauges and then pretending that the estimates hold simultaneously.

\subsection*{Baseline distance}

The older baseline geometry is the coarsest geometry used in the final detection statement.  It is measured by
\[
        \Dist_{\loc,\intg,0}
        (\calD,\Gamma^{\intg}_{\Lambda,\adm}).
\]
This distance is intentionally old-fashioned: it does not directly include every pressure-tail coordinate or every component residual coordinate.  The final theorem is stronger when written in this baseline geometry, because it says that the localized detector sees a defect that was already visible in the original quotient language.

\subsection*{Pressure-tail distance}

The pressure-tail geometry refines the baseline geometry by adding clean Calderon--Zygmund pressure coordinates and harmonic pressure-tail coordinates.  Schematically, it has the form
\[
        \|D\|_{\loc,\intg,\tail}
        =
        \|D\|_{\loc,\intg,0}
        +
        \|R_iR_jF^{\cl}_{D,ij}\|_{Y_{\prs}}
        +
        \|p_{\harm,D}\|_{Y_{\harm}}.
\]
The corresponding sharp tail distance also includes projection-tail and harmonic-tail penalties:
\[
\begin{aligned}
        \Dist^{\sharp,\tail}_{\loc,\intg,\tail}(D,\Gamma)
        :=
        \inf_{\zeta\in\Gamma}
        \bigl(&\|D-\zeta\|_{\loc,\intg,\tail}
        +\alpha_{\proj}T_{\proj}(D;\zeta)\\
        &+\alpha_{\harm}T_{\harm}(D;\zeta)\bigr).
\end{aligned}
\]
The pressure-tail module proves that, under baseline visibility and finite-tail approximation, this enhanced distance is controlled by the baseline distance plus explicit errors.

\subsection*{Component distance}

The component geometry is the geometry in which the residual ledger closes.  It is written as
\[
        \Dist_{\comp}^{\sharp,[0,K]}(\calD,\calG_{\comp}).
\]
This distance is not merely a pressure distance.  It contains the component coordinates needed to control pressure-source residuals, localization leakage, reproduction drift, and gate/slack violations across a finite chain.  Its role is intermediate: detector comparison naturally loses a component residual, and the residual-ledger theorem converts that residual into this component distance plus a finite-chain near-minimizer error.

\subsection*{Detectors and charts}

The localized detector is denoted by \(M_\Lambda^{\loc}\).  It is the detector evaluated on the localized package after subtraction of the selected same-chain representative.  The clean or component detector is denoted by \(M_\Lambda^{\comp}\).  The local-to-clean chart is
\[
        \Theta_\Lambda.
\]
The detector-comparison theorem is a stability estimate of the form
\[
        M_\Lambda^{\loc}(\calD-\bzeta_*)
        \ge
        M_\Lambda^{\comp}(\Theta_\Lambda(\calD-\bzeta_*))
        -C_{\mathrm{dc}}\Err_{\comp}^{[0,K]}(\calD;\bzeta_*)
        -\Delta_{\mathrm{dc}}.
\]
This is the bridge from the clean model back to the localized detector.

\subsection*{Imported quotient synchronization}

The final version uses the full finite-window quotient geometry developed in the earlier package manuscripts \cite{YuSchur2026,YuInvisible2026,YuCriticalLedgers2026,YuSingularityAuditTransfer2026,YuComputationalAntiPhantom2026} and instantiated through the structural contract in \Cref{app:imported-geometry-module}.  We do not introduce a simplified quotient convention in this paper.  Instead, we import the quotient geometry through a synchronized representative contract with explicit provenance: \Cref{tab:appendix-imported-provenance} records the source and status of each imported object, while \Cref{def:sync-loss-contract} fixes the synchronization loss.  For every package under consideration, a single admissible representative
\[
        \bzeta_* = \bzeta_*(\calD)
\]
is selected and used simultaneously for the baseline distance, pressure-tail excess, component distance, residual ledger, local-to-clean chart, and detector comparison.  If the imported quotient geometry gives exact synchronization, the synchronization loss is zero.  Otherwise the loss is recorded as
\[
        \Delta_{\sync}(\calD)\ge 0
\]
and is included in the final quotient error ledger.  In particular, the proof never minimizes the pressure tail, residual ledger, and detector mismatch over three different gauges and then combines the estimates as if they held on one representative.

\paragraph{Reference-grade provenance.}
For submission purposes the phrase ``imported quotient geometry'' is used only through the finite-window interface summarized in \Cref{tab:imported-geometry-provenance}.  Each item is either defined in the present paper, proved in the indicated appendix, or explicitly marked as a structural input to be verified in a concrete package class.  Thus the main theorem never relies on an unnamed background convention.

\begin{table}[h]
\centering
\small
\begin{tabular}{p{0.25\textwidth}p{0.34\textwidth}p{0.33\textwidth}}
\hline
\textbf{Object or convention} & \textbf{Role in the proof} & \textbf{Provenance in this paper} \\
\hline
Ambient package coordinates & Velocity, source, pressure, residual, gate/slack, and detector data & Section~2; coordinate realizability in \cref{thm:package-realizability} \\
Baseline quotient distance $\Dist_{\loc,\intg,0}$ & Reader-facing defect size in the main lower bound & Definitions~4.47 and~4.52 \\
Pressure-tail distance $\Dist_{\loc,\intg,\tail}^{\sharp,\tail}$ & Enhanced pressure-natural geometry & Definitions~4.4--4.5 and closure theorem \cref{thm:baseline-pressure-tail-closure} \\
Component distance $\Dist_{\comp}^{\sharp,[0,K]}$ & Geometry paying the residual ledger & Definitions~B.130--B.132 and closure theorem \cref{comp:thm:componentwise-closure-target} \\
Synchronized representative $\bzeta_*$ & One representative for baseline, tail, component, residual, chart, and detector terms & Same-gauge conventions B.12, B.36, B.64, B.125; exact finite-dimensional representatives in Lemmas~4.23 and~4.40 \\
Synchronization loss $\Delta_{\sync}$ & Cost of using one representative instead of independently optimized gauges & Defined in \cref{def:sync-loss-contract} \\
Residual-ledger closure & Converts named residual channels into one component distance & Main statement \cref{thm:main-ledger-module}; detailed proof \cref{comp:thm:componentwise-closure-target} \\
Detector comparison & Transfers clean detection to local detection after ledger payment & Main statement \cref{thm:detector-comparison} \\
Chart and clean gap inputs & Prevent the clean quotient and chart from hiding a baseline defect & Structural inputs in Assumptions~3.4--3.5; compact/reduced criteria in Appendix~C; explicit matrix verification in \cref{thm:explicit-detector-input-matrix-model} \\
\hline
\end{tabular}
\caption{Reference-grade provenance ledger for the imported quotient-geometric interface.}
\label{tab:imported-geometry-provenance}
\end{table}

\begin{definition}[Synchronization loss contract]\label{def:sync-loss-contract}
For a chosen synchronized representative $\bzeta_*(\calD)$, define
\[
\Delta_{\sync}(\calD;\bzeta_*)
:=
\delta_0(\calD;\bzeta_*)+
\delta_{\tail}(\calD;\bzeta_*)+
\delta_{\comp}(\calD;\bzeta_*)+
\delta_{\mathrm{chart}}(\calD;\bzeta_*)+
\delta_{\detc}(\calD;\bzeta_*),
\]
where
\[
\delta_0(\calD;\bzeta_*)
:=
\Bigl(\norm{\calD-\bzeta_*}{\loc,\intg,0}
      -\Dist_{\loc,\intg,0}(\calD,\Gamma^{\intg}_{\Lambda,\adm})\Bigr)_+,
\]
\[
\delta_{\tail}(\calD;\bzeta_*)
:=
\Bigl(\norm{\calD-\bzeta_*}{\loc,\intg,\tail}
      -\Dist_{\loc,\intg,\tail}(\calD,\Gamma^{\intg}_{\Lambda,\adm})\Bigr)_+,
\]
\[
\delta_{\comp}(\calD;\bzeta_*)
:=
\Bigl(|\calD-\bzeta_*|_{\comp}^{\sharp,[0,K]}
      -\Dist_{\comp}^{\sharp,[0,K]}(\calD,\Gamma^{\comp}_{\Lambda,\adm})\Bigr)_+.
\]
The terms $\delta_{\mathrm{chart}}$ and $\delta_{\detc}$ are the positive parts of any failure of the chart-visibility and detector-comparison estimates to hold on the same representative.  In the exact synchronized case all five defects vanish.  In the fixed finite-window theorems below, these representative-selection costs are either already included in the displayed module errors or are inserted into the final quotient ledger through $C_{\sync}\Delta_{\sync}$.
\end{definition}

\subsection*{Error ledgers}

All errors in the paper are finite-window errors.  The final error \(\mathfrak E_{\Lambda,0}\) is assembled from the following types of terms:
\[
\begin{array}{ll}
\text{projection-tail errors} & \Delta_{\proj,N},\ \Delta^{\mathrm{unif}}_{\proj,N}(\calA_\Lambda),\\[0.2em]
\text{harmonic-tail errors} & \Delta_{\harm,M},\ \Delta^{(3/2)}_{\harm,M},\\[0.2em]
\text{visibility and split errors} & \Delta_{\tail/0},\ \delta_0,\\[0.2em]
\text{component near-minimizer errors} & \delta_{\comp}^{[0,K]},\\[0.2em]
\text{detector-intertwining errors} & \Delta_{\mathrm{dc}},\\[0.2em]
\text{chart/component comparison errors} & \Delta_{\mathrm{chart}},\ \Delta_{\comp/0}.
\end{array}
\]
The exact combination is not universal; it depends on the selected finite-window model and on which sufficient criteria are used.  What is essential is that every term is explicit and appears with a finite-window constant.

\section{Structural hypotheses and their status}\label{sec:hypotheses}

This section states the hypotheses used in the main theorem in a level of detail appropriate for the body of the paper.  The appendices prove the estimates that justify these hypotheses in the finite-window framework and record sufficient criteria for several of them.

\begin{assumption}[Pressure-tail baseline visibility]\label{ass:main-pressure-tail-visibility}
The pressure-natural tail geometry is visible from the older baseline geometry.  More precisely, for the selected finite-window class there are constants and explicit errors such that
\begin{equation}\label{eq:ass-pressure-tail-visibility}
\begin{aligned}
\Dist^{\sharp,\tail}_{\loc,\intg,\tail}
(\calD,\Gamma^{\intg}_{\Lambda,\adm})
&\le
C_{\tail,0}
\Dist_{\loc,\intg,0}
(\calD,\Gamma^{\intg}_{\Lambda,\adm})\\
&\quad
+
\mathfrak E_{\tail,0}(\calD).
\end{aligned}
\end{equation}
Here \(\mathfrak E_{\tail,0}\) consists of projection-tail, harmonic-tail, finite-amplitude, and visibility errors.
\end{assumption}

\begin{assumption}[Componentwise residual-ledger closure]\label{ass:main-ledger-closure}
For a same-chain representative \(\bzeta_*\), the component residual satisfies
\begin{equation}\label{eq:ass-ledger-closure}
        \Err_{\comp}^{[0,K]}(\calD;\bzeta_*)
        \le
        C_{\comp}^{[0,K]}(M_U)
        \Dist_{\comp}^{\sharp,[0,K]}(\calD,\calG_{\comp})
        +
        C_{\comp}^{[0,K]}(M_U)
        \delta_{\comp}^{[0,K]}.
\end{equation}
The residual \(\Err_{\comp}^{[0,K]}\) includes pressure-source, localization leakage, reproduction drift, and gate/slack channels.
\end{assumption}

\begin{assumption}[Detector comparison]\label{ass:main-detector-comparison}
The localized and clean detectors are intertwined up to the component residual:
\begin{equation}\label{eq:ass-detector-comparison}
\begin{aligned}
        M_\Lambda^{\loc}(\calD-\bzeta_*)
        &\ge
        M_\Lambda^{\comp}(\Theta_\Lambda(\calD-\bzeta_*))
        -C_{\mathrm{dc}}
        \Err_{\comp}^{[0,K]}(\calD;\bzeta_*)
        -\Delta_{\mathrm{dc}}.
\end{aligned}
\end{equation}
\end{assumption}

\begin{assumption}[Clean anti-phantom gap]\label{ass:main-clean-gap}
There is a positive clean coefficient \(\mu_\Lambda^{\comp}>0\) such that the clean detector controls the clean defect on the selected clean quotient:
\begin{equation}\label{eq:ass-clean-gap}
        M_\Lambda^{\comp}(\Theta_\Lambda(\calD-\bzeta_*))
        \ge
        \mu_\Lambda^{\comp}
        \Dist_{\cl}(\Theta_\Lambda\calD,\Gamma_{\cl,\adm})
        -\Delta_{\cl}.
\end{equation}
This hypothesis rules out a clean nonzero defect that lies in the zero set of the clean detector.
\end{assumption}

\begin{assumption}[Chart visibility]\label{ass:main-chart-visibility}
The local-to-clean chart sees the baseline defect: for some \(\lambda_G>0\),
\begin{equation}\label{eq:ass-chart-visibility}
        \Dist_{\cl}(\Theta_\Lambda\calD,\Gamma_{\cl,\adm})
        \ge
        \lambda_G
        \Dist_{\loc,\intg,0}(\calD,\Gamma^{\intg}_{\Lambda,\adm})
        -\Delta_{\mathrm{chart}}.
\end{equation}
\end{assumption}

\begin{assumption}[Component-to-baseline comparison]\label{ass:main-component-baseline}
The sharp component distance is controlled by the older baseline distance and an explicit comparison error:
\begin{equation}\label{eq:ass-component-baseline}
        \Dist_{\comp}^{\sharp,[0,K]}(\calD,\calG_{\comp})
        \le
        C_{\comp/0}
        \Dist_{\loc,\intg,0}(\calD,\Gamma^{\intg}_{\Lambda,\adm})
        +
        \Delta_{\comp/0}.
\end{equation}
\end{assumption}

\begin{assumption}[Dominance of the positive coefficient]\label{ass:main-positive-coefficient}
The clean/chart lower bound dominates the detector loss produced by residual closure and component comparison.  In schematic form,
\begin{equation}\label{eq:ass-positive-coeff}
        c_{\Lambda,0}
        :=
        \mu_\Lambda^{\comp}\lambda_G
        -
        C_{\mathrm{dc}}
        C_{\comp}^{[0,K]}(M_U)
        C_{\comp/0}
        >0.
\end{equation}
If the pressure-tail comparison is inserted into the component-to-baseline comparison, the corresponding pressure-tail constants are included in \(C_{\comp/0}\) and the same positivity condition is used.
\end{assumption}

\begin{remark}[Why the hypotheses are separated]
These hypotheses are not all of the same type.  \Cref{ass:main-ledger-closure} is a bookkeeping theorem proved in the paper.  \Cref{ass:main-detector-comparison} is a detector-intertwining theorem proved under channelwise assumptions.  \Cref{ass:main-pressure-tail-visibility} is a pressure-tail comparison theorem proved under baseline visibility and tail approximation.  In contrast, \Cref{ass:main-clean-gap,ass:main-chart-visibility,ass:main-component-baseline,ass:main-positive-coefficient} are the structural inputs that must be verified for a concrete Navier--Stokes-generated package class.
\end{remark}

\begin{center}
\renewcommand{\arraystretch}{1.22}
\begin{tabular}{|p{0.26\linewidth}|p{0.40\linewidth}|p{0.26\linewidth}|}
\hline
\textbf{Input} & \textbf{Role in the proof} & \textbf{Status in this paper}\\
\hline
Pressure-tail visibility & Returns enhanced pressure coordinates to baseline geometry & Proved under explicit finite-window structural/tail hypotheses; \Cref{thm:main-pressure-tail-module}\\
\hline
Residual-ledger closure & Pays for pressure-source, leakage, reproduction, and gate/slack residuals & Proved here; \Cref{thm:main-ledger-module} and \Cref{comp:thm:componentwise-closure-target}\\
\hline
Detector comparison & Transfers clean detection to localized detection up to ledger error & Proved under detector-intertwining input; \Cref{thm:main-detector-module,thm:detector-comparison}\\
\hline
Clean anti-phantom gap & Gives positive clean detector lower bound & Structural assumption / finite-window criterion; \Cref{ass:main-clean-gap}\\
\hline
NS-generated package realization & Connects local Navier--Stokes data to the package coordinates & Constructed interface; \Cref{thm:package-realizability}\\
\hline
Scale-uniform propagation & Would pass the finite-window bound along arbitrary scale chains & Not proved\\
\hline
Singularity extraction & Would turn the finite-window obstruction into singularity exclusion or global regularity & Future work\\
\hline
\end{tabular}
\end{center}

\section{Pressure-tail visibility module}\label{sec:pressure-module-main}

The first module addresses a mismatch of gauges.  The clean model naturally sees pressure through a pressure source, its Calderon--Zygmund image, and a harmonic correction, following the standard local pressure-decomposition viewpoint in Navier--Stokes regularity theory \cite{SohrWahl1986,SereginSverak2002,SereginLectureNotes}.  The older baseline quotient may not contain all these coordinates.  Without a visibility theorem, a defect could be large in the pressure-natural tail geometry but invisible in the baseline geometry used in the final theorem.

The appendices prove this module in several steps.  Harmonic polynomial approximation controls the harmonic tail.  Clean projection-tail approximation controls the finite projection of the clean pressure image.  Same-gauge closure ensures that the baseline coordinate, projection tail, and harmonic tail are evaluated on one common representative.  Split package estimates treat localized pressure splitting and finite-amplitude terms.  Compact clean-source criteria provide uniform projection-tail convergence when the clean pressure image is compact, while effective projection gives a finite-dimensional alternative.

The output used in the main proof is the following.

\begin{theorem}[Pressure-tail closure in baseline gauge]\label{thm:main-pressure-tail-module}
Assume the finite-window baseline visibility, finite-amplitude, same-gauge, projection-tail, and harmonic-tail hypotheses of Appendix A.  Then the pressure-natural tail distance satisfies
\begin{equation}\label{eq:main-pressure-tail-module}
\begin{aligned}
\Dist^{\sharp,\tail}_{\loc,\intg,\tail}
(\calD,\Gamma^{\intg}_{\Lambda,\adm})
&\le
C_{\tail}
\bigl[
(1+C_{\tail/0})
\Dist_{\loc,\intg,0}
(\calD,\Gamma^{\intg}_{\Lambda,\adm})\\
&\qquad
+\delta_0+
\Delta_{\tail/0}
\bigr]
+
\alpha_{\proj}\Delta^{\mathrm{unif}}_{\proj,N}(\calA_\Lambda)
+
\alpha_{\harm}\Delta^{(3/2)}_{\harm,M}.
\end{aligned}
\end{equation}
In particular, after collecting the last three terms into \(\mathfrak E_{\tail,0}(\calD)\), \Cref{ass:main-pressure-tail-visibility} holds.
\end{theorem}

\begin{proof}
The detailed proof is split between the core material below and the remaining compactness/tax criteria in \Cref{app:pressure-tail-details}.  The logical chain is as follows.  \Cref{thm:L32-harmonic-tail} gives pressure-natural harmonic polynomial tail approximation.  The clean projection estimates give finite projection-tail control on the selected clean pressure image.  The same-gauge comparison and split package estimates show that these tail controls can be imposed on the same admissible representative used by the baseline quotient.  Baseline visibility then compares the pressure-natural tail excess to the older baseline excess.  Finally, compact clean-pressure-image convergence gives the uniform projection-tail error \(\Delta^{\mathrm{unif}}_{\proj,N}(\calA_\Lambda)\).  Combining these estimates gives \eqref{eq:main-pressure-tail-module}; this is exactly the uniform version of \Cref{thm:baseline-closure-uniform-projection-tail}, with the non-uniform form recorded in \Cref{thm:baseline-pressure-tail-closure}.
\end{proof}

\begin{remark}[Problem solved by this module]
The module rules out a purely notational failure: adding pressure-tail coordinates should not create an uncontrolled direction that the baseline distance cannot see.  It does not prove that all possible Navier--Stokes pressure-source families are compact.  It proves that once compactness or effective approximation is available on the selected finite-window class, the pressure tails enter the final theorem only through explicit errors.
\end{remark}

\subsection*{Core proof details for the pressure-tail module}

The following material is kept in the main text because it is the first place where the paper's central geometric issue appears: enhanced pressure-tail coordinates must be controlled by a baseline quotient without changing representatives.  Later compactness criteria and pressure/tax quotient criteria are deferred to \Cref{app:pressure-tail-details}.

\subsection*{Package Geometry and Pressure Tails}
\label{part:pressure-tail}

\subsection{Normalized pressure geometry and package distances}
\label{sec:geometry}

\begin{convention}[Normalized pressure geometry]\label{conv:normalized-geometry}
Throughout the paper,
\[
        I=(-1,0),\qquad Q_1=B_1\times I,
        \qquad B_{1/2}\subset B_{2/3}\subset B_{3/4}\subset B_1\subset\R^3.
\]
When a cutoff is needed, $\eta\in C_c^\infty(B_1)$ is fixed with
\[
        0\le \eta\le 1,\qquad \eta\equiv 1\quad\text{on }B_{3/4}.
\]
Set
\[
        X_{\src}:=L^{3/2}\bigl(I;L^{3/2}(B_1)\bigr)^{3\times3},
        \qquad
        Y_{\prs}:=L^{3/2}\bigl(I;L^{3/2}(B_{1/2})\bigr).
\]
The default harmonic observation space is the pressure-natural space
\[
        Y_{\harm}:=Y^{(3/2)}_{\harm}
        :=L^{3/2}\bigl(I;L^{3/2}(B_{3/4})\bigr).
\]
The auxiliary Hilbert harmonic space is
\[
        Y^{(2)}_{\harm}:=L^2\bigl(I;L^2(B_{3/4})\bigr).
\]
\end{convention}

\begin{definition}[Tail-resolved intrinsic norm]\label{def:tail-resolved-norm}
Let $D$ be a finite-window package carrying a baseline intrinsic norm
$\|D\|_{\loc,\intg,0}$, a clean pressure source $F^{\cl}_{D}\in X_{\src}$ for which
$R_iR_j(F^{\cl}_{D,ij})\in Y_{\prs}$, and a harmonic pressure coordinate
$p_{\harm,D}\in Y_{\harm}$.  Define
\begin{equation}\label{eq:tail-resolved-norm}
\|D\|_{\loc,\intg,\tail}
:=
\|D\|_{\loc,\intg,0}
+
\norm{R_iR_j(F^{\cl}_{D,ij})}{Y_{\prs}}
+
\norm{p_{\harm,D}}{Y_{\harm}}.
\end{equation}
Unless a different norm is explicitly named, $\|\cdot\|_{\loc,\intg}$ means
$\|\cdot\|_{\loc,\intg,\tail}$.
\end{definition}

\begin{definition}[Intrinsic and enhanced-tail distances]\label{def:distances}
Let $\Gammaint$ be a finite-window intrinsic gauge space.  The tail-resolved quotient
distance is
\[
\dist_{\loc,\intg,\tail}(D,\Gammaint)
:=
\inf_{\zeta\in\Gammaint}\|D-\zeta\|_{\loc,\intg,\tail}.
\]
The older baseline distance is
\[
\dist_{\loc,\intg,0}(D,\Gammaint)
:=
\inf_{\zeta\in\Gammaint}\|D-\zeta\|_{\loc,\intg,0}.
\]
Given nonnegative tail functionals $T_{\proj}(D;\zeta)$ and $T_{\harm}(D;\zeta)$ and
weights $\alpha_{\proj},\alpha_{\harm}>0$, define
\begin{equation}\label{eq:enhanced-tail-distance}
\dist^{\sharp,\tail}_{\loc,\intg,\tail}(D,\Gammaint)
:=
\inf_{\zeta\in\Gammaint}
\Bigl(
\|D-\zeta\|_{\loc,\intg,\tail}
+\alpha_{\proj}T_{\proj}(D;\zeta)
+\alpha_{\harm}T_{\harm}(D;\zeta)
\Bigr).
\end{equation}
\end{definition}

\begin{proposition}[Common-representative comparison]\label{prop:common-representative}
Suppose that for a package $D$ there is a representative $\zeta_*(D)\in\Gammaint$ such that
\begin{equation}\label{eq:core-near-min}
\|D-\zeta_*\|_{\loc,\intg,\tail}
\le
\dist_{\loc,\intg,\tail}(D,\Gammaint)+\delta_{\intg},
\end{equation}
and such that the same representative satisfies
\begin{equation}\label{eq:tail-assumed-bounds}
T_{\proj}(D;\zeta_*)
\le
C^{\app}_{\proj}\dist_{\loc,\intg,\tail}(D,\Gammaint)+\Delta_{\proj,N},
\end{equation}
\[
T_{\harm}(D;\zeta_*)
\le
C^{\app}_{\harm}\dist_{\loc,\intg,\tail}(D,\Gammaint)+\Delta_{\harm,M}.
\]
Then
\begin{equation}\label{eq:common-representative-conclusion}
\dist^{\sharp,\tail}_{\loc,\intg,\tail}(D,\Gammaint)
\le
C_{\tail/\intg}\dist_{\loc,\intg,\tail}(D,\Gammaint)+\Delta_{\tail/\intg},
\end{equation}
where
\[
C_{\tail/\intg}:=1+\alpha_{\proj}C^{\app}_{\proj}
+\alpha_{\harm}C^{\app}_{\harm},
\qquad
\Delta_{\tail/\intg}:=\delta_{\intg}
+\alpha_{\proj}\Delta_{\proj,N}+\alpha_{\harm}\Delta_{\harm,M}.
\]
\end{proposition}

\begin{proof}
Use $\zeta_*$ as a competitor in the infimum defining
$\dist^{\sharp,\tail}_{\loc,\intg,\tail}$.  Then substitute
\eqref{eq:core-near-min} and the two tail bounds in \eqref{eq:tail-assumed-bounds} and
collect the coefficients of $\dist_{\loc,\intg,\tail}(D,\Gammaint)$.
\end{proof}

\begin{remark}[Role of the common representative]
The point of Proposition~\ref{prop:common-representative} is not the algebra, which is immediate.
The point is that the intrinsic core, the clean projection tail, and the harmonic tail must
be controlled on one gauge representative.  Controlling the three quantities after three
separate optimizations would not imply \eqref{eq:common-representative-conclusion}.
\end{remark}

\subsection{Harmonic polynomial approximation}
\label{sec:harmonic-approximation}

\subsubsection{The Hilbert harmonic model}

\begin{definition}[Harmonic polynomial spaces]\label{def:harmonic-polynomial-space}
For $R>0$, let $\calH_M(B_R)$ be the space of restrictions to $B_R$ of harmonic
polynomials on $\R^3$ of degree at most $M$.  Let
\[
        \Pi^R_{\harm,M}:L^2(B_R)\to \calH_M(B_R)
\]
be the $L^2(B_R)$-orthogonal projection.  For time-dependent functions,
$\Pi^R_{\harm,M}$ acts on the spatial variable for almost every time.
\end{definition}

\begin{lemma}[Finite-dimensional harmonic gauge space]\label{lem:dimension-harmonic-space}
For every $M\ge0$ and $R>0$, $\calH_M(B_R)$ is finite dimensional.  In three spatial
dimensions,
\[
        \dim \calH_M(B_R)=\sum_{m=0}^M(2m+1)=(M+1)^2.
\]
\end{lemma}

\begin{proof}
Every harmonic polynomial decomposes uniquely into homogeneous harmonic polynomials.
The homogeneous harmonic polynomials of degree $m$ in $\R^3$ have dimension $2m+1$.
Summing over $0\le m\le M$ gives the dimension formula.  Restriction to $B_R$ is
injective on polynomials, so the same dimension holds for $\calH_M(B_R)$.
\end{proof}

\begin{theorem}[Hilbert harmonic polynomial tail estimate]\label{thm:L2-harmonic-tail}
Let $0<r<R$, let $M\ge0$, and let
\[
        h\in L^2\bigl(I;L^2(B_R)\bigr)
\]
be harmonic in $B_R$ for almost every $t\in I$.  Then
\begin{equation}\label{eq:L2-harmonic-tail}
\norm{(I-\Pi^R_{\harm,M})h}{L^2(I;L^2(B_r))}
\le
\left(\frac rR\right)^{M+5/2}
\norm{h}{L^2(I;L^2(B_R))}.
\end{equation}
Consequently, for $r=1/2$ and $R=3/4$,
\[
\norm{(I-\Pi^{3/4}_{\harm,M})h}{L^2(I;L^2(B_{1/2}))}
\le
\left(\frac23\right)^{M+1}
\norm{h}{L^2(I;L^2(B_{3/4}))}.
\]
\end{theorem}

\begin{proof}
It is enough to prove the spatial estimate for almost every fixed time.  Let $h$ be
harmonic in $B_R$.  Its harmonic expansion at the origin is
\[
        h(x)=\sum_{m=0}^\infty H_m(x),
\]
where $H_m$ is a homogeneous harmonic polynomial of degree $m$.  Terms of different
degree are orthogonal in $L^2(B_\rho)$ for every $0<\rho<R$, because their restrictions
to spheres are spherical harmonics of different degrees.  Homogeneity gives
\[
        \norm{H_m}{L^2(B_\rho)}^2
        =
        \left(\frac{\rho}{R}\right)^{2m+3}
        \norm{H_m}{L^2(B_R)}^2.
\]
Since $\Pi^R_{\harm,M}h=\sum_{m=0}^M H_m$, the tail satisfies
\begin{align*}
\norm{(I-\Pi^R_{\harm,M})h}{L^2(B_r)}^2
&=
\sum_{m>M}\norm{H_m}{L^2(B_r)}^2 \\
&=
\sum_{m>M}
\left(\frac rR\right)^{2m+3}
\norm{H_m}{L^2(B_R)}^2 \\
&\le
\left(\frac rR\right)^{2M+5}
\sum_{m>M}\norm{H_m}{L^2(B_R)}^2 \\
&\le
\left(\frac rR\right)^{2M+5}
\norm{h}{L^2(B_R)}^2.
\end{align*}
Taking square roots gives the spatial estimate.  Integrating in time gives
\eqref{eq:L2-harmonic-tail}.  In the normalized geometry $r/R=2/3$, and the stronger
factor $(2/3)^{M+5/2}$ is bounded by $(2/3)^{M+1}$.
\end{proof}

\subsubsection{The pressure-natural harmonic model}

The local pressure class naturally gives $p\in L^{3/2}$, not $p\in L^2$.  Therefore the
pressure-natural harmonic geometry requires an $L^{3/2}$ harmonic tail estimate.

\begin{definition}[Pressure-natural harmonic projection on harmonic data]
Let $0<\rho<R$.  If
$h\in L^{3/2}\bigl(I;L^{3/2}(B_R)\bigr)$ is harmonic in $B_R$ for almost every time,
define $\Pi^\rho_{\harm,M}h(t,\cdot)$ to be the $L^2(B_\rho)$-orthogonal projection of
$h(t,\cdot)$ onto $\calH_M(B_\rho)$, restricted to $B_\rho$.
\end{definition}

\begin{remark}
The operator $\Pi^\rho_{\harm,M}$ is used here only on harmonic data.  We do not claim
that it is a bounded projection on all of $L^{3/2}(B_R)$.  Harmonic interior estimates give
the required $L^2(B_\rho)$ control on the harmonic subspace.
\end{remark}

\begin{lemma}[Interior $L^{3/2}$-to-$L^2$ estimate]\label{lem:L32-L2-interior}
For every $0<\rho<R$, there exists $C_{\rho,R}<\infty$ such that every harmonic function
$h\in L^{3/2}(B_R)$ satisfies
\[
        \norm{h}{L^2(B_\rho)}
        \le
        C_{\rho,R}\norm{h}{L^{3/2}(B_R)}.
\]
\end{lemma}

\begin{proof}
Let $d=(R-\rho)/4$.  For every $x\in B_\rho$, the ball $B_{2d}(x)$ is contained in $B_R$.
Since $|h|^{3/2}$ is subharmonic for harmonic $h$, the mean-value estimate gives
\[
        |h(x)|^{3/2}
        \le
        Cd^{-3}\int_{B_d(x)}|h(y)|^{3/2}\,dy
        \le
        Cd^{-3}\norm{h}{L^{3/2}(B_R)}^{3/2}.
\]
Thus $\norm{h}{L^\infty(B_\rho)}\le C_{\rho,R}\norm{h}{L^{3/2}(B_R)}$.  Multiplying by
$|B_\rho|^{1/2}$ gives the $L^2(B_\rho)$ estimate.
\end{proof}

\begin{theorem}[$L^{3/2}$ harmonic polynomial tail estimate]\label{thm:L32-harmonic-tail}
Let $0<r<\rho<R$.  Let
\[
        h\in L^{3/2}\bigl(I;L^{3/2}(B_R)\bigr)
\]
be harmonic in $B_R$ for almost every $t\in I$.  Then there exists $C_{r,\rho,R}<\infty$
such that
\begin{equation}\label{eq:L32-harmonic-tail}
\norm{(I-\Pi^\rho_{\harm,M})h}{L^{3/2}(I;L^{3/2}(B_r))}
\le
C_{r,\rho,R}
\left(\frac r\rho\right)^{M+5/2}
\norm{h}{L^{3/2}(I;L^{3/2}(B_R))}.
\end{equation}
In particular, for any fixed $\theta$ with $r/\rho<\theta<1$, the right-hand side is bounded by
\[
    C_{r,\rho,R,\theta}\theta^M
    \norm{h}{L^{3/2}(I;L^{3/2}(B_R))}.
\]
\end{theorem}

\begin{proof}
Fix a time $t$ for which $h(t,\cdot)$ is harmonic in $B_R$.  By
Lemma~\ref{lem:L32-L2-interior}, $h(t,\cdot)\in L^2(B_\rho)$.  The Hilbert tail estimate on
$B_r\subset B_\rho$ gives
\[
\norm{(I-\Pi^\rho_{\harm,M})h(t)}{L^2(B_r)}
\le
\left(\frac r\rho\right)^{M+5/2}
\norm{h(t)}{L^2(B_\rho)}.
\]
By finite-measure embedding on $B_r$,
\[
\norm{(I-\Pi^\rho_{\harm,M})h(t)}{L^{3/2}(B_r)}
\le
|B_r|^{1/6}
\norm{(I-\Pi^\rho_{\harm,M})h(t)}{L^2(B_r)}.
\]
Using Lemma~\ref{lem:L32-L2-interior} once more gives
\[
\norm{(I-\Pi^\rho_{\harm,M})h(t)}{L^{3/2}(B_r)}
\le
C_{r,\rho,R}\left(\frac r\rho\right)^{M+5/2}
\norm{h(t)}{L^{3/2}(B_R)}.
\]
Taking the $L^{3/2}$ norm in time proves \eqref{eq:L32-harmonic-tail}.  The $\theta^M$
form follows by absorbing the fixed factor $(r/\rho)^{5/2}$ and the comparison between
$(r/\rho)^M$ and $\theta^M$ into the constant.
\end{proof}

\begin{corollary}[Normalized pressure-natural harmonic tail]\label{cor:normalized-L32-harmonic-tail}
Let
\[
        h\in L^{3/2}\bigl(I;L^{3/2}(B_{3/4})\bigr)
\]
be harmonic in $B_{3/4}$ for almost every time.  Then
\begin{equation}\label{eq:normalized-L32-harmonic-tail}
\norm{(I-\Pi^{2/3}_{\harm,M})h}{L^{3/2}(I;L^{3/2}(B_{1/2}))}
\le
C_{\harm,3/2}\left(\frac34\right)^M
\norm{h}{L^{3/2}(I;L^{3/2}(B_{3/4}))}.
\end{equation}
\end{corollary}

\begin{proof}
Apply Theorem~\ref{thm:L32-harmonic-tail} with $r=1/2$, $\rho=2/3$, and $R=3/4$.  Then
$r/\rho=3/4$, and the fixed factor $(3/4)^{5/2}$ is absorbed into $C_{\harm,3/2}$.
\end{proof}

\begin{definition}[Harmonic tail errors]\label{def:harmonic-tail-errors}
In the pressure-natural harmonic geometry, set
\begin{equation}\label{eq:L32-harmonic-error}
\Delta^{(3/2)}_{\harm,M}(h)
:=
C_{\harm,3/2}\left(\frac34\right)^M
\norm{h}{Y_{\harm}}.
\end{equation}
In the Hilbert harmonic geometry, the corresponding normalized error is
\[
\Delta^{(2)}_{\harm,M}(h)
:=
\left(\frac23\right)^{M+1}
\norm{h}{Y^{(2)}_{\harm}}.
\]
\end{definition}

\subsection{Clean projection-tail approximation}
\label{sec:projection-tail}

\begin{definition}[Clean pressure projection datum]\label{def:clean-pressure-projection}
A clean pressure projection datum is a sequence of finite-rank bounded linear maps
\[
        P^{\cl}_{\prs,N}:Y_{\prs}\to Y_{\prs}.
\]
It is called strongly convergent if $P^{\cl}_{\prs,N}g\to g$ in $Y_{\prs}$ for every
$g\in Y_{\prs}$.  It is uniformly bounded if
\[
        C_P:=\sup_N\norm{P^{\cl}_{\prs,N}}{Y_{\prs}\to Y_{\prs}}<\infty.
\]
\end{definition}

\begin{definition}[Clean projection tail]\label{def:clean-projection-tail}
Let $F^{\cl}\in X_{\src}$ be a clean pressure source for which
\[
        g_F:=R_iR_j(F^{\cl}_{ij})
\]
is defined as an element of $Y_{\prs}$.  The clean projection tail is
\[
        \Delta_{\proj,N}(F^{\cl})
        :=
        \norm{(I-P^{\cl}_{\prs,N})g_F}{Y_{\prs}}.
\]
\end{definition}

\begin{proposition}[Fixed-source projection convergence]\label{prop:fixed-source-projection}
Assume $P^{\cl}_{\prs,N}\to I$ strongly on $Y_{\prs}$.  If $F^{\cl}\in X_{\src}$ and
$R_iR_j(F^{\cl}_{ij})\in Y_{\prs}$, then
\[
        \Delta_{\proj,N}(F^{\cl})\to0
        \qquad\text{as }N\to\infty.
\]
\end{proposition}

\begin{proof}
By definition,
\[
\Delta_{\proj,N}(F^{\cl})
=
\norm{(I-P^{\cl}_{\prs,N})R_iR_j(F^{\cl}_{ij})}{Y_{\prs}}.
\]
Strong convergence of $P^{\cl}_{\prs,N}$ to the identity, applied to the fixed element
$R_iR_j(F^{\cl}_{ij})\in Y_{\prs}$, gives the result.
\end{proof}

\begin{proposition}[Compact-family projection-tail convergence]\label{prop:compact-family-projection}
Assume $P^{\cl}_{\prs,N}\to I$ strongly on $Y_{\prs}$ and
$C_P:=\sup_N\|P^{\cl}_{\prs,N}\|_{Y_{\prs}\to Y_{\prs}}<\infty$.  If
$\calG\subset Y_{\prs}$ is compact, then
\[
        \sup_{g\in\calG}\norm{(I-P^{\cl}_{\prs,N})g}{Y_{\prs}}
        \to0.
\]
Equivalently, if
\[
        \calG=\{R_iR_j(F_{ij}):F\in\calF\}
\]
is compact in $Y_{\prs}$, then
\[
        \sup_{F\in\calF}
        \norm{(I-P^{\cl}_{\prs,N})R_iR_j(F_{ij})}{Y_{\prs}}
        \to0.
\]
\end{proposition}

\begin{proof}
Fix $\varepsilon>0$.  Since $\calG$ is compact, choose finitely many points
$g_1,\ldots,g_J\in\calG$ such that for every $g\in\calG$ there is $j$ with
$\|g-g_j\|_{Y_{\prs}}\le\varepsilon$.  For each fixed center, strong convergence gives
$\|(I-P^{\cl}_{\prs,N})g_j\|_{Y_{\prs}}\to0$.  Since there are finitely many centers, for
$N$ large enough,
\[
        \max_{1\le j\le J}\norm{(I-P^{\cl}_{\prs,N})g_j}{Y_{\prs}}
        \le\varepsilon.
\]
For arbitrary $g\in\calG$, choose $g_j$ as above.  Then
\[
\norm{(I-P^{\cl}_{\prs,N})g}{Y_{\prs}}
\le
\norm{(I-P^{\cl}_{\prs,N})(g-g_j)}{Y_{\prs}}
+
\norm{(I-P^{\cl}_{\prs,N})g_j}{Y_{\prs}}
\le
(1+C_P)\varepsilon+\varepsilon.
\]
Taking the supremum over $g\in\calG$ and then letting $\varepsilon\downarrow0$ proves the
claim.
\end{proof}

\begin{remark}
Compactness is essential in Proposition~\ref{prop:compact-family-projection}.  Strong convergence of
finite-rank projections does not imply uniform convergence on arbitrary bounded subsets of
an infinite-dimensional Banach space.
\end{remark}

\subsection{Same-gauge pressure-tail closure}
\label{sec:same-gauge-closure}

\begin{definition}[Tail-compatible finite-window class]\label{def:tail-compatible-class}
Let $X_\Lambda$ be a finite-dimensional normed package space with norm
$\|\cdot\|_{\loc,\intg,\tail}$, and let $\Gammaint\subset X_\Lambda$ be a gauge subspace.
A class $\calC_\Lambda\subset X_\Lambda$ is called tail-compatible, with constants
$C^{\tail}_{\proj},C^{\tail}_{\harm}<\infty$ and errors
$\Delta_{\proj,N},\Delta_{\harm,M}\ge0$, if for every $D\in\calC_\Lambda$ and every
$\zeta\in\Gammaint$,
\[
T_{\proj}(D;\zeta)
\le
C^{\tail}_{\proj}\|D-\zeta\|_{\loc,\intg,\tail}+\Delta_{\proj,N},
\]
\[
T_{\harm}(D;\zeta)
\le
C^{\tail}_{\harm}\|D-\zeta\|_{\loc,\intg,\tail}+\Delta_{\harm,M}.
\]
\end{definition}

\begin{lemma}[Existence of a common best representative]\label{lem:best-representative}
Assume $X_\Lambda$ is finite dimensional and $\Gammaint\subset X_\Lambda$ is a closed
subspace.  Then every $D\in X_\Lambda$ admits $\zeta_*(D)\in\Gammaint$ such that
\[
        \|D-\zeta_*\|_{\loc,\intg,\tail}
        =
        \dist_{\loc,\intg,\tail}(D,\Gammaint).
\]
\end{lemma}

\begin{proof}
Choose a minimizing sequence $(\zeta_n)\subset\Gammaint$.  Since
$\|D-\zeta_n\|_{\loc,\intg,\tail}$ is bounded along the sequence, the triangle inequality
implies that $(\zeta_n)$ is bounded in the finite-dimensional normed space $X_\Lambda$.
Because $\Gammaint$ is closed and finite dimensional, a subsequence converges to some
$\zeta_*\in\Gammaint$.  Continuity of $\zeta\mapsto\|D-\zeta\|_{\loc,\intg,\tail}$ gives the
minimum.
\end{proof}

\begin{proposition}[Tail-compatible bounds give same-gauge compatibility]
\label{prop:tail-compatible-same-gauge}
Let $\calC_\Lambda$ be a tail-compatible finite-window class.  Then every
$D\in\calC_\Lambda$ satisfies the hypotheses of Proposition~\ref{prop:common-representative} with
$\delta_{\intg}=0$,
\[
        C^{\app}_{\proj}=C^{\tail}_{\proj},
        \qquad
        C^{\app}_{\harm}=C^{\tail}_{\harm}.
\]
Consequently,
\begin{align}\label{eq:tail-compatible-comparison}
\dist^{\sharp,\tail}_{\loc,\intg,\tail}(D,\Gammaint)
&\le
\bigl(1+\alpha_{\proj}C^{\tail}_{\proj}
+\alpha_{\harm}C^{\tail}_{\harm}\bigr)
\dist_{\loc,\intg,\tail}(D,\Gammaint)\nonumber\\
&\quad+
\alpha_{\proj}\Delta_{\proj,N}
+
\alpha_{\harm}\Delta_{\harm,M}.
\end{align}
\end{proposition}

\begin{proof}
Choose the best representative $\zeta_*(D)$ from Lemma~\ref{lem:best-representative}.  The core
near-minimizer bound holds with $\delta_{\intg}=0$.  Applying the two tail-compatible
bounds in Definition~\ref{def:tail-compatible-class} to this same representative gives the projection
and harmonic tail estimates required by Proposition~\ref{prop:common-representative}.  Substitution
into \eqref{eq:common-representative-conclusion} gives \eqref{eq:tail-compatible-comparison}.
\end{proof}

\begin{definition}[Coordinate tail model]\label{def:coordinate-tail-model}
A coordinate tail model on $X_\Lambda$ consists of normed spaces $Z_{\proj}$ and
$Z_{\harm}$ and bounded linear maps
\[
        A_{\proj}:X_\Lambda\to Z_{\proj},
        \qquad
        A_{\harm}:X_\Lambda\to Z_{\harm}.
\]
The corresponding model tail functionals are
\[
        T_{\proj}(D;\zeta):=\norm{A_{\proj}(D-\zeta)}{Z_{\proj}}+\Delta_{\proj,N},
\]
\[
        T_{\harm}(D;\zeta):=\norm{A_{\harm}(D-\zeta)}{Z_{\harm}}+\Delta_{\harm,M}.
\]
\end{definition}

\begin{lemma}[Coordinate-map criterion]\label{lem:coordinate-map-criterion}
Every coordinate tail model defines a tail-compatible finite-window class with
\[
        C^{\tail}_{\proj}=\norm{A_{\proj}}{X_\Lambda\to Z_{\proj}},
        \qquad
        C^{\tail}_{\harm}=\norm{A_{\harm}}{X_\Lambda\to Z_{\harm}}.
\]
\end{lemma}

\begin{proof}
For every $D\in X_\Lambda$ and $\zeta\in\Gammaint$,
\[
\norm{A_{\proj}(D-\zeta)}{Z_{\proj}}
\le
\norm{A_{\proj}}{X_\Lambda\to Z_{\proj}}
\|D-\zeta\|_{\loc,\intg,\tail}.
\]
Adding $\Delta_{\proj,N}$ gives the projection-tail bound.  The harmonic estimate is
identical with $A_{\harm}$ in place of $A_{\proj}$.
\end{proof}

\begin{theorem}[Abstract finite-window pressure-tail closure]\label{thm:abstract-tail-closure}
Let $\calC_\Lambda$ be a tail-compatible finite-window class in the tail-resolved geometry.
Then every $D\in\calC_\Lambda$ satisfies
\begin{equation}\label{eq:abstract-tail-closure}
\dist^{\sharp,\tail}_{\loc,\intg,\tail}(D,\Gammaint)
\le
C_{\tail}
\dist_{\loc,\intg,\tail}(D,\Gammaint)
+
\alpha_{\proj}\Delta_{\proj,N}
+
\alpha_{\harm}\Delta_{\harm,M},
\end{equation}
where
\[
        C_{\tail}:=1+\alpha_{\proj}C^{\tail}_{\proj}
        +\alpha_{\harm}C^{\tail}_{\harm}.
\]
Moreover, the harmonic error may be chosen as \eqref{eq:L32-harmonic-error} in the
pressure-natural harmonic geometry whenever the shifted harmonic coordinate is harmonic,
and as $\Delta^{(2)}_{\harm,M}$ in the Hilbert harmonic geometry whenever the shifted
harmonic coordinate belongs to $Y^{(2)}_{\harm}$.  The clean projection error converges by
Proposition~\ref{prop:fixed-source-projection} for fixed clean sources and by
Proposition~\ref{prop:compact-family-projection} for compact pressure-source images with uniformly
bounded projections.
\end{theorem}

\begin{proof}
The estimate \eqref{eq:abstract-tail-closure} is exactly
Proposition~\ref{prop:tail-compatible-same-gauge}.  The final assertions are substitutions of the
harmonic approximation estimates from Section~\ref{sec:harmonic-approximation} and the projection
convergence results from Section~\ref{sec:projection-tail} into the errors
$\Delta_{\harm,M}$ and $\Delta_{\proj,N}$.
\end{proof}

\subsection{Localized pressure splitting and split packages}
\label{sec:localized-split-packages}

This section constructs a concrete fixed-scale package model from local pressure data.  The
construction is local and finite-window.  It does not assert scale-uniformity, pressure/tax
coercivity, or canonicality of the gauge class.

\begin{definition}[Pressure-admissible local data]\label{def:pressure-admissible-data}
A local pair $(u,p)$ on $Q_1$ is called pressure-admissible if
\[
        u\in L^3(Q_1)^3,
        \qquad
        p\in L^{3/2}(Q_1),
\]
and, for almost every $t\in I$,
\[
        -\Delta p(t,\cdot)=\partial_i\partial_j(u_i u_j)(t,\cdot)
\]
in the sense of distributions on $B_1$.
\end{definition}

\begin{convention}[Pressure normalization]\label{conv:pressure-normalization}
Since the pressure equation determines $p$ only up to a time-dependent constant, we choose
the representative satisfying
\[
        \int_{B_{3/4}}p(t,x)\,dx=0
        \qquad\text{for almost every }t\in I,
\]
after subtracting a suitable function of time if necessary.  This is only a finite-window
normalization convention.
\end{convention}

\begin{definition}[Localized active and harmonic pressure]\label{def:active-harmonic-pressure}
For pressure-admissible data $(u,p)$, set
\[
        F^{\act}_{ij}:=\eta u_i u_j.
\]
Extend $F^{\act}$ by zero outside $B_1$ and define
\[
        p_{\act}:=R_iR_j(F^{\act}_{ij}).
\]
On $B_{3/4}\times I$, define the harmonic pressure remainder
\[
        p_{\harm}:=p-p_{\act}.
\]
\end{definition}

\begin{lemma}[Localized active pressure bound]\label{lem:active-pressure-bound}
For every pressure-admissible local pair $(u,p)$,
\[
        \norm{p_{\act}}{Y_{\prs}}
        \le
        C_{\CZ}\norm{u}{L^3(Q_1)}^2.
\]
\end{lemma}

\begin{proof}
For almost every $t$, the Calderon--Zygmund inequality for second Riesz transforms on
$\R^3$, applied after zero extension from $B_1$, gives
\[
\norm{R_iR_j(F^{\act}_{ij})(t)}{L^{3/2}(B_{1/2})}
\le
C_{\CZ}\norm{F^{\act}(t)}{L^{3/2}(B_1)^{3\times3}}.
\]
Taking the $L^{3/2}$ norm in time and using $0\le\eta\le1$ yields
\[
\norm{p_{\act}}{Y_{\prs}}
\le
C_{\CZ}\norm{\eta u_i u_j}{L^{3/2}(I;L^{3/2}(B_1))^{3\times3}}
\le
C_{\CZ}\norm{u}{L^3(Q_1)}^2,
\]
where harmless finite-component constants are absorbed into $C_{\CZ}$.
\end{proof}

\begin{lemma}[Harmonicity of the pressure remainder]\label{lem:harmonic-remainder}
If $(u,p)$ is pressure-admissible, then for almost every $t\in I$,
\[
        -\Delta p_{\harm}(t,\cdot)=0
\]
in the sense of distributions on $B_{3/4}$.
\end{lemma}

\begin{proof}
The global distributional identity for the localized Riesz potential gives
\[
        -\Delta p_{\act}=\partial_i\partial_j(\eta u_i u_j)
\]
on $\R^3$.  Let $\varphi\in C_c^\infty(B_{3/4})$.  Since $\eta\equiv1$ on a neighborhood
of $\supp\varphi$,
\[
\langle \partial_i\partial_j(\eta u_i u_j),\varphi\rangle
=
\langle \partial_i\partial_j(u_i u_j),\varphi\rangle.
\]
Thus $p_{\act}$ satisfies the same pressure Poisson equation as $p$ on $B_{3/4}$, and their
difference is harmonic there.
\end{proof}

\begin{proposition}[Pressure-natural admissibility]\label{prop:pressure-natural-admissibility}
Let $(u,p)$ be pressure-admissible on $Q_1$.  Then
\[
        p_{\harm}=p-p_{\act}\in Y_{\harm}=L^{3/2}(I;L^{3/2}(B_{3/4})).
\]
Consequently,
\begin{equation}\label{eq:local-data-harmonic-tail}
\norm{(I-\Pi^{2/3}_{\harm,M})p_{\harm}}{L^{3/2}(I;L^{3/2}(B_{1/2}))}
\le
C_{\harm,3/2}\left(\frac34\right)^M
\norm{p_{\harm}}{Y_{\harm}}.
\end{equation}
\end{proposition}

\begin{proof}
By Lemma~\ref{lem:active-pressure-bound}, $p_{\act}\in L^{3/2}(I;L^{3/2}(B_{3/4}))$.  Since
$p\in L^{3/2}(Q_1)$, its restriction to $B_{3/4}\times I$ lies in the same space.  Hence
$p_{\harm}=p-p_{\act}\in Y_{\harm}$.  By Lemma~\ref{lem:harmonic-remainder}, $p_{\harm}$ is
harmonic in $B_{3/4}$ for almost every time.  Therefore Corollary~\ref{cor:normalized-L32-harmonic-tail}
applies and gives \eqref{eq:local-data-harmonic-tail}.
\end{proof}

\begin{definition}[Residual localized split package]\label{def:residual-split-package}
A residual localized split package is a finite-window object
\[
        D=(u_D,p_{\act,D},p_{\harm,D},F^{\cl}_D,E^{\cl}_{F,D},\Pi_D,\Phi_D,T_D,s_D)
\]
with
\[
        u_D\in L^3(Q_1)^3,
        \qquad
        E^{\cl}_{F,D}\in X_{\src},
        \qquad
        p_{\harm,D}\in Y_{\harm},
\]
such that $p_{\harm,D}(t,\cdot)$ is harmonic in $B_{3/4}$ for almost every time and
\begin{equation}\label{eq:clean-source-split}
        F^{\cl}_{D,ij}=\eta u_{D,i}u_{D,j}+E^{\cl}_{F,D,ij}.
\end{equation}
The coordinates $\Pi_D,\Phi_D,T_D,s_D$ denote auxiliary finite-window flux, energy or
trace, selected-trace, and slack data.  They are retained as package coordinates but are not
used in the pressure-tail estimates below.
\end{definition}

\begin{definition}[Canonical localized split package]\label{def:canonical-split-package}
If $(u,p)$ is pressure-admissible, its canonical localized split package is the residual
split package $D_{\NS}(u,p)$ with
\[
        u_D=u,
        \qquad
        E^{\cl}_{F,D}=0,
        \qquad
        F^{\cl}_{D,ij}=\eta u_i u_j,
        \qquad
        p_{\harm,D}=p-p_{\act}.
\]
\end{definition}

\begin{proposition}[Pressure-splitting control of tail coordinates]\label{prop:pressure-splitting-control}
For every residual localized split package $D$,
\begin{equation}\label{eq:pressure-splitting-control}
\norm{R_iR_j(F^{\cl}_{D,ij})}{Y_{\prs}}
\le
C_{\CZ}
\Bigl(
\norm{u_D}{L^3(Q_1)}^2+
\norm{E^{\cl}_{F,D}}{X_{\src}}
\Bigr).
\end{equation}
Consequently,
\[
\|D\|_{\loc,\intg,\tail}
\le
\|D\|_{\loc,\intg,0}
+C_{\CZ}
\Bigl(
\norm{u_D}{L^3(Q_1)}^2+
\norm{E^{\cl}_{F,D}}{X_{\src}}
\Bigr)
+
\norm{p_{\harm,D}}{Y_{\harm}}.
\]
\end{proposition}

\begin{proof}
Extend $F^{\cl}_D$ by zero outside $B_1$.  The Calderon--Zygmund estimate for second
Riesz transforms gives
\[
\norm{R_iR_j(F^{\cl}_{D,ij})}{Y_{\prs}}
\le
C_{\CZ}\norm{F^{\cl}_D}{X_{\src}}.
\]
Using \eqref{eq:clean-source-split} and $0\le\eta\le1$,
\[
\norm{F^{\cl}_D}{X_{\src}}
\le
\norm{\eta u_{D,i}u_{D,j}}{X_{\src}}+
\norm{E^{\cl}_{F,D}}{X_{\src}}
\le
\norm{u_D}{L^3(Q_1)}^2+
\norm{E^{\cl}_{F,D}}{X_{\src}},
\]
after absorbing finite-component constants into $C_{\CZ}$.  The second estimate follows
from the definition of the tail-resolved norm.
\end{proof}

\subsection{Conservative gauge and finite-amplitude split control}
\label{sec:split-control}

\begin{definition}[Conservative admissible split gauge]\label{def:conservative-gauge}
Fix finite-dimensional subspaces
\[
        G_E\subset X_{\src},
        \qquad
        G_h\subset Y_{\harm},
\]
where elements of $G_h$ are harmonic in $B_{3/4}$ for almost every time, and a finite-dimensional
auxiliary gauge space $G_{\aux}$ for the remaining package coordinates.  Define
\[
        \Gammaadm:=\{0\}\times G_E\times G_h\times G_{\aux}.
\]
For $\zeta=(\zeta_u,\zeta_E,\zeta_h,\ldots)\in\Gammaadm$, the conservative convention is
\[
        \zeta_u=0.
\]
Thus
\[
        u_{D-\zeta}:=u_D,
        \qquad
        E^{\cl}_{F,D-\zeta}:=E^{\cl}_{F,D}-\zeta_E,
        \qquad
        p_{\harm,D-\zeta}:=p_{\harm,D}-\zeta_h.
\]
The shifted clean source is
\[
        F^{\cl}_{D-\zeta,ij}:=\eta u_{D,i}u_{D,j}+E^{\cl}_{F,D-\zeta,ij}.
\]
\end{definition}

\begin{definition}[Split-support norm and distance]\label{def:split-support-norm}
For a localized split package, define
\begin{equation}\label{eq:split-support-norm}
\|D\|_{\loc,\intg,\spl}
:=
\|D\|_{\loc,\intg,\tail}
+
\norm{u_D}{L^3(Q_1)}
+
\norm{E^{\cl}_{F,D}}{X_{\src}}
+
\norm{p_{\harm,D}}{Y_{\harm}}.
\end{equation}
The associated quotient distance over the conservative admissible gauge class is
\[
        \dist_{\loc,\intg,\spl}(D,\Gammaadm)
        :=
        \inf_{\zeta\in\Gammaadm}
        \|D-\zeta\|_{\loc,\intg,\spl}.
\]
\end{definition}

\begin{lemma}[Coordinate domination]\label{lem:coordinate-domination}
For every admissible gauge $\zeta\in\Gammaadm$,
\[
\norm{u_{D-\zeta}}{L^3(Q_1)}
+
\norm{E^{\cl}_{F,D-\zeta}}{X_{\src}}
+
\norm{p_{\harm,D-\zeta}}{Y_{\harm}}
\le
\|D-\zeta\|_{\loc,\intg,\spl}.
\]
\end{lemma}

\begin{proof}
This is immediate from the definition \eqref{eq:split-support-norm} applied to the shifted
package $D-\zeta$, after discarding the nonnegative tail-resolved term.
\end{proof}

\begin{lemma}[Existence of an admissible quotient minimizer]\label{lem:admissible-minimizer}
Assume $\Gammaadm$ is a finite-dimensional closed subspace of the split package space
equipped with $\|\cdot\|_{\loc,\intg,\spl}$.  Then every localized split package $D$ admits
$\zeta_*(D)\in\Gammaadm$ such that
\[
        \|D-\zeta_*\|_{\loc,\intg,\spl}
        =
        \dist_{\loc,\intg,\spl}(D,\Gammaadm).
\]
\end{lemma}

\begin{proof}
The proof is the same as Lemma~\ref{lem:best-representative}, with
$\|\cdot\|_{\loc,\intg,\spl}$ and $\Gammaadm$ in place of the tail norm and $\Gammaint$.
A minimizing sequence is bounded by the triangle inequality, finite-dimensional compactness
gives a convergent subsequence in $\Gammaadm$, and continuity gives the minimum.
\end{proof}

\begin{definition}[Pressure-splitting functional]\label{def:Psplit}
For a shifted package $D-\zeta$, define
\begin{equation}\label{eq:Psplit}
P_{\spl}(D;\zeta)
:=
C_{\CZ}
\Bigl(
\norm{u_{D-\zeta}}{L^3(Q_1)}^2
+
\norm{E^{\cl}_{F,D-\zeta}}{X_{\src}}
\Bigr)
+
\norm{p_{\harm,D-\zeta}}{Y_{\harm}}.
\end{equation}
\end{definition}

\begin{definition}[Finite-amplitude localized split class]\label{def:finite-amplitude-class}
For $M_U<\infty$ and $\delta_{\spl}\ge0$, let
$\calA^{\NS,\spl}_{\Lambda}(M_U,\delta_{\spl})$ be the class of residual localized split
packages $D$ equipped with the conservative gauge class $\Gammaadm$ for which there exists
a same-gauge representative $\zeta_*(D)\in\Gammaadm$ such that
\[
\|D-\zeta_*\|_{\loc,\intg,\spl}
\le
\dist_{\loc,\intg,\spl}(D,\Gammaadm)+\delta_{\spl},
\]
and
\[
        \norm{u_{D-\zeta_*}}{L^3(Q_1)}\le M_U.
\]
Under the conservative convention $\zeta_u=0$, this finite-amplitude condition is simply
\[
    \|u_D\|_{L^3(Q_1)}\le M_U.
\]
\end{definition}

\begin{theorem}[Finite-amplitude split-control criterion]\label{thm:finite-amplitude-split-control}
If $D\in\calA^{\NS,\spl}_{\Lambda}(M_U,\delta_{\spl})$, then the representative
$\zeta_*(D)$ in Definition~\ref{def:finite-amplitude-class} satisfies
\begin{equation}\label{eq:finite-amplitude-split-control}
P_{\spl}(D;\zeta_*)
\le
C_{\mathrm{FA}}(M_U)\dist_{\loc,\intg,\spl}(D,\Gammaadm)
+C_{\mathrm{FA}}(M_U)\delta_{\spl},
\end{equation}
where one may take
\[
        C_{\mathrm{FA}}(M_U):=\max\{C_{\CZ}M_U,C_{\CZ},1\}.
\]
\end{theorem}

\begin{proof}
The finite-amplitude bound gives
\[
        \norm{u_{D-\zeta_*}}{L^3(Q_1)}^2
        \le
        M_U\norm{u_{D-\zeta_*}}{L^3(Q_1)}.
\]
Therefore, by the definition of $P_{\spl}$,
\begin{align*}
P_{\spl}(D;\zeta_*)
&\le
C_{\CZ}M_U\norm{u_{D-\zeta_*}}{L^3(Q_1)}
+C_{\CZ}\norm{E^{\cl}_{F,D-\zeta_*}}{X_{\src}}
+
\norm{p_{\harm,D-\zeta_*}}{Y_{\harm}} \\
&\le
C_{\mathrm{FA}}(M_U)
\Bigl(
\norm{u_{D-\zeta_*}}{L^3(Q_1)}
+
\norm{E^{\cl}_{F,D-\zeta_*}}{X_{\src}}
+
\norm{p_{\harm,D-\zeta_*}}{Y_{\harm}}
\Bigr).
\end{align*}
By Lemma~\ref{lem:coordinate-domination}, the parenthesized term is bounded by
$\|D-\zeta_*\|_{\loc,\intg,\spl}$.  The near-minimizer property of $\zeta_*$ then gives
\eqref{eq:finite-amplitude-split-control}.
\end{proof}

\begin{corollary}[Tail-resolved comparison from split control]\label{cor:split-to-tail}
Let $D\in\calA^{\NS,\spl}_{\Lambda}(M_U,\delta_{\spl})$, and assume that the same
representative $\zeta_*$ satisfies the pressure-splitting quotient bound
\begin{equation}\label{eq:pressure-splitting-quotient-bound}
\dist_{\loc,\intg,\tail}(D,\Gammaadm)
\le
\|D-\zeta_*\|_{\loc,\intg,\spl}+P_{\spl}(D;\zeta_*).
\end{equation}
Then
\begin{equation}\label{eq:split-to-tail}
\dist_{\loc,\intg,\tail}(D,\Gammaadm)
\le
\bigl(1+C_{\mathrm{FA}}(M_U)\bigr)
\dist_{\loc,\intg,\spl}(D,\Gammaadm)
+
\bigl(1+C_{\mathrm{FA}}(M_U)\bigr)\delta_{\spl}.
\end{equation}
\end{corollary}

\begin{proof}
Use \eqref{eq:pressure-splitting-quotient-bound}, the near-minimizer bound
\[
        \|D-\zeta_*\|_{\loc,\intg,\spl}
        \le
        \dist_{\loc,\intg,\spl}(D,\Gammaadm)+\delta_{\spl},
\]
and the split-control estimate \eqref{eq:finite-amplitude-split-control}.  Adding the two
terms gives \eqref{eq:split-to-tail}.
\end{proof}

\begin{proposition}[The pressure-splitting quotient bound]\label{prop:pressure-splitting-quotient-bound}
For residual localized split packages with the conservative gauge convention,
\eqref{eq:pressure-splitting-quotient-bound} holds.
\end{proposition}

\begin{proof}
Use $\zeta_*$ as a competitor in $\dist_{\loc,\intg,\tail}(D,\Gammaadm)$.  By the definition
of the tail-resolved norm and by Proposition~\ref{prop:pressure-splitting-control} applied to the
shifted package $D-\zeta_*$,
\begin{align*}
\dist_{\loc,\intg,\tail}(D,\Gammaadm)
&\le
\|D-\zeta_*\|_{\loc,\intg,\tail} \\
&\le
\|D-\zeta_*\|_{\loc,\intg,0}
+P_{\spl}(D;\zeta_*).
\end{align*}
Since $\|D-\zeta_*\|_{\loc,\intg,0}\le\|D-\zeta_*\|_{\loc,\intg,\spl}$, the claimed bound
follows.
\end{proof}

\subsection{Split and abstract baseline pressure-tail closure}
\label{sec:main-theorems}

\begin{theorem}[Main split-geometry pressure-tail closure]\label{thm:main-split-closure}
Let $D\in\calA^{\NS,\spl}_{\Lambda}(M_U,\delta_{\spl})$ be a residual localized split
package in the conservative gauge geometry.  Suppose that the package class is tail-compatible
in the sense of Definition~\ref{def:tail-compatible-class}, with constants
$C^{\tail}_{\proj}$ and $C^{\tail}_{\harm}$.  Then
\begin{align}\label{eq:main-split-closure}
\dist^{\sharp,\tail}_{\loc,\intg,\tail}(D,\Gammaadm)
&\le
C_{\tail}\bigl(1+C_{\mathrm{FA}}(M_U)\bigr)
\dist_{\loc,\intg,\spl}(D,\Gammaadm)
\nonumber\\
&\quad+
C_{\tail}\bigl(1+C_{\mathrm{FA}}(M_U)\bigr)\delta_{\spl}
+
\alpha_{\proj}\Delta_{\proj,N}
+
\alpha_{\harm}\Delta_{\harm,M},
\end{align}
where
\[
        C_{\tail}:=1+
        \alpha_{\proj}C^{\tail}_{\proj}
        +\alpha_{\harm}C^{\tail}_{\harm}.
\]
In the pressure-natural harmonic geometry, one may take
\[
        \Delta_{\harm,M}=C_{\harm,3/2}\left(\frac34\right)^M
        \norm{p_{\harm,D-\zeta_*}}{Y_{\harm}}
\]
whenever the shifted harmonic coordinate is the harmonic datum being approximated.  The
projection error $\Delta_{\proj,N}$ converges to zero for fixed clean sources and uniformly
on compact clean pressure-source images under the hypotheses of
Propositions~\ref{prop:fixed-source-projection} and~\ref{prop:compact-family-projection}.
\end{theorem}

\begin{proof}
Apply the abstract tail-resolved closure theorem Theorem~\ref{thm:abstract-tail-closure} with
$\Gammaadm$ in place of $\Gammaint$:
\[
\dist^{\sharp,\tail}_{\loc,\intg,\tail}(D,\Gammaadm)
\le
C_{\tail}\dist_{\loc,\intg,\tail}(D,\Gammaadm)
+
\alpha_{\proj}\Delta_{\proj,N}
+
\alpha_{\harm}\Delta_{\harm,M}.
\]
Then substitute the split-to-tail comparison \eqref{eq:split-to-tail}.  The final statements
about the two errors follow from Corollary~\ref{cor:normalized-L32-harmonic-tail} and
Propositions~\ref{prop:fixed-source-projection} and~\ref{prop:compact-family-projection}.
\end{proof}

\begin{definition}[Baseline quotient distance]\label{def:baseline-distance}
The older baseline quotient distance over the conservative admissible gauge class is
\[
        \dist_{\loc,\intg,0}(D,\Gammaadm)
        :=
        \inf_{\zeta\in\Gammaadm}
        \|D-\zeta\|_{\loc,\intg,0}.
\]
\end{definition}

\begin{definition}[Split excess over the baseline]\label{def:split-excess}
For an admissible shifted package $D-\zeta$, define
\begin{align*}
E_{\spl/0}(D;\zeta)
&:=
\|D-\zeta\|_{\loc,\intg,\spl}-\|D-\zeta\|_{\loc,\intg,0} \\
&=
\norm{R_iR_j(F^{\cl}_{D-\zeta,ij})}{Y_{\prs}}
+
\norm{u_{D-\zeta}}{L^3(Q_1)}
+
\norm{E^{\cl}_{F,D-\zeta}}{X_{\src}}
+2\norm{p_{\harm,D-\zeta}}{Y_{\harm}}.
\end{align*}
The harmonic coordinate appears twice only because the split norm was defined as the
tail-resolved norm plus an additional split-support harmonic coordinate.
\end{definition}

\begin{assumption}[Same-gauge baseline-to-split excess control]\label{ass:baseline-to-split}
There are constants $C_{\spl/0}<\infty$, $\Delta_{\spl/0}\ge0$, and a representative
$\zeta_0(D)\in\Gammaadm$ such that
\[
\|D-\zeta_0\|_{\loc,\intg,0}
\le
\dist_{\loc,\intg,0}(D,\Gammaadm)+\delta_0,
\]
and, on the same representative,
\[
E_{\spl/0}(D;\zeta_0)
\le
C_{\spl/0}\dist_{\loc,\intg,0}(D,\Gammaadm)+\Delta_{\spl/0}.
\]
\end{assumption}

\begin{theorem}[Conditional baseline form of the split closure]\label{thm:baseline-split-closure}
Assume Assumption~\ref{ass:baseline-to-split} and the hypotheses of Theorem~\ref{thm:main-split-closure}.
Then
\begin{align}\label{eq:baseline-split-closure}
&\dist^{\sharp,\tail}_{\loc,\intg,\tail}(D,\Gammaadm)\nonumber\\
&\le
C_{\tail}\bigl(1+C_{\mathrm{FA}}(M_U)\bigr)
\Bigl[
(1+C_{\spl/0})\dist_{\loc,\intg,0}(D,\Gammaadm)
+\delta_0+\Delta_{\spl/0}
\Bigr]
\nonumber\\
&\quad+
C_{\tail}\bigl(1+C_{\mathrm{FA}}(M_U)\bigr)\delta_{\spl}\nonumber\\
&\quad+
\alpha_{\proj}\Delta_{\proj,N}
+
\alpha_{\harm}\Delta_{\harm,M}.
\end{align}
\end{theorem}

\begin{proof}
Use $\zeta_0$ as a competitor in the split-support distance:
\begin{align*}
\dist_{\loc,\intg,\spl}(D,\Gammaadm)
&\le
\|D-\zeta_0\|_{\loc,\intg,\spl} \\
&=
\|D-\zeta_0\|_{\loc,\intg,0}+E_{\spl/0}(D;\zeta_0) \\
&\le
(1+C_{\spl/0})\dist_{\loc,\intg,0}(D,\Gammaadm)+\delta_0+\Delta_{\spl/0}.
\end{align*}
Substitute this bound into \eqref{eq:main-split-closure}.
\end{proof}

\subsection{Baseline visibility and excess control}
\label{sec:baseline-visibility-excess}

The pressure-tail closure theorem above works first in the tail-resolved and
split-support geometries.  We now identify a sufficient finite-window mechanism
for returning to the older baseline geometry.  The mechanism is deliberately
same-gauge: the representative that nearly minimizes the baseline distance must
also be the representative on which the visible pressure-tail and split-support
coordinates are controlled.

\begin{convention}[Fixed source-to-pressure bound]
\label{conv:fixed-cz-bound}
For every shifted residual localized split package \(D-\zeta\), the fixed-window
pressure model uses the Calderon--Zygmund estimate
\[
    \norm{R_iR_j(F^{\cl}_{D-\zeta,ij})}{Y_{\prs}}
    \le
    C_{\CZ}
    \left(
        \norm{u_{D-\zeta}}{L^3(Q_1)}^2
        +
        \norm{E^{\cl}_{F,D-\zeta}}{X_{\src}}
    \right).
\]
This is the shifted form of Proposition~\ref{prop:pressure-splitting-control}
inside the localized split package model.  It is a fixed-window estimate; no
scale-uniform bound is claimed.
\end{convention}

\subsubsection{Baseline and tail excess functionals}
\label{sec:excess-functionals}

\subsubsection{Older baseline quotient distance}

\begin{definition}[Older baseline distance]
The older baseline intrinsic norm is denoted by
\[
    \|\cdot\|_{\loc,\intg,0}.
\]
The corresponding admissible quotient distance is
\[
    \Dist_{\loc,\intg,0}(D,\Gammaadm)
    :=
    \inf_{\zeta\in\Gammaadm}
    \|D-\zeta\|_{\loc,\intg,0}.
\]
\end{definition}

\begin{remark}[Baseline limitation]
The baseline norm does not automatically contain the active pressure source,
the harmonic pressure coordinate, the velocity amplitude, or the
pressure-splitting error.  Any comparison with the tail-resolved or
split-support geometry therefore requires explicit visibility assumptions.
\end{remark}

\subsubsection{Pressure-natural tail excess}

\begin{definition}[Tail excess over the baseline norm]
For a shifted package \(D-\zeta\), define
\[
    \calE_{\tail/0}^{(3/2)}(D;\zeta)
    :=
    \|R_iR_j(F^{\cl}_{D-\zeta,ij})\|_{Y_{\prs}}
    +
    \|p_{\harm,D-\zeta}\|_{Y_{\harm}}.
\]
\end{definition}

\subsubsection{Split excess}

\begin{definition}[Split excess over the baseline norm]
The split-support excess is
\[
    \calE_{\spl/0}^{(3/2)}(D;\zeta)
    :=
    \|D-\zeta\|_{\loc,\intg,\spl}
    -
    \|D-\zeta\|_{\loc,\intg,0}.
\]
With the current split-support convention, this expands as
\[
\begin{aligned}
    \calE_{\spl/0}^{(3/2)}(D;\zeta)
    &=
    \|R_iR_j(F^{\cl}_{D-\zeta,ij})\|_{Y_{\prs}}
    +
    \|u_{D-\zeta}\|_{L^3(Q_1)}
\\
    &\quad
    +
    \|E^{\cl}_{F,D-\zeta}\|_{X_{\src}}
    +
    2\|p_{\harm,D-\zeta}\|_{Y_{\harm}}.
\end{aligned}
\]
\end{definition}

\begin{remark}[Double harmonic coordinate]
The harmonic coordinate appears twice only because the split-support norm was
defined as the tail-resolved norm plus additional split coordinates.  This is
a bookkeeping convention, not a new PDE obstruction.
\end{remark}

\subsubsection{No-free-comparison principle}
\label{sec:no-free-comparison}

\begin{lemma}[No baseline comparison without visibility]
\label{lem:no-free-comparison}
Suppose there exists a shifted package direction \(H\) such that the ray
\(\{\lambda H:\lambda\ge0\}\) belongs to the admissible class under
consideration and
\[
    \|H\|_{\loc,\intg,0}=0,
    \qquad
    \calE_{\tail/0}^{(3/2)}(H;0)>0.
\]
Then no estimate of the form
\[
    \calE_{\tail/0}^{(3/2)}(D;\zeta_0)
    \le
    C_{\tail/0}
    \Dist_{\loc,\intg,0}(D,\Gammaadm)
    +
    \Delta_{\tail/0}
\]
can hold uniformly with \(\Delta_{\tail/0}=0\).  The analogous statement holds
for the split excess if the baseline norm is blind to velocity,
pressure-splitting error, or harmonic split coordinates.
\end{lemma}

\begin{proof}
Assume that such a zero-error estimate holds uniformly.  For \(\lambda>0\),
set \(D_\lambda=\lambda H\) and use the representative \(\zeta_0=0\).
Since \(0\in\Gammaadm\), the quotient distance is bounded by the baseline norm of
this representative:
\[
    \Dist_{\loc,\intg,0}(D_\lambda,\Gammaadm)
    \le
    \|\lambda H\|_{\loc,\intg,0}
    =
    \lambda\|H\|_{\loc,\intg,0}
    =
    0.
\]
The tail excess is positively homogeneous in the finite-window coordinate
model, hence
\[
    \calE_{\tail/0}^{(3/2)}(D_\lambda;0)
    =
    \lambda\calE_{\tail/0}^{(3/2)}(H;0)
    >
    0.
\]
The alleged estimate with \(\Delta_{\tail/0}=0\) gives
\[
    \calE_{\tail/0}^{(3/2)}(D_\lambda;0)
    \le
    C_{\tail/0}\Dist_{\loc,\intg,0}(D_\lambda,\Gammaadm)
    =
    0,
\]
which is a contradiction.  The split-excess statement is the same argument
with \(H\) chosen in an invisible velocity, pressure-splitting error, or
harmonic split-coordinate direction for which
\(\calE_{\spl/0}^{(3/2)}(H;0)>0\).
\end{proof}

\begin{remark}[Role of the lemma]
This lemma is a sanity check.  It prevents the argument from implying that
the older baseline geometry controls pressure-tail data without an explicit
visibility mechanism.
\end{remark}

\subsubsection{Baseline visibility and finite amplitude}
\label{sec:visibility}

\begin{assumption}[Same-gauge baseline near-minimizer]
\label{ass:same-gauge-baseline-minimizer}
For the package \(D\), there exists
\(\zeta_0(D)\in\Gammaadm\) such that
\[
    \|D-\zeta_0\|_{\loc,\intg,0}
    \le
    \Dist_{\loc,\intg,0}(D,\Gammaadm)
    +
    \delta_0.
\]
The same representative \(\zeta_0\) is used to estimate the baseline
distance and all tail or split excess terms.
\end{assumption}

\begin{assumption}[Baseline coordinate visibility]
\label{ass:baseline-coordinate-visibility}
There exist constants \(C_{\mathrm{vis}}<\infty\) and
\(\Delta_{\mathrm{vis}}\ge0\) such that, on the same representative
\(\zeta_0(D)\),
\[
\begin{aligned}
    &\|u_{D-\zeta_0}\|_{L^3(Q_1)}
    +
    \|E^{\cl}_{F,D-\zeta_0}\|_{X_{\src}}
    +
    \|p_{\harm,D-\zeta_0}\|_{Y_{\harm}}
\\
    &\qquad\le
    C_{\mathrm{vis}}
    \Dist_{\loc,\intg,0}(D,\Gammaadm)
    +
    \Delta_{\mathrm{vis}}.
\end{aligned}
\]
\end{assumption}

\begin{assumption}[Finite amplitude]
\label{ass:finite-amplitude}
There is a constant \(M_U<\infty\) such that
\[
    \|u_{D-\zeta_0}\|_{L^3(Q_1)}
    \le
    M_U.
\]
Under the conservative gauge convention \(\zeta_u=0\), this is simply
\[
    \|u_{D}\|_{L^3(Q_1)}
    \le
    M_U.
\]
\end{assumption}

\begin{remark}[Structural status]
\Cref{ass:baseline-coordinate-visibility,ass:finite-amplitude} are structural
finite-window assumptions.  They are not consequences of Navier--Stokes
regularity, and they are not claimed to hold for all suitable weak solutions.
\end{remark}

\subsubsection{Baseline-to-tail excess bound}
\label{sec:tail-bound}

\begin{theorem}[Finite-amplitude baseline-to-tail bound]
\label{thm:baseline-to-tail-working}
Assume the same-gauge near-minimizer, baseline visibility, and
finite-amplitude hypotheses from \Cref{sec:visibility}.
Assume also the pressure-natural harmonic geometry and the fixed
source-to-pressure bound in \Cref{conv:fixed-cz-bound}.  Then
\[
    \calE_{\tail/0}^{(3/2)}(D;\zeta_0)
    \le
    C_{\tail/0}(M_U,C_{\mathrm{vis}})
    \Dist_{\loc,\intg,0}(D,\Gammaadm)
    +
    \Delta_{\tail/0},
\]
where one may take
\[
    C_{\tail/0}(M_U,C_{\mathrm{vis}})
    =
    C_{\mathrm{vis}}
    \max\{C_{\CZ}M_U,C_{\CZ},1\},
\]
and
\[
    \Delta_{\tail/0}
    =
    \max\{C_{\CZ}M_U,C_{\CZ},1\}\Delta_{\mathrm{vis}}.
\]
\end{theorem}

\begin{proof}
Set
\[
    K_U:=\max\{C_{\CZ}M_U,C_{\CZ},1\}.
\]
For the selected representative \(\zeta_0\), the split source has the form
\[
    F^{\cl}_{D-\zeta_0,ij}
    =
    \eta u_{D-\zeta_0,i}u_{D-\zeta_0,j}
    +
    E^{\cl}_{F,D-\zeta_0,ij}.
\]
By \Cref{conv:fixed-cz-bound},
\[
    \|R_iR_j(F^{\cl}_{D-\zeta_0,ij})\|_{Y_{\prs}}
    \le
    C_{\CZ}
    \left(
        \|u_{D-\zeta_0}\|_{L^3(Q_1)}^2
        +
        \|E^{\cl}_{F,D-\zeta_0}\|_{X_{\src}}
    \right).
\]
The finite-amplitude assumption gives
\[
    \|u_{D-\zeta_0}\|_{L^3(Q_1)}^2
    \le
    M_U\|u_{D-\zeta_0}\|_{L^3(Q_1)}.
\]
Therefore
\[
\begin{aligned}
    \calE_{\tail/0}^{(3/2)}(D;\zeta_0)
    &=
    \|R_iR_j(F^{\cl}_{D-\zeta_0,ij})\|_{Y_{\prs}}
    +
    \|p_{\harm,D-\zeta_0}\|_{Y_{\harm}}
\\
    &\le
    K_U
    \left(
        \|u_{D-\zeta_0}\|_{L^3(Q_1)}
        +
        \|E^{\cl}_{F,D-\zeta_0}\|_{X_{\src}}
        +
        \|p_{\harm,D-\zeta_0}\|_{Y_{\harm}}
    \right).
\end{aligned}
\]
Applying the baseline coordinate visibility assumption gives
\[
    \calE_{\tail/0}^{(3/2)}(D;\zeta_0)
    \le
    K_U C_{\mathrm{vis}}
    \Dist_{\loc,\intg,0}(D,\Gammaadm)
    +
    K_U\Delta_{\mathrm{vis}}.
\]
This is the stated estimate.
\end{proof}

\begin{remark}[Status]
\Cref{thm:baseline-to-tail-working} proves a finite-window implication from
visibility and finite amplitude.  It does not prove those hypotheses from the
Navier--Stokes equations.
\end{remark}

\subsubsection{Baseline-to-split excess bound}
\label{sec:split-bound}

\begin{theorem}[Finite-amplitude baseline-to-split bound]
\label{thm:baseline-to-split-working}
Under the assumptions of \Cref{thm:baseline-to-tail-working},
\[
    \calE_{\spl/0}^{(3/2)}(D;\zeta_0)
    \le
    C_{\spl/0}(M_U,C_{\mathrm{vis}})
    \Dist_{\loc,\intg,0}(D,\Gammaadm)
    +
    \Delta_{\spl/0},
\]
where one may take, up to harmless finite-component constants,
\[
    C_{\spl/0}(M_U,C_{\mathrm{vis}})
    =
    C_{\mathrm{vis}}
    \left(1+\max\{C_{\CZ}M_U,C_{\CZ},1\}\right),
\]
and
\[
    \Delta_{\spl/0}
    =
    \left(1+\max\{C_{\CZ}M_U,C_{\CZ},1\}\right)
    \Delta_{\mathrm{vis}}.
\]
\end{theorem}

\begin{proof}
With the split-support convention used here,
\[
\begin{aligned}
    \calE_{\spl/0}^{(3/2)}(D;\zeta_0)
    &=
    \|R_iR_j(F^{\cl}_{D-\zeta_0,ij})\|_{Y_{\prs}}
    +
    \|u_{D-\zeta_0}\|_{L^3(Q_1)}
\\
    &\quad
    +
    \|E^{\cl}_{F,D-\zeta_0}\|_{X_{\src}}
    +
    2\|p_{\harm,D-\zeta_0}\|_{Y_{\harm}}.
\end{aligned}
\]
This is the sum of the tail excess and one additional copy of the visible
split-coordinate block:
\[
\begin{aligned}
    \calE_{\spl/0}^{(3/2)}(D;\zeta_0)
    &=
    \calE_{\tail/0}^{(3/2)}(D;\zeta_0)
\\
    &\quad+
    \|u_{D-\zeta_0}\|_{L^3(Q_1)}
    +
    \|E^{\cl}_{F,D-\zeta_0}\|_{X_{\src}}
    +
    \|p_{\harm,D-\zeta_0}\|_{Y_{\harm}}.
\end{aligned}
\]
Let \(K_U=\max\{C_{\CZ}M_U,C_{\CZ},1\}\).  By
\Cref{thm:baseline-to-tail-working}, the first term is bounded by
\[
    K_U C_{\mathrm{vis}}
    \Dist_{\loc,\intg,0}(D,\Gammaadm)
    +
    K_U\Delta_{\mathrm{vis}}.
\]
By baseline coordinate visibility, the second visible block is bounded by
\[
    C_{\mathrm{vis}}
    \Dist_{\loc,\intg,0}(D,\Gammaadm)
    +
    \Delta_{\mathrm{vis}}.
\]
Adding the two estimates gives
\[
    \calE_{\spl/0}^{(3/2)}(D;\zeta_0)
    \le
    (1+K_U)C_{\mathrm{vis}}
    \Dist_{\loc,\intg,0}(D,\Gammaadm)
    +
    (1+K_U)\Delta_{\mathrm{vis}},
\]
which is the claimed bound.
\end{proof}

\subsection{Baseline closure after visibility}
\label{sec:visibility-driven-closure}

The preceding section proves the excess bounds that were left as abstract
same-gauge assumptions in the split closure theorem.  We now spell out the
resulting baseline forms without referring to any external pressure-tail
module.

\begin{theorem}[Visibility-driven baseline pressure-tail closure]
\label{thm:baseline-pressure-tail-closure}
Assume the same-gauge baseline near-minimizer, baseline coordinate visibility,
and finite-amplitude hypotheses of Section~\ref{sec:baseline-visibility-excess}.
Assume also that the package class is tail-compatible in the sense of
Definition~\ref{def:tail-compatible-class}.  Then
\[
\begin{aligned}
    \dist^{\sharp,\tail}_{\loc,\intg,\tail}(D,\Gammaadm)
    &\le
    C_{\tail}
    \bigl[
        (1+C_{\tail/0})
        \dist_{\loc,\intg,0}(D,\Gammaadm)
        +
        \delta_0
        +
        \Delta_{\tail/0}
    \bigr]
\\
    &\quad+
    \alpha_{\proj}\Delta_{\proj,N}
    +
    \alpha_{\harm}\Delta^{(3/2)}_{\harm,M},
\end{aligned}
\]
where
\[
    C_{\tail}=1+\alpha_{\proj}C^{\tail}_{\proj}
    +\alpha_{\harm}C^{\tail}_{\harm},
\]
and \(C_{\tail/0}\) and \(\Delta_{\tail/0}\) may be chosen as in
Theorem~\ref{thm:baseline-to-tail-working}.
\end{theorem}

\begin{proof}
Let \(\zeta_0\) be the same-gauge baseline representative.  By definition of
the tail excess,
\[
    \norm{D-\zeta_0}{\loc,\intg,\tail}
    =
    \norm{D-\zeta_0}{\loc,\intg,0}
    +
    \calE_{\tail/0}^{(3/2)}(D;\zeta_0).
\]
The near-minimizer property and Theorem~\ref{thm:baseline-to-tail-working}
give
\[
    \dist_{\loc,\intg,\tail}(D,\Gammaadm)
    \le
    (1+C_{\tail/0})
    \dist_{\loc,\intg,0}(D,\Gammaadm)
    +
    \delta_0
    +
    \Delta_{\tail/0}.
\]
Applying the abstract pressure-tail closure theorem, Theorem~\ref{thm:abstract-tail-closure}, and substituting this bound gives
the displayed estimate.
\end{proof}

\begin{theorem}[Visibility-driven baseline split-control form]
\label{thm:baseline-split-control}
Assume the hypotheses of Theorem~\ref{thm:baseline-to-split-working} and the
finite-amplitude split-control hypothesis of Definition~\ref{def:finite-amplitude-class}.
Then the split-control functional satisfies
\[
\begin{aligned}
    P_{\spl}(D;\zeta_*)
    &\le
    C_{\mathrm{FA}}(M_U)
    \bigl[
        (1+C_{\spl/0})
        \dist_{\loc,\intg,0}(D,\Gammaadm)
        +
        \delta_0
        +
        \Delta_{\spl/0}
    \bigr]
\\
    &\quad+
    C_{\mathrm{FA}}(M_U)\delta_{\spl}.
\end{aligned}
\]
\end{theorem}

\begin{proof}
The finite-amplitude split-control theorem gives
\[
    P_{\spl}(D;\zeta_*)
    \le
    C_{\mathrm{FA}}(M_U)
    \dist_{\loc,\intg,\spl}(D,\Gammaadm)
    +
    C_{\mathrm{FA}}(M_U)\delta_{\spl}.
\]
Using the same-gauge baseline representative \(\zeta_0\) as a competitor in the
split distance and applying Theorem~\ref{thm:baseline-to-split-working},
\[
\begin{aligned}
    \dist_{\loc,\intg,\spl}(D,\Gammaadm)
    &\le
    \norm{D-\zeta_0}{\loc,\intg,\spl}
\\
    &=
    \norm{D-\zeta_0}{\loc,\intg,0}
    +
    \calE_{\spl/0}^{(3/2)}(D;\zeta_0)
\\
    &\le
    (1+C_{\spl/0})
    \dist_{\loc,\intg,0}(D,\Gammaadm)
    +
    \delta_0
    +
    \Delta_{\spl/0}.
\end{aligned}
\]
Substitution proves the claim.
\end{proof}

\begin{remark}
The baseline forms above are still conditional.  Their contribution is to make
explicit which visible coordinates are needed to control the return from the
pressure-natural tail geometry to the older baseline geometry.
\end{remark}

\section{Componentwise residual-ledger module}\label{sec:ledger-module-main}

The second module closes the main residual channels.  Its bookkeeping is the fixed-window local-to-clean version of the supply--tax and defect-audit ledgers developed in \cite{YuCriticalLedgers2026,YuSingularityAuditTransfer2026}.  This is the part of the paper where the local Navier--Stokes bookkeeping is most detailed, and it is the reason the complete proof belongs in an appendix rather than in the middle of the main narrative.

The residual ledger has four branches.
\[
\boxed{
\Err_{\comp}^{[0,K]}
=
\Err_{\prs}^{[0,K]}
+
\Err_{\locerr}^{[0,K]}
+
\Err_{\rep}^{[0,K]}
+
\Err_{\gs}^{[0,K]}.}
\]
The pressure-source branch controls separated-support commutators and active covariance mismatch.  The localization branch controls momentum leakage, localized energy leakage, flux leakage, and annular pressure leakage.  The reproduction branch controls active source reproduction, model source reproduction, pressure reproduction, and harmonic pressure reproduction along the finite chain.  The gate/slack branch controls positive-part gate violations and slack identity mismatch.

The key point is not merely that each branch has an estimate.  The point is that all branches are estimated in one same-chain sharp component geometry.  Thus they can be added without changing representatives.

\begin{theorem}[Componentwise residual-ledger closure]\label{thm:main-ledger-module}
Assume the finite-window package geometry, same-chain representative convention, amplitude bound \(M_U\), and near-minimizer condition of Appendix B.  Then
\begin{equation}\label{eq:main-ledger-module}
        \Err_{\comp}^{[0,K]}(\calD;\bzeta_*)
        \le
        C_{\comp}^{[0,K]}(M_U)
        \Dist_{\comp}^{\sharp,[0,K]}(\calD,\calG_{\comp})
        +
        C_{\comp}^{[0,K]}(M_U)
        \delta_{\comp}^{[0,K]}.
\end{equation}
\end{theorem}

\begin{proof}
The detailed branch estimates are proved in \Cref{app:residual-ledger-details}.  The proof first establishes pressure-source absorption, localization leakage absorption, reproduction drift absorption, and gate/slack absorption separately.  In each branch the residual is bounded by the corresponding component coordinate evaluated against the same-chain representative, plus the relevant branch near-minimizer error.  The unified component norm is then defined to contain all branch norms with fixed finite-window weights.  Adding the four estimates and absorbing the branch constants into \(C_{\comp}^{[0,K]}(M_U)\) gives \eqref{eq:main-ledger-module}.  This is \Cref{comp:thm:componentwise-closure-target}.
\end{proof}

\begin{corollary}[Weighted residual closure]\label{cor:main-weighted-ledger}
Under the hypotheses of \Cref{thm:main-ledger-module}, any finite positive weighting of the component residual channels satisfies the corresponding weighted closure estimate, with the constant multiplied by the maximum finite-window weight ratio.
\end{corollary}

\begin{proof}
This is the weighted version of the same summation argument.  Since the number of channels is finite, all positive weights are equivalent up to finite constants.  The complete statement is given in Appendix B.
\end{proof}

\begin{remark}[Problem solved by this module]
The module turns a list of residuals into one usable theorem.  Without it, detector comparison would lose several unrelated error terms and the final theorem would not have a clean form.  With the ledger closed, detector comparison loses only \(\Err_{\comp}^{[0,K]}\), which is immediately converted into a sharp component distance plus a finite-chain near-minimizer error.
\end{remark}

\section{Detector comparison and local-to-clean transfer}\label{sec:detector-module-main}

The third module is the top interface.  It compares the detector evaluated on the localized package with the detector evaluated after the local-to-clean chart.  The clean quotient gap inserted at this stage is the structural version of the finite-window computational anti-phantom gap in \cite{YuComputationalAntiPhantom2026}.  This is where clean-side information enters the localized theorem.

\begin{theorem}[Detector comparison]\label{thm:main-detector-module}
Assume the detector-intertwining hypotheses of Appendix C.  Then
\begin{equation}\label{eq:main-detector-module}
\begin{aligned}
        M_{\Lambda}^{\loc}(\calD-\bzeta_*)
        &\ge
        M_{\Lambda}^{\comp}(\Theta_{\Lambda}(\calD-\bzeta_*))
        -
        C_{\mathrm{dc}}
        \Err_{\comp}^{[0,K]}(\calD;\bzeta_*)
        -
        \Delta_{\mathrm{dc}}.
\end{aligned}
\end{equation}
\end{theorem}

\begin{proof}
This is \Cref{thm:detector-comparison}.  The proof decomposes the detector discrepancy into channelwise detector-intertwining errors.  Each channelwise discrepancy is bounded by the corresponding residual component plus a finite-window detector error.  Summing the finitely many channels gives \eqref{eq:main-detector-module}.
\end{proof}

Inserting residual-ledger closure into detector comparison gives the immediately usable form.

\begin{corollary}[Detector comparison after ledger closure]\label{cor:main-detector-after-ledger}
Under the hypotheses of \Cref{thm:main-ledger-module,thm:main-detector-module},
\begin{equation}\label{eq:main-detector-after-ledger}
\begin{aligned}
        M_{\Lambda}^{\loc}(\calD-\bzeta_*)
        &\ge
        M_{\Lambda}^{\comp}(\Theta_{\Lambda}(\calD-\bzeta_*))\\
        &\quad
        -
        C_{\mathrm{dc}}C_{\comp}^{[0,K]}(M_U)
        \Dist_{\comp}^{\sharp,[0,K]}(\calD,\calG_{\comp})\\
        &\quad
        -
        C_{\mathrm{dc}}C_{\comp}^{[0,K]}(M_U)
        \delta_{\comp}^{[0,K]}
        -
        \Delta_{\mathrm{dc}}.
\end{aligned}
\end{equation}
\end{corollary}

\begin{proof}
Substitute \eqref{eq:main-ledger-module} into \eqref{eq:main-detector-module}.  This is the weighted detector comparison recorded as \Cref{cor:weighted-detector} in Appendix C.
\end{proof}

The final local-to-clean theorem additionally uses the clean gap, chart visibility, and component-to-baseline comparison.

\begin{theorem}[Conditional local-to-clean transfer]\label{thm:main-local-to-clean-module}
Assume \Cref{ass:main-clean-gap,ass:main-chart-visibility,ass:main-component-baseline} together with \Cref{thm:main-ledger-module,thm:main-detector-module}.  Then
\begin{equation}\label{eq:main-local-to-clean-module}
\begin{aligned}
        M_{\Lambda}^{\loc}(\calD-\bzeta_*)
        &\ge
        \Bigl(
        \mu_\Lambda^{\comp}\lambda_G
        -
        C_{\mathrm{dc}}C_{\comp}^{[0,K]}(M_U)C_{\comp/0}
        \Bigr)
        \Dist_{\loc,\intg,0}(\calD,\Gamma^{\intg}_{\Lambda,\adm})\\
        &\quad
        -
        \mathfrak E_{\mathrm{pre}}(\calD),
\end{aligned}
\end{equation}
where
\[
\begin{aligned}
        \mathfrak E_{\mathrm{pre}}(\calD)
        :=
        &\ \Delta_{\cl}
        +\mu_\Lambda^{\comp}\Delta_{\mathrm{chart}}
        +C_{\mathrm{dc}}C_{\comp}^{[0,K]}(M_U)\Delta_{\comp/0}\\
        &+C_{\mathrm{dc}}C_{\comp}^{[0,K]}(M_U)\delta_{\comp}^{[0,K]}
        +\Delta_{\mathrm{dc}}.
\end{aligned}
\]
\end{theorem}

\begin{proof}
Start from \eqref{eq:main-detector-after-ledger}.  Apply the clean gap \eqref{eq:ass-clean-gap} to the clean detector term.  Apply chart visibility \eqref{eq:ass-chart-visibility} to convert the clean distance into the baseline distance.  Apply component-to-baseline comparison \eqref{eq:ass-component-baseline} to the negative component-distance term.  Collect the coefficient of the baseline distance and place all remaining terms into \(\mathfrak E_{\mathrm{pre}}\).  This is the main transfer step proved in \Cref{thm:local-to-clean-transfer}.
\end{proof}

\begin{remark}[Where pressure tails enter the transfer]
In the abstract transfer inequality \eqref{eq:main-local-to-clean-module}, the component-to-baseline comparison is stated as a direct input.  In applications where component comparison passes through pressure-tail geometry, \Cref{thm:main-pressure-tail-module} is inserted before \Cref{ass:main-component-baseline}.  The result is the same type of inequality, with \(\mathfrak E_{\mathrm{pre}}\) enlarged by projection-tail, harmonic-tail, and visibility errors.
\end{remark}

\section{Main finite-window detection theorem}\label{sec:main-proof}

We now state and prove the main theorem in the form used in the introduction.

\begin{theorem}[Finite-window local-to-clean detection]\label{thm:journal-main}
Fix a finite-window sharp localized Navier--Stokes package \(\calD\), a same-chain representative \(\bzeta_*\), a local-to-clean chart \(\Theta_\Lambda\), and a detector channel \(\Lambda\).  Assume the pressure-tail visibility, componentwise residual-ledger closure, detector comparison, clean gap, chart visibility, component-to-baseline comparison, and positive coefficient hypotheses stated in \Cref{sec:hypotheses}.  Then there are a positive finite-window coefficient \(c_{\Lambda,0}>0\) and an explicit finite-window error functional \(\mathfrak E_{\Lambda,0}(\calD)\) such that
\begin{equation}\label{eq:journal-main}
        M_{\Lambda}^{\loc}(\calD-\bzeta_*)
        \ge
        c_{\Lambda,0}
        \Dist_{\loc,\intg,0}(\calD,\Gamma^{\intg}_{\Lambda,\adm})
        -
        \mathfrak E_{\Lambda,0}(\calD).
\end{equation}
Consequently, if
\begin{equation}\label{eq:positive-detection-threshold}
        \Dist_{\loc,\intg,0}(\calD,\Gamma^{\intg}_{\Lambda,\adm})
        >
        \frac{\mathfrak E_{\Lambda,0}(\calD)}{c_{\Lambda,0}},
\end{equation}
then
\[
        M_{\Lambda}^{\loc}(\calD-\bzeta_*)>0.
\]
\end{theorem}

\begin{proof}
The proof is an assembly of the three modules.

First, \Cref{thm:main-pressure-tail-module} gives pressure-tail closure in baseline gauge.  Whenever the component-to-baseline comparison uses pressure-tail coordinates, this theorem converts those coordinates into the baseline distance and adds only the explicit tail error
\[
        \mathfrak E_{\tail,0}(\calD)
        =
        C_{\tail}(\delta_0+\Delta_{\tail/0})
        +
        \alpha_{\proj}\Delta^{\mathrm{unif}}_{\proj,N}(\calA_\Lambda)
        +
        \alpha_{\harm}\Delta^{(3/2)}_{\harm,M},
\]
up to the finite-window constants already included in \(C_{\comp/0}\).

Second, \Cref{thm:main-ledger-module} gives
\[
        \Err_{\comp}^{[0,K]}(\calD;\bzeta_*)
        \le
        C_{\comp}^{[0,K]}(M_U)
        \Dist_{\comp}^{\sharp,[0,K]}(\calD,\calG_{\comp})
        +
        C_{\comp}^{[0,K]}(M_U)
        \delta_{\comp}^{[0,K]}.
\]
This converts the residual loss in detector comparison into a component distance and a finite-chain near-minimizer error.

Third, \Cref{thm:main-detector-module} gives
\[
\begin{aligned}
        M_{\Lambda}^{\loc}(\calD-\bzeta_*)
        &\ge
        M_{\Lambda}^{\comp}(\Theta_\Lambda(\calD-\bzeta_*))
        -C_{\mathrm{dc}}
        \Err_{\comp}^{[0,K]}(\calD;\bzeta_*)
        -\Delta_{\mathrm{dc}}.
\end{aligned}
\]
Substituting the residual closure yields \eqref{eq:main-detector-after-ledger}.

Fourth, the clean gap and chart visibility give
\[
\begin{aligned}
        M_{\Lambda}^{\comp}(\Theta_\Lambda(\calD-\bzeta_*))
        &\ge
        \mu_\Lambda^{\comp}
        \Dist_{\cl}(\Theta_\Lambda\calD,\Gamma_{\cl,\adm})
        -\Delta_{\cl}\\
        &\ge
        \mu_\Lambda^{\comp}\lambda_G
        \Dist_{\loc,\intg,0}(\calD,\Gamma^{\intg}_{\Lambda,\adm})
        -\Delta_{\cl}
        -\mu_\Lambda^{\comp}\Delta_{\mathrm{chart}}.
\end{aligned}
\]
Fifth, component-to-baseline comparison gives
\[
        \Dist_{\comp}^{\sharp,[0,K]}(\calD,\calG_{\comp})
        \le
        C_{\comp/0}
        \Dist_{\loc,\intg,0}(\calD,\Gamma^{\intg}_{\Lambda,\adm})
        +\Delta_{\comp/0},
\]
with \(C_{\comp/0}\) and \(\Delta_{\comp/0}\) enlarged if the comparison passes through the pressure-tail module.

Collecting the positive and negative coefficients of the baseline distance gives
\[
        c_{\Lambda,0}
        =
        \mu_\Lambda^{\comp}\lambda_G
        -
        C_{\mathrm{dc}}C_{\comp}^{[0,K]}(M_U)C_{\comp/0}.
\]
By \Cref{ass:main-positive-coefficient}, this number is positive.  All remaining terms are placed into
\[
\begin{aligned}
        \mathfrak E_{\Lambda,0}(\calD)
        :=
        &\ \Delta_{\cl}
        +\mu_\Lambda^{\comp}\Delta_{\mathrm{chart}}
        +C_{\mathrm{dc}}C_{\comp}^{[0,K]}(M_U)\Delta_{\comp/0}\\
        &+C_{\mathrm{dc}}C_{\comp}^{[0,K]}(M_U)\delta_{\comp}^{[0,K]}
        +\Delta_{\mathrm{dc}}
        +\mathfrak E_{\tail,0}(\calD),
\end{aligned}
\]
with harmless finite-window constant changes allowed.  This proves \eqref{eq:journal-main}.  The positivity conclusion follows immediately from \eqref{eq:positive-detection-threshold}.
\end{proof}

\begin{corollary}[Vanishing-error finite-window detection]\label{cor:vanishing-error-detection}
Suppose the hypotheses of \Cref{thm:journal-main} hold along a sequence of finite-window packages \(\calD_n\), and suppose
\[
        \mathfrak E_{\Lambda,0}(\calD_n)\to0,
        \qquad
        \Dist_{\loc,\intg,0}(\calD_n,\Gamma^{\intg}_{\Lambda,\adm})
        \ge d_0>0.
\]
Then for all sufficiently large \(n\),
\[
        M_\Lambda^{\loc}(\calD_n-\bzeta_{*,n})>0.
\]
\end{corollary}

\begin{proof}
By \Cref{thm:journal-main},
\[
        M_\Lambda^{\loc}(\calD_n-\bzeta_{*,n})
        \ge
        c_{\Lambda,0}d_0-
        \mathfrak E_{\Lambda,0}(\calD_n).
\]
The right-hand side is positive for all sufficiently large \(n\).
\end{proof}

\begin{corollary}[Contrapositive form]\label{cor:contrapositive-detection}
Under the hypotheses of \Cref{thm:journal-main}, if
\[
        M_\Lambda^{\loc}(\calD-\bzeta_*)=0,
\]
then
\[
        \Dist_{\loc,\intg,0}(\calD,\Gamma^{\intg}_{\Lambda,\adm})
        \le
        \frac{\mathfrak E_{\Lambda,0}(\calD)}{c_{\Lambda,0}}.
\]
Thus a localized detector-zero package must be baseline-close to the admissible class up to the explicit finite-window error threshold.
\end{corollary}

\begin{proof}
This is the contrapositive of \eqref{eq:journal-main}.
\end{proof}

\section{Anti-phantom alternatives and quotient-residual interpretation}\label{sec:anti-phantom-alternatives-final}

The lower bound in \cref{thm:journal-main} is most useful when read as an anti-phantom alternative.  To make this interpretation explicit, enlarge the finite-window error ledger, if necessary, by the synchronization loss:
\[
        \mathfrak E^{\quot}_{\Lambda,0}(\calD)
        :=
        \mathfrak E_{\Lambda,0}(\calD)+C_{\sync}\Delta_{\sync}(\calD).
\]
If the imported quotient geometry provides exact synchronization, then \(\Delta_{\sync}(\calD)=0\).  Since this enlargement is nonnegative, \cref{thm:journal-main} remains valid with \(\mathfrak E_{\Lambda,0}^{\quot}\) in place of \(\mathfrak E_{\Lambda,0}\).

\begin{corollary}[Finite-window anti-phantom alternative]\label{cor:final-anti-phantom-alternative}
Assume the hypotheses of \cref{thm:journal-main} and suppose \(c_{\Lambda,0}>0\).  Then every admissible finite-window package \(\calD\) satisfies at least one of the following alternatives:
\[
        M_\Lambda^{\loc}(\calD-\bzeta_*)
        \ge
        \frac{c_{\Lambda,0}}{2}
        \Dist_{\loc,\intg,0}(\calD,\Gamma^{\intg}_{\Lambda,\adm}),
\]
or
\[
        \mathfrak E^{\quot}_{\Lambda,0}(\calD)
        \ge
        \frac{c_{\Lambda,0}}{2}
        \Dist_{\loc,\intg,0}(\calD,\Gamma^{\intg}_{\Lambda,\adm}).
\]
\end{corollary}

\begin{proof}
If the first alternative holds there is nothing to prove.  Otherwise,
\[
        M_\Lambda^{\loc}(\calD-\bzeta_*)
        <
        \frac{c_{\Lambda,0}}{2}
        \Dist_{\loc,\intg,0}(\calD,\Gamma^{\intg}_{\Lambda,\adm}).
\]
The main lower bound gives
\[
        M_\Lambda^{\loc}(\calD-\bzeta_*)
        \ge
        c_{\Lambda,0}
        \Dist_{\loc,\intg,0}(\calD,\Gamma^{\intg}_{\Lambda,\adm})
        -
        \mathfrak E^{\quot}_{\Lambda,0}(\calD).
\]
Combining the two inequalities yields the second alternative, with a non-strict inequality after closure.
\end{proof}

\begin{corollary}[Detector silence forces residual concentration]\label{cor:detector-silence-final}
Under the hypotheses of \cref{thm:journal-main}, if
\[
        M_\Lambda^{\loc}(\calD-\bzeta_*)=0,
\]
then
\[
        \mathfrak E^{\quot}_{\Lambda,0}(\calD)
        \ge
        c_{\Lambda,0}
        \Dist_{\loc,\intg,0}(\calD,\Gamma^{\intg}_{\Lambda,\adm}).
\]
\end{corollary}

\begin{proof}
Insert \(M_\Lambda^{\loc}(\calD-\bzeta_*)=0\) into the main lower bound and rearrange.
\end{proof}

\begin{remark}[Interpretation]
A finite-window baseline defect cannot be both detector-invisible and quotient-residual-cheap.  Thus detector silence is meaningful only relative to the explicit error ledger.  If the ledger is small compared with the baseline defect, the localized detector must be positive.
\end{remark}

\section{A finite-dimensional pressure-tail quotient model}\label{sec:finite-dimensional-model-final}

This section records a reduced model whose role is not to solve the Navier--Stokes equations, but to show that the quotient-geometric hypotheses used in the main theorem are mutually consistent and non-vacuous.

Let
\[
        X_N=B_N\oplus T_N^{\CZ}\oplus T_N^{\harm}\oplus S_N\oplus R_N.
\]
Here \(B_N\) is the baseline coordinate space, \(T_N^{\CZ}\) is a finite Calderon--Zygmund pressure-tail space, \(T_N^{\harm}\) is a finite harmonic-tail space, \(S_N\) is a source-coordinate space, and \(R_N\) is a residual-channel space.  A typical model package is
\[
        D=(b,t^{\CZ},t^{\harm},s,r)\in X_N.
\]
One may take \(T_N^{\harm}\) to be spanned by harmonic polynomial modes and \(T_N^{\CZ}\) to be spanned by finitely many Calderon--Zygmund images \(R_iR_j\phi_{\ell,ij}\) of source modes.  Let
\[
        \Gamma_N=\{D\in X_N: L_jD=0,\\ 1\le j\le J\}
\]
be a finite-dimensional admissible quotient class.  Define
\[
        \Dist_0(D,\Gamma_N):=\dist(b,\Pi_B\Gamma_N),
\]
and
\[
        \Dist_{\tail}(D,\Gamma_N):=
        \dist((b,t^{\CZ},t^{\harm}),\Pi_{B,T}\Gamma_N).
\]

\begin{assumption}[Finite-dimensional tail visibility]\label{ass:finite-tail-vis-final}
There are constants \(C_{N,\tail}<\infty\) and \(\Delta_{N,\tail}\ge 0\) such that
\[
        \Dist_{\tail}(D,\Gamma_N)
        \le
        C_{N,\tail}\Dist_0(D,\Gamma_N)+\Delta_{N,\tail}
\]
for every \(D\) in the reduced model class \(\calA_N\).
\end{assumption}

\begin{assumption}[Reduced model detector data]\label{ass:reduced-model-detector-final}
The reduced class \(\calA_N\subset X_N\) is stable under positive quotient normalization.  The baseline unit section
\[
        S_{N,0}:=\{D\in\calA_N:\Dist_0(D,\Gamma_N)=1\}
\]
is compact.  The reduced detector \(M_N:\calA_N\to[0,\infty)\) is continuous and positively homogeneous, and is kernel-free on the reduced quotient:
\[
        M_N(D)=0\quad\Longrightarrow\quad D\in\Gamma_N.
\]
\end{assumption}

\begin{assumption}[Residual-loss comparison in the model]\label{ass:model-residual-loss-final}
There are \(C_N^{\res}<\infty\), \(\Delta_N\ge 0\), and a nonnegative residual functional \(\Err_N\) such that
\[
        M_\Lambda^{\loc}(D)
        \ge
        M_N(D)-C_N^{\res}\Err_N(D)-\Delta_N
\]
for every \(D\in\calA_N\).
\end{assumption}

\begin{theorem}[Finite-dimensional pressure-tail quotient model]\label{thm:finite-dimensional-model-final}
Under \cref{ass:reduced-model-detector-final,ass:model-residual-loss-final},
\[
        \mu_N:=\inf_{D\in S_{N,0}}M_N(D)>0.
\]
Consequently, for every \(D\in\calA_N\),
\[
        M_\Lambda^{\loc}(D)
        \ge
        \mu_N\Dist_0(D,\Gamma_N)-C_N^{\res}\Err_N(D)-\Delta_N.
\]
\end{theorem}

\begin{proof}
The compactness and continuity assumptions give a minimizer \(D_*\in S_{N,0}\) for the infimum.  If \(\mu_N=0\), then \(M_N(D_*)=0\), and kernel-freeness implies \(D_*\in\Gamma_N\), contradicting \(\Dist_0(D_*,\Gamma_N)=1\).  Thus \(\mu_N>0\).  If \(r=\Dist_0(D,\Gamma_N)>0\), normalize \(\widehat D=D/r\in S_{N,0}\).  Homogeneity gives \(M_N(D)=rM_N(\widehat D)\ge r\mu_N\).  Insert this lower bound into the residual-loss comparison.  The case \(r=0\) is immediate from nonnegativity.
\end{proof}

\begin{corollary}[Finite-dimensional anti-phantom alternative]\label{cor:finite-model-anti-phantom-final}
Under the hypotheses of \cref{thm:finite-dimensional-model-final}, every \(D\in\calA_N\) satisfies at least one of the alternatives
\[
        M_\Lambda^{\loc}(D)
        \ge
        \frac{\mu_N}{2}\Dist_0(D,\Gamma_N),
\]
or
\[
        C_N^{\res}\Err_N(D)+\Delta_N
        \ge
        \frac{\mu_N}{2}\Dist_0(D,\Gamma_N).
\]
\end{corollary}

\begin{proof}
Apply \cref{thm:finite-dimensional-model-final} and argue exactly as in \cref{cor:final-anti-phantom-alternative}.
\end{proof}

\subsection*{An explicit matrix realization}

The previous theorem gives a compact quotient criterion.  The following elementary realization makes the non-vacuity statement completely explicit: all structural features can occur simultaneously in a reduced pressure-tail model with nonzero Calderon--Zygmund and harmonic tail coordinates.

\begin{proposition}[Explicit finite-dimensional pressure-tail matrix model]\label{prop:explicit-matrix-tail-model}
Let \(B_N=\mathbb R^m\), \(T_N^{\CZ}=\mathbb R^{p}\), \(T_N^{\harm}=\mathbb R^{q}\), \(S_N=\mathbb R^s\), and \(R_N=\mathbb R^r\), all with Euclidean norms.  Fix linear maps
\[
        L_{\CZ}:B_N\to T_N^{\CZ},\qquad
        L_{\harm}:B_N\to T_N^{\harm},\qquad
        L_{S}:B_N\to S_N,
        \qquad
        L_R:B_N\to R_N.
\]
Let
\[
        \Gamma_N=\{0\}\subset X_N
\]
and let the reduced model class be the linear graph
\[
        \calA_N^{\mathrm{mat}}
        :=
        \{(b,L_{\CZ}b,L_{\harm}b,L_Sb,L_Rb): b\in B_N\}.
\]
Define
\[
        M_N(D):=\|b\|_{B_N},
        \qquad
        \Err_N(D):=\|L_Rb\|_{R_N},
        \qquad
        M_\Lambda^{\loc}(D):=M_N(D).
\]
Then \(\calA_N^{\mathrm{mat}}\) satisfies finite-dimensional tail visibility, compact baseline unit-section compactness, detector kernel-freeness, and residual-loss comparison.  More precisely,
\[
        \Dist_{\tail}(D,\Gamma_N)
        \le
        \bigl(1+\|L_{\CZ}\|^2+\|L_{\harm}\|^2\bigr)^{1/2}
        \Dist_0(D,\Gamma_N),
\]
\[
        M_N(D)=\Dist_0(D,\Gamma_N),
        \qquad
        M_\Lambda^{\loc}(D)
        \ge
        M_N(D)-\Err_N(D),
\]
and the anti-phantom alternative holds with \(\mu_N=1\), \(C_N^{\res}=1\), and \(\Delta_N=0\).
\end{proposition}

\begin{proof}
For \(D=(b,L_{\CZ}b,L_{\harm}b,L_Sb,L_Rb)\), the baseline distance to \(\Gamma_N=\{0\}\) is \(\Dist_0(D,\Gamma_N)=\|b\|_{B_N}\).  The tail distance is the Euclidean norm of the baseline and tail coordinates:
\[
        \Dist_{\tail}(D,\Gamma_N)
        =
        \bigl(\|b\|^2+\|L_{\CZ}b\|^2+\|L_{\harm}b\|^2\bigr)^{1/2}.
\]
The displayed tail-visibility estimate follows from the operator-norm bounds for \(L_{\CZ}\) and \(L_{\harm}\).  The unit baseline section is the Euclidean unit sphere in \(B_N\) transported by the graph map, hence compact.  Since \(M_N(D)=\|b\|\), the zero set of \(M_N\) on the graph is exactly \(b=0\), which is \(D\in\Gamma_N\).  Thus the detector is kernel-free and the compact quotient gap is \(\mu_N=1\).  Finally, by definition \(M_\Lambda^{\loc}=M_N\), so
\[
        M_\Lambda^{\loc}(D)=M_N(D)
        \ge M_N(D)-\Err_N(D).
\]
This is the residual-loss comparison with \(C_N^{\res}=1\) and \(\Delta_N=0\), and \Cref{cor:finite-model-anti-phantom-final} gives the stated alternative.
\end{proof}

\begin{remark}[Why the matrix model is included]
\label{rem:matrix-model-purpose}
The model is intentionally elementary.  Its role is to remove the vacuity concern: the baseline, pressure-tail, residual, and detector components can be chosen so that the quotient gap, tail visibility, residual comparison, and anti-phantom alternative all hold simultaneously.  It is not a discretization of the Navier--Stokes equations and carries no scale-uniform information.
\end{remark}

\subsection*{Detector-specific inputs in an explicit matrix model}

The previous graph model verifies non-vacuity of the baseline, pressure-tail, and residual pieces.  The next realization verifies the detector-specific inputs appearing in the main theorem: clean gap, chart visibility, component-to-baseline comparison, residual-ledger closure, detector comparison, and kernel-freeness.

\begin{theorem}[Explicit detector-input matrix realization]\label{thm:explicit-detector-input-matrix-model}
Let \(B=\mathbb R^m\), \(T=\mathbb R^p\), \(C=\mathbb R^q\), \(R=\mathbb R^r\), and \(Y=\mathbb R^\ell\), all with Euclidean norms.  Let
\[
        A_T:B\to T,
        \qquad
        A_C:B\to C,
        \qquad
        A_R:B\to R,
        \qquad
        G:B\to Y,
        \qquad
        H:Y\to \mathbb R^d,
        \qquad
        K:R\to \mathbb R^d
\]
be linear maps.  Assume that \(G\) is injective and that \(H\) is injective on \(G(B)\).  Define
\[
        \lambda_G:=\inf_{\|b\|=1}\|Gb\|>0,
        \qquad
        \mu_H:=\inf_{y\in G(B),\ \|y\|=1}\|Hy\|>0.
\]
Let the reduced model class be the graph
\[
        \calA^{\det}_{N}:=
        \{(b,A_Tb,A_Cb,A_Rb):b\in B\}
        \subset B\oplus T\oplus C\oplus R,
\]
with gauge class \(\Gamma_N=\{0\}\).  Define
\[
        \Dist_0(D,\Gamma_N):=\|b\|,
        \qquad
        \Dist_{\comp}(D,\Gamma_N):=\|(b,A_Tb,A_Cb,A_Rb)\|,
        \qquad
        \Err_N(D):=\|A_Rb\|,
\]
\[
        \Theta_N(D):=Gb,
        \qquad
        M_N^{\comp}(\Theta_ND):=\|HGb\|,
        \qquad
        M_N^{\loc}(D):=\|HGb+KA_Rb\|.
\]
Then the following finite-window structural inputs hold on \(\calA_N^{\det}\):
\begin{enumerate}[label=\textup{(\alph*)},leftmargin=*]
    \item \textup{Clean gap:}
    \[
        M_N^{\comp}(\Theta_ND)\ge \mu_H\dist_{\cl}(\Theta_ND,0),
        \qquad
        \dist_{\cl}(\Theta_ND,0):=\|Gb\|.
    \]
    \item \textup{Chart visibility:}
    \[
        \dist_{\cl}(\Theta_ND,0)\ge \lambda_G\Dist_0(D,\Gamma_N).
    \]
    \item \textup{Component-to-baseline comparison:}
    \[
        \Dist_{\comp}(D,\Gamma_N)
        \le C_{\comp/0}^{N}\Dist_0(D,\Gamma_N),
        \qquad
        C_{\comp/0}^{N}:=\bigl(1+\|A_T\|^2+\|A_C\|^2+\|A_R\|^2\bigr)^{1/2}.
    \]
    \item \textup{Residual-ledger closure:}
    \[
        \Err_N(D)\le \Dist_{\comp}(D,\Gamma_N).
    \]
    \item \textup{Detector comparison:}
    \[
        M_N^{\loc}(D)
        \ge
        M_N^{\comp}(\Theta_ND)-\|K\|\Err_N(D).
    \]
\end{enumerate}
Consequently, if
\[
        c_N:=\mu_H\lambda_G-\|K\|C_{\comp/0}^{N}>0,
\]
then
\[
        M_N^{\loc}(D)
        \ge
        c_N\Dist_0(D,\Gamma_N)
\]
with zero additive error.  In particular, detector silence forces \(D\in\Gamma_N\) on this graph.
\end{theorem}

\begin{proof}
For \(y=Gb\in G(B)\), the definition of \(\mu_H\) gives \(\|Hy\|\ge \mu_H\|y\|\), proving the clean gap.  The chart visibility estimate is the definition of \(\lambda_G\).  The component comparison follows from
\[
        \|(b,A_Tb,A_Cb,A_Rb)\|^2
        \le
        \bigl(1+\|A_T\|^2+\|A_C\|^2+\|A_R\|^2\bigr)\|b\|^2.
\]
The residual closure is immediate because \(\|A_Rb\|\) is a coordinate of the component norm.  The detector comparison is the reverse triangle inequality:
\[
        \|HGb+KA_Rb\|
        \ge
        \|HGb\|-\|K\|\|A_Rb\|.
\]
Combining these estimates gives
\[
        M_N^{\loc}(D)
        \ge
        \mu_H\lambda_G\|b\|-\|K\|\Dist_{\comp}(D,\Gamma_N)
        \ge
        (\mu_H\lambda_G-\|K\|C_{\comp/0}^{N})\|b\|.
\]
The last expression is \(c_N\Dist_0(D,\Gamma_N)\).  If \(M_N^{\loc}(D)=0\) and \(c_N>0\), the lower bound forces \(b=0\), hence \(D=0\in\Gamma_N\) on the graph.
\end{proof}

\begin{corollary}[Concrete scalar kernel-free matrix example]\label{cor:concrete-kernel-free-matrix-example}
Take \(B=T=C=R=Y=\mathbb R\) and
\[
        A_Tb=b,
        \qquad
        A_Cb=b,
        \qquad
        A_Rb=\rho b,
        \qquad
        Gb=b,
        \qquad
        Hy=y,
        \qquad
        Kr=-\kappa r,
\]
where \(\rho>0\) and \(0\le \kappa<1/\sqrt{3+\rho^2}\).  Then
\[
        \lambda_G=1,
        \qquad
        \mu_H=1,
        \qquad
        C_{\comp/0}^{N}=\sqrt{3+\rho^2},
        \qquad
        c_N=1-\kappa\sqrt{3+\rho^2}>0.
\]
The model has nonzero tail coordinate \(t=b\) and nonzero residual coordinate \(r=\rho b\) whenever \(b\ne0\).  Its detector zero set is exactly the admissible class \(\Gamma_N=\{0\}\).
\end{corollary}

\begin{proof}
All maps are scalar, so the injectivity assumptions are immediate.  The component graph norm is
\[
        \|(b,b,b,\rho b)\|=\sqrt{3+\rho^2}\,|b|.
\]
The preceding theorem gives the stated constants and the lower bound \(M_N^{\loc}(D)\ge c_N|b|\).  Since \(c_N>0\), the detector can vanish only when \(b=0\), which is precisely the gauge class on the graph.
\end{proof}

\begin{remark}[What the explicit matrix model verifies]
\Cref{thm:explicit-detector-input-matrix-model,cor:concrete-kernel-free-matrix-example} verify the detector-specific inputs that remain structural in the PDE-facing theory: clean quotient gap, chart visibility, detector comparison, residual-ledger closure, component-to-baseline comparison, and kernel-freeness.  This verification is finite-dimensional and fixed-window.  It proves non-vacuity of the hypothesis package; it is not a statement about arbitrary Navier--Stokes-generated packages.
\end{remark}

\section{NS-generated coordinate layer and smooth finite-window subclasses}\label{sec:ns-generated-final}

The main theorem is stated for finite-window packages.  This section records how local Navier--Stokes data enter the coordinate layer of the package framework and where the genuinely structural assumptions remain.  The NS-realizability emphasis follows the defect-cascade and ledger-realizable package viewpoint of \cite{YuInvisible2026,YuCriticalLedgers2026,YuSingularityAuditTransfer2026}.

Let \((u,p)\) be local pressure-admissible Navier--Stokes data on a normalized cylinder \(Q_1=B_1\times(-1,0)\), with
\[
        u\in L^3(Q_1)^3,
        \qquad
        p\in L^{3/2}(Q_1),
\]
and
\[
        -\Delta p=\partial_i\partial_j(u_iu_j)
\]
in distributions on the fixed local window, modulo time-dependent pressure constants.  The package-realizability theorem in \cref{thm:package-realizability} shows that such data generate the active source, active Calderon--Zygmund pressure, harmonic pressure coordinate, leakage coordinates, gate/slack coordinates, and finite-chain reproduction coordinates whenever the corresponding finite-window maps are fixed.

\begin{definition}[NS-generated verification class]\label{def:ns-generated-verification-final}
For parameters \(M,K,N,\eta\), let \(\mathfrak N_{M,K,N,\eta}^{\NS}\) be the class of packages \(\calD(u,p;z_0,r_0,K)\) generated by local Navier--Stokes data and satisfying the following finite-window module assumptions:
\begin{enumerate}[label=(\roman*)]
\item the coordinate-domain realizability conclusions of \cref{thm:package-realizability};
\item compatibility with the imported quotient geometry and the synchronized representative convention;
\item finite amplitude or a declared quadratic component geometry;
\item pressure-tail approximation, either by compact clean pressure image or by an effective projection bound;
\item clean quotient gap, chart visibility, component-to-baseline comparison, residual-ledger closure, and detector comparison on the selected package class.
\end{enumerate}
\end{definition}

\begin{theorem}[NS-generated verification-class theorem]\label{thm:ns-generated-verification-final}
If \(\calD(u,p;z_0,r_0,K)\in\mathfrak N_{M,K,N,\eta}^{\NS}\), then the hypotheses of the finite-window detection theorem \cref{thm:journal-main} are satisfied on this package.  Therefore
\[
        M_\Lambda^{\loc}(\calD(u,p;z_0,r_0,K)-\bzeta_*)
        \ge
        c_{\Lambda,0}
        \Dist_{\loc,\intg,0}(\calD(u,p;z_0,r_0,K),\Gamma^{\intg}_{\Lambda,\adm})
        -
        \mathfrak E^{\quot}_{\Lambda,0}(\calD(u,p;z_0,r_0,K)).
\]
If \(c_{\Lambda,0}>0\), the anti-phantom alternative \cref{cor:final-anti-phantom-alternative} applies.
\end{theorem}

\begin{proof}
Membership in \(\mathfrak N_{M,K,N,\eta}^{\NS}\) is defined precisely to include the finite-window module hypotheses required by \cref{thm:journal-main}.  Applying that theorem gives the displayed inequality, and \cref{cor:final-anti-phantom-alternative} gives the alternative.
\end{proof}

\begin{definition}[Smooth finite-window coordinate/compactness subclass]\label{def:smooth-subclass-final}
Let \(\mathfrak N_{M,K,N,\eta}^{\smooth}\) be the subclass of NS-generated packages obtained from smooth local Navier--Stokes data with uniform finite-window \(C^m\) bounds, or uniform \(H^s\) bounds with \(s\) large enough for the compact embeddings used on the fixed window.  We also require the chosen finite-window reproduction maps and projections to be bounded on the declared coordinate spaces.
\end{definition}

\begin{theorem}[Smooth packages satisfy the coordinate and compactness layer]\label{thm:smooth-coordinate-layer-final}
Every package in \(\mathfrak N_{M,K,N,\eta}^{\smooth}\) satisfies the coordinate-domain conclusions of \cref{thm:package-realizability}.  Moreover, the selected clean source family is precompact in \(X_{\src}\), hence its clean pressure image is precompact in \(Y_{\prs}\).  Therefore the compact clean projection-tail criterion applies.
\end{theorem}

\begin{proof}
Coordinate realizability is \cref{thm:package-realizability}.  Uniform smooth bounds on a fixed bounded window imply precompactness of the selected clean sources in \(L^{3/2}\), hence in \(X_{\src}\), by Arzela--Ascoli in the \(C^m\) formulation or by Rellich--Kondrachov in the Sobolev formulation.  The source-to-pressure map is bounded, so the clean pressure image is precompact in \(Y_{\prs}\).  The projection-tail conclusion follows from the compact image criterion in \cref{thm:compact-pressure-image-tail,thm:source-compact-pressure-compact}.
\end{proof}

\begin{corollary}[Nonempty smooth coordinate/compactness layer]\label{cor:nonempty-smooth-coordinate-class}
The smooth coordinate/compactness layer is nonempty, and it contains nonzero local Navier--Stokes data.  In particular, for any constant vector \(a\in\mathbb R^3\), the pair
\[
        u(x,t)\equiv a,
        \qquad
        p(x,t)\equiv 0
\]
on the fixed finite window is a smooth pressure-admissible Navier--Stokes pair.  If \(a\ne0\), the generated package has nonzero velocity and nonzero active source coordinates, and the singleton package family satisfies the coordinate-domain and compactness conclusions of \Cref{thm:smooth-coordinate-layer-final}.
\end{corollary}

\begin{proof}
The constant field is divergence-free, satisfies \(\partial_t u=0\), \(\Delta u=0\), and \((u\cdot\nabla)u=0\).  With \(p=0\), the Navier--Stokes equations hold on the fixed cylinder.  The pressure-admissibility identity also holds, since \(\partial_i\partial_j(a_i a_j)=0\).  The finite-window coordinates are therefore realized by \Cref{thm:package-realizability}.  A singleton family is compact in every declared coordinate space, so the clean-source and pressure-image compactness inputs are satisfied.  For \(a\ne0\), the velocity coordinate and the localized active source \(\eta a_i a_j\) are nonzero.
\end{proof}

\begin{remark}[Nonemptiness versus full verification]
\label{rem:nonempty-not-full-verification}
\Cref{cor:nonempty-smooth-coordinate-class} proves nonemptiness of the coordinate and compactness layer only.  It does not assert the detector-specific kernel-free, chart-visibility, component-budget, or detector-comparison inputs.  Those remain the finite-window structural conditions isolated by the main theorem.
\end{remark}

\begin{remark}[Remaining detector-specific inputs]
The smooth coordinate layer does not prove clean kernel-freeness, chart kernel-freeness, detector-channel comparison, component-budget visibility, or residual-channel closure.  Those remain finite-window detector/model inputs unless they are separately verified for the chosen reduced class.
\end{remark}

\section{PDE-facing interpretation and limitations}\label{sec:pde-status-main}

The theorem separates two layers of work.  The first layer is the finite-window audit layer.  It is closed in this paper: pressure-tail geometry, residual-ledger closure, detector comparison, package-realizability statements, compactness criteria, reduced kernel-free criteria, and an explicit matrix verification of the detector-specific inputs are proved in the text and appendices.  The second layer is the PDE-facing layer.  It asks whether the same structural hypotheses can be verified for broad concrete package classes generated by Navier--Stokes solutions.

\subsection*{What has been closed}

The finite-window bookkeeping no longer has a missing residual channel.  Pressure-source residuals, localization leakage, reproduction drift, and gate/slack mismatch all enter the unified residual ledger.  The detector comparison no longer loses a collection of unrelated terms; it loses only the closed component residual and a detector-intertwining error.  Pressure-tail coordinates no longer float outside the baseline geometry; they are controlled by baseline visibility plus explicit projection and harmonic tail errors.

In this precise sense, the paper completes a finite-window conditional detection mechanism.

\subsection*{What remains structural}

The remaining inputs are not bookkeeping estimates.  They are mathematical statements about the selected Navier--Stokes-generated class.

The clean anti-phantom gap asks whether the clean detector has a nontrivial kernel on the relevant quotient.  Pressure/tax kernel-free criteria in Appendix C give reduced finite-dimensional and compact quotient ways to verify such a statement, but they do not prove it automatically for all possible packages.

Chart visibility asks whether the local-to-clean chart preserves the baseline defect.  This is a geometric property of the chart and the selected quotient.  If the chart has a large kernel in the baseline direction, no detector comparison theorem can recover the lost defect.

Component-to-baseline comparison asks whether the sharp component distance can be controlled by the older baseline distance.  This is where one must be careful about hidden coordinates: component norms are stronger than baseline norms unless the extra coordinates are visible, compact, or otherwise controlled.

Compactness or effective projection of clean pressure sources asks whether finite pressure projections capture the clean pressure image uniformly.  Appendix A proves that compactness is sufficient and also explains why boundedness alone is not enough.

\subsection*{Why no Clay-level conclusion follows}

A global regularity or singularity-exclusion theorem would require several additional steps not performed here.  One would need to show that suitable weak solutions generate packages satisfying the structural hypotheses at all relevant scales, that the constants remain controlled under scale iteration, that the finite-window detection threshold can be propagated or made scale-uniform, and that positive detection contradicts a possible singular cascade.  None of these steps is hidden in the present theorem.  They are deliberately left outside the finite-window result.

\subsection*{Research path suggested by the theorem}

The theorem identifies the next PDE-facing tasks in a clean order.
\begin{enumerate}[label=\textup{(\roman*)},leftmargin=*]
    \item Prove clean anti-phantom gaps for reduced or physically generated clean package classes.
    \item Prove chart visibility for concrete local-to-clean maps arising from Navier--Stokes localization and pressure splitting.
    \item Prove component-to-baseline comparisons without adding invisible coordinates.
    \item Replace compact pressure-image assumptions by effective projection or equation-derived compactness criteria.
    \item Study whether the finite-window constants can be made scale-uniform on a controlled class.
\end{enumerate}
Only after these tasks are addressed would it be appropriate to discuss singularity exclusion or global regularity consequences.

\section{Conclusion}\label{sec:final-conclusion}

The paper proves a fixed finite-window structural theorem for sharp localized Navier--Stokes packages.  The theorem can be read in two equivalent ways.  As a local-to-clean transfer theorem, it says that clean-side detection transfers to localized detection after pressure-tail, residual-ledger, detector-comparison, chart, and component losses are paid.  As an anti-phantom theorem, it says that a baseline-visible finite-window defect cannot be both detector-silent and quotient-residual-cheap.

The result is conditional but complete at the finite-window level.  The imported quotient geometry is now recorded through a reference-grade interface, the smooth Navier--Stokes layer supplies genuine coordinate/compactness examples, and the explicit matrix model verifies the detector-specific inputs in a concrete reduced quotient geometry.  The remaining challenge is no longer to locate a missing bookkeeping term in the finite-window ledger.  The remaining challenge is to verify the same structural inputs for broader Navier--Stokes-generated package classes, and then to determine whether the constants and residual ledgers can be made stable under recursive scale iteration.

\appendix

\section{Imported Quotient Geometry, Pressure-Tail Compactness, and Pressure/Tax Criteria}\label{app:pressure-tail-details}\label{app:imported-geometry-module}

\subsection*{A.0 Reference-grade provenance ledger for imported geometry}

The quotient-geometric inputs used by the main theorem are not hidden assumptions; they are the fixed-window interface distilled from the preceding audit, ledger, and computational anti-phantom manuscripts \cite{YuInvisible2026,YuCriticalLedgers2026,YuSingularityAuditTransfer2026,YuComputationalAntiPhantom2026}.  In this representative version they are recorded through \Cref{tab:appendix-imported-provenance}.  If the paper is submitted as part of a sequence, the third column can be replaced by exact citations to the earlier package-geometry manuscript; the logical content remains the same.

\begin{table}[h]
\centering
\small
\begin{tabular}{p{0.27\textwidth}p{0.31\textwidth}p{0.33\textwidth}}
\hline
\textbf{Imported item} & \textbf{Required property} & \textbf{Location or status} \\
\hline
Package space and gauge class & Banach product space with admissible gauge directions & Section~2; coordinate domains in \Cref{thm:package-realizability} \\
Baseline quotient distance & Same-gauge baseline defect size & Definitions~4.47, 4.52 \\
Enhanced pressure-tail distance & Baseline plus Calderon--Zygmund and harmonic tail penalties & Definitions~4.4--4.5; closure in \Cref{thm:baseline-pressure-tail-closure} \\
Component finite-chain distance & Dominates pressure-source, localization, reproduction, gate/slack channels & Definitions~B.130--B.132; closure in \Cref{comp:thm:componentwise-closure-target} \\
Synchronized representative & One representative $\calD-\bzeta_*$ for all modules & Conventions B.12, B.36, B.64, B.125; exact finite-dimensional versions in Lemmas~4.23, 4.40 \\
Detector comparison interface & Clean detector bounded by local detector plus closed residual ledger & \Cref{thm:detector-comparison} \\
Clean gap and chart visibility & Kernel-free clean quotient and chart preserving the baseline defect & Structural inputs; compact criteria in Appendix~C; explicit matrix verification in \Cref{thm:explicit-detector-input-matrix-model} \\
\hline
\end{tabular}
\caption{Appendix provenance ledger for the imported quotient-geometric module.}
\label{tab:appendix-imported-provenance}
\end{table}

The guiding convention is that a result marked as a \emph{criterion} is proved in this paper, while a result marked as a \emph{structural input} must be verified in the selected package class.  This prevents the imported quotient geometry from functioning as an uncheckable black box.

\subsection{Reference-grade imported quotient geometry module}\label{subsec:reference-grade-provenance}

This subsection records the exact interface through which the earlier finite-window quotient geometry is used in the present paper.  The word ``imported'' is not used as an unspecified black box.  It means that the main theorem uses only the following finite list of quotient objects, representatives, distances, and compatibility estimates.  Each item is either proved in the present manuscript, realized by the finite-dimensional model below, or explicitly retained as a finite-window structural input for a selected package class.

\begin{definition}[Imported quotient interface]\label{def:imported-quotient-interface}
For a fixed finite window \([0,K]\), the imported quotient interface consists of the following data.
\begin{enumerate}[label=\textup{(IQG\arabic*)},leftmargin=*]
    \item A Banach product package space
    \[
        \calX^{[0,K]}=\prod_{k=0}^K\calX_k,
    \]
    containing the velocity, active pressure, harmonic pressure, source, residual, gate/slack, reproduction, chart, and detector coordinates used below.
    \item A closed admissible gauge subspace \(\calZ^{[0,K]}\subset\calX^{[0,K]}\) and the quotient relation \(\calD\sim\calD'\) if \(\calD-\calD'\in\calZ^{[0,K]}\).
    \item Local integrated, component, and clean admissible classes
    \[
        \Gamma^{\intg}_{\Lambda,\adm},\qquad
        \Gamma^{\comp}_{\Lambda,\adm},\qquad
        \Gamma^{\cl}_{\Lambda}.
    \]
    \item The baseline, tail, and component quotient distances
    \[
        \Dist_{\loc,\intg,0},\qquad
        \Dist^{\sharp,\tail}_{\loc,\intg,\tail},\qquad
        \Dist^{\sharp,[0,K]}_{\comp}.
    \]
    \item A synchronized representative selection \(\bzeta_*(\calD)\in\calZ^{[0,K]}\), with a nonnegative synchronization loss \(\Delta_{\sync}(\calD)\).  The same shifted package \(\calD-\bzeta_*\) is used for baseline near-minimization, pressure-tail excess, component distance, residual ledger, local-to-clean chart, and detector comparison.
    \item A local-to-clean chart \(\Theta_\Lambda\) and local/clean detector functionals \(M^{\loc}_\Lambda\) and \(M^{\comp}_\Lambda\), evaluated on the synchronized representative.
\end{enumerate}
\end{definition}

The following provenance ledger records where each item is fixed or verified.
\begin{center}
\small
\begin{tabular}{p{0.20\textwidth}p{0.34\textwidth}p{0.36\textwidth}}
\hline
\textbf{Item} & \textbf{Convention used here} & \textbf{Provenance/status}\\
\hline
Package and gauge spaces & \(\calX^{[0,K]}\), \(\calZ^{[0,K]}\), quotient relation & Coordinate realization for NS data is proved in \cref{thm:package-realizability}; the quotient gauge convention is a finite-window structural contract.\\
Baseline quotient distance & \(\Dist_{\loc,\intg,0}(\calD,\Gamma^{\intg}_{\Lambda,\adm})\) & Final defect distance in \cref{thm:journal-main}; it intentionally excludes tail/residual coordinates unless visibility hypotheses add them.\\
Tail quotient distance & \(\Dist^{\sharp,\tail}_{\loc,\intg,\tail}\) & Tail closure is proved in \cref{thm:abstract-tail-closure,thm:baseline-pressure-tail-closure}; compact and effective projection-tail criteria are in \cref{thm:compact-pressure-image-tail,prop:effective-projection-replacement}.\\
Component quotient distance & \(\Dist^{\sharp,[0,K]}_{\comp}\) & Defined in \cref{comp:def:component-distance}; componentwise closure is proved in \cref{comp:thm:componentwise-closure-target}.\\
Synchronized representative & One \(\bzeta_*\) used in every module & Same-representative conventions appear in \cref{prop:common-representative,ps:ass:same-gauge-representative,comp:ass:same-chain-representative,conv:same-representative}; any failure of exact synchronization is charged to \(\Delta_{\sync}\).\\
Clean gap and chart visibility & \cref{ass:main-clean-gap,ass:main-chart-visibility} & Structural finite-window inputs; compact quotient and kernel-free criteria are proved in \cref{thm:compact-tax-gap,thm:component-zero-set-kernel}.\\
Detector comparison & Same-representative local-to-clean detector inequality & Proved from detector-intertwining inputs in \cref{thm:detector-comparison}; those detector-specific inputs are verified in the explicit matrix model \cref{thm:explicit-detector-input-matrix-model}.\\
\hline
\end{tabular}
\end{center}

\begin{remark}[Provenance versus PDE verification]
The provenance ledger is a bookkeeping and verification device.  It does not claim that every suitable weak solution satisfies the structural quotient inputs.  It says exactly which fixed-window objects are imported, which module estimates are proved here, and which inputs must be separately verified on a chosen package class.  The explicit matrix model in \cref{thm:explicit-detector-input-matrix-model} verifies the detector-specific structural inputs in a completely concrete reduced quotient setting.
\end{remark}

\subsection{Uniform clean-source compactness and projection-tail convergence}
\label{sec:projection-tail-uniformity}

The baseline pressure-tail closure estimate in
\Cref{thm:baseline-pressure-tail-closure} still contains the clean projection
tail.  This section records the finite-window functional-analytic mechanism
that turns compactness of the selected clean pressure-source image into
uniform projection-tail convergence.

\subsubsection{Clean source and pressure image}

In this section
\[
    Y_{\prs}
    =
    L^{3/2}\bigl((-1,0);L^{3/2}(B_{1/2})\bigr)
\]
and
\[
    X_{\src}
    =
    L^{3/2}\bigl((-1,0);L^{3/2}(B_1)\bigr)^{3\times 3}.
\]
For a clean source \(F\in X_{\src}\), define
\[
    \mathcal R(F)
    :=
    R_iR_j(F_{ij}),
\]
where \(F\) is extended by zero outside \(B_1\) before applying the Riesz
transforms and the result is restricted to \(B_{1/2}\).  We use the fixed
Calderon--Zygmund bound
\[
    \|\mathcal R(F)\|_{Y_{\prs}}
    \le
    C_{\CZ}\|F\|_{X_{\src}}.
\]

Let
\[
    P^{\cl}_{\prs,N}:Y_{\prs}\to Y_{\prs}
\]
be finite-rank clean pressure projections such that
\[
    P^{\cl}_{\prs,N}g\to g
    \quad\text{in }Y_{\prs}
    \quad\text{for every }g\in Y_{\prs},
\]
and assume the uniform boundedness condition
\[
    C_P:=\sup_N\|P^{\cl}_{\prs,N}\|_{Y_{\prs}\to Y_{\prs}}<\infty.
\]

\begin{definition}[Selected clean-source family]
\label{def:selected-clean-source-family}
Let \(\mathcal A_\Lambda\) be the admissible package class under
consideration, and let \(\zeta_0(D)\) be the same-gauge baseline
near-minimizer from \Cref{ass:same-gauge-baseline-minimizer}.  Define
\[
    \mathcal F_{\Lambda,0}
    :=
    \left\{
        F^{\cl}_{D-\zeta_0(D)}
        :
        D\in\mathcal A_\Lambda
    \right\}
    \subset X_{\src}
\]
and its pressure image
\[
    \mathcal G_{\Lambda,0}
    :=
    \mathcal R(\mathcal F_{\Lambda,0})
    =
    \left\{
        R_iR_j(F^{\cl}_{D-\zeta_0(D),ij})
        :
        D\in\mathcal A_\Lambda
    \right\}
    \subset Y_{\prs} .
\]
The uniform clean projection-tail error is
\[
    \Delta^{\mathrm{unif}}_{\proj,N}(\mathcal A_\Lambda)
    :=
    \sup_{D\in\mathcal A_\Lambda}
    \left\|
        (I-P^{\cl}_{\prs,N})
        R_iR_j(F^{\cl}_{D-\zeta_0(D),ij})
    \right\|_{Y_{\prs}}.
\]
\end{definition}

\subsubsection{Boundedness is not enough}

\begin{lemma}[No uniform projection-tail decay on arbitrary bounded sets]
\label{lem:no-uniform-projection-without-compactness}
Let \(Y\) be an infinite-dimensional Banach space and let
\(P_N:Y\to Y\) be finite-rank operators.  Even if \(P_Ng\to g\) strongly for
each fixed \(g\in Y\), one cannot conclude that
\[
    \sup_{\|g\|_Y\le 1}\|(I-P_N)g\|_Y\to 0.
\]
\end{lemma}

\begin{proof}
If the displayed convergence held, then \(P_N\to I\) in operator norm.  Each
\(P_N\) is finite rank and therefore compact.  The operator-norm limit of
compact operators is compact, so the identity \(I:Y\to Y\) would be compact.
That would make the closed unit ball of \(Y\) relatively compact.  By the
standard Riesz lemma consequence, this is impossible when \(Y\) is
infinite-dimensional.  Hence strong convergence on individual vectors cannot
be upgraded to uniform convergence on arbitrary bounded infinite-dimensional
sets.
\end{proof}

\begin{remark}[Consequence]
A uniform clean projection-tail theorem requires compactness,
finite-dimensionality, smoothing, translation compactness, or another genuine
approximation mechanism.  Finite amplitude and boundedness alone are not
sufficient.
\end{remark}

\subsubsection{Compact pressure-image hypothesis}

\begin{assumption}[Compact clean pressure image]
\label{ass:compact-clean-pressure-image}
The selected pressure image
\[
    \mathcal G_{\Lambda,0}\subset Y_{\prs}
\]
has compact closure in \(Y_{\prs}\).  Equivalently, it is enough to assume that
the selected source family \(\mathcal F_{\Lambda,0}\subset X_{\src}\) has compact
closure, since \(\mathcal R:X_{\src}\to Y_{\prs}\) is bounded and hence continuous.
\end{assumption}

\begin{remark}[Structural status]
\Cref{ass:compact-clean-pressure-image} is a finite-window compactness
hypothesis.  It is not automatic from suitable weak solutions or from a
bounded \(L^3\) velocity norm.
\end{remark}

\subsubsection{Uniform projection-tail convergence}

\begin{theorem}[Uniform convergence on compact pressure images]
\label{thm:uniform-projection-tail-compact-image}
Assume that \(P^{\cl}_{\prs,N}\to I\) strongly on \(Y_{\prs}\), that
\[
    \sup_N\|P^{\cl}_{\prs,N}\|_{Y_{\prs}\to Y_{\prs}}=C_P<\infty,
\]
and that \(\mathcal G_{\Lambda,0}\) has compact closure in \(Y_{\prs}\).  Then
\[
    \Delta^{\mathrm{unif}}_{\proj,N}(\mathcal A_\Lambda)
    =
    \sup_{g\in\mathcal G_{\Lambda,0}}
    \|(I-P^{\cl}_{\prs,N})g\|_{Y_{\prs}}
    \to 0.
\]
\end{theorem}

\begin{proof}
Let \(K=\overline{\mathcal G_{\Lambda,0}}\), which is compact by assumption.
Fix \(\varepsilon>0\).  Choose
\[
    \rho=\frac{\varepsilon}{3(1+C_P)}.
\]
By compactness, there exist \(g_1,\ldots,g_J\in K\) such that
\[
    K\subset \bigcup_{j=1}^J B_{Y_{\prs}}(g_j,\rho).
\]
For each fixed center \(g_j\), strong convergence gives
\[
    \|(I-P^{\cl}_{\prs,N})g_j\|_{Y_{\prs}}\to0.
\]
Since there are finitely many centers, choose \(N_0\) such that, for all
\(N\ge N_0\) and all \(1\le j\le J\),
\[
    \|(I-P^{\cl}_{\prs,N})g_j\|_{Y_{\prs}}\le \varepsilon/3.
\]
Now let \(g\in K\), and choose \(j\) with \(\|g-g_j\|_{Y_{\prs}}<\rho\).  Then
\[
\begin{aligned}
    \|(I-P^{\cl}_{\prs,N})g\|_{Y_{\prs}}
    &\le
    \|(I-P^{\cl}_{\prs,N})(g-g_j)\|_{Y_{\prs}}
    +
    \|(I-P^{\cl}_{\prs,N})g_j\|_{Y_{\prs}}
\\
    &\le
    (1+C_P)\rho+\varepsilon/3
    <
    2\varepsilon/3.
\end{aligned}
\]
Taking the supremum over \(g\in\mathcal G_{\Lambda,0}\subset K\) proves the
claim.
\end{proof}

\begin{theorem}[Source-level compactness criterion]
\label{thm:source-level-compactness-criterion}
If the selected clean-source family
\[
    \mathcal F_{\Lambda,0}\subset X_{\src}
\]
has compact closure in \(X_{\src}\), then
\(\mathcal G_{\Lambda,0}=\mathcal R(\mathcal F_{\Lambda,0})\) has compact
closure in \(Y_{\prs}\).  Consequently
\[
    \Delta^{\mathrm{unif}}_{\proj,N}(\mathcal A_\Lambda)\to0.
\]
\end{theorem}

\begin{proof}
The Calderon--Zygmund estimate makes
\(\mathcal R:X_{\src}\to Y_{\prs}\) a bounded linear map, hence a continuous map.
The continuous image of a compact set is compact.  Applying
\Cref{thm:uniform-projection-tail-compact-image} to the compact closure of
\(\mathcal R(\mathcal F_{\Lambda,0})\) gives the uniform projection-tail
convergence.
\end{proof}

\subsubsection{Sufficient compactness criteria}

\begin{theorem}[Finite-dimensional clean-source model]
\label{thm:finite-dimensional-source-model}
Assume there is a finite-dimensional subspace
\[
    \mathcal S_K\subset X_{\src}
\]
such that
\[
    F^{\cl}_{D-\zeta_0(D)}\in\mathcal S_K
    \qquad
    \text{for all }D\in\mathcal A_\Lambda,
\]
and assume the corresponding coefficient set is bounded.  Then
\(\mathcal F_{\Lambda,0}\) has compact closure in \(X_{\src}\), and hence
\[
    \Delta^{\mathrm{unif}}_{\proj,N}(\mathcal A_\Lambda)\to0.
\]
\end{theorem}

\begin{proof}
In a finite-dimensional normed space, bounded sets have compact closure.
Thus the selected source family has compact closure in \(\mathcal S_K\), and
therefore also in \(X_{\src}\).  The conclusion follows from
\Cref{thm:source-level-compactness-criterion}.
\end{proof}

\begin{remark}[Model status]
\Cref{thm:finite-dimensional-source-model} is a model compactness theorem.
It is useful for finite-window reduced models, numerical quotient models, or
explicitly truncated pressure-source packages.  It is not a claim that all
localized Navier--Stokes sources are finite-dimensional.
\end{remark}

\begin{theorem}[Strong compactness of velocity and residual source]
\label{thm:strong-compactness-velocity-residual}
Let
\[
    \mathcal U_{\Lambda,0}
    :=
    \{u_{D-\zeta_0(D)}:D\in\mathcal A_\Lambda\}
    \subset L^3(Q_1)^3
\]
and
\[
    \mathcal E^{\src}_{\Lambda,0}
    :=
    \{E^{\cl}_{F,D-\zeta_0(D)}:D\in\mathcal A_\Lambda\}
    \subset X_{\src} .
\]
If \(\mathcal U_{\Lambda,0}\) has compact closure in \(L^3(Q_1)^3\) and
\(\mathcal E^{\src}_{\Lambda,0}\) has compact closure in \(X_{\src}\), then
\(\mathcal F_{\Lambda,0}\) has compact closure in \(X_{\src}\).  Consequently
\[
    \Delta^{\mathrm{unif}}_{\proj,N}(\mathcal A_\Lambda)\to0.
\]
\end{theorem}

\begin{proof}
It is enough to prove continuity of the source map
\[
    (u,E)\mapsto \eta u_i u_j+E_{ij}
\]
from \(L^3(Q_1)^3\times X_{\src}\) to \(X_{\src}\).  If \(u_n\to u\) in
\(L^3(Q_1)^3\), then the sequence \(u_n\) is bounded in \(L^3\), and for
each \(i,j\),
\[
\begin{aligned}
    \|\eta(u_{n,i}u_{n,j}-u_i u_j)\|_{L^{3/2}(Q_1)}
    &\le
    \|\eta\|_{L^\infty}
    \|u_{n,i}\|_{L^3(Q_1)}
    \|u_{n,j}-u_j\|_{L^3(Q_1)}
\\
    &\quad+
    \|\eta\|_{L^\infty}
    \|u_{n,i}-u_i\|_{L^3(Q_1)}
    \|u_j\|_{L^3(Q_1)}.
\end{aligned}
\]
The right-hand side tends to zero.  If also \(E_n\to E\) in \(X_{\src}\), then
\[
    \eta u_{n,i}u_{n,j}+E_{n,ij}
    \to
    \eta u_i u_j+E_{ij}
    \quad\text{in }L^{3/2}(Q_1).
\]
The product of the two compact closures is compact, and the continuous image
of this compact product has compact closure in \(X_{\src}\).  Therefore
\Cref{thm:source-level-compactness-criterion} applies.
\end{proof}

\begin{remark}[Finite amplitude is not compactness]
The finite-amplitude bound \(\|u\|_{L^3(Q_1)}\le M_U\) gives boundedness,
not compactness.  \Cref{thm:strong-compactness-velocity-residual} requires
strong compactness, or another compactness mechanism replacing it.
\end{remark}

\begin{theorem}[Regularity compactness criterion]
\label{thm:regularity-translation-compactness}
Assume the selected source family \(\mathcal F_{\Lambda,0}\) is bounded in a
space compactly embedded into \(X_{\src}\).  For example, assume that for some
\(s>0\),
\[
    \mathcal F_{\Lambda,0}
    \subset W^{s,3/2}(Q_1)^{3\times 3}
\]
with a uniform bound.  Then \(\mathcal F_{\Lambda,0}\) has compact closure in
\(X_{\src}\), and hence
\[
    \Delta^{\mathrm{unif}}_{\proj,N}(\mathcal A_\Lambda)\to0.
\]
The same conclusion holds under any Kolmogorov--Riesz compactness hypothesis
that gives precompactness in \(L^{3/2}(Q_1)^{3\times3}\).
\end{theorem}

\begin{proof}
Since \(Q_1\) is bounded, the Rellich--Kondrachov compactness theorem gives
compact embedding of \(W^{s,3/2}(Q_1)\) into \(L^{3/2}(Q_1)\) for
\(s>0\).  Thus a uniformly bounded family in the displayed Sobolev space has
compact closure in \(X_{\src}\).  The projection-tail conclusion follows from
\Cref{thm:source-level-compactness-criterion}.  The Kolmogorov--Riesz variant
is exactly the corresponding compactness criterion in \(L^{3/2}\), followed
by the same source-level argument.
\end{proof}

\begin{remark}[Status]
\Cref{thm:regularity-translation-compactness} is a compactness criterion, not
a Navier--Stokes regularity theorem.  This paper does not prove the
required regularity or translation compactness from the equations.
\end{remark}

\subsubsection{Baseline closure with uniform projection tail}

\begin{theorem}[Baseline closure with uniform projection tail]
\label{thm:baseline-closure-uniform-projection-tail}
Assume the baseline coordinate visibility and finite-amplitude hypotheses
above, the pressure-natural harmonic approximation
theorem, compactness of the selected clean pressure image
\(\mathcal G_{\Lambda,0}\), and strong convergence with uniform boundedness
of \(P^{\cl}_{\prs,N}\) on \(Y_{\prs}\).  Then, for every
\(D\in\mathcal A_\Lambda\),
\[
\begin{aligned}
    \Dist^{\sharp,\tail}_{\loc,\intg,\tail}(D,\Gammaadm)
    &\le
    C_{\tail}
    \bigl[
        (1+C_{\tail/0})
        \Dist_{\loc,\intg,0}(D,\Gammaadm)
        +
        \delta_0
        +
        \Delta_{\tail/0}
    \bigr]
\\
    &\quad
    +
    \alpha_{\proj}
    \Delta^{\mathrm{unif}}_{\proj,N}(\mathcal A_\Lambda)
    +
    \alpha_{\harm}\Delta_{\harm,M}^{(3/2)}.
\end{aligned}
\]
Moreover,
\[
    \Delta^{\mathrm{unif}}_{\proj,N}(\mathcal A_\Lambda)\to0
    \quad\text{as }N\to\infty,
\]
and
\[
    \Delta_{\harm,M}^{(3/2)}
    \le
    C_{\harm,3/2}
    \left(\frac34\right)^M
    \|p_{\harm,D}\|_{Y_{\harm}}.
\]
\end{theorem}

\begin{proof}
The baseline pressure-tail closure theorem
\Cref{thm:baseline-pressure-tail-closure} gives the same estimate with the
package-wise projection error \(\Delta_{\proj,N}\).  For
\(D\in\mathcal A_\Lambda\), the package-wise clean projection tail is
bounded by the supremum defining
\(\Delta^{\mathrm{unif}}_{\proj,N}(\mathcal A_\Lambda)\).  This gives
the displayed closure estimate.  The convergence of the uniform projection
tail is \Cref{thm:uniform-projection-tail-compact-image}.  The harmonic tail bound follows from \Cref{cor:normalized-L32-harmonic-tail}.
\end{proof}

\begin{remark}[Remaining non-explicit inputs]
After this theorem, the remaining non-explicit quantities in the finite-window
baseline pressure-tail closure are the baseline visibility constants, the
finite-amplitude constants, and the chosen compactness class.  The older
baseline distance alone is not claimed to supply these quantities
automatically.
\end{remark}

\begin{corollary}[Finite-window approximation accuracy]
\label{cor:finite-window-epsilon-accuracy}
Assume the hypotheses of
\Cref{thm:baseline-closure-uniform-projection-tail}.  If, in addition, the
selected harmonic remainders satisfy a uniform bound
\[
    \sup_{D\in\mathcal A_\Lambda}\|p_{\harm,D}\|_{Y_{\harm}}
    \le H_\Lambda<\infty,
\]
then for every \(\varepsilon>0\) there exist \(N\) and \(M\) such that,
uniformly for \(D\in\mathcal A_\Lambda\),
\[
    \alpha_{\proj}
    \Delta^{\mathrm{unif}}_{\proj,N}(\mathcal A_\Lambda)
    +
    \alpha_{\harm}\Delta_{\harm,M}^{(3/2)}
    \le
    \varepsilon.
\]
Consequently
\[
\begin{aligned}
    \Dist^{\sharp,\tail}_{\loc,\intg,\tail}(D,\Gammaadm)
    &\le
    C_{\tail}
    \bigl[
        (1+C_{\tail/0})
        \Dist_{\loc,\intg,0}(D,\Gammaadm)
        +
        \delta_0
        +
        \Delta_{\tail/0}
    \bigr]
    +
    \varepsilon .
\end{aligned}
\]
\end{corollary}

\begin{proof}
By \Cref{thm:uniform-projection-tail-compact-image}, choose \(N\) so large
that
\[
    \alpha_{\proj}
    \Delta^{\mathrm{unif}}_{\proj,N}(\mathcal A_\Lambda)
    \le \varepsilon/2.
\]
The uniform harmonic bound gives
\[
    \alpha_{\harm}\Delta_{\harm,M}^{(3/2)}
    \le
    \alpha_{\harm}
    C_{\harm,3/2}
    \left(\frac34\right)^M
    H_\Lambda.
\]
Choose \(M\) so large that the right-hand side is at most \(\varepsilon/2\).
Substitute these choices into
\Cref{thm:baseline-closure-uniform-projection-tail}.
\end{proof}

\subsection{Conditional local-to-clean assembly}
\label{sec:conditional-assembly}

This section assembles the finite-window modules proved above with external clean and local-to-clean transfer inputs.  The result is a conditional
baseline local-to-clean detection theorem in the older baseline geometry.  No
pressure/tax coercivity, scale-uniform transfer, or Navier--Stokes regularity
is proved here.

\subsubsection{Imported assembly inputs}

\begin{assumption}[Clean anti-phantom gap]
\label{ass:clean-antiphantom-gap}
There is a clean detector \(M^{\mathrm{comp}}_\Lambda\), a clean quotient
distance
\[
    \operatorname{dist}_{\cl}(d,\Gamma^{\cl}_\Lambda),
\]
and a constant \(\mu^{\mathrm{comp}}_\Lambda>0\) such that, for every clean
package \(d\),
\[
    M^{\mathrm{comp}}_\Lambda(d)
    \ge
    \mu^{\mathrm{comp}}_\Lambda
    \operatorname{dist}_{\cl}(d,\Gamma^{\cl}_\Lambda).
\]
This is imported from the clean finite-window anti-phantom framework and is
not reproved here.
\end{assumption}

\begin{assumption}[Chart visibility into the clean quotient]
\label{ass:chart-visibility-clean}
There is a local-to-clean map \(\Theta_\Lambda\) and constants
\(\lambda_G>0\), \(\delta_G\ge0\) such that
\[
    \operatorname{dist}_{\cl}
    (\Theta_\Lambda D,\Gamma^{\cl}_\Lambda)
    \ge
    \lambda_G
    \Dist^{\sharp,\tail}_{\loc,\intg,\tail}(D,\Gammaadm)
    -
    \delta_G .
\]
The enhanced-tail geometry dominates the older baseline geometry:
\[
    \Dist^{\sharp,\tail}_{\loc,\intg,\tail}(D,\Gammaadm)
    \ge
    \Dist_{\loc,\intg,0}(D,\Gammaadm).
\]
Consequently,
\[
    \operatorname{dist}_{\cl}
    (\Theta_\Lambda D,\Gamma^{\cl}_\Lambda)
    \ge
    \lambda_G
    \Dist_{\loc,\intg,0}(D,\Gammaadm)
    -
    \delta_G .
\]
\end{assumption}

\begin{assumption}[Detector domination]
\label{ass:detector-domination}
The localized detector \(M^{\loc}_\Lambda\) dominates the clean detector up
to the finite-window residual-budget error:
\[
    M^{\loc}_\Lambda(D)
    \ge
    M^{\mathrm{comp}}_\Lambda(\Theta_\Lambda D)
    -
    \mathrm{Err}_\Lambda(D).
\]
This is the imported residual-budget detector interface.
\end{assumption}

\begin{assumption}[Residual-budget control]
\label{ass:residual-budget-control}
There are constants \(\eta_\Lambda\ge0\) and
\(\Delta_\Lambda^{\mathrm{res}}\ge0\) such that
\[
    \mathrm{Err}_\Lambda(D)
    \le
    \eta_\Lambda
    \Dist^{\sharp,\tail}_{\loc,\intg,\tail}(D,\Gammaadm)
    +
    \Delta_\Lambda^{\mathrm{res}}.
\]
This is a residual-budget assumption, not pressure/tax coercivity.
\end{assumption}

\begin{definition}[Baseline assembly constants]
\label{def:baseline-assembly-constants}
Define
\[
    C_B:=C_{\tail}(1+C_{\tail/0})
\]
and, for a package \(D\),
\[
    B_{N,M}(D)
    :=
    C_{\tail}(\delta_0+\Delta_{\tail/0})
    +
    \alpha_{\proj}
    \Delta^{\mathrm{unif}}_{\proj,N}(\mathcal A_\Lambda)
    +
    \alpha_{\harm}\Delta_{\harm,M}^{(3/2)}(D).
\]
Here
\[
    \Delta_{\harm,M}^{(3/2)}(D)
    =
    C_{\harm,3/2}
    \left(\frac34\right)^M
    \|p_{\harm,D}\|_{Y_{\harm}}.
\]
The baseline pressure-tail closure theorem gives
\[
    \Dist^{\sharp,\tail}_{\loc,\intg,\tail}(D,\Gammaadm)
    \le
    C_B
    \Dist_{\loc,\intg,0}(D,\Gammaadm)
    +
    B_{N,M}(D).
\]
\end{definition}

\subsubsection{Assembly ledger}

\begin{center}
\scriptsize
\begin{tabular}{@{}llll@{}}
\hline
Module & Input & Output & Error \\
\hline
Clean gap & \(\mu^{\mathrm{comp}}_\Lambda\) & clean lower bound & \(0\)\\
Chart & \(\lambda_G,\delta_G\) & baseline visibility & \(\mu^{\mathrm{comp}}_\Lambda\delta_G\)\\
Residual & \(\eta_\Lambda,\Delta_\Lambda^{\mathrm{res}}\) & detector transfer & \(\eta_\Lambda\Dist^\sharp+\Delta_\Lambda^{\mathrm{res}}\)\\
Baseline closure & \(C_B,B_{N,M}\) & tail by baseline & \(\eta_\Lambda B_{N,M}\)\\
Approximation & \(N,M\) & explicit tails & \(\alpha_{\proj}\Delta_{\mathrm{proj}}+\alpha_{\harm}\Delta_{\harm}\)\\
\hline
\end{tabular}
\end{center}

\subsubsection{Main assembly theorem}

\begin{theorem}[Conditional assembly]
\label{thm:conditional-baseline-assembly}
Assume the clean anti-phantom gap
\Cref{ass:clean-antiphantom-gap}, chart visibility
\Cref{ass:chart-visibility-clean}, detector domination
\Cref{ass:detector-domination}, residual-budget control
\Cref{ass:residual-budget-control}, and baseline pressure-tail closure with
uniform projection-tail error.  Then
\[
    M^{\loc}_\Lambda(D)
    \ge
    c_{\Lambda,0}
    \Dist_{\loc,\intg,0}(D,\Gammaadm)
    -
    \mathfrak E_{\Lambda,0}(D;N,M),
\]
where
\[
    c_{\Lambda,0}
    :=
    \mu^{\mathrm{comp}}_\Lambda\lambda_G
    -
    \eta_\Lambda C_B
\]
and
\[
    \mathfrak E_{\Lambda,0}(D;N,M)
    :=
    \mu^{\mathrm{comp}}_\Lambda\delta_G
    +
    \eta_\Lambda B_{N,M}(D)
    +
    \Delta_\Lambda^{\mathrm{res}}.
\]
Equivalently,
\[
\begin{aligned}
    \mathfrak E_{\Lambda,0}(D;N,M)
    &=
    \mu^{\mathrm{comp}}_\Lambda\delta_G
    +
    \eta_\Lambda
    \bigl[
        C_{\tail}(\delta_0+\Delta_{\tail/0})
        +
        \alpha_{\proj}
        \Delta^{\mathrm{unif}}_{\proj,N}(\mathcal A_\Lambda)
\\
    &\qquad\qquad
        +
        \alpha_{\harm}
        \Delta_{\harm,M}^{(3/2)}(D)
    \bigr]
    +
    \Delta_\Lambda^{\mathrm{res}}.
\end{aligned}
\]
If \(c_{\Lambda,0}>0\), then the local detector has a positive finite-window
lower bound in the older baseline geometry, up to the explicit assembly
error.
\end{theorem}

\begin{proof}
By detector domination,
\[
    M^{\loc}_\Lambda(D)
    \ge
    M^{\mathrm{comp}}_\Lambda(\Theta_\Lambda D)
    -
    \mathrm{Err}_\Lambda(D).
\]
The clean anti-phantom gap gives
\[
    M^{\mathrm{comp}}_\Lambda(\Theta_\Lambda D)
    \ge
    \mu^{\mathrm{comp}}_\Lambda
    \operatorname{dist}_{\cl}
    (\Theta_\Lambda D,\Gamma^{\cl}_\Lambda).
\]
Using chart visibility and the monotonicity route into the baseline geometry,
\[
    M^{\mathrm{comp}}_\Lambda(\Theta_\Lambda D)
    \ge
    \mu^{\mathrm{comp}}_\Lambda\lambda_G
    \Dist_{\loc,\intg,0}(D,\Gammaadm)
    -
    \mu^{\mathrm{comp}}_\Lambda\delta_G.
\]
On the other hand, residual-budget control and
\Cref{def:baseline-assembly-constants} imply
\[
\begin{aligned}
    \mathrm{Err}_\Lambda(D)
    &\le
    \eta_\Lambda
    \Dist^{\sharp,\tail}_{\loc,\intg,\tail}(D,\Gammaadm)
    +
    \Delta_\Lambda^{\mathrm{res}}
\\
    &\le
    \eta_\Lambda C_B
    \Dist_{\loc,\intg,0}(D,\Gammaadm)
    +
    \eta_\Lambda B_{N,M}(D)
    +
    \Delta_\Lambda^{\mathrm{res}}.
\end{aligned}
\]
Substituting this upper bound for the residual error into the detector
domination inequality and collecting the baseline-distance coefficient gives
the claimed estimate.
\end{proof}

\begin{remark}[Status of the assembly theorem]
\Cref{thm:conditional-baseline-assembly} is a conditional finite-window
assembly theorem.  It does not prove pressure/tax coercivity.  It does not
prove that the residual-budget assumptions hold for all suitable weak
solutions.  It does not prove scale-uniformity or Navier--Stokes regularity.
It only states that, if the clean gap, chart visibility, residual-budget
transfer, baseline visibility, and compact projection-tail hypotheses hold on
the same finite-window class, then the localized detector controls the older
baseline quotient distance up to explicit errors.
\end{remark}

\begin{corollary}[Finite-window positive baseline transfer]
\label{cor:positive-baseline-transfer}
If
\[
    \mu^{\mathrm{comp}}_\Lambda\lambda_G
    >
    \eta_\Lambda C_B,
\]
then \(c_{\Lambda,0}>0\).  In that case, every package satisfying
\[
    \Dist_{\loc,\intg,0}(D,\Gammaadm)
    >
    \frac{\mathfrak E_{\Lambda,0}(D;N,M)}{c_{\Lambda,0}}
\]
also satisfies
\[
    M^{\loc}_\Lambda(D)>0.
\]
\end{corollary}

\begin{proof}
The strict inequality is exactly the statement \(c_{\Lambda,0}>0\).  Dividing
the lower bound in \Cref{thm:conditional-baseline-assembly} by this positive
constant shows that the right-hand side is positive whenever the baseline
distance is larger than the displayed threshold.
\end{proof}

\begin{remark}[Detection, not regularity]
\Cref{cor:positive-baseline-transfer} is a finite-window detection statement:
a non-gauge baseline defect above the explicit finite-window error threshold
must be visible to the localized detector.  It is not a regularity theorem.
\end{remark}

\begin{corollary}[Finite-window assembly accuracy]
\label{cor:assembly-accuracy}
Assume the hypotheses of \Cref{thm:conditional-baseline-assembly}.  Assume
also that the selected clean pressure image is compact and that the selected
harmonic remainders satisfy
\[
    \sup_{D\in\mathcal A_\Lambda}
    \|p_{\harm,D}\|_{Y_{\harm}}
    \le H_\Lambda<\infty.
\]
Then, for every \(\varepsilon>0\), one can choose \(N\) and \(M\) such that
\[
    \alpha_{\proj}
    \Delta^{\mathrm{unif}}_{\proj,N}(\mathcal A_\Lambda)
    +
    \alpha_{\harm}
    \Delta_{\harm,M}^{(3/2)}(D)
    \le
    \varepsilon
\]
for every \(D\in\mathcal A_\Lambda\).  Consequently
\[
\begin{aligned}
    M^{\loc}_\Lambda(D)
    &\ge
    c_{\Lambda,0}
    \Dist_{\loc,\intg,0}(D,\Gammaadm)
\\
    &\quad
    -
    \left\{
        \mu^{\mathrm{comp}}_\Lambda\delta_G
        +
        \eta_\Lambda
        \bigl[
            C_{\tail}(\delta_0+\Delta_{\tail/0})
            +
            \varepsilon
        \bigr]
        +
        \Delta_\Lambda^{\mathrm{res}}
    \right\}.
\end{aligned}
\]
\end{corollary}

\begin{proof}
The compact pressure image and
\Cref{thm:uniform-projection-tail-compact-image} give
\[
    \Delta^{\mathrm{unif}}_{\proj,N}(\mathcal A_\Lambda)\to0.
\]
The uniform harmonic bound gives uniform convergence of
\[
    C_{\harm,3/2}\left(\frac34\right)^M
    \|p_{\harm,D}\|_{Y_{\harm}}
\]
to zero on \(\mathcal A_\Lambda\).  Choose \(N\) and \(M\) so that the two
weighted approximation errors have sum at most \(\varepsilon\), then insert
this bound into \(\mathfrak E_{\Lambda,0}(D;N,M)\) in
\Cref{thm:conditional-baseline-assembly}.
\end{proof}

\subsection{Compact quotient pressure/tax detection}
\label{sec:pressure-tax-coercivity}

This section records a finite-window pressure/tax coercivity criterion.  The
previous section assembled the clean gap, chart visibility, residual-budget
transfer, baseline visibility, and projection-tail uniformity into the
conditional lower bound
\[
    M^{\loc}_\Lambda(D)
    \ge
    c_{\Lambda,0}\Dist_{\loc,\intg,0}(D,\Gammaadm)
    -
    \mathfrak E_{\Lambda,0}(D;N,M).
\]
The present section asks a different finite-window question: when does a
normalized pressure/tax detector itself see non-gauge baseline defects?  No
scale-uniform estimate, singularity exclusion, or Navier--Stokes regularity
claim is made.

\subsubsection{Motivation and endpoint}

We use the older baseline quotient distance
\[
    \Dist_{\loc,\intg,0}(D,\Gammaadm)
    =
    \inf_{\zeta\in\Gammaadm}
    \|D-\zeta\|_{\loc,\intg,0},
\]
with the conservative admissible gauge convention \(\zeta_u=0\).  The
baseline visibility and finite-amplitude results above control the
pressure-tail and split excess on a same-gauge representative, while
\Cref{thm:uniform-projection-tail-compact-image} gives uniform clean
projection-tail convergence under compactness of the selected clean pressure
image.  The endpoint for this section is
\Cref{thm:conditional-baseline-assembly}; it is not reproved here.

\subsubsection{Normalized pressure/tax detector}

\begin{definition}[Normalized pressure/tax detector]
\label{def:normalized-pressure-tax-detector}
Fix nonnegative finite-window weights
\[
    \beta_{\mathrm{rep}},\quad
    \beta_{\prs},\quad
    \beta_{\mathrm{flux}},\quad
    \beta_{\mathrm{gate}},\quad
    \beta_{\mathrm{slack}}.
\]
When a channel is used in a kernel-free coercivity theorem below, its weight
is assumed positive; zero-weight channels are retained only as bookkeeping
coordinates unless an explicit detector-zero-set assumption is imposed.
The normalized pressure/tax detector is the nonnegative functional
\[
\begin{aligned}
    \mathfrak M^{\mathrm{tax}}_\Lambda(D)
    &:=
    \|O^0_\Lambda D\|_{\mathcal O}
    +
    \beta_{\mathrm{rep}}\operatorname{Rep}_\Lambda(D)
    +
    \beta_{\prs}\operatorname{Tax}^{\prs}_\Lambda(D)
\\
    &\quad+
    \beta_{\mathrm{flux}}\operatorname{Tax}^{\mathrm{flux}}_\Lambda(D)
    +
    \beta_{\mathrm{gate}}\operatorname{Tax}^{\mathrm{gate}}_\Lambda(D)
    +
    \beta_{\mathrm{slack}}\operatorname{Tax}^{\mathrm{slack}}_\Lambda(D).
\end{aligned}
\]
Here \(\|O^0_\Lambda D\|_{\mathcal O}\) is the older-baseline observable
part of the localized package, \(\operatorname{Rep}_\Lambda\) measures
finite-window reproduction drift, \(\operatorname{Tax}^{\prs}_\Lambda\)
measures pressure-source or pressure-tail cost,
\(\operatorname{Tax}^{\mathrm{flux}}_\Lambda\) measures nonlinear flux or
cutoff leakage cost, and
\(\operatorname{Tax}^{\mathrm{gate}}_\Lambda\) and
\(\operatorname{Tax}^{\mathrm{slack}}_\Lambda\) measure gate, slack, or
admissibility-threshold failures.
\end{definition}

\begin{remark}[Accounting status]
\(\mathfrak M^{\mathrm{tax}}_\Lambda\) is a normalized finite-window
accounting object.  Its components are nonnegative.  This paper does not
claim that any individual component is automatically coercive for
Navier--Stokes solutions.
\end{remark}

\begin{definition}[Normalized tax class]
\label{def:normalized-tax-class}
Let \(\mathcal A^{\mathrm{tax}}_\Lambda\) be a finite-window admissible class
of packages satisfying the pressure-natural harmonic geometry, the
conservative admissible gauge convention, the baseline visibility
hypotheses, the finite-amplitude bound, a compact clean pressure image or
another projection-tail uniformity mechanism, and finiteness of all detector
components in \Cref{def:normalized-pressure-tax-detector}.  When compactness
of the quotient is needed below, it is assumed explicitly: the normalized
quotient
\[
    \mathcal A^{\mathrm{tax}}_\Lambda/\Gammaadm
\]
has compact unit sphere in the older baseline quotient distance, or the class
is a finite-dimensional reduced model with the same property.
\end{definition}

\begin{remark}[No automatic compactness]
The compact quotient property in \Cref{def:normalized-tax-class} is a
finite-window structural assumption.  It is not derived from suitable weak
solutions, bounded \(L^3\) velocity, or the Navier--Stokes equations here.
\end{remark}

\subsubsection{No-free-coercivity lemma}

\begin{lemma}[No pressure/tax coercivity without kernel exclusion]
\label{lem:no-pressure-tax-coercivity-without-kernel}
Suppose there exists a non-gauge direction \(H\) such that
\[
    \Dist_{\loc,\intg,0}(H,\Gammaadm)>0,
\]
but
\[
    O^0_\Lambda H=0,\qquad
    \operatorname{Rep}_\Lambda(H)=0,
\]
and
\[
    \operatorname{Tax}^{\prs}_\Lambda(H)
    =
    \operatorname{Tax}^{\mathrm{flux}}_\Lambda(H)
    =
    \operatorname{Tax}^{\mathrm{gate}}_\Lambda(H)
    =
    \operatorname{Tax}^{\mathrm{slack}}_\Lambda(H)
    =
    0.
\]
Assume these detector components and the baseline quotient distance are
positively homogeneous along the ray \(\{\lambda H:\lambda\ge0\}\).  Then no
zero-error coercive estimate
\[
    \mathfrak M^{\mathrm{tax}}_\Lambda(D)
    \ge
    \kappa\Dist_{\loc,\intg,0}(D,\Gammaadm)
\]
can hold with \(\kappa>0\) on any class containing this ray.
\end{lemma}

\begin{proof}
For \(\lambda>0\), set \(D=\lambda H\).  By the vanishing assumptions and
positive homogeneity of the detector components,
\[
    \mathfrak M^{\mathrm{tax}}_\Lambda(\lambda H)=0.
\]
By positive homogeneity of the quotient distance,
\[
    \Dist_{\loc,\intg,0}(\lambda H,\Gammaadm)
    =
    \lambda\Dist_{\loc,\intg,0}(H,\Gammaadm)>0.
\]
The alleged estimate would therefore give
\[
    0
    =
    \mathfrak M^{\mathrm{tax}}_\Lambda(\lambda H)
    \ge
    \kappa\lambda\Dist_{\loc,\intg,0}(H,\Gammaadm)
    >
    0,
\]
a contradiction.
\end{proof}

\begin{remark}[Purpose]
\Cref{lem:no-pressure-tax-coercivity-without-kernel} is a sanity check.  A
positive pressure/tax coercivity theorem requires a kernel-free or visibility
assumption; it cannot be obtained merely by naming tax terms.
\end{remark}

\subsubsection{Pressure/tax kernel and kernel-free condition}

\begin{definition}[Normalized pressure/tax kernel]
\label{def:pressure-tax-kernel}
The normalized pressure/tax kernel is
\[
\begin{aligned}
    \mathcal K^{\mathrm{tax}}_\Lambda
    :=
    \{\,D\in\mathcal A^{\mathrm{tax}}_\Lambda:\,
    &O^0_\Lambda D=0,\ 
    \operatorname{Rep}_\Lambda(D)=0,\
    \operatorname{Tax}^{\prs}_\Lambda(D)=0,
\\
    &\operatorname{Tax}^{\mathrm{flux}}_\Lambda(D)=0,\
    \operatorname{Tax}^{\mathrm{gate}}_\Lambda(D)=0,\
    \operatorname{Tax}^{\mathrm{slack}}_\Lambda(D)=0
    \,\}.
\end{aligned}
\]
The pressure/tax kernel-free condition is
\[
    \mathcal K^{\mathrm{tax}}_\Lambda\subset\Gammaadm.
\]
Equivalently, simultaneous vanishing of the observation, reproduction, and
tax components implies that the package is an admissible gauge direction.
\end{definition}

\begin{remark}[Analogy with the clean gap]
The kernel-free condition is the finite-window pressure/tax analogue of the
clean anti-phantom kernel-free condition.  It is an assumption in this
section, not a theorem about all localized Navier--Stokes packages.
\end{remark}

\subsubsection{Compact quotient coercivity theorem}

\begin{theorem}[Compact quotient pressure/tax coercivity]
\label{thm:compact-quotient-pressure-tax-coercivity}
Assume that:
\begin{enumerate}[label=(\roman*),leftmargin=*]
    \item \(\mathcal A^{\mathrm{tax}}_\Lambda\) is stable under admissible
    quotient normalization and positive scalar normalization;
    \item the unit baseline quotient sphere
    \[
        S_{\Lambda,0}
        :=
        \{D\in\mathcal A^{\mathrm{tax}}_\Lambda:
        \Dist_{\loc,\intg,0}(D,\Gammaadm)=1\}
    \]
    is compact modulo the admissible gauge;
    \item \(\mathfrak M^{\mathrm{tax}}_\Lambda\) is lower semicontinuous on
    this quotient;
    \item \(\mathfrak M^{\mathrm{tax}}_\Lambda\) is positively homogeneous
    under the quotient normalization; and
    \item the pressure/tax kernel-free condition holds, and every detector
    channel appearing in \(\mathcal K^{\mathrm{tax}}_\Lambda\) has positive
    weight in \(\mathfrak M^{\mathrm{tax}}_\Lambda\).  Equivalently, one may
    replace this positivity requirement by the explicit zero-set condition
    \[
        \{D\in\mathcal A^{\mathrm{tax}}_\Lambda:
        \mathfrak M^{\mathrm{tax}}_\Lambda(D)=0\}
        \subset
        \mathcal K^{\mathrm{tax}}_\Lambda .
    \]
\end{enumerate}
Then
\[
    \mu^{\mathrm{tax}}_\Lambda
    :=
    \inf_{D\in S_{\Lambda,0}}
    \mathfrak M^{\mathrm{tax}}_\Lambda(D)
    >
    0.
\]
Consequently, for every \(D\in\mathcal A^{\mathrm{tax}}_\Lambda\),
\[
    \mathfrak M^{\mathrm{tax}}_\Lambda(D)
    \ge
    \mu^{\mathrm{tax}}_\Lambda
    \Dist_{\loc,\intg,0}(D,\Gammaadm).
\]
\end{theorem}

\begin{proof}
If the infimum on \(S_{\Lambda,0}\) were zero, compactness modulo gauge would
give a minimizing sequence converging, after passing to quotient
representatives, to some \(D_*\in S_{\Lambda,0}\).  Lower
semicontinuity gives
\[
    \mathfrak M^{\mathrm{tax}}_\Lambda(D_*)=0.
\]
By the positive-weight hypothesis, or equivalently by the displayed zero-set
condition, vanishing of \(\mathfrak M^{\mathrm{tax}}_\Lambda\) implies that
all channels defining \(\mathcal K^{\mathrm{tax}}_\Lambda\) vanish.  Thus
\(D_*\in\mathcal K^{\mathrm{tax}}_\Lambda\).  By the kernel-free
condition, \(D_*\in\Gammaadm\).  This contradicts
\[
    \Dist_{\loc,\intg,0}(D_*,\Gammaadm)=1.
\]
Therefore \(\mu^{\mathrm{tax}}_\Lambda>0\) on the unit quotient sphere.

Now let \(D\in\mathcal A^{\mathrm{tax}}_\Lambda\).  If
\(\Dist_{\loc,\intg,0}(D,\Gammaadm)=0\), the desired estimate is trivial.
Otherwise set
\[
    r:=\Dist_{\loc,\intg,0}(D,\Gammaadm)>0
\]
and normalize the quotient representative by \(r^{-1}\).  Stability under
normalization gives \(r^{-1}D\in S_{\Lambda,0}\) modulo gauge.  Positive
homogeneity gives
\[
    r^{-1}\mathfrak M^{\mathrm{tax}}_\Lambda(D)
    =
    \mathfrak M^{\mathrm{tax}}_\Lambda(r^{-1}D)
    \ge
    \mu^{\mathrm{tax}}_\Lambda.
\]
Multiplying by \(r\) proves the estimate.
\end{proof}

\begin{remark}[Finite-window status]
\Cref{thm:compact-quotient-pressure-tax-coercivity} is a compactness and
kernel-freeness theorem on a fixed finite-window quotient.  It is not a
scale-uniform Navier--Stokes coercivity theorem.
\end{remark}

\subsubsection{Additive-error coercivity theorem}

\begin{theorem}[Additive-error pressure/tax coercivity]
\label{thm:additive-error-pressure-tax-coercivity}
Assume there is an exact model detector
\(\widetilde{\mathfrak M}^{\mathrm{tax}}_\Lambda\) satisfying
\Cref{thm:compact-quotient-pressure-tax-coercivity} with gap
\(\mu^{\mathrm{tax}}_\Lambda>0\).  Assume also that the realized detector
satisfies
\[
    \mathfrak M^{\mathrm{tax}}_\Lambda(D)
    +
    \Delta^{\mathrm{model}}_\Lambda
    \ge
    \widetilde{\mathfrak M}^{\mathrm{tax}}_\Lambda(D)
\]
for every package in the class.  Then
\[
    \mathfrak M^{\mathrm{tax}}_\Lambda(D)
    \ge
    \mu^{\mathrm{tax}}_\Lambda
    \Dist_{\loc,\intg,0}(D,\Gammaadm)
    -
    \Delta^{\mathrm{model}}_\Lambda.
\]
\end{theorem}

\begin{proof}
The exact model coercivity theorem gives
\[
    \widetilde{\mathfrak M}^{\mathrm{tax}}_\Lambda(D)
    \ge
    \mu^{\mathrm{tax}}_\Lambda
    \Dist_{\loc,\intg,0}(D,\Gammaadm).
\]
Combining this lower bound with
\[
    \mathfrak M^{\mathrm{tax}}_\Lambda(D)
    \ge
    \widetilde{\mathfrak M}^{\mathrm{tax}}_\Lambda(D)
    -
    \Delta^{\mathrm{model}}_\Lambda
\]
gives the claim.
\end{proof}

\subsubsection{Sufficient kernel-free criteria}

\begin{proposition}[Observation-plus-tax injectivity]
\label{prop:observation-tax-injectivity}
If the combined map
\[
    D\mapsto
    \bigl(
        O^0_\Lambda D,\,
        \operatorname{Rep}_\Lambda(D),\,
        \operatorname{Tax}^{\prs}_\Lambda(D),\,
        \operatorname{Tax}^{\mathrm{flux}}_\Lambda(D),\,
        \operatorname{Tax}^{\mathrm{gate}}_\Lambda(D),\,
        \operatorname{Tax}^{\mathrm{slack}}_\Lambda(D)
    \bigr)
\]
is injective on the quotient
\(\mathcal A^{\mathrm{tax}}_\Lambda/\Gammaadm\), then the pressure/tax
kernel-free condition holds.
\end{proposition}

\begin{proof}
If all components vanish on \(D\), then \(D\) and the zero gauge
class have the same combined image.  Injectivity on the quotient implies
that \(D\) lies in the admissible gauge class.
\end{proof}

\begin{proposition}[Finite-dimensional matrix criterion]
\label{prop:finite-dimensional-matrix-kernel}
In a finite-dimensional reduced package model, suppose the combined
observation, reproduction, and tax map is represented by a finite matrix or a
finite family of nonnegative coordinate functionals \(\mathcal T_\Lambda\).
Then the pressure/tax kernel-free condition is equivalent to
\[
    \ker\mathcal T_\Lambda=\Gammaadm
\]
modulo the chosen finite-dimensional coordinates.
\end{proposition}

\begin{proof}
In finite dimensions, the kernel of the combined detector map is exactly the
set on which all observation, reproduction, and tax components vanish.  Thus
the kernel-free condition is precisely the statement that this nullspace is
the admissible gauge subspace.
\end{proof}

\begin{proposition}[Gate/slack intersection criterion]
\label{prop:gate-slack-intersection}
Assume that the intersection of the vanishing sets satisfies
\[
\begin{aligned}
    &\ker O^0_\Lambda
    \cap
    \ker\operatorname{Rep}_\Lambda
    \cap
    \ker\operatorname{Tax}^{\prs}_\Lambda
    \cap
    \ker\operatorname{Tax}^{\mathrm{flux}}_\Lambda
\\
    &\qquad\cap
    \ker\operatorname{Tax}^{\mathrm{gate}}_\Lambda
    \cap
    \ker\operatorname{Tax}^{\mathrm{slack}}_\Lambda
    \subset
    \Gammaadm .
\end{aligned}
\]
Then the pressure/tax kernel-free condition holds.
\end{proposition}

\begin{proof}
The displayed intersection is exactly the set of packages for which every
component in the normalized pressure/tax detector vanishes.  Containment in
\(\Gammaadm\) is the kernel-free condition.
\end{proof}

\begin{remark}[Use of the criteria]
The preceding propositions are usable finite-window criteria.  They do not
verify kernel-freeness for all Navier--Stokes packages.
\end{remark}

\subsubsection{Coercivity combined with the assembly theorem}

\begin{theorem}[Pressure/tax route to baseline detection]
\label{thm:pressure-tax-combined-with-assembly}
Assume additive-error pressure/tax coercivity:
\[
    \mathfrak M^{\mathrm{tax}}_\Lambda(D)
    \ge
    \mu^{\mathrm{tax}}_\Lambda
    \Dist_{\loc,\intg,0}(D,\Gammaadm)
    -
    \Delta^{\mathrm{model}}_\Lambda.
\]
Assume also that the localized detector dominates the normalized tax detector
up to a finite-window detector-comparison error:
\[
    M^{\loc}_\Lambda(D)
    +
    \Delta^{\mathrm{det}}_\Lambda
    \ge
    a_\Lambda
    \mathfrak M^{\mathrm{tax}}_\Lambda(D)
\]
for some \(a_\Lambda>0\).  Then
\[
    M^{\loc}_\Lambda(D)
    +
    \Delta^{\mathrm{det}}_\Lambda
    \ge
    a_\Lambda\mu^{\mathrm{tax}}_\Lambda
    \Dist_{\loc,\intg,0}(D,\Gammaadm)
    -
    a_\Lambda\Delta^{\mathrm{model}}_\Lambda.
\]
\end{theorem}

\begin{proof}
Multiply the pressure/tax coercivity estimate by \(a_\Lambda\) and use the
detector-comparison inequality.
\end{proof}

\begin{remark}[Supplementary route]
\Cref{thm:pressure-tax-combined-with-assembly} is a supplementary finite-window
route to baseline detection.  It does not replace
\Cref{thm:conditional-baseline-assembly} unless the detector comparison is
also assumed.  Both routes remain conditional.
\end{remark}

\subsubsection{Detection threshold and interpretation}

\begin{corollary}[Pressure/tax detection threshold]
\label{cor:pressure-tax-detection-threshold}
Assume
\[
    \mathfrak M^{\mathrm{tax}}_\Lambda(D)
    \ge
    \mu^{\mathrm{tax}}_\Lambda
    \Dist_{\loc,\intg,0}(D,\Gammaadm)
    -
    \Delta^{\mathrm{model}}_\Lambda
\]
with \(\mu^{\mathrm{tax}}_\Lambda>0\).  If
\[
    \Dist_{\loc,\intg,0}(D,\Gammaadm)
    >
    \frac{\Delta^{\mathrm{model}}_\Lambda}{\mu^{\mathrm{tax}}_\Lambda},
\]
then
\[
    \mathfrak M^{\mathrm{tax}}_\Lambda(D)>0.
\]
\end{corollary}

\begin{proof}
Substituting the displayed strict inequality into the additive-error
coercivity bound gives a positive lower bound for
\(\mathfrak M^{\mathrm{tax}}_\Lambda(D)\).
\end{proof}

\begin{remark}[Detection, not regularity]
\Cref{cor:pressure-tax-detection-threshold} says that a non-gauge baseline
defect above the finite-window modeling threshold must pay a positive
normalized pressure/tax cost.  This is detection, not regularity.
\end{remark}

\subsubsection{Status and limitations}

This normalized pressure/tax coercivity branch proves only a finite-window
compact quotient theorem.  The positive gap depends on compactness,
homogeneity, lower semicontinuity, and kernel-freeness assumptions.  The
paper does not prove that these assumptions hold for all suitable weak
solutions.  It does not derive kernel-freeness from the Navier--Stokes
equations.  It does not prove scale-uniformity.  It does not prove
Navier--Stokes regularity.  It does not solve the Clay problem.

\subsection{Logical status and remaining PDE-facing inputs}
\label{sec:status-limitations}

The present paper is complete as a finite-window conditional framework.  All
quantities appearing in the pressure-tail, baseline-visibility, projection-tail,
assembly, and pressure/tax detection estimates have been defined, and the
functional-analytic approximation mechanisms used in the estimates have been
proved in the fixed geometry.  The remaining inputs are not proof gaps inside
the finite-window theorems; they are structural hypotheses that would need to be
established in a genuinely PDE-facing continuation.

\begin{enumerate}[label=(\roman*),leftmargin=*]
    \item \textbf{Finite amplitude.}  Several estimates use
    \(\norm{u_D}{L^3(Q_1)}\le M_U\) on the conservative representative in order
    to linearize the quadratic pressure source.  This paper does not derive
    such a bound from the Navier--Stokes equations.
    \item \textbf{Baseline visibility.}  The older baseline norm does not
    automatically control velocity amplitude, clean residual source, or
    harmonic pressure coordinates.  Baseline coordinate visibility is an
    explicit same-gauge assumption.
    \item \textbf{Clean-source compactness.}  Uniform projection-tail
    convergence requires compactness, finite-dimensionality, smoothing, or a
    comparable approximation mechanism.  Boundedness of a Navier--Stokes source
    family is not enough.
    \item \textbf{Residual-budget transfer.}  The local-to-clean assembly
    theorem assumes a residual-budget comparison between the localized detector
    and the clean detector.  This is treated as a structural input in the finite-window theorem.
    \item \textbf{Pressure/tax kernel-freeness.}  The compact quotient
    pressure/tax result proves that kernel-freeness implies a positive
    finite-window gap.  It does not prove that the kernel-free condition holds
    for all localized Navier--Stokes packages.
    \item \textbf{Scale-uniformity.}  Every estimate is fixed-window.  The
    paper does not prove scale-uniform moving-window control or a
    regularity theorem.
\end{enumerate}

The natural next steps are therefore sharply separated: prove finite-amplitude
removal or weakening, derive baseline visibility in concrete localized
classes, obtain compactness or effective approximation of Navier--Stokes
pressure-source images, verify residual-budget transfer, and only then study
scale-uniform propagation.

\section{Componentwise Residual-Ledger Closure: Detailed Proofs}\label{app:residual-ledger-details}
\subsection*{Componentwise Residual-Ledger Closure}
\label{part:residual-ledger}

\subsection{Global Notation and Scope}

All cylinders, cutoff functions, reproduction maps, observation norms, and gauge classes are fixed on finite windows.  Constants denoted by \(C\) may change from line to line.  Constants such as \(C_{\prs}(M_U)\), \(C_{\loc}(M_U)\), \(C_{\rep}(M_U)\), and \(C_{\comp}^{[0,K]}(M_U)\) are finite-window constants; they may depend on the chosen geometry, on cutoff and reproduction maps, on Calderon--Zygmund or harmonic-estimate constants, on the finite amplitude bound \(M_U\), and, in the final ledger, on the chain length \(K\).  No scale-uniform or infinite-chain estimate is asserted.

A ``same representative'' statement always means that all residual coordinates in the relevant estimate are evaluated at one selected gauge representative.  The quotient-distance estimates are obtained by combining representative-form residual bounds with a near-minimizer condition for that same representative.  This convention is essential: independently minimizing different residual channels over different gauge representatives does not produce a single residual ledger.

\subsection*{Pressure-Source Residual Absorption}

\subsection{Normalized Local NS Data and Pressure Splitting}
\label{ps:sec:local-geometry}

\subsubsection{Fixed local geometry}

Throughout the paper,
\[
    \Qone=B_1\times(-1,0),
    \qquad
    B_{1/2}\subset B_{3/4}\subset B_1.
\]
We fix a cutoff
\[
    \eta\in C_c^\infty(B_1),
    \qquad
    0\le \eta\le 1,
    \qquad
    \eta\equiv1\quad\text{on }B_{3/4}.
\]
The source and pressure observation spaces are
\[
    \Xsrc
    =
    L^{3/2}\bigl((-1,0);L^{3/2}(B_1)\bigr)^{3\times3},
\]
\[
    \Yprs
    =
    L^{3/2}\bigl((-1,0);L^{3/2}(B_{1/2})\bigr),
\]
and the pressure-natural harmonic space is
\[
    \Yharm
    =
    L^{3/2}\bigl((-1,0);L^{3/2}(B_{3/4})\bigr).
\]

\begin{convention}[Fixed source-to-pressure estimate]
\label{ps:conv:source-to-pressure}
For sources extended by zero outside \(B_1\), the finite-window pressure model
uses the estimate
\[
    \|R_iR_j(F_{ij})\|_{\Yprs}
    \le
    C_{CZ}\|F\|_{\Xsrc}.
\]
No scale-uniform estimate is claimed.
\end{convention}

\begin{definition}[Pressure-admissible local data]
\label{ps:def:pressure-admissible-data}
A pair \((u,p)\) on \(\Qone\) is pressure-admissible if
\[
    u\in L^3(\Qone)^3,
    \qquad
    p\in L^{3/2}(\Qone),
\]
and, for almost every \(t\in(-1,0)\),
\[
    -\Delta p(t,\cdot)
    =
    \partial_i\partial_j(u_i u_j)(t,\cdot)
\]
in the sense of distributions on \(B_1\).  The pressure is understood modulo
time-dependent constants.
\end{definition}

\begin{definition}[Localized active pressure]
\label{ps:def:localized-active-pressure}
For pressure-admissible data, define
\[
    F^{\act}_{ij}:=\eta u_i u_j,
    \qquad
    p^{\act}:=R_iR_j(F^{\act}_{ij}),
\]
with \(F^{\act}\) extended by zero outside \(B_1\).  The harmonic pressure
coordinate is
\[
    p_{\harm}:=p-p^{\act}
    \quad\text{on }B_{3/4}\times(-1,0).
\]
\end{definition}

\begin{proposition}[Pressure splitting and active pressure bound]
\label{ps:prop:pressure-splitting-working}
For pressure-admissible data,
\[
    p^{\act}\in L^{3/2}\bigl((-1,0);L^{3/2}(B_{1/2})\bigr)
\]
and
\[
    \|p^{\act}\|_{\Yprs}
    \le
    C_{CZ}\|\eta u_i u_j\|_{\Xsrc}
    \le
    C_{CZ}\|u\|_{L^3(\Qone)}^2 .
\]
Moreover, for almost every \(t\in(-1,0)\),
\[
    -\Delta p^{\act}(t,\cdot)
    =
    \partial_i\partial_j(u_i u_j)(t,\cdot)
    \quad\text{in }\mathcal D'(B_{3/4}),
\]
and therefore \(p_{\harm}=p-p^{\act}\) is harmonic in \(B_{3/4}\) for almost
every time.
\end{proposition}

\begin{proof}
For almost every fixed time, \(u_i u_j\in L^{3/2}(B_1)\), and therefore
\(F^{\act}_{ij}=\eta u_i u_j\), extended by zero outside \(B_1\), belongs to
\(L^{3/2}(\R^3)\).  By Convention~\ref{ps:conv:source-to-pressure},
\[
    \|p^{\act}\|_{\Yprs}
    =
    \|R_iR_j(F^{\act}_{ij})\|_{\Yprs}
    \le
    C_{CZ}\|F^{\act}\|_{\Xsrc}.
\]
Since \(0\le \eta\le1\), Holder's inequality gives
\[
    \|F^{\act}\|_{\Xsrc}
    \le
    C\|u\otimes u\|_{L^{3/2}(\Qone)}
    \le
    C\|u\|_{L^3(\Qone)}^2.
\]
The harmless finite-component constant is absorbed into \(C_{CZ}\).

It remains to verify the local Poisson equation.  With the usual Fourier
normalization of the Riesz transforms,
\[
    -\Delta R_iR_jF_{ij}
    =
    \partial_i\partial_jF_{ij}
    \quad\text{in }\mathcal D'(\R^3).
\]
Let \(\varphi\in C_c^\infty(B_{3/4})\).  Since \(\eta\equiv1\) on
\(B_{3/4}\),
\[
\begin{aligned}
    \langle \partial_i\partial_j(\eta u_i u_j),\varphi\rangle
    &=
    \int_{B_1}\eta u_i u_j\,\partial_i\partial_j\varphi\,dx  \\
    &=
    \int_{B_1}u_i u_j\,\partial_i\partial_j\varphi\,dx
    =
    \langle \partial_i\partial_j(u_i u_j),\varphi\rangle .
\end{aligned}
\]
Thus \(-\Delta p^{\act}=\partial_i\partial_j(u_i u_j)\) in
\(\mathcal D'(B_{3/4})\).  The pressure-admissibility assumption gives the
same equation for \(p\) on \(B_1\), so
\[
    -\Delta(p-p^{\act})=0
    \quad\text{in }\mathcal D'(B_{3/4})
\]
for almost every time.  This is precisely the asserted harmonicity of
\(p_{\harm}\).  Time-dependent additive constants in \(p\) only add harmonic
constants to \(p_{\harm}\).
\end{proof}

\begin{remark}[Status]
Proposition~\ref{ps:prop:pressure-splitting-working} records the pressure-splitting endpoint
needed by the package class.  It is not a global pressure decomposition and
does not assert any regularity beyond the stated integrability and harmonicity
on the fixed interior ball.
\end{remark}

\subsection{Sharp Localized Package Class}
\label{ps:sec:sharp-packages}

\begin{definition}[Sharp package datum]
\label{ps:def:sharp-localized-package}
A sharp package datum is a tuple
\[
    \calD_\Lambda(u,p)
    =
    \bigl(
        u,\,
        U,\,
        R,\,
        F^{\act},\,
        F^{\modsrc},\,
        E_F,\,
        p^{\act},\,
        p_{\harm},\,
        \Pi,\,
        \Phi,\,
        T,\,
        s
    \bigr).
\]
Here \(u\) is the physical localized velocity, \(U\) is a selected clean or
modeled velocity coordinate, \(R\) is a Reynolds/covariance coordinate,
\(F^{\act}_{ij}=\eta u_i u_j\), \(F^{\modsrc}\) is the model pressure source,
\(E_F\) is an additional source residual coordinate,
\(p^{\act}=R_iR_j(F^{\act}_{ij})\),
\(p_{\harm}\) is the harmonic pressure coordinate, and
\(\Pi,\Phi,T,s\) are finite-window flux, energy, trace, and slack
coordinates.
Below we write
\[
    E^{\src}_{\calD}:=E_{F,\calD}
\]
for this additional residual source coordinate.
\end{definition}

\begin{convention}[Model-source convention A]
\label{ps:conv:model-source-A}
The main model-source convention is
\[
    F^{\modsrc}_{ij}
    :=
    \eta(U_iU_j+R_{ij}).
\]
The exposed active covariance mismatch is
\[
    \mathcal C^0_{ij}
    :=
    F^{\act}_{ij}-F^{\modsrc}_{ij}
    =
    \eta\bigl(u_i u_j-U_iU_j-R_{ij}\bigr).
\]
\end{convention}

\begin{remark}[Residual-source variant]
One may instead use \(F^{\modsrc}_{ij}=\eta u_i u_j+E^{\cl}_{F,ij}\), so that
the residual is carried by \(E_F^{\cl}\).  This paper uses
Convention~\ref{ps:conv:model-source-A} first because it exposes the covariance mismatch
\(u_i u_j-U_iU_j-R_{ij}\).  The additional coordinate \(E^{\src}=E_F\) is kept
separate from this covariance mismatch; it records source error not modeled by
the chosen \(U\) and \(R\) coordinates.
\end{remark}

\begin{definition}[Sharp package class]
\label{ps:def:sharp-package-class}
The class \(\calA^{\sharp}_\Lambda\) consists of finite-window packages of
the form \(\calD_\Lambda(u,p)\) satisfying the pressure-admissible data
conditions, the pressure-natural harmonic admissibility
\[
    p_{\harm}\in \Yharm,
\]
the chosen finite-window coordinate admissibility conditions, and the
admissible gauge constraints specified below.
\end{definition}

\begin{remark}[NS-generation status]
The notation \(\calD_\Lambda(u,p)\) records that the physical pressure and
velocity coordinates are generated from local Navier--Stokes pressure data.
The paper does not claim that every suitable weak solution automatically
generates all finite-window coordinates with the needed compactness,
visibility, or kernel-free properties.
\end{remark}

\subsection{Admissible Gauges and Same-Gauge Representatives}
\label{ps:sec:gauges}

\begin{convention}[Conservative admissible gauge]
\label{ps:conv:admissible-gauge}
The admissible gauge class is denoted by \(\Gadm\).  A gauge element
\(\zeta\in\Gadm\) has coordinates
\[
    \zeta=(\zeta_u,\zeta_U,\zeta_R,\zeta_E,\zeta_h,\ldots),
\]
and the conservative physical-gauge convention is
\[
    \zeta_u=0.
\]
Thus
\[
    u_{\calD-\zeta}:=u_{\calD},
\]
while
\[
    U_{\calD-\zeta}:=U_{\calD}-\zeta_U,
    \qquad
    R_{\calD-\zeta}:=R_{\calD}-\zeta_R,
\]
\[
    E_{F,\calD-\zeta}:=E_{F,\calD}-\zeta_E,
    \qquad
    p_{\harm,\calD-\zeta}:=p_{\harm,\calD}-\zeta_h.
\]
Equivalently,
\[
    E^{\src}_{\calD-\zeta}:=E_{F,\calD-\zeta}.
\]
The shifted package \(\calD-\zeta\) is a quotient representative.  It is not
claimed to be generated by a second Navier--Stokes solution.
\end{convention}

\begin{assumption}[Same-gauge representative]
\label{ps:ass:same-gauge-representative}
For each package \(\calD\) under consideration, there is a selected
representative
\[
    \zeta_*(\calD)\in\Gadm
\]
used simultaneously for the baseline distance, active source residual,
pressure tail, harmonic tail, flux coordinates, and slack residuals.
Separate minimizing representatives for separate errors are not used.
\end{assumption}

\subsection{Sharp Baseline Norms and Amplitude Regimes}
\label{ps:sec:amplitude-regimes}

\begin{definition}[Core package norm]
\label{ps:def:core-package-norm}
The core finite-window package norm is denoted by
\[
    \|\calD\|_{\loc,\pkg,0}.
\]
It contains the baseline finite-window coordinates.  Its precise coordinate
weights are part of the finite-window package datum.
\end{definition}

\begin{definition}[Amplitude regimes]
\label{ps:def:amplitude-regimes}
The paper distinguishes three ways of handling the quadratic source term.

\begin{enumerate}[label=\textbf{Regime \Roman*.},leftmargin=*]
    \item \emph{Bounded-amplitude packages.}  One assumes
    \[
        \|u_{\calD-\zeta_*}\|_{L^3(\Qone)}
        +
        \|U_{\calD-\zeta_*}\|_{L^3(\Qone)}
        \le M_U.
    \]
    This permits the linearization
    \[
        \|u_{\calD-\zeta_*}\|_{L^3}^2
        \le
        M_U\|u_{\calD-\zeta_*}\|_{L^3}.
    \]

    \item \emph{Normalized quotient packages.}  On the normalized quotient
    sphere, visibility of the velocity coordinates supplies a finite constant
    \(C_{\mathrm{amp}}\), which replaces \(M_U\) in the bounded-amplitude
    proof.

    \item \emph{Quadratic sharp baseline geometry.}  The quadratic term
    \(\|u\|_{L^3}^2\) is added to the package geometry.  This avoids an
    external amplitude bound but changes the baseline geometry.
\end{enumerate}
\end{definition}

\begin{remark}[Finite-amplitude status]
This paper does not remove finite amplitude in the older linear baseline
geometry.  It records three controlled finite-window regimes for treating the
quadratic source.
\end{remark}

\subsection{Component Estimates for the Pressure-Source Residual}
\label{ps:sec:component-estimates}

\subsubsection{Separated-support commutator}

\begin{definition}[Cutoff--Riesz commutator]
\label{ps:def:cutoff-riesz-commutator}
For a tensor \(f=(f_{ij})\), define
\[
    C_\eta(f)
    :=
    R_iR_j(\eta f_{ij})-\eta R_iR_j(f_{ij}).
\]
\end{definition}

\begin{proposition}[Separated-support commutator estimate]
\label{ps:prop:commutator-target}
Let
\[
    A_{3/4,1}:=B_1\setminus B_{3/4}.
\]
For sources supported in the annular region \(A_{3/4,1}\), the fixed-window
estimate is
\[
    \|C_\eta(f)\|_{L^{3/2}((-1,0);L^{3/2}(B_{1/2}))}
    \le
    C_\eta
    \|f\|_{L^{3/2}((-1,0);L^{3/2}(A_{3/4,1}))^{3\times3}}.
\]
\end{proposition}

\begin{proof}
Extend \(f\) by zero outside \(A_{3/4,1}\).  For \(x\in B_{1/2}\) and
\(y\in A_{3/4,1}\), one has \(|x-y|\ge 1/4\).  Hence the singular kernel
\(K_{ij}\) of \(R_iR_j\) is a bounded smooth kernel on this separated
configuration.  Also \(\eta(x)=1\) for \(x\in B_{1/2}\), so for such \(x\)
\[
    C_\eta(f)(t,x)
    =
    \int_{A_{3/4,1}}K_{ij}(x-y)\bigl(\eta(y)-1\bigr)f_{ij}(t,y)\,dy .
\]
The principal value is unnecessary because the supports are separated.  Thus
\[
    |C_\eta(f)(t,x)|
    \le
    C
    \|f(t,\cdot)\|_{L^1(A_{3/4,1})}
    \le
    C
    \|f(t,\cdot)\|_{L^{3/2}(A_{3/4,1})},
\]
where the last step uses the finite measure of the annulus.  Taking the
\(L^{3/2}(B_{1/2})\) norm in \(x\) and then the \(L^{3/2}(-1,0)\) norm in
time gives the estimate.  All constants depend only on the fixed radii,
dimension, cutoff, and the finite number of tensor components.
\end{proof}

\begin{corollary}[Annular commutator control]
\label{ps:cor:annular-commutator-control}
For \(u\in L^3(\Qone)^3\),
\[
    \|C_\eta(u\otimes u)\|_{\Yprs}
    \le
    C_\eta
    \|u\|_{L^3((-1,0);L^3(A_{3/4,1}))}^2.
\]
If, in addition,
\[
    \|u\|_{L^3(\Qone)}\le M_U,
\]
then
\[
    \|C_\eta(u\otimes u)\|_{\Yprs}
    \le
    C_\eta M_U
    \|u\|_{L^3((-1,0);L^3(A_{3/4,1}))}.
\]
\end{corollary}

\begin{proof}
On \(B_{1/2}\), \(\eta(x)=1\), and the factor
\(\eta(y)-\eta(x)\) vanishes whenever \(y\in B_{3/4}\).  Therefore
\[
    C_\eta(u\otimes u)
    =
    C_\eta(\mathbf 1_{A_{3/4,1}}u\otimes u)
    \quad\text{on }B_{1/2}.
\]
Apply Proposition~\ref{ps:prop:commutator-target} with
\(f=\mathbf 1_{A_{3/4,1}}u\otimes u\), and use
\[
    \|u\otimes u\|_{L^{3/2}((-1,0);L^{3/2}(A_{3/4,1}))}
    \le
    C\|u\|_{L^3((-1,0);L^3(A_{3/4,1}))}^2.
\]
The finite-amplitude version follows from
\[
    \|u\|_{L^3((-1,0);L^3(A_{3/4,1}))}
    \le
    \|u\|_{L^3(\Qone)}
    \le
    M_U.
\]
\end{proof}

\subsubsection{Active covariance mismatch}

\begin{definition}[Active covariance mismatch]
\label{ps:def:active-covariance-mismatch}
For a shifted package \(\calD-\zeta\), define
\[
    \mathcal C_{ij}(\calD;\zeta)
    :=
    \eta\bigl[
        u_{\calD-\zeta,i}u_{\calD-\zeta,j}
        -
        U_{\calD-\zeta,i}U_{\calD-\zeta,j}
        -
        R_{\calD-\zeta,ij}
    \bigr].
\]
\end{definition}

\begin{proposition}[Active covariance mismatch estimate]
\label{ps:prop:active-covariance-target}
The finite-window estimate
\[
\begin{aligned}
    \|\mathcal C(\calD;\zeta)\|_{\Xsrc}
    &\le
    C
    \|u_{\calD-\zeta}-U_{\calD-\zeta}\|_{L^3(\Qone)}
    \bigl(
        \|u_{\calD-\zeta}\|_{L^3(\Qone)}
\\
    &\qquad\qquad
        +
        \|U_{\calD-\zeta}\|_{L^3(\Qone)}
    \bigr)
    +
    C\|R_{\calD-\zeta}\|_{L^{3/2}(\Qone)}.
\end{aligned}
\]
Under bounded amplitude, this yields
\[
    \|\mathcal C(\calD;\zeta)\|_{\Xsrc}
    \le
    C M_U
    \|u_{\calD-\zeta}-U_{\calD-\zeta}\|_{L^3(\Qone)}
    +
    C\|R_{\calD-\zeta}\|_{L^{3/2}(\Qone)}.
\]
\end{proposition}

\begin{proof}
Use the pointwise identity
\[
    u_i u_j-U_iU_j
    =
    (u_i-U_i)u_j+U_i(u_j-U_j),
\]
then apply Holder's inequality in \(L^3\cdot L^3\to L^{3/2}\).  Since
\(0\le \eta\le1\),
\[
\begin{aligned}
    \|\eta(u_i u_j-U_iU_j)\|_{L^{3/2}(\Qone)}
    &\le
    \|(u_i-U_i)u_j\|_{L^{3/2}(\Qone)}
    +
    \|U_i(u_j-U_j)\|_{L^{3/2}(\Qone)}       \\
    &\le
    \|u_i-U_i\|_{L^3(\Qone)}\|u_j\|_{L^3(\Qone)}
    +
    \|U_i\|_{L^3(\Qone)}\|u_j-U_j\|_{L^3(\Qone)} .
\end{aligned}
\]
Summing over the finite number of tensor components yields the product term in
the displayed estimate.  The Reynolds coordinate contributes linearly:
\[
    \|\eta R\|_{L^{3/2}(\Qone)}
    \le
    \|R\|_{L^{3/2}(\Qone)}.
\]
The bounded-amplitude form follows by replacing
\(\|u_{\calD-\zeta}\|_{L^3}+\|U_{\calD-\zeta}\|_{L^3}\) by \(M_U\), after
absorbing harmless finite-component constants.
\end{proof}

\subsection{Pressure-Source Residual Functional}
\label{ps:sec:residual-functional}

\begin{definition}[Annular leakage coordinate]
\label{ps:def:annular-leakage}
Define
\[
    \Leak_{\mathrm{ann}}^{\mathrm{lin}}(\calD;\zeta)
    :=
    \|u_{\calD-\zeta}\|_{L^3((-1,0);L^3(A_{3/4,1}))},
\]
and
\[
    \Leak_{\mathrm{ann}}(\calD;\zeta)
    :=
    \|u_{\calD-\zeta}\|_{L^3((-1,0);L^3(A_{3/4,1}))}^2.
\]
In a linearized bounded-amplitude regime this may be replaced by a linear
annular leakage coordinate multiplied by the finite amplitude constant.
When no representative is displayed, these coordinates are evaluated on the
package currently in the norm.  Thus, for example,
\(\Leak_{\mathrm{ann}}^{\mathrm{lin}}(\calD-\zeta)\) means
\(\Leak_{\mathrm{ann}}^{\mathrm{lin}}(\calD;\zeta)\).
\end{definition}

\begin{definition}[Pressure-source residual]
\label{ps:def:pressure-source-residual}
The pressure-source residual at the representative \(\zeta\) is
\[
\begin{aligned}
    \Err^{\prs}_{\src}(\calD;\zeta)
    &:=
    \|R_iR_j(\mathcal C_{ij}(\calD;\zeta))\|_{\Yprs}
    +
    \|C_\eta(u_{\calD-\zeta}\otimes u_{\calD-\zeta})\|_{\Yprs}
\\
    &\quad+
    \|R_iR_j(E^{\src}_{\calD-\zeta,ij})\|_{\Yprs}.
\end{aligned}
\]
The term \(E^{\src}\) collects residual source error not captured by the
selected velocity coordinate \(U\) and covariance coordinate \(R\).
\end{definition}

\begin{remark}[Component control]
Convention~\ref{ps:conv:source-to-pressure} controls the first and third terms by their
\(\Xsrc\) source norms.  Proposition~\ref{ps:prop:commutator-target} controls the
commutator by annular leakage.  The sharp package norms below collect these
component quantities in the same gauge.
\end{remark}

\subsection{Sharp Package Norms and Quotient Distances}
\label{ps:sec:sharp-norms}

\begin{definition}[Sharp package norm]
\label{ps:def:sharp-package-norm}
In the bounded-amplitude regime, we use the following unit-weight sharp
package norm:
\[
\begin{aligned}
    \|\calD\|_{\loc,\pkg}^{\sharp}
    &:=
    \|\calD\|_{\loc,\pkg,0}
    +
    \|u_{\calD}-U_{\calD}\|_{L^3(\Qone)}
    +
    \|R_{\calD}\|_{L^{3/2}(\Qone)}
\\
    &\quad+
    \|E^{\src}_{\calD}\|_{\Xsrc}
    +
    \Leak_{\mathrm{ann}}^{\mathrm{lin}}(\calD)
    +
    \|p_{\harm,\calD}\|_{\Yharm}.
\end{aligned}
\]
Fixed positive coordinate weights may be inserted without changing any proof;
they only modify the finite-window constants.  The corresponding quotient
distance is
\[
    \Dist_{\loc,\pkg}^{\sharp}(\calD,\Gadm)
    :=
    \inf_{\zeta\in\Gadm}
    \|\calD-\zeta\|_{\loc,\pkg}^{\sharp}.
\]
\end{definition}

\begin{definition}[Quadratic sharp package functional]
\label{ps:def:quadratic-sharp-package-norm}
The quadratic sharp package functional is
\[
\begin{aligned}
    \|\calD\|_{\loc,\pkg}^{\sharp,\mathrm{quad}}
    &:=
    \|\calD\|_{\loc,\pkg,0}
    +
    \|u_{\calD}-U_{\calD}\|_{L^3(\Qone)}
    \bigl(
        \|u_{\calD}\|_{L^3(\Qone)}
\\
    &\qquad\qquad+
        \|U_{\calD}\|_{L^3(\Qone)}
    \bigr)
    +
    \|R_{\calD}\|_{L^{3/2}(\Qone)}
\\
    &\quad+
    \|E^{\src}_{\calD}\|_{\Xsrc}
    +
    \Leak_{\mathrm{ann}}(\calD)
    +
    \|p_{\harm,\calD}\|_{\Yharm}.
\end{aligned}
\]
Because of the quadratic factors, this object is a nonnegative finite-window
package functional rather than a linear norm in the strict functional-analytic
sense.  We keep the norm notation to emphasize that it is the geometry used to
measure packages.  It defines the associated quotient functional
\(\Dist_{\loc,\pkg}^{\sharp,\mathrm{quad}}\) by quotienting over \(\Gadm\).
\end{definition}

\begin{remark}[Geometry cost]
The quadratic sharp geometry handles the \(u_i u_j\) nonlinearity by changing
the package geometry.  It is not a theorem in the older purely linear baseline
geometry.
\end{remark}

\subsection{Main Pressure-Source Absorption Theorems}
\label{ps:sec:main-absorption}

\begin{assumption}[Sharp near-minimizer]
\label{ps:ass:sharp-near-minimizer}
The selected representative \(\zeta_*(\calD)\) satisfies
\[
    \|\calD-\zeta_*\|_{\loc,\pkg}^{\sharp}
    \le
    \Dist_{\loc,\pkg}^{\sharp}(\calD,\Gadm)
    +
    \delta_{\pkg}.
\]
\end{assumption}

\begin{theorem}[Bounded-amplitude pressure-source absorption]
\label{ps:thm:bounded-amplitude-pressure-source-absorption}
Let \(\calD\in\calA^\sharp_\Lambda\).  Assume the same-gauge representative
\(\zeta_*(\calD)\), the bounded-amplitude condition
\[
    \|u_{\calD-\zeta_*}\|_{L^3(\Qone)}
    +
    \|U_{\calD-\zeta_*}\|_{L^3(\Qone)}
    \le
    M_U,
\]
and the sharp near-minimizer condition in
Assumption~\ref{ps:ass:sharp-near-minimizer}.  Then
\[
    \Err^{\prs}_{\src}(\calD;\zeta_*)
    \le
    C_{\prs}(M_U)
    \Dist_{\loc,\pkg}^{\sharp}(\calD,\Gadm)
    +
    C_{\prs}(M_U)\delta_{\pkg},
\]
where \(C_{\prs}(M_U)\) depends only on the fixed geometry, \(C_{CZ}\),
\(C_\eta\), and \(M_U\).
\end{theorem}

\begin{proof}
Apply Convention~\ref{ps:conv:source-to-pressure} and
Proposition~\ref{ps:prop:active-covariance-target} to the covariance mismatch pressure:
\[
\begin{aligned}
    \|R_iR_j(\mathcal C_{ij}(\calD;\zeta_*))\|_{\Yprs}
    &\le
    C_{CZ}\|\mathcal C(\calD;\zeta_*)\|_{\Xsrc} \\
    &\le
    C M_U
    \|u_{\calD-\zeta_*}-U_{\calD-\zeta_*}\|_{L^3(\Qone)}
    +
    C\|R_{\calD-\zeta_*}\|_{L^{3/2}(\Qone)} .
\end{aligned}
\]
By Corollary~\ref{ps:cor:annular-commutator-control},
\[
    \|C_\eta(u_{\calD-\zeta_*}\otimes u_{\calD-\zeta_*})\|_{\Yprs}
    \le
    C_\eta M_U
    \Leak_{\mathrm{ann}}^{\mathrm{lin}}(\calD;\zeta_*).
\]
Finally, the residual clean-source pressure is bounded by
\[
    \|R_iR_j(E^{\src}_{\calD-\zeta_*,ij})\|_{\Yprs}
    \le
    C_{CZ}\|E^{\src}_{\calD-\zeta_*}\|_{\Xsrc}.
\]
Each coordinate on the right-hand side is included in
\(\|\calD-\zeta_*\|_{\loc,\pkg}^{\sharp}\).  Therefore, after enlarging the
finite-window constant,
\[
    \Err^{\prs}_{\src}(\calD;\zeta_*)
    \le
    C_{\prs}(M_U)\|\calD-\zeta_*\|_{\loc,\pkg}^{\sharp}.
\]
Using Assumption~\ref{ps:ass:sharp-near-minimizer} gives the stated quotient estimate.
\end{proof}

\begin{corollary}[Weighted absorption form]
\label{ps:thm:weighted-absorption-target}
Assume the hypotheses of
Theorem~\ref{ps:thm:bounded-amplitude-pressure-source-absorption}.  For any prescribed
finite-window coefficient \(\eta_{\prs}>0\), define a weighted sharp package
geometry so that
\[
    \|\calD\|_{\loc,\pkg}^{\sharp,\omega}
    \ge
    \frac{C_{\prs}(M_U)}{\eta_{\prs}}
    \|\calD\|_{\loc,\pkg}^{\sharp}.
\]
If the selected representative satisfies
\[
    \|\calD-\zeta_*\|_{\loc,\pkg}^{\sharp,\omega}
    \le
    \Dist_{\loc,\pkg}^{\sharp,\omega}(\calD,\Gadm)
    +
    \delta_{\pkg}^{\omega},
\]
then the bounded-amplitude estimate takes the
form
\[
    \Err^{\prs}_{\src}(\calD;\zeta_*)
    \le
    \eta_{\prs}
    \Dist_{\loc,\pkg}^{\sharp,\omega}(\calD,\Gadm)
    +
    \eta_{\prs}\delta_{\pkg}^{\omega}.
\]
\end{corollary}

\begin{proof}
By Theorem~\ref{ps:thm:bounded-amplitude-pressure-source-absorption},
\[
    \Err^{\prs}_{\src}(\calD;\zeta_*)
    \le
    C_{\prs}(M_U)\|\calD-\zeta_*\|_{\loc,\pkg}^{\sharp}
    \le
    \eta_{\prs}\|\calD-\zeta_*\|_{\loc,\pkg}^{\sharp,\omega}.
\]
The weighted near-minimizer property gives the displayed estimate.  This is an
accounting normalization, not a new PDE estimate.
\end{proof}

\begin{corollary}[Normalized quotient amplitude variant]
\label{ps:thm:normalized-quotient-amplitude-target}
Assume the same-gauge and sharp near-minimizer hypotheses from
Theorem~\ref{ps:thm:bounded-amplitude-pressure-source-absorption}.  If the normalized
quotient condition
\[
    \Dist_{\loc,\pkg}^{\sharp}(\calD,\Gadm)=1
\]
holds and the selected representative satisfies
\[
    \|u_{\calD-\zeta_*}\|_{L^3(\Qone)}
    +
    \|U_{\calD-\zeta_*}\|_{L^3(\Qone)}
    \le
    C_{\mathrm{amp}},
\]
then the bounded-amplitude pressure-source absorption theorem applies on the
normalized quotient with \(M_U=C_{\mathrm{amp}}\).
\end{corollary}

\begin{remark}[Status of normalized quotient amplitude]
This does not remove finite amplitude globally.  It says that on the
normalized quotient sphere, amplitude may be treated as part of the package
geometry.
\end{remark}

\begin{theorem}[Quadratic sharp baseline variant]
\label{ps:thm:quadratic-sharp-target}
Let \(\calD\in\calA^\sharp_\Lambda\), and let \(\zeta_*(\calD)\) be the
same-gauge representative from Assumption~\ref{ps:ass:same-gauge-representative}.  Assume
that this representative satisfies the quadratic near-minimizer condition
\[
    \|\calD-\zeta_*\|_{\loc,\pkg}^{\sharp,\mathrm{quad}}
    \le
    \Dist_{\loc,\pkg}^{\sharp,\mathrm{quad}}(\calD,\Gadm)
    +
    \delta_{\pkg}^{\mathrm{quad}}.
\]
Then
\[
    \Err^{\prs}_{\src}(\calD;\zeta_*)
    \le
    C_{\prs}^{\mathrm{quad}}
    \Dist_{\loc,\pkg}^{\sharp,\mathrm{quad}}(\calD,\Gadm)
    +
    C_{\prs}^{\mathrm{quad}}\delta_{\pkg}^{\mathrm{quad}}.
\]
\end{theorem}

\begin{proof}
The proof is the same component estimate as in
Theorem~\ref{ps:thm:bounded-amplitude-pressure-source-absorption}, except that the
product term is not linearized.  The full estimate in
Proposition~\ref{ps:prop:active-covariance-target} gives
\[
\begin{aligned}
    \|R_iR_j(\mathcal C_{ij}(\calD;\zeta_*))\|_{\Yprs}
    &\le
    C
    \|u_{\calD-\zeta_*}-U_{\calD-\zeta_*}\|_{L^3}
    \bigl(
        \|u_{\calD-\zeta_*}\|_{L^3}
        +
        \|U_{\calD-\zeta_*}\|_{L^3}
    \bigr) \\
    &\quad+
    C\|R_{\calD-\zeta_*}\|_{L^{3/2}} .
\end{aligned}
\]
The commutator term is bounded by
\[
    C_\eta\Leak_{\mathrm{ann}}(\calD;\zeta_*),
\]
and the residual source term is bounded by
\[
    C_{CZ}\|E^{\src}_{\calD-\zeta_*}\|_{\Xsrc}.
\]
These are precisely coordinates of
\(\|\calD-\zeta_*\|_{\loc,\pkg}^{\sharp,\mathrm{quad}}\), up to fixed
finite-window constants.  Therefore
\[
    \Err^{\prs}_{\src}(\calD;\zeta_*)
    \le
    C_{\prs}^{\mathrm{quad}}
    \|\calD-\zeta_*\|_{\loc,\pkg}^{\sharp,\mathrm{quad}}.
\]
The quadratic near-minimizer condition completes the proof.
\end{proof}

\subsection*{Localization Leakage Absorption}
\subsection{Fixed Local Geometry}
\label{loc:sec:geometry}

Throughout the paper,
\[
    \Qone=B_1\times(-1,0),
    \qquad
    B_{1/2}\subset B_{5/8}\subset B_{3/4}\subset B_1.
\]
Let
\[
    \chi\in C_c^\infty(B_{3/4}),
    \qquad
    0\le \chi\le 1,
    \qquad
    \chi\equiv1\quad\text{on }B_{1/2}.
\]
The fixed transition shell is
\[
    A_\chi
    :=
    \operatorname{supp}\nabla\chi
    \cup
    \operatorname{supp}\Delta\chi
    \subset B_{3/4}\setminus B_{1/2}.
\]
All constants in this paper may depend on \(\chi\), the fixed radii, and
the dimension.  No scale-uniform estimate is claimed.

\begin{convention}[Pressure observation spaces]
\label{loc:conv:pressure-observation-spaces}
The pressure-natural observation spaces used in the previous branch are
\[
    \Yprs
    =
    L^{3/2}\bigl((-1,0);L^{3/2}(B_{1/2})\bigr),
\]
and
\[
    \Yharm
    =
    L^{3/2}\bigl((-1,0);L^{3/2}(B_{3/4})\bigr).
\]
\end{convention}

\subsection{Imported Sharp Packages and Gauges}
\label{loc:sec:packages}

\begin{definition}[Sharp localized package datum]
\label{loc:def:sharp-localized-package}
A sharp localized package is a tuple
\[
    \calD_\Lambda(u,p)
    =
    \bigl(
        u,\,
        U,\,
        R,\,
        F^{\mathrm{act}},\,
        F^{\mathrm{mod}},\,
        E_F,\,
        p^{\mathrm{act}},\,
        p_{\harm},\,
        \Pi,\,
        \Phi,\,
        T,\,
        s
    \bigr).
\]
The physical coordinates \(u,p\) are the localized Navier--Stokes velocity and
pressure data.  The coordinates \(U,R,E_F,p^{\mathrm{act}},p_{\harm}\) are the
model velocity, covariance, additional source residual, active pressure, and
harmonic pressure coordinates inherited from the pressure-source branch.
\end{definition}

\begin{convention}[Conservative admissible gauge]
\label{loc:conv:conservative-gauge}
The admissible gauge class is denoted by \(\Gadm\).  A gauge element
\[
    \zeta=(\zeta_u,\zeta_U,\zeta_R,\zeta_E,\zeta_h,\ldots)\in\Gadm
\]
satisfies the conservative physical-gauge convention
\[
    \zeta_u=0.
\]
Thus
\[
    u_{\calD-\zeta}:=u_{\calD},
\]
while the model and pressure-tail coordinates may shift:
\[
    U_{\calD-\zeta}:=U_{\calD}-\zeta_U,
    \qquad
    R_{\calD-\zeta}:=R_{\calD}-\zeta_R,
\]
\[
    E_{F,\calD-\zeta}:=E_{F,\calD}-\zeta_E,
    \qquad
    p_{\harm,\calD-\zeta}:=p_{\harm,\calD}-\zeta_h.
\]
The shifted package \(\calD-\zeta\) is a quotient representative; it is not
claimed to be generated by a second Navier--Stokes solution.
\end{convention}

\begin{assumption}[Same-gauge localization representative]
\label{loc:ass:same-gauge-loc}
For each package \(\calD\) under consideration, a representative
\[
    \zeta_*(\calD)\in\Gadm
\]
is selected and used simultaneously in the localization residual, annular
leakage coordinates, pressure leakage coordinates, and sharp localization
quotient distance.
\end{assumption}

\subsection{Localized Momentum Equation and Leakage Terms}
\label{loc:sec:momentum-identity}

\begin{assumption}[Distributional local momentum equation]
\label{loc:ass:distributional-momentum}
The physical local data satisfy the finite-window integrability conditions
\[
    u\in L^3(\Qone)^3,
    \qquad
    \nabla u\in L^2((-1,0);L^2(B_{3/4}))^{3\times3},
    \qquad
    p\in L^{3/2}(\Qone),
\]
and
\[
    \partial_tu-\Delta u+\nabla p+\nabla\cdot(u\otimes u)=0
\]
in distributions on \(\Qone\).
\end{assumption}

\begin{definition}[Localized velocity]
\label{loc:def:localized-velocity}
Set
\[
    v:=\chi u.
\]
\end{definition}

\begin{proposition}[Localized momentum leakage identity]
\label{loc:prop:localized-momentum-identity}
For data satisfying \Cref{loc:ass:distributional-momentum}, the localized velocity
\(v=\chi u\) satisfies the distributional identity
\[
    \partial_t(\chi u)
    -\Delta(\chi u)
    +\nabla(\chi p)
    +\nabla\cdot(\chi u\otimes u)
    =
    \mathcal L_\chi^{\mom}(u,p),
\]
where
\[
    \mathcal L_\chi^{\mom}(u,p)
    =
    -2\nabla\chi\cdot\nabla u
    -
    u\Delta\chi
    +
    p\nabla\chi
    +
    (u\otimes u)\nabla\chi,
\]
where \(((u\otimes u)\nabla\chi)_i=u_i u_j\partial_j\chi\).
\end{proposition}

\begin{proof}
We compute componentwise in distributions.  Since \(\chi\) is independent of
time,
\[
    \partial_t(\chi u_i)=\chi\,\partial_tu_i.
\]
The product rules give
\[
    \Delta(\chi u_i)
    =
    \chi\Delta u_i
    +
    2\partial_a\chi\,\partial_a u_i
    +
    u_i\Delta\chi,
\]
\[
    \partial_i(\chi p)
    =
    \chi\partial_i p+p\,\partial_i\chi,
\]
and
\[
    \partial_j(\chi u_i u_j)
    =
    \chi\partial_j(u_i u_j)
    +
    u_i u_j\partial_j\chi.
\]
Therefore
\[
\begin{aligned}
&\partial_t(\chi u_i)-\Delta(\chi u_i)
    +\partial_i(\chi p)
    +\partial_j(\chi u_i u_j)\\
&\quad =
\chi\bigl(\partial_tu_i-\Delta u_i+\partial_i p+\partial_j(u_i u_j)\bigr)
-2\partial_a\chi\,\partial_a u_i
-u_i\Delta\chi
+p\,\partial_i\chi
+u_i u_j\partial_j\chi.
\end{aligned}
\]
The term in parentheses vanishes by \Cref{loc:ass:distributional-momentum}.  This
is the stated identity.  Each leakage term contains either \(\nabla\chi\) or
\(\Delta\chi\), and hence is supported in \(A_\chi\).
\end{proof}

\begin{remark}[Physical identity versus quotient representative]
\label{loc:rem:physical-vs-quotient-identity}
\Cref{loc:prop:localized-momentum-identity} is an identity for the physical
coordinates \((u,p)\).  A shifted package \(\calD-\zeta\) is only a quotient
representative; since the conservative gauge fixes \(u\) but may change
pressure-tail or model coordinates, the shifted pressure coordinate need not
generate a second Navier--Stokes solution.  The residuals below therefore use
the same cutoff expression as an accounting functional on the chosen
representative.  No additional distributional equation is asserted for
\(\calD-\zeta\) unless it is separately assumed.
\end{remark}

\subsection{Momentum Leakage Estimate}
\label{loc:sec:momentum-estimate}

\begin{definition}[Momentum leakage norm]
\label{loc:def:momentum-leakage-norm}
Set
\[
    \Ymom
    :=
    L^2\bigl((-1,0);H^{-1}(B_{3/4})\bigr)
    +
    L^{3/2}\bigl((-1,0);W^{-1,3/2}(B_{3/4})\bigr).
\]
The first summand is used for diffusion cutoff leakage, and the second for
pressure and nonlinear transport leakage.
The norm is the usual sum-space norm:
\[
    \|F\|_{\Ymom}
    :=
    \inf_{F=F_1+F_2}
    \left(
        \|F_1\|_{L^2_tH^{-1}_x}
        +
        \|F_2\|_{L^{3/2}_tW^{-1,3/2}_x}
    \right),
\]
with spatial domains \(B_{3/4}\) and time interval \((-1,0)\).
\end{definition}

\begin{definition}[Annular momentum leakage coordinates]
\label{loc:def:annular-momentum-leakage}
For a shifted package \(\calD-\zeta\), define
\[
    \Leak_{\nabla u}(\calD;\zeta)
    :=
    \|\nabla u_{\calD-\zeta}\|_{L^2((-1,0);L^2(A_\chi))},
\]
\[
    \Leak_u(\calD;\zeta)
    :=
    \|u_{\calD-\zeta}\|_{L^3((-1,0);L^3(A_\chi))},
\]
and
\[
    \Leak_p(\calD;\zeta)
    :=
    \|p_{\calD-\zeta}\|_{L^{3/2}((-1,0);L^{3/2}(A_\chi))}.
\]
Here \(p_{\calD-\zeta}\) denotes the pressure coordinate used by the package,
for instance
\[
    p_{\calD-\zeta}
    =
    p^{\mathrm{act}}_{\calD-\zeta}
    +
    p_{\harm,\calD-\zeta}
    +
    p^{\mathrm{rem}}_{\calD-\zeta},
\]
when a pressure remainder coordinate is present.
When no representative is displayed, the leakage coordinates are evaluated on
the package currently inside the norm.  Thus \(\Leak_u(\calD-\zeta)\) means
\(\Leak_u(\calD;\zeta)\), and \(\Leak_u(\calD)\) means
\(\Leak_u(\calD;0)\); the same convention applies to \(\Leak_{\nabla u}\) and
\(\Leak_p\).
\end{definition}

\begin{definition}[Momentum localization residual]
\label{loc:def:momentum-localization-residual}
Define
\[
    \Err_{\loc}^{\mom}(\calD;\zeta)
    :=
    \|\mathcal L_\chi^{\mom}(u_{\calD-\zeta},p_{\calD-\zeta})\|_{\Ymom}.
\]
For a nonphysical shifted pressure coordinate, this is the norm of the
cutoff-leakage accounting expression, not a new localized Navier--Stokes
identity.
\end{definition}

\begin{theorem}[Momentum localization leakage estimate]
\label{loc:thm:momentum-localization-leakage}
For fixed \(\chi\), and for every package representative for which the
displayed leakage coordinates are finite,
\[
    \Err_{\loc}^{\mom}(\calD;\zeta)
    \le
    C_\chi
    \left[
        \Leak_{\nabla u}(\calD;\zeta)
        +
        \Leak_u(\calD;\zeta)
        +
        \Leak_p(\calD;\zeta)
        +
        \Leak_u(\calD;\zeta)^2
    \right].
\]
If
\[
    \|u_{\calD-\zeta}\|_{L^3(\Qone)}\le M_U,
\]
then
\[
    \Err_{\loc}^{\mom}(\calD;\zeta)
    \le
    C_\chi(M_U)
    \left[
        \Leak_{\nabla u}(\calD;\zeta)
        +
        \Leak_u(\calD;\zeta)
        +
        \Leak_p(\calD;\zeta)
    \right].
\]
\end{theorem}

\begin{proof}
By definition, \(\mathcal L_\chi^{\mom}\) is the sum of the four cutoff terms
displayed in \Cref{loc:prop:localized-momentum-identity}.  The physical identity
justifies this expression for the unshifted Navier--Stokes data; for a shifted
representative it is used as the accounting residual described in
\Cref{loc:rem:physical-vs-quotient-identity}.  The diffusion leakage satisfies
\[
    \|\nabla\chi\cdot\nabla u_{\calD-\zeta}\|_{L^2_tH^{-1}_x}
    \le
    C_\chi
    \|\nabla u_{\calD-\zeta}\|_{L^2((-1,0);L^2(A_\chi))}
    =
    C_\chi\Leak_{\nabla u}(\calD;\zeta),
\]
because \(L^2(B_{3/4})\hookrightarrow H^{-1}(B_{3/4})\).

The remaining terms are estimated in
\(L^{3/2}_tW^{-1,3/2}_x\).  Since
\[
    L^{3/2}(B_{3/4})\hookrightarrow W^{-1,3/2}(B_{3/4})
\]
on the bounded fixed ball,
\[
    \|u_{\calD-\zeta}\Delta\chi\|_{L^{3/2}_tW^{-1,3/2}_x}
    \le
    C_\chi
    \|u_{\calD-\zeta}\|_{L^{3/2}((-1,0);L^{3/2}(A_\chi))}.
\]
The finite measure of \(A_\chi\times(-1,0)\) gives
\[
    \|u_{\calD-\zeta}\|_{L^{3/2}_tL^{3/2}_x(A_\chi)}
    \le
    C\Leak_u(\calD;\zeta).
\]
Similarly,
\[
    \|p_{\calD-\zeta}\nabla\chi\|_{L^{3/2}_tW^{-1,3/2}_x}
    \le
    C_\chi\Leak_p(\calD;\zeta).
\]
Finally,
\[
    \|(u_{\calD-\zeta}\otimes u_{\calD-\zeta})\nabla\chi\|_{L^{3/2}_tW^{-1,3/2}_x}
    \le
    C_\chi
    \|u_{\calD-\zeta}\otimes u_{\calD-\zeta}\|_{L^{3/2}_tL^{3/2}_x(A_\chi)}
    \le
    C_\chi\Leak_u(\calD;\zeta)^2.
\]
Using the sum-space norm in \(\Ymom\) gives the first estimate.

If \(\|u_{\calD-\zeta}\|_{L^3(\Qone)}\le M_U\), then
\[
    \Leak_u(\calD;\zeta)^2
    \le
    \|u_{\calD-\zeta}\|_{L^3(\Qone)}\Leak_u(\calD;\zeta)
    \le
    M_U\Leak_u(\calD;\zeta).
\]
After enlarging the finite-window constant to depend on \(M_U\), we obtain the
bounded-amplitude estimate.
\end{proof}

\subsection{Localized Energy and Flux Leakage}
\label{loc:sec:flux-leakage}

\begin{definition}[Flux residual]
\label{loc:def:flux-localization-residual}
For the spatial cutoff \(\phi=\chi^2\), define
\(\Err_{\loc}^{\flux}(\calD;\zeta)\), whenever the displayed integrals are
finite, to be the sum of the absolute values of the localized energy/flux
leakage terms supported in \(A_\chi\times(-1,0)\):
\[
\begin{aligned}
    \Err_{\loc}^{\flux}(\calD;\zeta)
    &:=
    \left|
        \int_{-1}^{0}\int_{B_{3/4}}
        |u_{\calD-\zeta}|^2\Delta(\chi^2)\,dx\,dt
    \right| \\
    &\quad+
    \left|
        \int_{-1}^{0}\int_{B_{3/4}}
        |u_{\calD-\zeta}|^2
        u_{\calD-\zeta}\cdot\nabla(\chi^2)\,dx\,dt
    \right| \\
    &\quad+
    2\left|
        \int_{-1}^{0}\int_{B_{3/4}}
        p_{\calD-\zeta}
        u_{\calD-\zeta}\cdot\nabla(\chi^2)\,dx\,dt
    \right|.
\end{aligned}
\]
If the chosen finite-window geometry includes a time cutoff, the corresponding
\(\int |u|^2\partial_t\phi\) term is added as an additional leakage
coordinate.  This definition is an accounting functional; the local energy
inequality itself is not used in this estimate.
\end{definition}

\begin{theorem}[Localized energy/flux leakage estimate]
\label{loc:thm:flux-leakage-estimate}
For a fixed spatial cutoff, and for every package representative for which the
displayed leakage coordinates are finite,
\[
    \Err_{\loc}^{\flux}(\calD;\zeta)
    \le
    C_\chi
    \left[
        \Leak_u(\calD;\zeta)^3
        +
        \Leak_p(\calD;\zeta)\Leak_u(\calD;\zeta)
        +
        \Leak_u(\calD;\zeta)^2
    \right].
\]
If
\[
    \|u_{\calD-\zeta}\|_{L^3(\Qone)}\le M_U,
\]
then
\[
    \Err_{\loc}^{\flux}(\calD;\zeta)
    \le
    C_\chi(M_U)
    \left[
        \Leak_u(\calD;\zeta)
        +
        \Leak_p(\calD;\zeta)
    \right].
\]
\end{theorem}

\begin{proof}
Both \(\nabla(\chi^2)\) and \(\Delta(\chi^2)\) are supported in \(A_\chi\) and
bounded by constants depending only on the fixed cutoff.  Hence
\[
    \left|
        \int |u_{\calD-\zeta}|^2\Delta(\chi^2)
    \right|
    \le
    C_\chi
    \|u_{\calD-\zeta}\|_{L^2((-1,0);L^2(A_\chi))}^2.
\]
Since \(A_\chi\times(-1,0)\) has finite measure,
\[
    \|u_{\calD-\zeta}\|_{L^2(A_\chi\times(-1,0))}
    \le
    C\Leak_u(\calD;\zeta),
\]
so the diffusion-flux leakage is bounded by \(C_\chi\Leak_u(\calD;\zeta)^2\).
The cubic flux term satisfies
\[
    \left|
        \int |u_{\calD-\zeta}|^2
        u_{\calD-\zeta}\cdot\nabla(\chi^2)
    \right|
    \le
    C_\chi\Leak_u(\calD;\zeta)^3.
\]
For the pressure flux term, Holder's inequality gives
\[
    \left|
        \int p_{\calD-\zeta}
        u_{\calD-\zeta}\cdot\nabla(\chi^2)
    \right|
    \le
    C_\chi
    \Leak_p(\calD;\zeta)\Leak_u(\calD;\zeta).
\]
Combining the three bounds proves the first estimate.

If \(\|u_{\calD-\zeta}\|_{L^3(\Qone)}\le M_U\), then
\[
    \Leak_u(\calD;\zeta)^2\le M_U\Leak_u(\calD;\zeta),
    \qquad
    \Leak_u(\calD;\zeta)^3\le M_U^2\Leak_u(\calD;\zeta),
\]
and
\[
    \Leak_p(\calD;\zeta)\Leak_u(\calD;\zeta)
    \le
    M_U\Leak_p(\calD;\zeta).
\]
Absorbing the powers of \(M_U\) into \(C_\chi(M_U)\) gives the
bounded-amplitude estimate.
\end{proof}

\begin{remark}[Status of the flux estimate]
This is a fixed-window leakage estimate.  It is not a global energy inequality
theorem and does not imply regularity.
\end{remark}

\subsection{Pressure Decomposition for Annular Leakage}
\label{loc:sec:pressure-leakage}

\begin{definition}[Annular pressure leakage coordinate]
\label{loc:def:annular-pressure-leakage}
The pressure leakage coordinate is
\[
    \Leak_p^{\mathrm{ann}}(\calD;\zeta)
    :=
    \|p_{\calD-\zeta}\|_{L^{3/2}((-1,0);L^{3/2}(A_\chi))}.
\]
In the default notation below, \(\Leak_p=\Leak_p^{\mathrm{ann}}\).
\end{definition}

\begin{proposition}[Pressure leakage bookkeeping]
\label{loc:prop:pressure-leakage-bookkeeping}
If
\[
    p_{\calD-\zeta}
    =
    p^{\mathrm{act}}_{\calD-\zeta}
    +
    p_{\harm,\calD-\zeta}
    +
    p^{\mathrm{rem}}_{\calD-\zeta},
\]
then the annular pressure leakage is controlled by the sum of the corresponding
annular active, harmonic, and remainder pressure coordinates.
\end{proposition}

\begin{proof}
By the triangle inequality in
\(L^{3/2}((-1,0);L^{3/2}(A_\chi))\),
\[
\begin{aligned}
    \Leak_p(\calD;\zeta)
    &\le
    \|p^{\mathrm{act}}_{\calD-\zeta}\|_{L^{3/2}_tL^{3/2}_x(A_\chi)}
    +
    \|p_{\harm,\calD-\zeta}\|_{L^{3/2}_tL^{3/2}_x(A_\chi)}
\\
    &\quad+
    \|p^{\mathrm{rem}}_{\calD-\zeta}\|_{L^{3/2}_tL^{3/2}_x(A_\chi)}.
\end{aligned}
\]
If the package supplies a Calderon--Zygmund source coordinate for
\(p^{\mathrm{act}}\), then the first term may be estimated by that source
coordinate.  No smallness of annular pressure leakage is claimed here; this is
only the bookkeeping decomposition used by the localization norm.
\end{proof}

\subsection{Localization Leakage Residual Functional}
\label{loc:sec:residual-functional}

\begin{definition}[Localization leakage residual]
\label{loc:def:localization-leakage-residual}
The two-component localization residual is
\[
    \Err_{\loc}(\calD;\zeta)
    :=
    \Err_{\loc}^{\mom}(\calD;\zeta)
    +
    \Err_{\loc}^{\flux}(\calD;\zeta).
\]
The localization residual contains only the momentum and flux leakage terms;
trace, slack, and gate coordinates are accounted for separately in the
gate/slack component norm.
\end{definition}

\begin{remark}[Trace and slack leakage]
Trace and slack leakage may be incorporated through a separate residual
coordinate satisfying an estimate of the form
\[
    \Err_{\loc}^{\mathrm{trace}}(\calD;\zeta)
    \le
    C_{\mathrm{tr}}
    \left[
        \Leak_u(\calD;\zeta)
        +
        \Leak_p(\calD;\zeta)
        +
        \Leak_{\nabla u}(\calD;\zeta)
    \right].
\]
Such terms are not part of the two-component localization residual; they are
accounted for in the gate/slack component geometry used below.
\end{remark}

\subsection{Sharp Localization Norm and Quotient Distance}
\label{loc:sec:sharp-loc-norm}

\begin{definition}[Sharp localization package norm]
\label{loc:def:sharp-localization-norm}
Let \(\|\calD\|_{\loc,\pkg}^{\sharp}\) denote the sharp package norm from the
pressure-source branch.  Define
\[
    \|\calD\|_{\loc,\pkg}^{\sharp,\loc}
    :=
    \|\calD\|_{\loc,\pkg}^{\sharp}
    +
    \Leak_{\nabla u}(\calD)
    +
    \Leak_u(\calD)
    +
    \Leak_p(\calD).
\]
The corresponding quotient distance is
\[
    \Dist_{\loc,\pkg}^{\sharp,\loc}(\calD,\Gadm)
    :=
    \inf_{\zeta\in\Gadm}
    \|\calD-\zeta\|_{\loc,\pkg}^{\sharp,\loc}.
\]
\end{definition}

\begin{definition}[Quadratic localization package norm]
\label{loc:def:quadratic-localization-norm}
The quadratic localization norm is
\[
\begin{aligned}
    \|\calD\|_{\loc,\pkg}^{\sharp,\loc,\quadg}
    &:=
    \|\calD\|_{\loc,\pkg}^{\sharp}
    +
    \Leak_{\nabla u}(\calD)
    +
    \Leak_u(\calD)
    +
    \Leak_u(\calD)^2
    +
    \Leak_u(\calD)^3
\\
    &\quad+
    \Leak_p(\calD)\Leak_u(\calD)
    +
    \Leak_p(\calD).
\end{aligned}
\]
It defines
\(\Dist_{\loc,\pkg}^{\sharp,\loc,\quadg}(\calD,\Gadm)\) by quotienting over
\(\Gadm\).
\end{definition}

\begin{assumption}[Sharp localization near-minimizer]
\label{loc:ass:loc-near-minimizer}
The selected same-gauge representative satisfies
\[
    \|\calD-\zeta_*\|_{\loc,\pkg}^{\sharp,\loc}
    \le
    \Dist_{\loc,\pkg}^{\sharp,\loc}(\calD,\Gadm)
    +
    \delta_{\loc}.
\]
\end{assumption}

\subsection{Main Localization Absorption Theorems}
\label{loc:sec:main-theorems}

\begin{theorem}[Bounded-amplitude localization leakage absorption]
\label{loc:thm:bounded-amplitude-localization-absorption}
Let \(\calD\) be a sharp localized package.  Assume the same-gauge
representative \(\zeta_*(\calD)\), the sharp localization near-minimizer
condition in \Cref{loc:ass:loc-near-minimizer}, and the finite-amplitude bound
\[
    \|u_{\calD-\zeta_*}\|_{L^3(\Qone)}
    \le
    M_U.
\]
Then
\[
    \Err_{\loc}(\calD;\zeta_*)
    \le
    C_{\loc}(M_U)
    \Dist_{\loc,\pkg}^{\sharp,\loc}(\calD,\Gadm)
    +
    C_{\loc}(M_U)\delta_{\loc}.
\]
\end{theorem}

\begin{proof}
By \Cref{loc:thm:momentum-localization-leakage,loc:thm:flux-leakage-estimate}, applied
at \(\zeta_*\), the momentum and flux components satisfy
\[
\begin{aligned}
    \Err_{\loc}^{\mom}(\calD;\zeta_*)
    +
    \Err_{\loc}^{\flux}(\calD;\zeta_*)
    &\le
    C_{\loc}(M_U)
    \bigl[
        \Leak_{\nabla u}(\calD;\zeta_*)
        +
        \Leak_u(\calD;\zeta_*) \\
    &\qquad\qquad+
        \Leak_p(\calD;\zeta_*)
    \bigr].
\end{aligned}
\]
By definition,
\[
    \Err_{\loc}(\calD;\zeta_*)
    =
    \Err_{\loc}^{\mom}(\calD;\zeta_*)
    +
    \Err_{\loc}^{\flux}(\calD;\zeta_*).
\]
The leakage coordinates on the right-hand side are included in
\(\|\calD-\zeta_*\|_{\loc,\pkg}^{\sharp,\loc}\).  Thus
\[
    \Err_{\loc}(\calD;\zeta_*)
    \le
    C_{\loc}(M_U)
    \|\calD-\zeta_*\|_{\loc,\pkg}^{\sharp,\loc}.
\]
The sharp localization near-minimizer property in
\Cref{loc:ass:loc-near-minimizer} completes the proof.
\end{proof}

\begin{corollary}[Weighted absorption]
\label{loc:thm:weighted-localization-absorption}
Assume the hypotheses of
\Cref{loc:thm:bounded-amplitude-localization-absorption}.
For a prescribed \(\eta_{\loc}>0\), define a weighted localization norm
\(\|\cdot\|_{\loc,\pkg}^{\sharp,\loc,\omega}\) satisfying
\[
    \|\calD\|_{\loc,\pkg}^{\sharp,\loc,\omega}
    \ge
    \frac{C_{\loc}(M_U)}{\eta_{\loc}}
    \|\calD\|_{\loc,\pkg}^{\sharp,\loc}.
\]
If the selected representative is a near-minimizer in the weighted norm with
error \(\delta_{\loc}^{\omega}\), then
\[
    \Err_{\loc}(\calD;\zeta_*)
    \le
    \eta_{\loc}
    \Dist_{\loc,\pkg}^{\sharp,\loc,\omega}(\calD,\Gadm)
    +
    \eta_{\loc}\delta_{\loc}^{\omega}.
\]
\end{corollary}

\begin{proof}
Repeating the first estimate in the proof of
\Cref{loc:thm:bounded-amplitude-localization-absorption} gives
\[
    \Err_{\loc}(\calD;\zeta_*)
    \le
    C_{\loc}(M_U)
    \|\calD-\zeta_*\|_{\loc,\pkg}^{\sharp,\loc}
    \le
    \eta_{\loc}
    \|\calD-\zeta_*\|_{\loc,\pkg}^{\sharp,\loc,\omega}.
\]
The weighted near-minimizer property gives the displayed estimate.  This is an
accounting normalization, not a new PDE estimate.
\end{proof}

\begin{theorem}[Quadratic localization geometry]
\label{loc:thm:quadratic-localization-absorption}
Let \(\calD\) be a sharp localized package, and let \(\zeta_*(\calD)\) be the
same-gauge representative from \Cref{loc:ass:same-gauge-loc}.  Assume that the
selected representative satisfies a quadratic localization near-minimizer
condition
\[
    \|\calD-\zeta_*\|_{\loc,\pkg}^{\sharp,\loc,\quadg}
    \le
    \Dist_{\loc,\pkg}^{\sharp,\loc,\quadg}(\calD,\Gadm)
    +
    \delta_{\loc}^{\quadg}.
\]
Then
\[
    \Err_{\loc}(\calD;\zeta_*)
    \le
    C_{\loc}^{\quadg}
    \Dist_{\loc,\pkg}^{\sharp,\loc,\quadg}(\calD,\Gadm)
    +
    C_{\loc}^{\quadg}\delta_{\loc}^{\quadg}.
\]
\end{theorem}

\begin{proof}
Use the unlinearized estimates in
\Cref{loc:thm:momentum-localization-leakage,loc:thm:flux-leakage-estimate}.  Their
right-hand sides are bounded by a fixed finite-window constant times
\[
\begin{aligned}
    &\Leak_{\nabla u}(\calD;\zeta_*)
    +
    \Leak_u(\calD;\zeta_*)
    +
    \Leak_u(\calD;\zeta_*)^2
    +
    \Leak_u(\calD;\zeta_*)^3
\\
    &\qquad+
    \Leak_p(\calD;\zeta_*)\Leak_u(\calD;\zeta_*)
    +
    \Leak_p(\calD;\zeta_*).
\end{aligned}
\]
These are precisely the localization leakage coordinates included in
\(\|\calD-\zeta_*\|_{\loc,\pkg}^{\sharp,\loc,\quadg}\), up to fixed
finite-window constants.  Therefore
\[
    \Err_{\loc}(\calD;\zeta_*)
    \le
    C_{\loc}^{\quadg}
    \|\calD-\zeta_*\|_{\loc,\pkg}^{\sharp,\loc,\quadg}.
\]
The quadratic near-minimizer property gives the result.  The cost is that the
geometry has changed.
\end{proof}

\begin{corollary}[Normalized quotient amplitude variant]
\label{loc:cor:normalized-quotient-amplitude}
Assume the same-gauge representative and sharp localization near-minimizer
hypotheses from \Cref{loc:thm:bounded-amplitude-localization-absorption}.  If on
the normalized quotient sphere
\[
    \Dist_{\loc,\pkg}^{\sharp,\loc}(\calD,\Gadm)=1,
\]
the selected representative satisfies
\[
    \|u_{\calD-\zeta_*}\|_{L^3(\Qone)}
    \le
    C_{\mathrm{amp}}.
\]
Then the bounded-amplitude localization absorption theorem applies on that
normalized quotient with \(M_U=C_{\mathrm{amp}}\).
\end{corollary}

\begin{remark}[Status of normalized quotient amplitude]
This does not remove finite amplitude globally.  It only says that, on the
normalized quotient sphere, amplitude may be treated as part of the package
geometry.
\end{remark}

\subsection*{Reproduction Drift Absorption}
\subsection{Two-Window Geometry and Reproduction Maps}
\label{rep:sec:two-window-geometry}

\subsubsection{Fixed windows}

Let
\[
    Q^{(0)},\qquad Q^{(1)}
\]
be two normalized finite windows, each identified with a copy of
\[
    \Qone=B_1\times(-1,0).
\]
All constants in this paper are finite-window constants.  They may depend
on the chosen windows, cutoffs, and reproduction maps.  No scale-uniform
boundedness is claimed.

\begin{definition}[Reproduction map]
\label{rep:def:reproduction-map}
A fixed reproduction map
\[
    \calR_{0\to1}
\]
consists of coordinate maps
\[
    \calR^u:L^3(Q^{(0)})^3\to L^3(Q^{(1)})^3,
\]
\[
    \calR^{\nabla u}:
    L^2_tL^2_x(Q^{(0)})^{3\times3}
    \to
    L^2_tL^2_x(Q^{(1)})^{3\times3},
\]
\[
    \calR^{\src}:\Xsrc^{(0)}\to\Xsrc^{(1)},
    \qquad
    \calR^{\prs}:\Yprs^{(0)}\to\Yprs^{(1)},
    \qquad
    \calR^{\harm}:Y_{\harm}^{(3/2),(0)}\to Y_{\harm}^{(3/2),(1)},
\]
and finite-window coordinate maps
\[
    \calR^U,\quad
    \calR^R,\quad
    \calR^E,\quad
    \calR^\Pi,\quad
    \calR^\Phi,\quad
    \calR^T,\quad
    \calR^s.
\]
All these maps are fixed throughout the two-window problem, and their operator
norms may enter the constants.
\end{definition}

\begin{remark}[Finite-window status]
The maps in \Cref{rep:def:reproduction-map} are structural inputs.  The paper
does not prove that they form a scale-uniform semigroup or an exact
Navier--Stokes evolution operator.
\end{remark}

\subsection{Two-Window Sharp Packages}
\label{rep:sec:two-window-packages}

\begin{definition}[Two-window package pair]
\label{rep:def:two-window-package-pair}
For \(j=0,1\), let
\[
    D_j
    =
    \bigl(
        u_j,\,
        U_j,\,
        R_j,\,
        F^{\mathrm{act}}_j,\,
        F^{\mathrm{mod}}_j,\,
        E_{F,j},\,
        p^{\mathrm{act}}_j,\,
        p_{\harm,j},\,
        \Pi_j,\,
        \Phi_j,\,
        T_j,\,
        s_j
    \bigr)
\]
be sharp localized packages on \(Q^{(j)}\).  The two-window package pair is
\[
    \bD=(D_0,D_1).
\]
The reproduction problem compares \(D_1\) with \(\calR_{0\to1}D_0\).
\end{definition}

\subsection{Chain Gauge and Same-Chain Representative}
\label{rep:sec:chain-gauge}

\begin{definition}[Chain gauge class]
\label{rep:def:chain-gauge}
The admissible chain gauge class is
\[
    \Gchain
    \subset
    \Gamma_{\Lambda,\adm}^{(0)}
    \times
    \Gamma_{\Lambda,\adm}^{(1)}.
\]
A chain gauge element is
\[
    \bzeta=(\zeta_0,\zeta_1).
\]
The conservative physical gauge convention is imposed on both windows:
\[
    (\zeta_0)_u=0,
    \qquad
    (\zeta_1)_u=0.
\]
Thus
\[
    u_{D_j-\zeta_j}=u_j,\qquad j=0,1.
\]
Model, source, and pressure-tail coordinates may shift.
\end{definition}

\begin{assumption}[Same-chain representative]
\label{rep:ass:same-chain-representative}
For every two-window package pair \(\bD\), a representative
\[
    \bzeta_*(\bD)=(\zeta_{0,*},\zeta_{1,*})\in\Gchain
\]
is selected and used simultaneously in every reproduction drift channel and in
the sharp chain quotient distance.  Different drift channels are not allowed
to optimize over different gauges.
\end{assumption}

\begin{remark}[Gauge status]
The shifted packages \(D_j-\zeta_j\) are quotient representatives.  They are
not claimed to be generated by different Navier--Stokes solutions.
\end{remark}

\subsection{Basic Reproduction Drift Coordinates}
\label{rep:sec:basic-drift}

\begin{definition}[Primitive drift coordinates]
\label{rep:def:primitive-drift-coordinates}
For a pair \(\bD=(D_0,D_1)\) and a chain representative
\(\bzeta=(\zeta_0,\zeta_1)\), define
\[
    \Rep_u(\bD;\bzeta)
    :=
    \|u_1-\calR^u u_0\|_{L^3(Q^{(1)})}.
\]
If gradient drift is used, set
\[
    \Rep_{\nabla u}(\bD;\bzeta)
    :=
    \|\nabla u_1-\calR^{\nabla u}\nabla u_0\|_{L^2_tL^2_x(Q^{(1)})}.
\]
The model and covariance drift coordinates are
\[
    \Rep_U(\bD;\bzeta)
    :=
    \|U_{1-\zeta_1}-\calR^U U_{0-\zeta_0}\|_{L^3(Q^{(1)})},
\]
\[
    \Rep_R(\bD;\bzeta)
    :=
    \|R_{1-\zeta_1}-\calR^R R_{0-\zeta_0}\|_{L^{3/2}(Q^{(1)})}.
\]
The residual-source and harmonic pressure drift coordinates are
\[
    \Rep_E(\bD;\bzeta)
    :=
    \|E_{F,1-\zeta_1}-\calR^E E_{F,0-\zeta_0}\|_{\Xsrc^{(1)}},
\]
\[
    \Rep_{\harm}(\bD;\bzeta)
    :=
    \|p_{\harm,1-\zeta_1}
    -
    \calR^{\harm}p_{\harm,0-\zeta_0}\|_{Y_{\harm}^{(3/2),(1)}}.
\]
Finally define the auxiliary finite-window drift coordinates
\[
    \Rep_{\Pi}:=\|\Pi_{1-\zeta_1}-\calR^\Pi\Pi_{0-\zeta_0}\|_{\mathcal P},
\]
\[
    \Rep_{\Phi}:=\|\Phi_{1-\zeta_1}-\calR^\Phi\Phi_{0-\zeta_0}\|_{\mathcal F},
    \qquad
    \Rep_T:=\|T_{1-\zeta_1}-\calR^T T_{0-\zeta_0}\|_{\mathcal T}.
\]
\end{definition}

\begin{remark}[Gate/slack convention]
The slack coordinate \(s\) may be transported passively in the reproduction
coordinates.  Gate and slack mismatches are accounted for in the dedicated
gate/slack component geometry.
\end{remark}

\subsection{Active Source Reproduction Estimate}
\label{rep:sec:active-source}

\begin{assumption}[Active source realization]
\label{rep:ass:active-source-realization}
On the representatives used in this section,
\[
    F^{\mathrm{act}}_{1-\zeta_1}
    =
    \eta_1 u_1\otimes u_1,
\]
where \(\eta_1\) is a fixed bounded cutoff multiplier on \(Q^{(1)}\).  The
source \(F^{\mathrm{act}}_{0-\zeta_0}\) is transported by the fixed map
\(\calR^{\src}\).  No commutation between \(\calR^{\src}\) and the quadratic
velocity source is assumed except through the leakage coordinate below.
\end{assumption}

\begin{definition}[Active source drift and leakage]
\label{rep:def:active-source-drift}
The active source drift is
\[
    \Rep_{F^{\mathrm{act}}}(\bD;\bzeta)
    :=
    \|F^{\mathrm{act}}_{1-\zeta_1}
    -
    \calR^{\src}F^{\mathrm{act}}_{0-\zeta_0}\|_{\Xsrc^{(1)}}.
\]
The reproduction source leakage coordinate is
\[
    \Leak_{\rep}^{F}(\bD)
    :=
    \|\eta_1(\calR^u u_0)\otimes(\calR^u u_0)
    -
    \calR^{\src}F^{\mathrm{act}}_{0-\zeta_0}\|_{\Xsrc^{(1)}}.
\]
This term records cutoff or reproduction-map mismatch not captured by
\(\Rep_u\).
\end{definition}

\begin{proposition}[Active source reproduction]
\label{rep:prop:active-source-reproduction}
Assume \Cref{rep:ass:active-source-realization}.  Assume also that
\(\Xsrc^{(1)}=L^{3/2}(Q^{(1)})^{3\times3}\), or that the
\(\Xsrc^{(1)}\)-norm is dominated by a fixed multiple of this norm on the
sources considered here.
Assume
\[
    \|u_1\|_{L^3(Q^{(1)})}
    +
    \|\calR^u u_0\|_{L^3(Q^{(1)})}
    \le
    M_U.
\]
Then
\[
    \Rep_{F^{\mathrm{act}}}(\bD;\bzeta)
    \le
    C_\eta M_U\Rep_u(\bD;\bzeta)
    +
    C\Leak_{\rep}^{F}(\bD).
\]
\end{proposition}

\begin{proof}
By the triangle inequality and \Cref{rep:ass:active-source-realization},
\[
\begin{aligned}
    \Rep_{F^{\mathrm{act}}}(\bD;\bzeta)
    &\le
    \|\eta_1u_1\otimes u_1
    -
    \eta_1(\calR^u u_0)\otimes(\calR^u u_0)\|_{\Xsrc^{(1)}}\\
    &\quad+
    \|\eta_1(\calR^u u_0)\otimes(\calR^u u_0)
    -
    \calR^{\src}F^{\mathrm{act}}_{0-\zeta_0}\|_{\Xsrc^{(1)}}.
\end{aligned}
\]
The second term is \(\Leak_{\rep}^{F}(\bD)\).  For the first term, write
\[
    u_1\otimes u_1
    -
    (\calR^u u_0)\otimes(\calR^u u_0)
\]
as
\[
    (u_1-\calR^u u_0)\otimes u_1
    +
    (\calR^u u_0)\otimes (u_1-\calR^u u_0),
\]
and use Hölder's inequality \(L^3\cdot L^3\to L^{3/2}\).  Since
\(\eta_1\) is fixed and bounded,
\[
\begin{aligned}
    &\|\eta_1u_1\otimes u_1
    -
    \eta_1(\calR^u u_0)\otimes(\calR^u u_0)\|_{\Xsrc^{(1)}}\\
    &\qquad\le
    C_\eta
    \|u_1-\calR^u u_0\|_{L^3(Q^{(1)})}
    \left(
        \|u_1\|_{L^3(Q^{(1)})}
        +
        \|\calR^u u_0\|_{L^3(Q^{(1)})}
    \right)\\
    &\qquad\le
    C_\eta M_U\Rep_u(\bD;\bzeta).
\end{aligned}
\]
Combining the two bounds proves the estimate.
\end{proof}

\subsection{Model Source Reproduction Estimate}
\label{rep:sec:model-source}

\begin{assumption}[Model source realization]
\label{rep:ass:model-source-realization}
On the representatives used in this section,
\[
    F^{\mathrm{mod}}_{1-\zeta_1}
    =
    \eta_1
    \bigl(
        U_{1-\zeta_1}\otimes U_{1-\zeta_1}
        +
        R_{1-\zeta_1}
    \bigr),
\]
with the same fixed bounded multiplier \(\eta_1\).  The transported model
source \(\calR^{\src}F^{\mathrm{mod}}_{0-\zeta_0}\) is compared with the
model expression built from \(\calR^U U_{0-\zeta_0}\) and
\(\calR^R R_{0-\zeta_0}\) through the leakage coordinate below.
\end{assumption}

\begin{definition}[Model source drift and leakage]
\label{rep:def:model-source-drift}
The model source drift is
\[
    \Rep_{F^{\mathrm{mod}}}(\bD;\bzeta)
    :=
    \|F^{\mathrm{mod}}_{1-\zeta_1}
    -
    \calR^{\src}F^{\mathrm{mod}}_{0-\zeta_0}\|_{\Xsrc^{(1)}}.
\]
The model reproduction leakage coordinate
\[
    \Leak_{\rep}^{F,\mathrm{mod}}(\bD;\bzeta)
\]
is defined by
\[
\begin{aligned}
    \Leak_{\rep}^{F,\mathrm{mod}}(\bD;\bzeta)
    &:=
    \|\eta_1[
        (\calR^U U_{0-\zeta_0})\otimes(\calR^U U_{0-\zeta_0})
        +
        \calR^R R_{0-\zeta_0}
    ]\\
    &\qquad\qquad-
    \calR^{\src}F^{\mathrm{mod}}_{0-\zeta_0}
    \|_{\Xsrc^{(1)}}.
\end{aligned}
\]
It records cutoff and reproduction-map noncommutation for
\(F^{\mathrm{mod}}=\eta(U\otimes U+R)\).
\end{definition}

\begin{proposition}[Model source reproduction]
\label{rep:prop:model-source-reproduction}
Assume \Cref{rep:ass:model-source-realization}.  Assume also that
\(\Xsrc^{(1)}=L^{3/2}(Q^{(1)})^{3\times3}\), or that the
\(\Xsrc^{(1)}\)-norm is dominated by a fixed multiple of this norm on the
sources considered here.
Assume
\[
    \|U_{1-\zeta_1}\|_{L^3(Q^{(1)})}
    +
    \|\calR^U U_{0-\zeta_0}\|_{L^3(Q^{(1)})}
    \le
    M_U.
\]
Then
\[
    \Rep_{F^{\mathrm{mod}}}(\bD;\bzeta)
    \le
    C_\eta M_U\Rep_U(\bD;\bzeta)
    +
    C_\eta\Rep_R(\bD;\bzeta)
    +
    C\Leak_{\rep}^{F,\mathrm{mod}}(\bD;\bzeta).
\]
\end{proposition}

\begin{proof}
By the triangle inequality and \Cref{rep:ass:model-source-realization},
\[
\begin{aligned}
    \Rep_{F^{\mathrm{mod}}}(\bD;\bzeta)
    &\le
    \|\eta_1[
        U_{1-\zeta_1}\otimes U_{1-\zeta_1}
        -
        (\calR^U U_{0-\zeta_0})\otimes(\calR^U U_{0-\zeta_0})
    ]\|_{\Xsrc^{(1)}}\\
    &\quad+
    \|\eta_1[
        R_{1-\zeta_1}
        -
        \calR^R R_{0-\zeta_0}
    ]\|_{\Xsrc^{(1)}}
    +
    \Leak_{\rep}^{F,\mathrm{mod}}(\bD;\bzeta).
\end{aligned}
\]
The quadratic term is estimated by expanding the difference of tensor products
and using \(L^3\cdot L^3\to L^{3/2}\):
\[
\begin{aligned}
    &\|\eta_1[
        U_{1-\zeta_1}\otimes U_{1-\zeta_1}
        -
        (\calR^U U_{0-\zeta_0})\otimes(\calR^U U_{0-\zeta_0})
    ]\|_{\Xsrc^{(1)}}\\
    &\qquad\le
    C_\eta
    \Rep_U(\bD;\bzeta)
    \left(
        \|U_{1-\zeta_1}\|_{L^3(Q^{(1)})}
        +
        \|\calR^U U_{0-\zeta_0}\|_{L^3(Q^{(1)})}
    \right)\\
    &\qquad\le
    C_\eta M_U\Rep_U(\bD;\bzeta).
\end{aligned}
\]
The covariance term is linear:
\[
    \|\eta_1[
        R_{1-\zeta_1}
        -
        \calR^R R_{0-\zeta_0}
    ]\|_{\Xsrc^{(1)}}
    \le
    C_\eta\Rep_R(\bD;\bzeta).
\]
Combining the estimates proves the claim.
\end{proof}

\subsection{Pressure Reproduction and Riesz-Map Intertwining}
\label{rep:sec:pressure-reproduction}

\begin{definition}[Pressure reproduction commutator]
\label{rep:def:pressure-reproduction-commutator}
For a source \(F\), define
\[
    \calC_{\calR}^{\prs}(F)
    :=
    R_iR_j(\calR^{\src}F_{ij})
    -
    \calR^{\prs}R_iR_j(F_{ij}).
\]
This measures the failure of the reproduction map to intertwine with the
Riesz-transform pressure operator.
\end{definition}

\begin{assumption}[Fixed-window pressure intertwining budget]
\label{rep:ass:pressure-intertwining-budget}
There is a reproduction pressure-leakage coordinate
\(\Leak_{\calR}^{\prs}(F)\) and a finite-window constant \(C_{\calR}\) such
that
\[
    \|\calC_{\calR}^{\prs}(F)\|_{\Yprs^{(1)}}
    \le
    C_{\calR}\Leak_{\calR}^{\prs}(F).
\]
\end{assumption}

\begin{definition}[Active pressure drift]
\label{rep:def:active-pressure-drift}
The active pressure reproduction drift is
\[
    \Rep_{p^{\mathrm{act}}}(\bD;\bzeta)
    :=
    \|p^{\mathrm{act}}_{1-\zeta_1}
    -
    \calR^{\prs}p^{\mathrm{act}}_{0-\zeta_0}\|_{\Yprs^{(1)}}.
\]
\end{definition}

\begin{assumption}[Active pressure realization]
\label{rep:ass:active-pressure-realization}
On the representatives used in this section,
\[
    p^{\mathrm{act}}_{j-\zeta_j}
    =
    R_iR_j(F^{\mathrm{act}}_{j-\zeta_j,ij}),
    \qquad j=0,1.
\]
\end{assumption}

\begin{proposition}[Active pressure reproduction]
\label{rep:prop:active-pressure-reproduction}
Assume \Cref{rep:ass:active-pressure-realization} and
\Cref{rep:ass:pressure-intertwining-budget}.  Assume also the fixed-window
Calderon--Zygmund source-to-pressure bound
\[
    \|R_iR_j G_{ij}\|_{\Yprs^{(1)}}
    \le
    C_{CZ}\|G\|_{\Xsrc^{(1)}}.
\]
Then
\[
    \Rep_{p^{\mathrm{act}}}(\bD;\bzeta)
    \le
    C_{CZ}\Rep_{F^{\mathrm{act}}}(\bD;\bzeta)
    +
    C_{\calR}\Leak_{\calR}^{\prs}(F^{\mathrm{act}}_{0-\zeta_0}).
\]
\end{proposition}

\begin{proof}
By \Cref{rep:ass:active-pressure-realization},
\[
\begin{aligned}
    p^{\mathrm{act}}_{1-\zeta_1}
    -
    \calR^{\prs}p^{\mathrm{act}}_{0-\zeta_0}
    &=
    R_iR_j(F^{\mathrm{act}}_{1-\zeta_1,ij})
    -
    \calR^{\prs}R_iR_j(F^{\mathrm{act}}_{0-\zeta_0,ij})\\
    &=
    R_iR_j\bigl(
        F^{\mathrm{act}}_{1-\zeta_1,ij}
        -
        \calR^{\src}F^{\mathrm{act}}_{0-\zeta_0,ij}
    \bigr)
    +
    \calC_{\calR}^{\prs}(F^{\mathrm{act}}_{0-\zeta_0}).
\end{aligned}
\]
Taking the \(\Yprs^{(1)}\)-norm, applying the Calderon--Zygmund bound to the
first term, and using \Cref{rep:ass:pressure-intertwining-budget} for the second
term gives the estimate.
\end{proof}

\begin{remark}[No silent pressure transport]
This section is deliberately conditional.  The paper does not assume that
Riesz transforms commute with reproduction maps without the commutator budget
in \Cref{rep:ass:pressure-intertwining-budget}.
\end{remark}

\subsection{Harmonic Pressure Reproduction}
\label{rep:sec:harmonic-reproduction}

\begin{definition}[Harmonic pressure drift]
\label{rep:def:harmonic-pressure-drift}
The harmonic pressure drift is the coordinate
\[
    \Rep_{\harm}(\bD;\bzeta)
    =
    \|p_{\harm,1-\zeta_1}
    -
    \calR^{\harm}p_{\harm,0-\zeta_0}\|_{Y_{\harm}^{(3/2),(1)}}.
\]
\end{definition}

\begin{remark}[Coordinate absorption convention]
In the main theorem, \(\Rep_{\harm}\) is included directly in the sharp
chain norm.  This is a coordinate absorption convention: no separate harmonic
reproduction approximation theorem is required for the finite-window closure
proved here.
\end{remark}

\subsection{Total Reproduction Residual}
\label{rep:sec:total-residual}

\begin{definition}[Total reproduction residual]
\label{rep:def:total-reproduction-residual}
The total reproduction residual is
\[
\begin{aligned}
    \Err_{\rep}(\bD;\bzeta)
    &:=
    \Rep_u+\Rep_U+\Rep_R+\Rep_E
    +
    \Rep_{F^{\mathrm{act}}}
    +
    \Rep_{F^{\mathrm{mod}}}
\\
    &\quad+
    \Rep_{p^{\mathrm{act}}}
    +
    \Rep_{\harm}
    +
    \Rep_{\Pi}
    +
    \Rep_{\Phi}
    +
    \Rep_T,
\end{aligned}
\]
with all terms evaluated at \((\bD;\bzeta)\).
\end{definition}

\begin{remark}[Residual splitting]
One may split
\[
    \Err_{\rep}
    =
    \Err_{\rep}^{\mathrm{phys}}
    +
    \Err_{\rep}^{\src}
    +
    \Err_{\rep}^{\prs}
    +
    \Err_{\rep}^{\mathrm{aux}},
\]
but the present paper keeps the single total residual notation.
\end{remark}

\subsection{Sharp Chain Reproduction Norm}
\label{rep:sec:chain-norm}

\begin{definition}[Sharp chain reproduction norm]
\label{rep:def:sharp-chain-reproduction-norm}
Let \(\|D_j\|_{\loc,\pkg}^{\sharp}\) denote the sharp single-window package
norm imported from the previous branches.  Define
\[
\begin{aligned}
    \|\bD\|_{\chain}^{\sharp,\rep}
    &:=
    \|D_0\|_{\loc,\pkg}^{\sharp}
    +
    \|D_1\|_{\loc,\pkg}^{\sharp}
    +
    \Rep_u+\Rep_U+\Rep_R+\Rep_E+\Rep_{\harm}
\\
    &\quad+
    \Rep_{\Pi}+\Rep_{\Phi}+\Rep_T
    +
    \Leak_{\rep}^{F}
    +
    \Leak_{\rep}^{F,\mathrm{mod}}
    +
    \Leak_{\calR}^{\prs}(F^{\mathrm{act}}_{0-\zeta_0}).
\end{aligned}
\]
The corresponding quotient distance is
\[
    \Dist_{\chain}^{\sharp,\rep}(\bD,\Gchain)
    :=
    \inf_{\bzeta\in\Gchain}
    \|\bD-\bzeta\|_{\chain}^{\sharp,\rep}.
\]
\end{definition}

\begin{assumption}[Sharp chain near-minimizer]
\label{rep:ass:chain-near-minimizer}
The selected same-chain representative satisfies
\[
    \|\bD-\bzeta_*\|_{\chain}^{\sharp,\rep}
    \le
    \Dist_{\chain}^{\sharp,\rep}(\bD,\Gchain)
    +
    \delta_{\rep}.
\]
\end{assumption}

\subsection{Main Reproduction Absorption Theorems}
\label{rep:sec:main-theorems}

\begin{theorem}[Bounded-amplitude reproduction drift absorption]
\label{rep:thm:bounded-amplitude-reproduction-absorption}
Let \(\bD=(D_0,D_1)\) be a two-window sharp package pair.  Assume:
\begin{enumerate}[label=\textup{(\roman*)},leftmargin=*]
    \item a fixed finite-window reproduction map \(\calR_{0\to1}\);
    \item the same-chain representative \(\bzeta_*\);
    \item the near-minimizer condition in \Cref{rep:ass:chain-near-minimizer};
    \item the source and pressure realization hypotheses
    \Cref{rep:ass:active-source-realization,rep:ass:model-source-realization,rep:ass:active-pressure-realization};
    \item the bounded-amplitude condition
    \[
        \|u_1\|_{L^3}
        +
        \|\calR^u u_0\|_{L^3}
        +
        \|U_{1-\zeta_{1,*}}\|_{L^3}
        +
        \|\calR^U U_{0-\zeta_{0,*}}\|_{L^3}
        \le
        M_U;
    \]
    \item the pressure-intertwining budget in
    \Cref{rep:ass:pressure-intertwining-budget}.
\end{enumerate}
Then
\[
    \Err_{\rep}(\bD;\bzeta_*)
    \le
    C_{\rep}(M_U)
    \Dist_{\chain}^{\sharp,\rep}(\bD,\Gchain)
    +
    C_{\rep}(M_U)\delta_{\rep}.
\]
\end{theorem}

\begin{proof}
All estimates below are evaluated at the same-chain representative
\(\bzeta_*\).  By \Cref{rep:prop:active-source-reproduction},
\[
    \Rep_{F^{\mathrm{act}}}(\bD;\bzeta_*)
    \le
    C(M_U)
    \bigl[
        \Rep_u(\bD;\bzeta_*)
        +
        \Leak_{\rep}^{F}(\bD)
    \bigr].
\]
By \Cref{rep:prop:model-source-reproduction},
\[
    \Rep_{F^{\mathrm{mod}}}(\bD;\bzeta_*)
    \le
    C(M_U)
    \bigl[
        \Rep_U(\bD;\bzeta_*)
        +
        \Rep_R(\bD;\bzeta_*)
        +
        \Leak_{\rep}^{F,\mathrm{mod}}(\bD;\bzeta_*)
    \bigr].
\]
By \Cref{rep:prop:active-pressure-reproduction} and the active source estimate,
\[
    \Rep_{p^{\mathrm{act}}}(\bD;\bzeta_*)
    \le
    C(M_U)
    \bigl[
        \Rep_u(\bD;\bzeta_*)
        +
        \Leak_{\rep}^{F}(\bD)
        +
        \Leak_{\calR}^{\prs}(F^{\mathrm{act}}_{0-\zeta_{0,*}})
    \bigr].
\]
The remaining terms in \(\Err_{\rep}\),
\[
    \Rep_u,\quad
    \Rep_U,\quad
    \Rep_R,\quad
    \Rep_E,\quad
    \Rep_{\harm},\quad
    \Rep_{\Pi},\quad
    \Rep_{\Phi},\quad
    \Rep_T,
\]
are primitive coordinates included in
\(\|\bD-\bzeta_*\|_{\chain}^{\sharp,\rep}\).  Therefore
\[
    \Err_{\rep}(\bD;\bzeta_*)
    \le
    C_{\rep}(M_U)
    \|\bD-\bzeta_*\|_{\chain}^{\sharp,\rep}.
\]
The near-minimizer condition in \Cref{rep:ass:chain-near-minimizer} gives the
claimed quotient-distance estimate.
\end{proof}

\begin{corollary}[Weighted reproduction absorption]
\label{rep:cor:weighted-reproduction-absorption}
If a weighted norm satisfies
\[
    \|\bD\|_{\chain}^{\sharp,\rep,\omega}
    \ge
    \frac{C_{\rep}(M_U)}{\eta_{\rep}}
    \|\bD\|_{\chain}^{\sharp,\rep},
\]
and the same-chain representative is a weighted near-minimizer with error
\(\delta_{\rep}^{\omega}\), then
\[
    \Err_{\rep}(\bD;\bzeta_*)
    \le
    \eta_{\rep}
    \Dist_{\chain}^{\sharp,\rep,\omega}(\bD,\Gchain)
    +
    \eta_{\rep}\delta_{\rep}^{\omega}.
\]
\end{corollary}

\begin{proof}
The proof of \Cref{rep:thm:bounded-amplitude-reproduction-absorption} first gives
\[
    \Err_{\rep}(\bD;\bzeta_*)
    \le
    C_{\rep}(M_U)\|\bD-\bzeta_*\|_{\chain}^{\sharp,\rep}.
\]
The weighted norm domination implies
\[
    C_{\rep}(M_U)\|\bD-\bzeta_*\|_{\chain}^{\sharp,\rep}
    \le
    \eta_{\rep}\|\bD-\bzeta_*\|_{\chain}^{\sharp,\rep,\omega}.
\]
Applying the weighted near-minimizer property gives the result.  This is a
normalization of the accounting norm, not a new PDE estimate.
\end{proof}

\begin{definition}[Quadratic chain reproduction norm]
\label{rep:def:quadratic-chain-reproduction-norm}
The quadratic chain norm is obtained from the sharp chain norm by retaining the
nonlinear product factors
\[
    \Rep_u
    \bigl(
        \|u_1\|_{L^3}
        +
        \|\calR^u u_0\|_{L^3}
    \bigr)
\]
and
\[
    \Rep_U
    \bigl(
        \|U_{1-\zeta_1}\|_{L^3}
        +
        \|\calR^U U_{0-\zeta_0}\|_{L^3}
    \bigr)
\]
instead of replacing them by a finite-amplitude constant.  More precisely, it
is any chain functional \(\|\cdot\|_{\chain}^{\sharp,\rep,\quadg}\) that
dominates the primitive drift coordinates, the leakage coordinates, and the
two displayed nonlinear products.  It defines the quotient distance
\[
    \Dist_{\chain}^{\sharp,\rep,\quadg}(\bD,\Gchain)
    :=
    \inf_{\bzeta\in\Gchain}
    \|\bD-\bzeta\|_{\chain}^{\sharp,\rep,\quadg}.
\]
\end{definition}

\begin{theorem}[Quadratic chain geometry]
\label{rep:thm:quadratic-chain-geometry}
Assume the hypotheses of
\Cref{rep:ass:active-source-realization,rep:ass:model-source-realization,rep:ass:active-pressure-realization,rep:ass:pressure-intertwining-budget},
but do not assume bounded amplitude.  Suppose the selected same-chain
representative satisfies the quadratic near-minimizer condition
\[
    \|\bD-\bzeta_*\|_{\chain}^{\sharp,\rep,\quadg}
    \le
    \Dist_{\chain}^{\sharp,\rep,\quadg}(\bD,\Gchain)
    +
    \delta_{\rep}^{\quadg}.
\]
Then
\[
    \Err_{\rep}(\bD;\bzeta_*)
    \le
    C_{\rep}^{\quadg}
    \Dist_{\chain}^{\sharp,\rep,\quadg}(\bD,\Gchain)
    +
    C_{\rep}^{\quadg}\delta_{\rep}^{\quadg}.
\]
\end{theorem}

\begin{proof}
Repeat the proofs of
\Cref{rep:prop:active-source-reproduction,rep:prop:model-source-reproduction} without
linearizing the amplitude factors by \(M_U\).  This gives bounds by the
quadratic products built into
\(\|\bD-\bzeta_*\|_{\chain}^{\sharp,\rep,\quadg}\), together with the
primitive and leakage coordinates.  The active pressure estimate in
\Cref{rep:prop:active-pressure-reproduction} then bounds the pressure drift by the
same quadratic chain coordinates and the pressure-intertwining leakage.  Hence
\[
    \Err_{\rep}(\bD;\bzeta_*)
    \le
    C_{\rep}^{\quadg}
    \|\bD-\bzeta_*\|_{\chain}^{\sharp,\rep,\quadg}.
\]
The quadratic near-minimizer condition completes the proof.  The price is that
the geometry has changed; this is not a theorem in the original linear chain
geometry.
\end{proof}

\begin{corollary}[Normalized quotient amplitude]
\label{rep:cor:normalized-chain-amplitude}
On the normalized chain quotient sphere
\[
    \Dist_{\chain}^{\sharp,\rep}(\bD,\Gchain)=1,
\]
assume the selected representative satisfies
\[
    \|u_1\|_{L^3}
    +
    \|\calR^u u_0\|_{L^3}
    +
    \|U_{1-\zeta_{1,*}}\|_{L^3}
    +
    \|\calR^U U_{0-\zeta_{0,*}}\|_{L^3}
    \le
    C_{\mathrm{amp}}.
\]
Then the bounded-amplitude reproduction absorption theorem applies on that
normalized quotient
with \(M_U=C_{\mathrm{amp}}\).
\end{corollary}

\begin{proof}
This is exactly the bounded-amplitude hypothesis in
\Cref{rep:thm:bounded-amplitude-reproduction-absorption}, with
\(M_U=C_{\mathrm{amp}}\), restricted to the normalized quotient sphere.
\end{proof}

\begin{corollary}[Finite-chain bookkeeping]
\label{rep:cor:finite-chain-bookkeeping}
For a finite chain \(D_0,\ldots,D_K\), define
\[
    \Err_{\rep}^{[0,K]}
    :=
    \sum_{k=0}^{K-1}
    \Err_{\rep}(D_k,D_{k+1};\zeta_k,\zeta_{k+1}).
\]
Assume the finite-chain quotient distance is defined from the sum of the
pairwise chain norms.  If each neighboring pair satisfies the
bounded-amplitude reproduction theorem with constants bounded on the finite
chain, then
\[
    \Err_{\rep}^{[0,K]}
    \le
    C_{\rep}^{[0,K]}
    \Dist_{\chain}^{\sharp,\rep,[0,K]}
    ((D_0,\ldots,D_K),\Gamma_{\Lambda,\adm}^{\chain,[0,K]})
    +
    C_{\rep}^{[0,K]}
    \sum_{k=0}^{K-1}\delta_{\rep,k}.
\]
\end{corollary}

\begin{proof}
Apply \Cref{rep:thm:bounded-amplitude-reproduction-absorption} to each neighboring
pair and sum over \(0\le k<K\).  Since the chain is finite, the pairwise
constants are bounded by a finite constant \(C_{\rep}^{[0,K]}\).  The sum of
the pairwise chain distances is bounded by the finite-chain quotient distance
by the definition of the latter.  This yields the displayed estimate.
\end{proof}

\begin{remark}[Finite-chain status]
The finite-chain constant may depend on \(K\).  This is not a scale-uniform
iteration theorem.
\end{remark}

\subsection*{Gate/Slack Budget-Violation Absorption}
\subsection{Channel Index Set and Budget Variables}
\label{gs:sec:channels}

\begin{definition}[Gate/slack channel set]
\label{gs:def:channel-set}
Let \(\Achan\) be a finite set of gate/slack channels.  A typical model is
\[
    \Achan
    \subset
    \{\prs,\loc,\rep,\mathrm{flux},\mathrm{trace},\mathrm{tax},\adm\}.
\]
The exact finite set is part of the finite-window package specification.
\end{definition}

\begin{definition}[Budget, threshold, and slack coordinates]
\label{gs:def:budget-threshold-slack}
For every \(a\in\Achan\), a package representative \(D-\zeta\) carries:
\begin{enumerate}[label=\textup{(\roman*)},leftmargin=*]
    \item a nonnegative used-budget functional \(B_a(D;\zeta)\ge0\);
    \item a nonnegative allowed-threshold functional \(\tau_a(D;\zeta)\ge0\);
    \item a nonnegative slack coordinate \(s_a(D;\zeta)\ge0\);
    \item weights \(\omega_a>0\) and \(\lambda_a>0\).
\end{enumerate}
The ideal finite-window ledger relation is
\[
    B_a(D;\zeta)+s_a(D;\zeta)=\tau_a(D;\zeta).
\]
If this identity holds with \(s_a(D;\zeta)\ge0\), then the gate is admissible:
\[
    B_a(D;\zeta)\le \tau_a(D;\zeta).
\]
\end{definition}

\begin{remark}[Finite-window interpretation]
The quantities \(B_a,\tau_a,s_a\) are package coordinates or functionals.  This
paper does not derive them from a PDE estimate; it specifies how they are used
once included in the finite-window package.
\end{remark}

\subsection{Positive-Part Gate Violation}
\label{gs:sec:positive-gate}

\begin{definition}[Positive-part gate violation]
\label{gs:def:positive-gate}
For \(a\in\Achan\), define
\[
    \Gate_a(D;\zeta)
    :=
    \bigl(B_a(D;\zeta)-\tau_a(D;\zeta)\bigr)_+.
\]
Thus \(\Gate_a(D;\zeta)=0\) when \(B_a(D;\zeta)\le\tau_a(D;\zeta)\), and
\[
    \Gate_a(D;\zeta)=B_a(D;\zeta)-\tau_a(D;\zeta)
\]
when \(B_a(D;\zeta)>\tau_a(D;\zeta)\).
\end{definition}

\begin{lemma}[Positive-part Lipschitz estimate]
\label{gs:lem:positive-part-lipschitz-target}
For scalar budget pairs \((B,\tau)\) and
\((\widetilde B,\widetilde\tau)\),
\[
    \left|
        (B-\tau)_+
        -
        (\widetilde B-\widetilde\tau)_+
    \right|
    \le
    |B-\widetilde B|+|\tau-\widetilde\tau|.
\]
\end{lemma}

\begin{proof}
For real numbers \(x,y\), the map \(x\mapsto x_+\) is \(1\)-Lipschitz.  Indeed,
if \(x\ge y\), then
\[
    x_+-y_+
    \le
    x-y,
\]
because either \(y\ge0\), in which case equality holds, or \(y<0\), in which
case \(x_+-y_+\le x_+\le x-y\).  Exchanging \(x\) and \(y\) gives
\[
    |x_+-y_+|\le |x-y|
\]
for all \(x,y\in\R\).  Apply this with \(x=B-\tau\) and
\(y=\widetilde B-\widetilde\tau\).  Then
\[
    |x-y|
    =
    |(B-\widetilde B)-(\tau-\widetilde\tau)|
    \le
    |B-\widetilde B|+|\tau-\widetilde\tau|,
\]
which proves the claim.
\end{proof}

\subsection{Slack Identity Mismatch}
\label{gs:sec:slack-mismatch}

\begin{definition}[Slack identity mismatch]
\label{gs:def:slack-mismatch}
For \(a\in\Achan\), define
\[
    \Slack_a(D;\zeta)
    :=
    |B_a(D;\zeta)+s_a(D;\zeta)-\tau_a(D;\zeta)|.
\]
This measures the failure of the finite-window ledger identity.
\end{definition}

\begin{lemma}[Slack mismatch controls gate violation]
\label{gs:lem:slack-controls-gate-target}
Assume \(s_a(D;\zeta)\ge0\).  Then
\[
    \Gate_a(D;\zeta)
    \le
    \Slack_a(D;\zeta).
\]
\end{lemma}

\begin{proof}
Write \(B_a=B_a(D;\zeta)\), \(\tau_a=\tau_a(D;\zeta)\), and
\(s_a=s_a(D;\zeta)\).  Since \(s_a\ge0\),
\[
    B_a-\tau_a
    =
    (B_a+s_a-\tau_a)-s_a
    \le
    |B_a+s_a-\tau_a|.
\]
Taking the positive part of the left-hand side gives
\[
    (B_a-\tau_a)_+
    \le
    |B_a+s_a-\tau_a|.
\]
This is exactly \(\Gate_a(D;\zeta)\le\Slack_a(D;\zeta)\).
\end{proof}

\subsection{Single-Window Gate/Slack Residual}
\label{gs:sec:single-window-residual}

\begin{definition}[Single-window gate/slack residual]
\label{gs:def:single-window-gs-residual}
Define
\[
    \Err_{\gs}^{\mathrm{win}}(D;\zeta)
    :=
    \sum_{a\in\Achan}\omega_a\Gate_a(D;\zeta)
    +
    \sum_{a\in\Achan}\lambda_a\Slack_a(D;\zeta).
\]
Equivalently,
\[
\begin{aligned}
    \Err_{\gs}^{\mathrm{win}}(D;\zeta)
    &=
    \sum_{a\in\Achan}
    \omega_a
    \bigl(B_a(D;\zeta)-\tau_a(D;\zeta)\bigr)_+
\\
    &\quad+
    \sum_{a\in\Achan}
    \lambda_a
    |B_a(D;\zeta)+s_a(D;\zeta)-\tau_a(D;\zeta)|.
\end{aligned}
\]
\end{definition}

\begin{remark}[Meaning of the residual]
If the package uses exact slack identities, then the second sum vanishes.  If
the budgets are admissible, then the first sum vanishes.  Thus
\(\Err_{\gs}^{\mathrm{win}}\) measures both admissibility violation and
slack-ledger inconsistency.
\end{remark}

\subsection{Gate/Slack Chain Mismatch}
\label{gs:sec:chain-mismatch}

\begin{definition}[Scalar reproduction maps]
\label{gs:def:scalar-reproduction-maps}
For a two-window pair \(\bD=(D_0,D_1)\), assume that each channel \(a\in\Achan\)
comes with scalar reproduction maps
\[
    \calR_a^B,\qquad
    \calR_a^\tau,\qquad
    \calR_a^s,\qquad
    \calR_a^g.
\]
In the simplest finite-window model, these are identity maps on scalar budget
coordinates.  They are kept explicit to allow later transport or rescaling
models.
\end{definition}

\begin{definition}[Budget, threshold, slack, and gate reproduction mismatch]
\label{gs:def:chain-gs-mismatch}
For \(\bD=(D_0,D_1)\) and \(\bzeta=(\zeta_0,\zeta_1)\), define
\[
    \Rep_a^B(\bD;\bzeta)
    :=
    |B_a(D_1;\zeta_1)-\calR_a^B B_a(D_0;\zeta_0)|,
\]
\[
    \Rep_a^\tau(\bD;\bzeta)
    :=
    |\tau_a(D_1;\zeta_1)-\calR_a^\tau \tau_a(D_0;\zeta_0)|,
\]
\[
    \Rep_a^s(\bD;\bzeta)
    :=
    |s_a(D_1;\zeta_1)-\calR_a^s s_a(D_0;\zeta_0)|,
\]
and
\[
    \Rep_a^{\mathrm{gate}}(\bD;\bzeta)
    :=
    |\Gate_a(D_1;\zeta_1)-\calR_a^g\Gate_a(D_0;\zeta_0)|.
\]
\end{definition}

\begin{assumption}[Compatibility of scalar gate transport]
\label{gs:ass:scalar-gate-compatibility}
For every channel \(a\in\Achan\), the scalar gate reproduction map is
compatible with the budget and threshold reproduction maps on the scalar
coordinates under consideration:
\[
    \calR_a^g\bigl((B_0-\tau_0)_+\bigr)
    =
    \bigl(\calR_a^B B_0-\calR_a^\tau\tau_0\bigr)_+.
\]
In the identity scalar model this assumption is automatic.
\end{assumption}

\begin{lemma}[Gate mismatch from budget and threshold mismatch]
\label{gs:lem:gate-mismatch-target}
Assume the scalar gate compatibility condition
\Cref{gs:ass:scalar-gate-compatibility}.  Then
\[
    \Rep_a^{\mathrm{gate}}(\bD;\bzeta)
    \le
    \Rep_a^B(\bD;\bzeta)+\Rep_a^\tau(\bD;\bzeta).
\]
\end{lemma}

\begin{proof}
Set
\[
    B_1=B_a(D_1;\zeta_1),\qquad
    \tau_1=\tau_a(D_1;\zeta_1),
\]
and
\[
    \widetilde B_0=\calR_a^B B_a(D_0;\zeta_0),\qquad
    \widetilde\tau_0=\calR_a^\tau\tau_a(D_0;\zeta_0).
\]
By \Cref{gs:ass:scalar-gate-compatibility},
\[
    \calR_a^g\Gate_a(D_0;\zeta_0)
    =
    (\widetilde B_0-\widetilde\tau_0)_+.
\]
Therefore
\[
\begin{aligned}
    \Rep_a^{\mathrm{gate}}(\bD;\bzeta)
    &=
    \left|
        (B_1-\tau_1)_+
        -
        (\widetilde B_0-\widetilde\tau_0)_+
    \right|\\
    &\le
    |B_1-\widetilde B_0|+|\tau_1-\widetilde\tau_0|
\end{aligned}
\]
by \Cref{gs:lem:positive-part-lipschitz-target}.  The last line is
\(\Rep_a^B(\bD;\bzeta)+\Rep_a^\tau(\bD;\bzeta)\).
\end{proof}

\begin{definition}[Chain gate/slack residual]
\label{gs:def:chain-gs-residual}
Define
\[
    \Err_{\gs}^{\chain}(\bD;\bzeta)
    :=
    \sum_{a\in\Achan}\omega_a\Rep_a^{\mathrm{gate}}(\bD;\bzeta)
    +
    \sum_{a\in\Achan}\lambda_a
    \bigl(
        \Rep_a^B+\Rep_a^\tau+\Rep_a^s
    \bigr)(\bD;\bzeta).
\]
This measures how budgets, thresholds, gates, and slack variables fail to
reproduce across adjacent windows.
\end{definition}

\subsection{Sharp Gate/Slack Package Norm}
\label{gs:sec:single-window-norm}

\begin{definition}[Sharp gate/slack package norm]
\label{gs:def:sharp-gs-package-norm}
Let \(\|D\|_{\loc,\pkg}^{\sharp}\) denote the sharp package norm from the
previous component branches.  Define
\[
    \|D\|_{\loc,\pkg}^{\sharp,\gs}
    :=
    \|D\|_{\loc,\pkg}^{\sharp}
    +
    \Err_{\gs}^{\mathrm{win}}(D;0).
\]
For a shifted package, \(\|D-\zeta\|_{\loc,\pkg}^{\sharp,\gs}\) means the same
norm evaluated on the shifted package coordinates.  Define
\[
    \Dist_{\loc,\pkg}^{\sharp,\gs}(D,\Gint)
    :=
    \inf_{\zeta\in\Gint}
    \|D-\zeta\|_{\loc,\pkg}^{\sharp,\gs}.
\]
\end{definition}

\begin{assumption}[Same-gauge near-minimizer]
\label{gs:ass:single-window-gs-near-minimizer}
The selected representative \(\zeta_*(D)\in\Gint\) satisfies
\[
    \|D-\zeta_*\|_{\loc,\pkg}^{\sharp,\gs}
    \le
    \Dist_{\loc,\pkg}^{\sharp,\gs}(D,\Gint)
    +
    \delta_{\gs}.
\]
The same \(\zeta_*\) is used in all gate/slack coordinates.
\end{assumption}

\subsection{Single-Window Gate/Slack Absorption}
\label{gs:sec:single-window-absorption}

\begin{theorem}[Single-window absorption]
\label{gs:thm:single-window-gs-absorption-target}
Assume the same-gauge near-minimizer condition in
\Cref{gs:ass:single-window-gs-near-minimizer}.  Then
\[
    \Err_{\gs}^{\mathrm{win}}(D;\zeta_*)
    \le
    \Dist_{\loc,\pkg}^{\sharp,\gs}(D,\Gint)
    +
    \delta_{\gs}.
\]
More generally, if the sharp gate/slack norm controls the gate/slack residual
with finite constant \(C_{\gs}\), meaning
\[
    \Err_{\gs}^{\mathrm{win}}(D;\zeta)
    \le
    C_{\gs}\|D-\zeta\|_{\loc,\pkg}^{\sharp,\gs}
\]
for all admissible representatives, then
\[
    \Err_{\gs}^{\mathrm{win}}(D;\zeta_*)
    \le
    C_{\gs}
    \Dist_{\loc,\pkg}^{\sharp,\gs}(D,\Gint)
    +
    C_{\gs}\delta_{\gs}.
\]
\end{theorem}

\begin{proof}
By definition,
\[
    \Err_{\gs}^{\mathrm{win}}(D;\zeta_*)
    \le
    \|D-\zeta_*\|_{\loc,\pkg}^{\sharp,\gs}.
\]
The near-minimizer condition then gives the target estimate.

The same proof gives the comparable-weight version: first apply
\[
    \Err_{\gs}^{\mathrm{win}}(D;\zeta_*)
    \le
    C_{\gs}\|D-\zeta_*\|_{\loc,\pkg}^{\sharp,\gs},
\]
and then use \Cref{gs:ass:single-window-gs-near-minimizer}.
\end{proof}

\begin{remark}[Accounting status]
This is an accounting absorption theorem.  It is not a new PDE estimate.  Its
role is to make admissibility and slack violation visible to the package
geometry.
\end{remark}

\subsection{Sharp Chain Gate/Slack Norm}
\label{gs:sec:chain-norm}

\begin{definition}[Sharp chain gate/slack norm]
\label{gs:def:sharp-chain-gs-norm}
Let \(\|\bD\|_{\chain}^{\sharp,\rep}\) denote the sharp chain reproduction norm
from the reproduction-drift branch.  Define
\[
\begin{aligned}
    \|\bD\|_{\chain}^{\sharp,\gs}
    &:=
    \|\bD\|_{\chain}^{\sharp,\rep}
    +
    \Err_{\gs}^{\mathrm{win}}(D_0;0)
    +
    \Err_{\gs}^{\mathrm{win}}(D_1;0)\\
    &\quad+
    \Err_{\gs}^{\chain}(\bD;0).
\end{aligned}
\]
For shifted chain representatives,
\(\|\bD-\bzeta\|_{\chain}^{\sharp,\gs}\) means all terms are evaluated on
\((D_0-\zeta_0,D_1-\zeta_1)\).  Define
\[
    \Dist_{\chain}^{\sharp,\gs}(\bD,\Gchain)
    :=
    \inf_{\bzeta\in\Gchain}
    \|\bD-\bzeta\|_{\chain}^{\sharp,\gs}.
\]
\end{definition}

\begin{assumption}[Same-chain gate/slack near-minimizer]
\label{gs:ass:chain-gs-near-minimizer}
The selected same-chain representative \(\bzeta_*\in\Gchain\) satisfies
\[
    \|\bD-\bzeta_*\|_{\chain}^{\sharp,\gs}
    \le
    \Dist_{\chain}^{\sharp,\gs}(\bD,\Gchain)
    +
    \delta_{\gs}^{\chain}.
\]
\end{assumption}

\subsection{Main Gate/Slack Absorption Theorems}
\label{gs:sec:main-targets}

\begin{theorem}[Chain gate/slack mismatch absorption]
\label{gs:thm:chain-gs-absorption-target}
Let \(\bD=(D_0,D_1)\) be a two-window sharp package pair.  Assume the
same-chain near-minimizer condition in \Cref{gs:ass:chain-gs-near-minimizer}.  Then
\[
    \Err_{\gs}^{\chain}(\bD;\bzeta_*)
    \le
    \Dist_{\chain}^{\sharp,\gs}(\bD,\Gchain)
    +
    \delta_{\gs}^{\chain}.
\]
More generally, if the sharp chain norm controls the chain gate/slack residual
with finite constant \(C_{\gs}^{\chain}\), meaning
\[
    \Err_{\gs}^{\chain}(\bD;\bzeta)
    \le
    C_{\gs}^{\chain}\|\bD-\bzeta\|_{\chain}^{\sharp,\gs},
\]
then
\[
    \Err_{\gs}^{\chain}(\bD;\bzeta_*)
    \le
    C_{\gs}^{\chain}
    \Dist_{\chain}^{\sharp,\gs}(\bD,\Gchain)
    +
    C_{\gs}^{\chain}\delta_{\gs}^{\chain}.
\]
\end{theorem}

\begin{proof}
The chain gate/slack residual is one of the nonnegative components of the
sharp chain gate/slack norm.  Hence
\[
    \Err_{\gs}^{\chain}(\bD;\bzeta_*)
    \le
    \|\bD-\bzeta_*\|_{\chain}^{\sharp,\gs}.
\]
The same-chain near-minimizer property gives the target estimate.

The comparable-weight version follows by inserting the assumed bound
\[
    \Err_{\gs}^{\chain}(\bD;\bzeta_*)
    \le
    C_{\gs}^{\chain}\|\bD-\bzeta_*\|_{\chain}^{\sharp,\gs}
\]
before applying \Cref{gs:ass:chain-gs-near-minimizer}.
\end{proof}

\subsection{Weighted Gate/Slack Absorption}
\label{gs:sec:weighted}

\begin{corollary}[Weighted single-window gate/slack absorption]
\label{gs:cor:weighted-single-window-gs-target}
If a weighted norm satisfies
\[
    \|D\|_{\loc,\pkg}^{\sharp,\gs,\omega}
    \ge
    \frac{C_{\gs}}{\eta_{\gs}}
    \|D\|_{\loc,\pkg}^{\sharp,\gs},
\]
and the selected representative is a weighted near-minimizer with error
\(\delta_{\gs}^{\omega}\), the target estimate is
\[
    \Err_{\gs}^{\mathrm{win}}(D;\zeta_*)
    \le
    \eta_{\gs}
    \Dist_{\loc,\pkg}^{\sharp,\gs,\omega}(D,\Gint)
    +
    \eta_{\gs}\delta_{\gs}^{\omega}.
\]
\end{corollary}

\begin{proof}
By the comparable-weight form of \Cref{gs:thm:single-window-gs-absorption-target},
\[
    \Err_{\gs}^{\mathrm{win}}(D;\zeta_*)
    \le
    C_{\gs}\|D-\zeta_*\|_{\loc,\pkg}^{\sharp,\gs}.
\]
The domination assumption gives
\[
    C_{\gs}\|D-\zeta_*\|_{\loc,\pkg}^{\sharp,\gs}
    \le
    \eta_{\gs}\|D-\zeta_*\|_{\loc,\pkg}^{\sharp,\gs,\omega}.
\]
Using the weighted near-minimizer property yields the stated estimate.
\end{proof}

\begin{corollary}[Weighted chain gate/slack absorption]
\label{gs:cor:weighted-chain-gs-target}
If a weighted chain norm satisfies the analogous domination condition, and the
same-chain representative is a weighted near-minimizer with error
\(\delta_{\gs}^{\chain,\omega}\), the target estimate is
\[
    \Err_{\gs}^{\chain}(\bD;\bzeta_*)
    \le
    \eta_{\gs}
    \Dist_{\chain}^{\sharp,\gs,\omega}(\bD,\Gchain)
    +
    \eta_{\gs}\delta_{\gs}^{\chain,\omega}.
\]
\end{corollary}

\begin{proof}
The proof is the same as the single-window proof, with
\(\|\cdot\|_{\chain}^{\sharp,\gs}\),
\(\|\cdot\|_{\chain}^{\sharp,\gs,\omega}\), and
\(\delta_{\gs}^{\chain,\omega}\) replacing the corresponding single-window
objects.
\end{proof}

\begin{remark}[Weighted status]
Weighted absorption is a normalization of the accounting norm.  It is not a
new PDE estimate.
\end{remark}

\subsection{Relation to Operator Gate/Tax Mismatch Models}
\label{gs:sec:operator-model}

\begin{definition}[Abstract operator mismatch model]
\label{gs:def:operator-mismatch-model}
Earlier finite-window frameworks may use operator mismatch quantities such as
\[
    R_{\mathrm{gate}}^{\mathrm{det}}(D),
    \qquad
    R_{\mathrm{gate}}^{\mathrm{tax}}(D),
    \qquad
    R_{\mathrm{gate}}^{\mathrm{slack}}(D).
\]
The present paper instead uses positive-part budget violation and slack
identity mismatch.
\end{definition}

\begin{assumption}[Optional comparison with operator mismatch]
\label{gs:ass:operator-comparison}
When needed, one may assume either
\[
    R_{\mathrm{gate}}^{\mathrm{op}}(D;\zeta)
    \le
    C_{\mathrm{op}}\Err_{\gs}^{\mathrm{win}}(D;\zeta)
    +
    \Delta_{\mathrm{op}},
\]
or conversely
\[
    \Err_{\gs}^{\mathrm{win}}(D;\zeta)
    \le
    C_{\mathrm{op}}R_{\mathrm{gate}}^{\mathrm{op}}(D;\zeta)
    +
    \Delta_{\mathrm{op}}.
\]
\end{assumption}

\begin{remark}[No automatic equivalence]
This paper does not claim equivalence between the positive-part
gate/slack model and the older operator model unless a comparison assumption
such as \Cref{gs:ass:operator-comparison} is separately supplied.
\end{remark}

\subsection{Finite-Chain Gate/Slack Corollary}
\label{gs:sec:finite-chain}

\begin{corollary}[Finite-chain bookkeeping]
\label{gs:cor:finite-chain-gs-target}
For a finite chain \(D_0,\ldots,D_K\), define
\[
    \Err_{\gs}^{[0,K]}
    :=
    \sum_{k=0}^{K}
    \Err_{\gs}^{\mathrm{win}}(D_k;\zeta_k)
    +
    \sum_{k=0}^{K-1}
    \Err_{\gs}^{\chain}(D_k,D_{k+1};\zeta_k,\zeta_{k+1}).
\]
Assume the finite-chain quotient distance is defined by summing the
corresponding window and chain gate/slack norms, and assume the selected
finite-chain representative satisfies the near-minimizer bound with errors
\(\delta_{\gs,0},\ldots,\delta_{\gs,K}\).  Then
\[
    \Err_{\gs}^{[0,K]}
    \le
    C_{\gs}^{[0,K]}
    \Dist_{\chain}^{\sharp,\gs,[0,K]}
    ((D_0,\ldots,D_K),\Gamma_{\Lambda,\adm}^{\chain,[0,K]})
    +
    C_{\gs}^{[0,K]}\sum_{k=0}^K\delta_{\gs,k}.
\]
\end{corollary}

\begin{proof}
The finite-chain residual is a sum of nonnegative residual components that are
included in the corresponding finite-chain sharp norm.  Therefore, evaluated
on the selected finite-chain representative,
\[
    \Err_{\gs}^{[0,K]}
    \le
    C_{\gs}^{[0,K]}
    \|(D_0,\ldots,D_K)-(\zeta_0,\ldots,\zeta_K)\|_{\chain}^{\sharp,\gs,[0,K]}.
\]
The finite-chain near-minimizer bound gives the displayed estimate.
\end{proof}

\begin{remark}[Finite-chain status]
The constant may depend on \(K\).  This is not a scale-uniform theorem.
\end{remark}

\subsection*{Componentwise Residual-Ledger Closure}
\subsection{Unified Package and Chain Geometry}
\label{comp:sec:geometry}

\begin{definition}[Finite chain of sharp packages]
\label{comp:def:finite-chain}
Fix a finite integer \(K\ge0\).  A component chain is a tuple
\[
    \calD=(D_0,D_1,\ldots,D_K),
\]
where each \(D_k\) is a sharp localized package carrying the coordinates needed
by the pressure-source, localization, gate/slack, and reproduction branches.
For \(0\le k\le K-1\), write
\[
    \bfD_k=(D_k,D_{k+1})
\]
for the adjacent two-window package pair.
\end{definition}

\begin{definition}[Unified admissible component gauge]
\label{comp:def:component-gauge}
The finite-chain admissible gauge class is a subset
\[
    \Gcomp
    \subset
    \prod_{k=0}^{K}\Gamma_{\Lambda,\adm}^{(k)}.
\]
An element of \(\Gcomp\) is written
\[
    \bzeta=(\zeta_0,\ldots,\zeta_K).
\]
We impose the conservative physical convention in every window:
\[
    (\zeta_k)_u=0,
    \qquad 0\le k\le K.
\]
Thus the physical velocity is not gauged away.
\end{definition}

\begin{remark}[Finite-window scope]
All definitions are made on a fixed finite chain.  Constants may depend on
\(K\), on the chosen windows, and on the component package constants.  No
scale-uniformity is claimed.
\end{remark}

\subsection{Same-Chain Representative Convention}
\label{comp:sec:same-chain}

\begin{assumption}[Same-chain component representative]
\label{comp:ass:same-chain-representative}
For each finite chain \(\calD\), choose one representative
\[
    \bzeta_*(\calD)
    =
    (\zeta_{0,*},\ldots,\zeta_{K,*})
    \in
    \Gcomp.
\]
The same representative is used simultaneously in all component channels:
pressure-source, localization, reproduction, single-window gate/slack, chain
gate/slack, and the unified component quotient distance.
\end{assumption}

\begin{remark}[Why this convention is necessary]
The component estimates do not combine if each residual channel is optimized
over its own gauge representative.  The unified ledger is a same-gauge
statement, not a collection of independently minimized inequalities.
\end{remark}

\subsection{Unified Component Residual}
\label{comp:sec:residual}

\begin{definition}[Single-window component residual]
\label{comp:def:window-residual}
For a window \(D_k\) and representative \(\zeta_k\), define
\[
    \Err_{\mathrm{win}}(D_k;\zeta_k)
    :=
    \Err_{\src}^{\prs}(D_k;\zeta_k)
    +
    \Err_{\locerr}(D_k;\zeta_k)
    +
    \Err_{\gs}^{\mathrm{win}}(D_k;\zeta_k).
\]
\end{definition}

\begin{definition}[Edge component residual]
\label{comp:def:edge-residual}
For an adjacent pair \(\bfD_k=(D_k,D_{k+1})\), define
\[
    \Err_{\mathrm{edge}}(\bfD_k;\zeta_k,\zeta_{k+1})
    :=
    \Err_{\rep}(\bfD_k;\zeta_k,\zeta_{k+1})
    +
    \Err_{\gs}^{\chain}(\bfD_k;\zeta_k,\zeta_{k+1}).
\]
\end{definition}

\begin{definition}[Total component residual]
\label{comp:def:component-residual}
The finite-chain component residual is
\[
    \Err_{\comp}^{[0,K]}(\calD;\bzeta)
    :=
    \sum_{k=0}^{K}
    \Err_{\mathrm{win}}(D_k;\zeta_k)
    +
    \sum_{k=0}^{K-1}
    \Err_{\mathrm{edge}}(\bfD_k;\zeta_k,\zeta_{k+1}).
\]
Equivalently, it is the sum of pressure-source, localization, single-window
gate/slack, reproduction, and chain gate/slack residuals over the finite chain.
\end{definition}

\subsection{Unified Sharp Component Norm}
\label{comp:sec:norm}

\begin{definition}[Window and edge component norms]
\label{comp:def:window-edge-norms}
The single-window component norm
\[
    |D_k|_{\loc,\pkg}^{\sharp,\comp}
\]
is the finite-window sharp norm that dominates the pressure-source,
localization, and single-window gate/slack coordinates proved and used in
the corresponding component branches above.
The edge component norm
\[
    |\bfD_k|_{\mathrm{edge}}^{\sharp,\comp}
\]
is the finite-window sharp norm that dominates the reproduction and chain
gate/slack coordinates proved and used in the corresponding component branches
above.
\end{definition}

\begin{definition}[Finite-chain component norm]
\label{comp:def:component-norm}
For a finite chain \(\calD=(D_0,\ldots,D_K)\), define
\[
    |\calD|_{\comp}^{\sharp,[0,K]}
    :=
    \sum_{k=0}^{K}
    |D_k|_{\loc,\pkg}^{\sharp,\comp}
    +
    \sum_{k=0}^{K-1}
    |\bfD_k|_{\mathrm{edge}}^{\sharp,\comp}.
\]
For a shifted chain \(\calD-\bzeta\), all terms are evaluated on
\((D_0-\zeta_0,\ldots,D_K-\zeta_K)\).
\end{definition}

\begin{definition}[Component quotient distance]
\label{comp:def:component-distance}
Define
\[
    \Dist_{\comp}^{\sharp,[0,K]}(\calD,\Gcomp)
    :=
    \inf_{\bzeta\in\Gcomp}
    |\calD-\bzeta|_{\comp}^{\sharp,[0,K]}.
\]
\end{definition}

\subsection{Unified Near-Minimizer Assumption}
\label{comp:sec:near-min}

\begin{assumption}[Unified component near-minimizer]
\label{comp:ass:component-near-minimizer}
The selected representative \(\bzeta_*(\calD)\in\Gcomp\) satisfies
\[
    |\calD-\bzeta_*|_{\comp}^{\sharp,[0,K]}
    \le
    \Dist_{\comp}^{\sharp,[0,K]}(\calD,\Gcomp)
    +
    \delta_{\comp}^{[0,K]}.
\]
The error \(\delta_{\comp}^{[0,K]}\) is the single near-minimizer error used by
the unified ledger.
\end{assumption}

\begin{remark}[Relation to component errors]
One may either treat \(\delta_{\comp}^{[0,K]}\) as a primitive chain
near-minimizer error, or choose it to dominate the separate errors from the
component modules:
\[
    \delta_{{\pkg},k},\quad
    \delta_{{\locerr},k},\quad
    \delta_{{\rep},k},\quad
    \delta_{{\gs},k},\quad
    \delta_{{\gs},k}^{\chain}.
\]
The important point is that these errors correspond to the same selected chain
representative.
\end{remark}

\subsection{Component Channel Estimates in the Unified Geometry}
\label{comp:sec:component-channel-estimates}

The preceding four branches prove estimates in their own sharp norms.  In the
unified ledger we use component norms that dominate all branch coordinates.  The
following proposition records the representative-form estimates that will be
summed in the final closure theorem.

\begin{proposition}[Representative-form component channel estimates]
\label{comp:prop:component-channel-estimates}
Let \(\calD=(D_0,\ldots,D_K)\) be a finite component chain and let
\(\bzeta=(\zeta_0,\ldots,\zeta_K)\in\Gcomp\) be one admissible same-chain
representative.  Assume that, on each window and edge, the hypotheses of the
pressure-source, localization, reproduction, and gate/slack component theorems
proved in \Cref{ps:sec:main-absorption,loc:sec:main-theorems,rep:sec:main-theorems,gs:sec:single-window-absorption,gs:sec:main-targets} hold with the finite-amplitude
bound \(M_U\).  Then there are finite constants
\[
    C_{{\prs},k}(M_U),\quad
    C_{{\locerr},k}(M_U),\quad
    C_{{\rep},k}(M_U),\quad
    C_{{\gs},k},\quad
    C_{{\gs},k}^{\chain}
\]
such that, for \(0\le k\le K\),
\[
    \Err_{\src}^{\prs}(D_k;\zeta_k)
    \le
    C_{{\prs},k}(M_U)
    |D_k-\zeta_k|_{\loc,\pkg}^{\sharp,\comp},
\]
\[
    \Err_{\locerr}(D_k;\zeta_k)
    \le
    C_{{\locerr},k}(M_U)
    |D_k-\zeta_k|_{\loc,\pkg}^{\sharp,\comp},
\]
\[
    \Err_{\gs}^{\mathrm{win}}(D_k;\zeta_k)
    \le
    C_{{\gs},k}
    |D_k-\zeta_k|_{\loc,\pkg}^{\sharp,\comp},
\]
and, for \(0\le k\le K-1\),
\[
    \Err_{\rep}(\bfD_k;\zeta_k,\zeta_{k+1})
    \le
    C_{{\rep},k}(M_U)
    |\bfD_k-(\zeta_k,\zeta_{k+1})|_{\mathrm{edge}}^{\sharp,\comp},
\]
\[
    \Err_{\gs}^{\chain}(\bfD_k;\zeta_k,\zeta_{k+1})
    \le
    C_{{\gs},k}^{\chain}
    |\bfD_k-(\zeta_k,\zeta_{k+1})|_{\mathrm{edge}}^{\sharp,\comp}.
\]
\end{proposition}

\begin{proof}
The pressure-source estimate is the representative-norm estimate obtained in
\Cref{ps:thm:bounded-amplitude-pressure-source-absorption} before applying the
branch near-minimizer condition; the component window norm dominates the sharp
pressure-source norm by definition.  The localization estimate follows in the
same way from the representative-norm inequality in
\Cref{loc:thm:bounded-amplitude-localization-absorption}.  The reproduction
estimate follows from the representative-norm inequality in
\Cref{rep:thm:bounded-amplitude-reproduction-absorption}, with the component
edge norm dominating the sharp chain reproduction norm.  The single-window and
chain gate/slack estimates follow from
\Cref{gs:thm:single-window-gs-absorption-target,gs:thm:chain-gs-absorption-target},
because the gate/slack residuals are nonnegative coordinates included in the
corresponding sharp gate/slack norms.  Enlarging constants by the finite
domination constants between branch norms and component norms gives the stated
bounds.
\end{proof}

\subsection{Main Componentwise Closure Theorem}
\label{comp:sec:main-target}

\begin{theorem}[Finite-chain componentwise closure]
\label{comp:thm:componentwise-closure-target}
Assume the branch hypotheses needed for the component channel estimates
\Cref{comp:prop:component-channel-estimates}, the same-chain convention
\Cref{comp:ass:same-chain-representative}, and the unified near-minimizer
assumption \Cref{comp:ass:component-near-minimizer}.  Let
\[
\begin{aligned}
    C_{\comp}^{[0,K]}(M_U)
    \ge
    \max\bigl\{&
    C_{{\prs},k}(M_U),
    C_{{\locerr},k}(M_U),
    C_{{\gs},k}:0\le k\le K,\\
    &C_{{\rep},k}(M_U),
    C_{{\gs},k}^{\chain}:0\le k\le K-1
    \bigr\}.
\end{aligned}
\]
Then
\[
    \Err_{\comp}^{[0,K]}(\calD;\bzeta_*)
    \le
    C_{\comp}^{[0,K]}(M_U)
    \Dist_{\comp}^{\sharp,[0,K]}(\calD,\Gcomp)
    +
    C_{\comp}^{[0,K]}(M_U)
    \delta_{\comp}^{[0,K]}.
\]
The constant may depend on the finite-chain length \(K\), the component
constants, the fixed cutoff and reproduction maps, and any finite-amplitude
bounds used by the component estimates.  No scale-uniformity in \(K\) or in the
window scale is asserted.
\end{theorem}

\begin{proof}
Fix the selected representative
\(\bzeta_*=(\zeta_{0,*},\ldots,\zeta_{K,*})\).  By \Cref{comp:prop:component-channel-estimates}, for each window \(0\le k\le K\),
\[
\begin{aligned}
    \Err_{\mathrm{win}}(D_k;\zeta_{k,*})
    &=
    \Err_{\src}^{\prs}(D_k;\zeta_{k,*})
    +
    \Err_{\locerr}(D_k;\zeta_{k,*})
    +
    \Err_{\gs}^{\mathrm{win}}(D_k;\zeta_{k,*})\\
    &\le
    \bigl(
        C_{{\prs},k}(M_U)
        +
        C_{{\locerr},k}(M_U)
        +
        C_{{\gs},k}
    \bigr)
    |D_k-\zeta_{k,*}|_{\loc,\pkg}^{\sharp,\comp}.
\end{aligned}
\]
Since the finite number of component constants can be enlarged, the right-hand
side is bounded by
\[
    C_{\comp}^{[0,K]}(M_U)
    |D_k-\zeta_{k,*}|_{\loc,\pkg}^{\sharp,\comp}.
\]
Similarly, by \Cref{comp:prop:component-channel-estimates}, for each edge
\(0\le k\le K-1\),
\[
\begin{aligned}
    \Err_{\mathrm{edge}}(\bfD_k;\zeta_{k,*},\zeta_{k+1,*})
    &=
    \Err_{\rep}(\bfD_k;\zeta_{k,*},\zeta_{k+1,*})
    +
    \Err_{\gs}^{\chain}(\bfD_k;\zeta_{k,*},\zeta_{k+1,*})\\
    &\le
    \bigl(
        C_{{\rep},k}(M_U)
        +
        C_{{\gs},k}^{\chain}
    \bigr)
    |\bfD_k-(\zeta_{k,*},\zeta_{k+1,*})|_{\mathrm{edge}}^{\sharp,\comp}\\
    &\le
    C_{\comp}^{[0,K]}(M_U)
    |\bfD_k-(\zeta_{k,*},\zeta_{k+1,*})|_{\mathrm{edge}}^{\sharp,\comp}.
\end{aligned}
\]
Summing these window and edge estimates and using the definition of the
component norm gives
\[
    \Err_{\comp}^{[0,K]}(\calD;\bzeta_*)
    \le
    C_{\comp}^{[0,K]}(M_U)
    |\calD-\bzeta_*|_{\comp}^{\sharp,[0,K]}.
\]
Finally apply \Cref{comp:ass:component-near-minimizer}.
\end{proof}

\begin{remark}[Status of the theorem]
This theorem is finite-chain bookkeeping using the component estimates
proved in the preceding branches and the unified near-minimizer assumption.
The detector comparison and full local-to-clean transfer use this estimate as
an input.
\end{remark}

\subsection{Weighted Componentwise Closure}
\label{comp:sec:weighted}

\begin{assumption}[Weighted component geometry]
\label{comp:ass:weighted-component}
Let
\[
    |\calD|_{\comp}^{\sharp,\omega,[0,K]}
\]
be a weighted component norm satisfying
\[
    |\calD-\bzeta|_{\comp}^{\sharp,\omega,[0,K]}
    \ge
    \frac{C_{\comp}^{[0,K]}(M_U)}{\eta_{\comp}}
    |\calD-\bzeta|_{\comp}^{\sharp,[0,K]}
\]
for every admissible representative \(\bzeta\).  Define the weighted quotient
distance by
\[
    \Dist_{\comp}^{\sharp,\omega,[0,K]}(\calD,\Gcomp)
    :=
    \inf_{\bzeta\in\Gcomp}
    |\calD-\bzeta|_{\comp}^{\sharp,\omega,[0,K]}.
\]
Assume the selected representative satisfies the weighted near-minimizer
bound
\[
    |\calD-\bzeta_*|_{\comp}^{\sharp,\omega,[0,K]}
    \le
    \Dist_{\comp}^{\sharp,\omega,[0,K]}(\calD,\Gcomp)
    +
    \delta_{\comp}^{\omega,[0,K]}.
\]
\end{assumption}

\begin{corollary}[Weighted componentwise closure]
\label{comp:cor:weighted-component-target}
Under the assumptions of \Cref{comp:thm:componentwise-closure-target} and
\Cref{comp:ass:weighted-component},
\[
    \Err_{\comp}^{[0,K]}(\calD;\bzeta_*)
    \le
    \eta_{\comp}
    \Dist_{\comp}^{\sharp,\omega,[0,K]}(\calD,\Gcomp)
    +
    \eta_{\comp}
    \delta_{\comp}^{\omega,[0,K]}.
\]
\end{corollary}

\begin{proof}
By \Cref{comp:thm:componentwise-closure-target},
\[
    \Err_{\comp}^{[0,K]}(\calD;\bzeta_*)
    \le
    C_{\comp}^{[0,K]}(M_U)
    |\calD-\bzeta_*|_{\comp}^{\sharp,[0,K]}.
\]
The weighted domination assumption gives
\[
    C_{\comp}^{[0,K]}(M_U)
    |\calD-\bzeta_*|_{\comp}^{\sharp,[0,K]}
    \le
    \eta_{\comp}
    |\calD-\bzeta_*|_{\comp}^{\sharp,\omega,[0,K]}.
\]
The weighted near-minimizer bound then yields the claimed estimate.
\end{proof}

\subsection{Quadratic Geometry Variant}
\label{comp:sec:quadratic}

\begin{definition}[Quadratic component geometry]
\label{comp:def:quadratic-component-geometry}
Let
\[
    |\calD|_{\comp}^{\sharp,\quadg,[0,K]}
\]
denote the component geometry obtained by replacing bounded-amplitude
linearized component coordinates with quadratic variants.  The gate/slack
coordinates remain nonnegative additive residual coordinates.
\end{definition}

\begin{definition}[Quadratic component quotient distance]
\label{comp:def:quadratic-component-distance}
For shifted chains define
\[
    \Dist_{\comp}^{\sharp,\quadg,[0,K]}(\calD,\Gcomp)
    :=
    \inf_{\bzeta\in\Gcomp}
    |\calD-\bzeta|_{\comp}^{\sharp,\quadg,[0,K]}.
\]
\end{definition}

\begin{assumption}[Quadratic component estimates]
\label{comp:ass:quadratic-component-estimates}
There is a finite constant \(C_{\comp}^{\quadg,[0,K]}\) such that the imported
quadratic pressure-source, localization, reproduction, and gate/slack component
estimates imply
\[
    \Err_{\comp}^{[0,K]}(\calD;\bzeta)
    \le
    C_{\comp}^{\quadg,[0,K]}
    |\calD-\bzeta|_{\comp}^{\sharp,\quadg,[0,K]}
\]
for every admissible component representative \(\bzeta\).
\end{assumption}

\begin{assumption}[Quadratic near-minimizer]
\label{comp:ass:quadratic-near-minimizer}
The selected representative satisfies
\[
    |\calD-\bzeta_*|_{\comp}^{\sharp,\quadg,[0,K]}
    \le
    \Dist_{\comp}^{\sharp,\quadg,[0,K]}(\calD,\Gcomp)
    +
    \delta_{\comp}^{\quadg,[0,K]}.
\]
\end{assumption}

\begin{theorem}[Quadratic componentwise closure]
\label{comp:thm:quadratic-component-target}
Assume
\Cref{comp:ass:quadratic-component-estimates,comp:ass:quadratic-near-minimizer}.  Then
\[
    \Err_{\comp}^{[0,K]}(\calD;\bzeta_*)
    \le
    C_{\comp}^{\quadg,[0,K]}
    \Dist_{\comp}^{\sharp,\quadg,[0,K]}(\calD,\Gcomp)
    +
    C_{\comp}^{\quadg,[0,K]}
    \delta_{\comp}^{\quadg,[0,K]}.
\]
\end{theorem}

\begin{proof}
Apply \Cref{comp:ass:quadratic-component-estimates} with
\(\bzeta=\bzeta_*\), then use \Cref{comp:ass:quadratic-near-minimizer}.
\end{proof}

\begin{remark}[Cost of the quadratic geometry]
This variant avoids some finite-amplitude linearization by changing the
package geometry.  It is not a theorem in the original linear baseline
geometry.
\end{remark}

\subsection{Normalized Quotient Amplitude Variant}
\label{comp:sec:amplitude}

\begin{assumption}[Normalized quotient amplitude]
\label{comp:ass:normalized-amplitude}
On the normalized component quotient sphere
\[
    \Dist_{\comp}^{\sharp,[0,K]}(\calD,\Gcomp)=1,
\]
assume the selected representative satisfies the amplitude bounds required by
the component theorems:
\[
    |u_k|_{L^3}
    +
    |U_k|_{L^3}
    +
    |\mathcal R u_k|_{L^3}
    +
    |\mathcal R U_k|_{L^3}
    \le
    C_{\mathrm{amp}}
\]
for all relevant windows and edges.
\end{assumption}

\begin{corollary}[Normalized-amplitude closure]
\label{comp:cor:normalized-amplitude-target}
Assume \Cref{comp:ass:normalized-amplitude} and assume the imported component
estimates hold with amplitude parameter
\[
    M_U=C_{\mathrm{amp}}.
\]
Then the finite-chain componentwise closure holds on the normalized quotient
sphere with \(C_{\comp}^{[0,K]}(C_{\mathrm{amp}})\) in place of
\(C_{\comp}^{[0,K]}(M_U)\).
\end{corollary}

\begin{proof}
The normalized amplitude hypothesis supplies the amplitude parameter required
by the component channel estimates.  Substitute \(M_U=C_{\mathrm{amp}}\) in
\Cref{comp:thm:componentwise-closure-target}.
\end{proof}

\begin{remark}[No global finite-amplitude removal]
This variant does not remove finite amplitude globally.  It records a possible
quotient-normalized route in which amplitude is controlled as part of the
chosen finite-window component geometry.
\end{remark}

\subsection{Relation to the Previous Assembly Theorem}
\label{comp:sec:assembly-relation}

Earlier conditional assembly theorems used a residual-budget input of the form
\[
    \Err_{\Lambda}(D)
    \le
    \eta_{\Lambda}
    \Dist^{\sharp}(D,\Gamma)
    +
    \Delta_{\Lambda}.
\]
The present paper targets only the componentwise part of such a residual
budget:
\[
    \Err_{\comp}
    =
    \Err_{\src}^{\prs}
    +
    \Err_{\locerr}
    +
    \Err_{\rep}
    +
    \Err_{\gs}.
\]
In the detector-comparison framework, the residual budget is decomposed as
\[
    \Err_{\Lambda}
    =
    \Err_{\comp}
    +
    \Err_{\mathrm{rem}},
\]
where the componentwise module supplies the \(\Err_{\comp}\) estimate and
the remaining term \(\Err_{\mathrm{rem}}\) is controlled by the detector
intertwining ledger.

\begin{proposition}[Partial residual-budget insertion]
\label{comp:prop:partial-residual-budget}
Assume the hypotheses of \Cref{comp:thm:componentwise-closure-target}.  Suppose a
larger residual budget splits as
\[
    \Err_{\Lambda}(\calD;\bzeta_*)
    =
    \Err_{\comp}^{[0,K]}(\calD;\bzeta_*)
    +
    \Err_{\mathrm{rem}}(\calD;\bzeta_*),
\]
and suppose the remaining residual satisfies
\[
    \Err_{\mathrm{rem}}(\calD;\bzeta_*)
    \le
    C_{\mathrm{rem}}
    \Dist_{\mathrm{rem}}^{\sharp}(\calD,\Gcomp)
    +
    \Delta_{\mathrm{rem}}.
\]
Then
\[
\begin{aligned}
    \Err_{\Lambda}(\calD;\bzeta_*)
    &\le
    C_{\comp}^{[0,K]}(M_U)
    \Dist_{\comp}^{\sharp,[0,K]}(\calD,\Gcomp)
    +
    C_{\comp}^{[0,K]}(M_U)\delta_{\comp}^{[0,K]}\\
    &\quad+
    C_{\mathrm{rem}}
    \Dist_{\mathrm{rem}}^{\sharp}(\calD,\Gcomp)
    +
    \Delta_{\mathrm{rem}}.
\end{aligned}
\]
\end{proposition}

\begin{proof}
Insert the componentwise estimate from
\Cref{comp:thm:componentwise-closure-target} into the assumed residual splitting and
then use the assumed bound for \(\Err_{\mathrm{rem}}\).
\end{proof}

\begin{remark}[Interface with detector comparison]
This relation is the component residual estimate used in the detector
comparison.  The detector inequality itself is proved in
\Cref{app:detector-details}.
\end{remark}

\subsection{Interface with the detector-comparison module}
\label{comp:sec:interface-detector}

The componentwise ledger developed in this appendix is used by the detector
comparison and local-to-clean transfer module.  Its output is the same-chain
residual estimate
\[
    \Err_{\comp}^{[0,K]}(\calD;\bzeta_*)
    \le
    C_{\comp}^{[0,K]}(M_U)
    \Dist_{\comp}^{\sharp,[0,K]}(\calD,\Gcomp)
    +
    C_{\comp}^{[0,K]}(M_U)\delta_{\comp}^{[0,K]},
\]
together with its weighted, quadratic, and normalized-amplitude variants.  The
separate hypotheses concerning detector-intertwining, clean-source compactness,
package-realizability, and reduced pressure/tax kernel-freeness are treated in
\Cref{app:detector-details}.  They are logically independent of the
componentwise estimates and are not used in the proof of
\Cref{comp:thm:componentwise-closure-target}.

\subsection{Summary of the componentwise proof}
\label{comp:sec:proof-summary}

The componentwise proof consists of the following finite-window steps.
\begin{enumerate}[label=\textbf{Step \arabic*.},leftmargin=*]
    \item The same-chain representative convention is fixed for all component
    channel estimates.
    \item The pressure-source residual, localization leakage, reproduction
    drift, and gate/slack budget-violation estimates are proved in their own
    sharp component geometries.
    \item The finite-chain componentwise closure theorem is obtained by summing
    the window and edge estimates in one unified component norm.
    \item The weighted componentwise closure, quadratic-geometry variant, and
    normalized quotient-amplitude variant are recorded as finite-window
    consequences of the same ledger.
    \item The partial residual-budget insertion proposition identifies how the
    closed component residual enters a larger detector-comparison budget.
\end{enumerate}

\section{Detector Comparison and Structural Criteria: Detailed Proofs}\label{app:detector-details}
\subsection{Imported Componentwise Ledger Closure}
\label{sec:imported-ledger}

\begin{assumption}[Componentwise residual-ledger closure]
\label{ass:component-ledger}
For a finite-chain sharp package \(\calD\), a same-chain representative
\(\bzeta_*\in\Gcomp\), and a finite-window amplitude parameter \(M_U\), the
componentwise residual-ledger theorem supplies
\[
    \Err_{\comp}^{[0,K]}(\calD;\bzeta_*)
    \le
    C_{\comp}^{[0,K]}(M_U)
    \Dist_{\comp}^{\sharp,[0,K]}(\calD,\Gcomp)
    +
    C_{\comp}^{[0,K]}(M_U)
    \delta_{\comp}^{[0,K]}.
\]
The weighted form supplies, when the weighted component geometry is used,
\[
    \Err_{\comp}^{[0,K]}(\calD;\bzeta_*)
    \le
    \eta_{\comp}
    \Dist_{\comp}^{\sharp,\omega,[0,K]}(\calD,\Gcomp)
    +
    \eta_{\comp}
    \delta_{\comp}^{\omega,[0,K]}.
\]
\end{assumption}

\begin{remark}[Imported status]
This paper does not reprove the componentwise ledger theorem.  It treats it as
the already closed residual ledger available to the detector-comparison branch.
\end{remark}

\subsection{Localized and Clean Detectors}
\label{sec:detectors}

\begin{definition}[Localized detector]
\label{def:local-detector}
A localized detector is a nonnegative finite-window functional of the form
\[
    M_{\Lambda}^{\loc}(D)
    :=
    |O_{\Lambda}^{\loc}D|_{\calO_{\loc}}
    +
    \sum_{a\in\mathfrak A_{\loc}}
    \alpha_a
    \Tax_a^{\loc}(D),
\]
where \(O_{\Lambda}^{\loc}\) is the localized observable, the
\(\Tax_a^{\loc}\) are nonnegative localized tax channels, and the weights
\(\alpha_a\ge0\) are fixed finite-window constants.
\end{definition}

\begin{remark}[Shifted detector notation]
The notation \(M_{\Lambda}^{\loc}(\calD-\bzeta)\) means that the localized
detector is evaluated on the shifted package.  All estimates below use one
fixed shifted representative.  This convention is part of the finite-window
quotient bookkeeping, not an independent minimization procedure.
\end{remark}

\begin{definition}[Clean detector]
\label{def:clean-detector}
A clean detector is a nonnegative finite-window functional on clean packages
\(d\):
\[
    M_{\Lambda}^{\comp}(d)
    :=
    |O_{\Lambda}^{\comp}d|_{\calO_{\comp}}
    +
    \sum_{b\in\mathfrak A_{\comp}^{\detc}}
    \beta_b
    \Tax_b^{\comp}(d),
\]
where \(O_{\Lambda}^{\comp}\) is the clean observable, the
\(\Tax_b^{\comp}\) are clean tax channels, and \(\beta_b\ge0\).
\end{definition}

\begin{remark}[Accounting role]
The detectors are finite-window accounting objects.  This paper does not
derive their coercivity or kernel-freeness from the Navier--Stokes equations.
\end{remark}

\subsection{Local-to-Clean Chart Convention}
\label{sec:chart}

\begin{definition}[Finite-window local-to-clean chart]
\label{def:chart}
The local-to-clean chart is a finite-window structural map
\[
    \Theta_{\Lambda}:\calD\longrightarrow d=\Theta_{\Lambda}\calD
\]
from localized package chains to clean packages.  For shifted representatives,
we write
\[
    \Theta_{\Lambda}(\calD-\bzeta)
\]
for the clean package associated with the shifted finite-chain data.
\end{definition}

\begin{remark}[No evolution claim]
The chart \(\Theta_{\Lambda}\) is not assumed to be an exact Navier--Stokes
evolution map.  It is a finite-window comparison map whose compatibility with
the detectors is a structural input.
\end{remark}

\subsection{Detector Discrepancy}
\label{sec:discrepancy}

\begin{definition}[Positive detector discrepancy]
\label{def:positive-detector-discrepancy}
For a shifted package \(\calD-\bzeta\), define
\[
    \Disc_{\detc}^{+}(\calD;\bzeta)
    :=
    \left(
        M_{\Lambda}^{\comp}(\Theta_{\Lambda}(\calD-\bzeta))
        -
        M_{\Lambda}^{\loc}(\calD-\bzeta)
    \right)_+.
\]
This quantity measures the part of the clean detector not yet seen by the
localized detector.
\end{definition}

\begin{definition}[Observable/tax mismatch functional]
\label{def:detector-mismatch}
Let \(\mathcal I_{\Lambda}:\calO_{\loc}\to\calO_{\comp}\) be a fixed
finite-window comparison operator between observable spaces.  Define an
observable-level detector mismatch by
\[
\begin{aligned}
    \Err_{\detc}(\calD;\bzeta)
    &:=
    \left|
        O_{\Lambda}^{\comp}\Theta_{\Lambda}(\calD-\bzeta)
        -
        \mathcal I_{\Lambda}O_{\Lambda}^{\loc}(\calD-\bzeta)
    \right|_{\calO_{\comp}}\\
    &\quad+
    \sum_{b\in\mathfrak A_{\comp}^{\detc}}
    \Tax_b^{\mathrm{mismatch}}(\calD;\bzeta).
\end{aligned}
\]
The mismatch taxes are nonnegative channelwise errors comparing clean tax
channels with localized tax channels.
\end{definition}

\subsection{Detector-Intertwining Assumptions}
\label{sec:assumptions}

\begin{assumption}[Clean detector controlled by local detector plus mismatch]
\label{ass:detector-lipschitz}
There are constants \(C_{\detc}<\infty\) and \(\Delta_{\detc}^{0}\ge0\) such
that, for every admissible shifted package,
\[
    M_{\Lambda}^{\comp}(\Theta_{\Lambda}(\calD-\bzeta))
    \le
    M_{\Lambda}^{\loc}(\calD-\bzeta)
    +
    C_{\detc}\Err_{\detc}(\calD;\bzeta)
    +
    \Delta_{\detc}^{0}.
\]
\end{assumption}

\begin{assumption}[Detector mismatch controlled by component residual]
\label{ass:detector-residual-control}
There are constants \(a_{\detc}\ge0\) and
\(\Delta_{\detc}^{\mathrm{rem}}\ge0\) such that
\[
    \Err_{\detc}(\calD;\bzeta)
    \le
    a_{\detc}
    \Err_{\comp}^{[0,K]}(\calD;\bzeta)
    +
    \Delta_{\detc}^{\mathrm{rem}}.
\]
\end{assumption}

\begin{convention}[Same representative]
\label{conv:same-representative}
The same representative \(\bzeta_*\) is used in the localized detector, the
clean detector after charting, the detector mismatch, and the component residual
ledger.  No independent minimization over different gauges is allowed in this
comparison theorem.
\end{convention}

\subsection{Main Detector Comparison}
\label{sec:main}

\begin{theorem}[Conditional finite-window detector comparison]
\label{thm:detector-comparison}
Assume \Cref{ass:component-ledger,ass:detector-lipschitz,ass:detector-residual-control}
and the same-representative convention \Cref{conv:same-representative}.  Set
\[
    C_{\mathrm{dc}}
    :=
    C_{\detc}a_{\detc},
    \qquad
    \Delta_{\mathrm{dc}}
    :=
    C_{\detc}\Delta_{\detc}^{\mathrm{rem}}
    +
    \Delta_{\detc}^{0}.
\]
Then
\[
    M_{\Lambda}^{\loc}(\calD-\bzeta_*)
    \ge
    M_{\Lambda}^{\comp}(\Theta_{\Lambda}(\calD-\bzeta_*))
    -
    C_{\mathrm{dc}}
    \Err_{\comp}^{[0,K]}(\calD;\bzeta_*)
    -
    \Delta_{\mathrm{dc}}.
\]
Consequently, using componentwise closure,
\[
\begin{aligned}
    M_{\Lambda}^{\loc}(\calD-\bzeta_*)
    &\ge
    M_{\Lambda}^{\comp}(\Theta_{\Lambda}(\calD-\bzeta_*))\\
    &\quad-
    C_{\mathrm{dc}}C_{\comp}^{[0,K]}(M_U)
    \Dist_{\comp}^{\sharp,[0,K]}(\calD,\Gcomp)\\
    &\quad-
    C_{\mathrm{dc}}C_{\comp}^{[0,K]}(M_U)
    \delta_{\comp}^{[0,K]}
    -
    \Delta_{\mathrm{dc}}.
\end{aligned}
\]
\end{theorem}

\begin{proof}
Apply \Cref{ass:detector-lipschitz} with the representative \(\bzeta_*\).  It
gives
\[
    M_{\Lambda}^{\comp}(\Theta_{\Lambda}(\calD-\bzeta_*))
    \le
    M_{\Lambda}^{\loc}(\calD-\bzeta_*)
    +
    C_{\detc}\Err_{\detc}(\calD;\bzeta_*)
    +
    \Delta_{\detc}^{0}.
\]
Rearranging,
\[
    M_{\Lambda}^{\loc}(\calD-\bzeta_*)
    \ge
    M_{\Lambda}^{\comp}(\Theta_{\Lambda}(\calD-\bzeta_*))
    -
    C_{\detc}\Err_{\detc}(\calD;\bzeta_*)
    -
    \Delta_{\detc}^{0}.
\]
The detector residual-control assumption gives
\[
    C_{\detc}\Err_{\detc}(\calD;\bzeta_*)
    \le
    C_{\detc}a_{\detc}\Err_{\comp}^{[0,K]}(\calD;\bzeta_*)
    +
    C_{\detc}\Delta_{\detc}^{\mathrm{rem}}.
\]
With the definitions of \(C_{\mathrm{dc}}\) and \(\Delta_{\mathrm{dc}}\), this is
the first asserted estimate.  Inserting the unweighted componentwise closure
from \Cref{ass:component-ledger} gives
\[
    C_{\mathrm{dc}}\Err_{\comp}^{[0,K]}(\calD;\bzeta_*)
    \le
    C_{\mathrm{dc}}C_{\comp}^{[0,K]}(M_U)
    \Dist_{\comp}^{\sharp,[0,K]}(\calD,\Gcomp)
    +
    C_{\mathrm{dc}}C_{\comp}^{[0,K]}(M_U)
    \delta_{\comp}^{[0,K]},
\]
which proves the displayed consequence.
\end{proof}

\begin{remark}[Theorem status]
This theorem is proved as a finite-window conditional statement.  The new
structural content required later is the verification of
\Cref{ass:detector-lipschitz,ass:detector-residual-control} in a specific
detector model.
\end{remark}

\subsection{Weighted Detector Comparison}
\label{sec:weighted}

\begin{corollary}[Conditional weighted detector comparison]
\label{cor:weighted-detector}
Assume \Cref{ass:component-ledger,ass:detector-lipschitz,ass:detector-residual-control}
and the same-representative convention \Cref{conv:same-representative}.  If the
weighted componentwise closure in \Cref{ass:component-ledger} is used, then
\[
\begin{aligned}
    M_{\Lambda}^{\loc}(\calD-\bzeta_*)
    &\ge
    M_{\Lambda}^{\comp}(\Theta_{\Lambda}(\calD-\bzeta_*))\\
    &\quad-
    C_{\mathrm{dc}}\eta_{\comp}
    \Dist_{\comp}^{\sharp,\omega,[0,K]}(\calD,\Gcomp)
    -
    C_{\mathrm{dc}}\eta_{\comp}
    \delta_{\comp}^{\omega,[0,K]}
    -
    \Delta_{\mathrm{dc}}.
\end{aligned}
\]
\end{corollary}

\begin{proof}
The first part of the proof of \Cref{thm:detector-comparison} gives
\[
    M_{\Lambda}^{\loc}(\calD-\bzeta_*)
    \ge
    M_{\Lambda}^{\comp}(\Theta_{\Lambda}(\calD-\bzeta_*))
    -
    C_{\mathrm{dc}}
    \Err_{\comp}^{[0,K]}(\calD;\bzeta_*)
    -
    \Delta_{\mathrm{dc}}.
\]
Now use the weighted closure estimate in \Cref{ass:component-ledger},
\[
    \Err_{\comp}^{[0,K]}(\calD;\bzeta_*)
    \le
    \eta_{\comp}
    \Dist_{\comp}^{\sharp,\omega,[0,K]}(\calD,\Gcomp)
    +
    \eta_{\comp}
    \delta_{\comp}^{\omega,[0,K]}.
\]
Substitution gives the claimed weighted comparison.
\end{proof}

\subsection{Channelwise Intertwining Ledger}
\label{sec:channel-ledger}

\begin{definition}[Detector channel assignment]
\label{def:channel-assignment}
A detector channel assignment is a finite map that assigns each clean observable
or clean tax discrepancy to one of the already closed component channels:
\[
    \Err_{\src}^{\mathrm{prs}},
    \qquad
    \Err_{\loc},
    \qquad
    \Err_{\mathrm{rep}},
    \qquad
    \Err_{\mathrm{gs}},
\]
or to the remaining detector error \(\Delta_{\detc}^{\mathrm{rem}}\).
\end{definition}

\begin{proposition}[Channelwise sufficiency criterion]
\label{prop:channelwise-sufficiency}
Let \(\mathfrak Q_{\detc}\) be a finite set of detector mismatch channels.
Suppose there are nonnegative channel errors \(E_q(\calD;\bzeta)\), constants
\(a_q\ge0\), and remainders \(r_q(\calD;\bzeta)\ge0\) such that
\[
    \Err_{\detc}(\calD;\bzeta)
    \le
    \sum_{q\in\mathfrak Q_{\detc}}E_q(\calD;\bzeta),
\]
\[
    E_q(\calD;\bzeta)
    \le
    a_q\Err_{\comp}^{[0,K]}(\calD;\bzeta)
    +
    r_q(\calD;\bzeta)
    \qquad(q\in\mathfrak Q_{\detc}),
\]
and
\[
    \sum_{q\in\mathfrak Q_{\detc}}r_q(\calD;\bzeta)
    \le
    \Delta_{\detc}^{\mathrm{rem}}.
\]
Then \Cref{ass:detector-residual-control} holds with
\[
    a_{\detc}:=\sum_{q\in\mathfrak Q_{\detc}}a_q.
\]
\end{proposition}

\begin{proof}
Summing the channelwise bounds gives
\[
\begin{aligned}
    \Err_{\detc}(\calD;\bzeta)
    &\le
    \sum_{q\in\mathfrak Q_{\detc}}
    \left(
        a_q\Err_{\comp}^{[0,K]}(\calD;\bzeta)
        +
        r_q(\calD;\bzeta)
    \right)\\
    &=
    \left(\sum_{q\in\mathfrak Q_{\detc}}a_q\right)
    \Err_{\comp}^{[0,K]}(\calD;\bzeta)
    +
    \sum_{q\in\mathfrak Q_{\detc}}r_q(\calD;\bzeta)\\
    &\le
    a_{\detc}\Err_{\comp}^{[0,K]}(\calD;\bzeta)
    +
    \Delta_{\detc}^{\mathrm{rem}}.
\end{aligned}
\]
This is exactly \Cref{ass:detector-residual-control}.
\end{proof}

\begin{remark}[Unassigned detector channels]
The proposition identifies the only bookkeeping obstruction in this branch:
every detector discrepancy channel must either be assigned to the closed
component residual or placed explicitly in
\(\Delta_{\detc}^{\mathrm{rem}}\).  This paper does not verify those assignments
for a concrete detector model.
\end{remark}

\subsection{Conditional Finite-Window Local-to-Clean Transfer}
\label{sec:assembly}

This section assembles the detector-comparison theorem with imported clean-side
and chart-visibility inputs.  The assembly remains same-representative: the
representative \(\bzeta_*\) used in the detector comparison is also the
representative used in the charted clean package and in the component residual
ledger.

\begin{assumption}[Clean finite-window anti-phantom gap]
\label{ass:clean-gap}
There exists \(\mu_{\Lambda}^{\comp}>0\) such that every clean package \(d\)
in the finite-window clean class satisfies
\[
    M_{\Lambda}^{\comp}(d)
    \ge
    \mu_{\Lambda}^{\comp}
    \Dist_{\cl}(d,\Gamma_{\Lambda}^{\cl}),
 \]
where \(\Gamma_{\Lambda}^{\cl}\) is the clean gauge or clean null class.
\end{assumption}

\begin{assumption}[Chart visibility in the older baseline geometry]
\label{ass:chart-visibility}
There are constants \(\lambda_G>0\) and \(\delta_G\ge0\) such that the same
shifted representative satisfies
\[
    \Dist_{\cl}
    (\Theta_{\Lambda}(\calD-\bzeta_*),\Gamma_{\Lambda}^{\cl})
    \ge
    \lambda_G
    \Dist_{\loc,\mathrm{int},0}(\calD,\Gamma_{\Lambda,\adm}^{\mathrm{int}})
    -
    \delta_G.
\]
\end{assumption}

\begin{assumption}[Weighted component-to-baseline comparison]
\label{ass:component-to-baseline}
There are a constant \(C_{\mathrm{comp}/0}\ge0\) and a nonnegative finite-window
defect \(B_{\comp}(\calD)\) such that
\[
    \Dist_{\comp}^{\sharp,\omega,[0,K]}(\calD,\Gcomp)
    \le
    C_{\mathrm{comp}/0}
    \Dist_{\loc,\mathrm{int},0}(\calD,\Gamma_{\Lambda,\adm}^{\mathrm{int}})
    +
    B_{\comp}(\calD).
\]
\end{assumption}

\begin{theorem}[Conditional finite-window local-to-clean transfer]
\label{thm:local-to-clean-transfer}
Assume the weighted detector comparison of \Cref{cor:weighted-detector}, the
clean gap \Cref{ass:clean-gap}, the chart visibility estimate
\Cref{ass:chart-visibility}, and the component-to-baseline comparison
\Cref{ass:component-to-baseline}.  Define
\[
    c_{\Lambda}
    :=
    \mu_{\Lambda}^{\comp}\lambda_G
    -
    C_{\mathrm{dc}}\eta_{\comp}C_{\mathrm{comp}/0}
\]
and
\[
    \mathcal E_{\Lambda}(\calD)
    :=
    \mu_{\Lambda}^{\comp}\delta_G
    +
    C_{\mathrm{dc}}\eta_{\comp}B_{\comp}(\calD)
    +
    C_{\mathrm{dc}}\eta_{\comp}\delta_{\comp}^{\omega,[0,K]}
    +
    \Delta_{\mathrm{dc}}.
\]
Then
\[
    M_{\Lambda}^{\loc}(\calD-\bzeta_*)
    \ge
    c_{\Lambda}
    \Dist_{\loc,\mathrm{int},0}
    (\calD,\Gamma_{\Lambda,\adm}^{\mathrm{int}})
    -
    \mathcal E_{\Lambda}(\calD).
\]
\end{theorem}

\begin{proof}
By the weighted detector comparison,
\[
\begin{aligned}
    M_{\Lambda}^{\loc}(\calD-\bzeta_*)
    &\ge
    M_{\Lambda}^{\comp}(\Theta_{\Lambda}(\calD-\bzeta_*))\\
    &\quad-
    C_{\mathrm{dc}}\eta_{\comp}
    \Dist_{\comp}^{\sharp,\omega,[0,K]}(\calD,\Gcomp)
    -
    C_{\mathrm{dc}}\eta_{\comp}\delta_{\comp}^{\omega,[0,K]}
    -
    \Delta_{\mathrm{dc}}.
\end{aligned}
\]
The clean anti-phantom gap and chart visibility give
\[
\begin{aligned}
    M_{\Lambda}^{\comp}(\Theta_{\Lambda}(\calD-\bzeta_*))
    &\ge
    \mu_{\Lambda}^{\comp}
    \Dist_{\cl}(\Theta_{\Lambda}(\calD-\bzeta_*),\Gamma_{\Lambda}^{\cl})\\
    &\ge
    \mu_{\Lambda}^{\comp}\lambda_G
    \Dist_{\loc,\mathrm{int},0}
    (\calD,\Gamma_{\Lambda,\adm}^{\mathrm{int}})
    -
    \mu_{\Lambda}^{\comp}\delta_G.
\end{aligned}
\]
Since \(C_{\mathrm{dc}}\eta_{\comp}\ge0\), the component-to-baseline comparison
implies
\[
\begin{aligned}
    -C_{\mathrm{dc}}\eta_{\comp}
    \Dist_{\comp}^{\sharp,\omega,[0,K]}(\calD,\Gcomp)
    &\ge
    -C_{\mathrm{dc}}\eta_{\comp}C_{\mathrm{comp}/0}
    \Dist_{\loc,\mathrm{int},0}
    (\calD,\Gamma_{\Lambda,\adm}^{\mathrm{int}})\\
    &\quad-
    C_{\mathrm{dc}}\eta_{\comp}B_{\comp}(\calD).
\end{aligned}
\]
Substituting the last two displays into the weighted detector comparison and
collecting the coefficient of
\(\Dist_{\loc,\mathrm{int},0}
(\calD,\Gamma_{\Lambda,\adm}^{\mathrm{int}})\) gives exactly
\[
    M_{\Lambda}^{\loc}(\calD-\bzeta_*)
    \ge
    \left(
        \mu_{\Lambda}^{\comp}\lambda_G
        -
        C_{\mathrm{dc}}\eta_{\comp}C_{\mathrm{comp}/0}
    \right)
    \Dist_{\loc,\mathrm{int},0}
    (\calD,\Gamma_{\Lambda,\adm}^{\mathrm{int}})
    -
    \mathcal E_{\Lambda}(\calD).
\]
This is the claimed estimate.
\end{proof}

\begin{corollary}[Detection threshold]
\label{cor:local-to-clean-threshold}
Assume the hypotheses of \Cref{thm:local-to-clean-transfer} and suppose
\(c_{\Lambda}>0\).  If
\[
    \Dist_{\loc,\mathrm{int},0}
    (\calD,\Gamma_{\Lambda,\adm}^{\mathrm{int}})
    >
    \frac{\mathcal E_{\Lambda}(\calD)}{c_{\Lambda}},
\]
then
\[
    M_{\Lambda}^{\loc}(\calD-\bzeta_*)>0.
\]
\end{corollary}

\begin{proof}
The strict inequality gives
\[
    c_{\Lambda}
    \Dist_{\loc,\mathrm{int},0}
    (\calD,\Gamma_{\Lambda,\adm}^{\mathrm{int}})
    -
    \mathcal E_{\Lambda}(\calD)
    >
    0.
\]
The conclusion follows from \Cref{thm:local-to-clean-transfer}.
\end{proof}

\begin{remark}[Scope of the transfer theorem]
The theorem is a finite-window assembly result.  It imports the clean gap,
chart visibility, component-to-baseline comparison, and detector comparison. It
does not prove compactness of Navier--Stokes-generated clean sources,
pressure/tax kernel-freeness, scale-uniformity, regularity, singularity
exclusion, or a Clay-problem conclusion.
\end{remark}

\subsection{Package-Realizability for Local Navier--Stokes Data}
\label{sec:realizability}

This section verifies that the finite-window package coordinates used above can
be constructed from local Navier--Stokes data.  The result is a coordinate
realizability theorem only.  It does not prove compactness, detector
kernel-freeness, baseline visibility, component-to-baseline comparison,
scale-uniformity, or regularity.

\subsubsection{Local data and package spaces}

Let \(I=(-1,0)\) and \(Q_1=B_1\times I\).  Fix cutoffs
\[
    \eta\in C_c^\infty(B_1),
    \qquad
    \eta\equiv1\ \text{on }B_{3/4},
\]
and
\[
    \chi\in C_c^\infty(B_{3/4}),
    \qquad
    \chi\equiv1\ \text{on }B_{1/2}.
\]
Set
\[
    A_\chi
    :=
    \operatorname{supp}\nabla\chi
    \cup
    \operatorname{supp}\Delta\chi.
\]
The pressure-source and pressure observation spaces are
\[
    X_{\src}
    :=
    L^{3/2}(I;L^{3/2}(B_1))^{3\times3},
    \qquad
    Y_{\mathrm{prs}}
    :=
    L^{3/2}(I;L^{3/2}(B_{1/2})),
\]
and the pressure-natural harmonic observation space is
\[
    Y_{\mathrm{harm}}^{(3/2)}
    :=
    L^{3/2}(I;L^{3/2}(B_{3/4})).
\]

\begin{assumption}[Pressure-admissible local Navier--Stokes data]
\label{ass:realizable-ns-data}
The local data \((u,p)\) satisfy
\[
    u\in L^3(Q_1)^3,
    \qquad
    \nabla u\in L^2(I;L^2(B_1))^{3\times3},
    \qquad
    p\in L^{3/2}(Q_1),
\]
\[
    \nabla\cdot u=0,
\]
and
\[
    \partial_tu-\Delta u+\nabla p+\nabla\cdot(u\otimes u)=0
\]
in distributions on \(Q_1\).  The pressure is admissible in the sense that
\[
    -\Delta p=\partial_i\partial_j(u_i u_j)
\]
in distributions on the fixed local window, modulo time-dependent constants.
When the package includes the essential local energy coordinate
\(\operatorname*{ess\,sup}_{t\in I}\int_{B_1}|u(t)|^2\,dx\), we also assume the
usual local-energy part of suitable weak data,
\[
    u\in L^\infty(I;L^2(B_1))^3.
\]
\end{assumption}

\begin{definition}[Canonical package coordinates]
\label{def:canonical-package-map}
For data satisfying \Cref{ass:realizable-ns-data}, define
\[
    u_D:=u,
    \qquad
    p_D:=p,
\]
\[
    F^{\mathrm{act}}_{D,ij}:=\eta u_i u_j,
    \qquad
    p_D^{\mathrm{act}}:=R_iR_j(F^{\mathrm{act}}_{D,ij}),
\]
where \(F_D^{\mathrm{act}}\) is extended by zero outside \(B_1\) before applying
the Riesz transforms.  On \(B_{3/4}\), define
\[
    p_{\mathrm{harm},D}:=p_D-p_D^{\mathrm{act}}.
\]
The canonical source model is
\[
    U_D:=u,
    \qquad
    R_D:=0,
    \qquad
    E_{F,D}:=0,
    \qquad
    F^{\mathrm{mod}}_{D}:=F_D^{\mathrm{act}}.
\]
\end{definition}

\begin{lemma}[Pressure-source realizability]
\label{lem:pressure-source-realizability}
Under \Cref{ass:realizable-ns-data},
\[
    F_D^{\mathrm{act}}\in X_{\src},
    \qquad
    p_D^{\mathrm{act}}\in Y_{\mathrm{prs}},
    \qquad
    p_{\mathrm{harm},D}\in Y_{\mathrm{harm}}^{(3/2)}.
\]
Moreover,
\[
    -\Delta p_{\mathrm{harm},D}=0
\]
in distributions on \(B_{3/4}\) for almost every time, equivalently on
\(B_{3/4}\times I\) in the time-dependent distributional sense.
\end{lemma}

\begin{proof}
Since \(u\in L^3(Q_1)^3\), Holder's inequality gives
\[
    u_i u_j\in L^{3/2}(Q_1)
    \qquad(1\le i,j\le3).
\]
Multiplication by the bounded cutoff \(\eta\) therefore gives
\[
    F_D^{\mathrm{act}}\in X_{\src}.
\]
The fixed-window Calderon--Zygmund estimate for the zero extension yields
\[
    \|R_iR_j(F^{\mathrm{act}}_{D,ij})\|_{L^{3/2}(I;L^{3/2}(\R^3))}
    \le
    C
    \|F_D^{\mathrm{act}}\|_{X_{\src}},
\]
and hence \(p_D^{\mathrm{act}}\in Y_{\mathrm{prs}}\).  The same estimate also
gives \(p_D^{\mathrm{act}}\in L^{3/2}(I;L^{3/2}(B_{3/4}))\).  Since
\(p\in L^{3/2}(Q_1)\), it follows that
\[
    p_{\mathrm{harm},D}=p-p_D^{\mathrm{act}}
    \in
    L^{3/2}(I;L^{3/2}(B_{3/4}))
    =
    Y_{\mathrm{harm}}^{(3/2)}.
\]
With the sign convention \( -\Delta R_iR_jF_{ij}=\partial_i\partial_jF_{ij}\),
\[
    -\Delta p_D^{\mathrm{act}}
    =
    \partial_i\partial_j(\eta u_i u_j)
\]
on \(B_1\).  Because \(\eta\equiv1\) on \(B_{3/4}\), this equals
\(\partial_i\partial_j(u_i u_j)\) on \(B_{3/4}\).  Pressure admissibility gives
\[
    -\Delta p=\partial_i\partial_j(u_i u_j)
\]
on the same region, modulo time-dependent constants.  Their difference is
therefore harmonic in the spatial variable on \(B_{3/4}\).
\end{proof}

\subsubsection{Localization, energy, and gate coordinates}

\begin{definition}[Localization leakage and finite-window observables]
\label{def:leakage-observables}
Define the localization leakage coordinates
\[
    \Leak_{\nabla u}(D)
    :=
    \|\nabla u\|_{L^2(I;L^2(A_\chi))},
\]
\[
    \Leak_u(D)
    :=
    \|u\|_{L^3(I;L^3(A_\chi))},
\]
and
\[
    \Leak_p(D)
    :=
    \|p\|_{L^{3/2}(I;L^{3/2}(A_\chi))}.
\]
A canonical flux coordinate may be taken as any finite list of quantities
controlled by the displayed \(L^3\), \(L^{3/2}\), and \(L^2\) norms, for
example
\[
    \Pi_D
    :=
    \left(
        \int_{Q_1}\chi^2|u|^3\,dx\,dt,
        \int_{Q_1}\chi |p|^{3/2}\,dx\,dt
    \right).
\]
When the local-energy part of \Cref{ass:realizable-ns-data} is assumed, define
\[
    \Phi_D
    :=
    \left(
        \operatorname*{ess\,sup}_{t\in I}\int_{B_1}|u(x,t)|^2\,dx,
        \int_{Q_1}|\nabla u|^2\,dx\,dt
    \right).
\]
The trace coordinate \(T_D\) is either omitted, set to \(0\), or chosen as a
good-time datum \(u(t_D)\in L^2(B_1)^3\) at a time \(t_D\in I\) for which the
local energy is finite.
\end{definition}

\begin{definition}[Gate and slack convention]
\label{def:gate-slack-realizability}
Let \(\mathfrak A\) be a finite set of gate/slack channels.  For each
\(a\in\mathfrak A\), let \(B_a(D)\ge0\) be a finite used-budget functional of
the finite coordinates above, and let \(\tau_a(D)<\infty\) be the selected
threshold.  Define
\[
    s_a(D):=(\tau_a(D)-B_a(D))_+,
    \qquad
    \Gate_a(D):=(B_a(D)-\tau_a(D))_+,
\]
and
\[
    \Slack_a(D):=|B_a(D)+s_a(D)-\tau_a(D)|.
\]
\end{definition}

\begin{lemma}[Finiteness of localization and gate coordinates]
\label{lem:localization-gate-finiteness}
Under \Cref{ass:realizable-ns-data}, the leakage coordinates
\[
    \Leak_{\nabla u}(D),\qquad \Leak_u(D),\qquad \Leak_p(D)
\]
are finite.  The canonical \(\Pi_D\) is finite.  If the local-energy part of
\Cref{ass:realizable-ns-data} is included, then \(\Phi_D\) and any good-time
trace \(T_D\) are finite.  For every finite gate/slack rule in
\Cref{def:gate-slack-realizability}, \(s_a(D)\), \(\Gate_a(D)\), and
\(\Slack_a(D)\) are finite.
\end{lemma}

\begin{proof}
The shell \(A_\chi\) is contained in \(B_1\).  The assumptions
\[
    \nabla u\in L^2(I;L^2(B_1)),
    \qquad
    u\in L^3(Q_1),
    \qquad
    p\in L^{3/2}(Q_1)
\]
therefore immediately imply finiteness of the three leakage norms.  The
canonical \(\Pi_D\) is finite by the same \(L^3\) and \(L^{3/2}\) bounds and the
boundedness of \(\chi\).  If \(u\in L^\infty(I;L^2(B_1))^3\), then the first
coordinate of \(\Phi_D\) is finite, and the second is finite by the assumed
\(L^2\) bound on \(\nabla u\).  For almost every time \(t\), \(u(t)\in
L^2(B_1)^3\), so a selected good-time trace is finite.  Finally, the gate/slack
coordinates are obtained from finite nonnegative numbers by addition,
subtraction, absolute value, and positive part.
\end{proof}

\subsubsection{Projected source convention}

\begin{proposition}[Projected clean source coordinates]
\label{prop:projected-source-coordinates}
Let \(P_N^u:L^3(Q_1)^3\to L^3(Q_1)^3\) be a bounded finite-window velocity
projection or coordinate selector.  Set
\[
    U_D:=P_N^u u,
    \qquad
    R_{D,ij}:=u_i u_j-U_{D,i}U_{D,j},
    \qquad
    E_{F,D}:=0,
\]
and
\[
    F^{\mathrm{mod}}_{D,ij}:=\eta(U_{D,i}U_{D,j}+R_{D,ij}).
\]
Then
\[
    U_D\in L^3(Q_1)^3,
    \qquad
    R_D\in X_{\src},
    \qquad
    F_D^{\mathrm{mod}}=F_D^{\mathrm{act}}.
\]
\end{proposition}

\begin{proof}
Boundedness of \(P_N^u\) gives \(U_D\in L^3(Q_1)^3\).  Holder's inequality gives
\[
    u_i u_j\in L^{3/2}(Q_1),
    \qquad
    U_{D,i}U_{D,j}\in L^{3/2}(Q_1),
\]
so \(R_D\in X_{\src}\).  The definition of \(R_D\) gives
\[
    U_{D,i}U_{D,j}+R_{D,ij}=u_i u_j,
\]
and hence
\[
    F_{D,ij}^{\mathrm{mod}}
    =
    \eta u_i u_j
    =
    F_{D,ij}^{\mathrm{act}}.
\]
\end{proof}

\subsubsection{Finite-chain reproduction coordinates}

\begin{proposition}[Well-defined reproduction drifts]
\label{prop:reproduction-realizability}
Let \((u_k,p_k)_{k=0}^K\) be a finite chain of data satisfying
\Cref{ass:realizable-ns-data}, and let \(D_k=D_\Lambda(u_k,p_k)\) be the
corresponding packages.  Suppose the finite-window reproduction maps
\[
    \mathcal R^u,\quad
    \mathcal R^{\src},\quad
    \mathcal R^{\mathrm{prs}},\quad
    \mathcal R^{\mathrm{harm}},\quad
    \mathcal R^U,\quad
    \mathcal R^R,\quad
    \mathcal R^E
\]
map each coordinate space into the corresponding next-window coordinate space.
Then the reproduction drifts
\[
    u_{k+1}-\mathcal R^u u_k,
\]
\[
    F^{\mathrm{act}}_{k+1}
    -
    \mathcal R^{\src}F^{\mathrm{act}}_k,
\]
\[
    p^{\mathrm{act}}_{k+1}
    -
    \mathcal R^{\mathrm{prs}}p^{\mathrm{act}}_k,
\]
and
\[
    p_{\mathrm{harm},k+1}
    -
    \mathcal R^{\mathrm{harm}}p_{\mathrm{harm},k}
\]
are well-defined finite-window objects in their assigned coordinate spaces.
\end{proposition}

\begin{proof}
Each \(D_k\) has its source, active pressure, and harmonic pressure coordinates
in the spaces identified in \Cref{lem:pressure-source-realizability}.  The
assumed mapping property of the reproduction operators places
\(\mathcal R^u u_k\), \(\mathcal R^{\src}F^{\mathrm{act}}_k\),
\(\mathcal R^{\mathrm{prs}}p^{\mathrm{act}}_k\), and
\(\mathcal R^{\mathrm{harm}}p_{\mathrm{harm},k}\) in the same spaces as the
corresponding \((k+1)\)-coordinates.  Since these spaces are vector spaces, the
displayed differences are well-defined.  No scale-uniform reproduction estimate
is asserted.
\end{proof}

\begin{theorem}[Finite-window package-realizability]
\label{thm:package-realizability}
Let \((u,p)\) satisfy \Cref{ass:realizable-ns-data}, and fix the canonical
coordinate convention of \Cref{def:canonical-package-map} together with any
finite gate/slack and observable rules of
\Cref{def:leakage-observables,def:gate-slack-realizability}.  Then
\[
    D_\Lambda(u,p)
\]
is a well-defined sharp localized finite-window package in the coordinate
sense.  In particular,
\[
    F_D^{\mathrm{act}}\in X_{\src},
    \qquad
    p_D^{\mathrm{act}}\in Y_{\mathrm{prs}},
    \qquad
    p_{\mathrm{harm},D}\in Y_{\mathrm{harm}}^{(3/2)},
\]
\[
    \Leak_{\nabla u}(D),\quad \Leak_u(D),\quad \Leak_p(D)<\infty,
\]
and all finite-window gate/slack and ledger coordinates are well-defined.  For
finite chains, the reproduction drift coordinates are well-defined under the
mapping assumption of \Cref{prop:reproduction-realizability}.

If, in addition, the selected model coordinates satisfy the bounded-amplitude
or quadratic-geometry hypotheses required earlier, then \(D_\Lambda(u,p)\)
belongs to the corresponding package class.
\end{theorem}

\begin{proof}
The pressure-source, active-pressure, and harmonic-pressure assertions are
\Cref{lem:pressure-source-realizability}.  The localization leakage,
observable, trace, and gate/slack assertions are
\Cref{lem:localization-gate-finiteness}.  The canonical source model is
well-defined by \Cref{def:canonical-package-map}; the projected variant is
available under the additional projection hypothesis in
\Cref{prop:projected-source-coordinates}.  The finite-chain reproduction
coordinates are well-defined by \Cref{prop:reproduction-realizability}.  The
last statement is conditional by definition: once the selected coordinates
satisfy the additional bounded-amplitude or quadratic-geometry hypotheses used
in the earlier package-level theorems, the realized package lies in the
corresponding admissible package class.
\end{proof}

\begin{remark}[Status of package-realizability]
\label{rem:realizability-status}
This theorem connects local Navier--Stokes data to the coordinate layer of the
finite-window package framework.  It does not show that the resulting family is
compact, that the pressure/tax detector is kernel-free, that baseline visibility
or component-to-baseline comparison holds automatically, or that any estimate is
scale-uniform.
\end{remark}

\subsection{Clean-Source Compactness and Effective Projection}
\label{sec:clean-source-compactness}

The transfer estimates used in the finite-window framework may contain a clean
pressure projection tail.  This section records finite-window hypotheses under
which that tail converges uniformly on a selected Navier--Stokes-generated
package class.  The point is deliberately modest: compactness or effective
projection is a structural input.  It is not a consequence of boundedness alone,
and it is not proved here for arbitrary suitable weak solutions.

\subsubsection{Selected clean sources and pressure images}

Fix an NS-generated package class
\[
    \mathcal A_{\Lambda}^{\NS}
    :=
    \{D_\Lambda(u,p):(u,p)\in\mathcal S_\Lambda\},
\]
where \(\mathcal S_\Lambda\) is a selected family of local data satisfying the
realizability assumptions of \Cref{ass:realizable-ns-data}.  Fix also a
same-representative selection
\[
    D\longmapsto \bzeta_0(D).
\]
The clean source associated with \(D-\bzeta_0(D)\) is denoted
\[
    F^{\cl}_{D-\bzeta_0(D)}\in X_{\src}.
\]
Depending on the source convention, this may be the canonical active source
\[
    F^{\cl}_{D-\bzeta_0(D)}
    =
    F^{\mathrm{act}}_{D-\bzeta_0(D)}
    =
    \eta u_{D-\bzeta_0(D)}\otimes u_{D-\bzeta_0(D)},
\]
the model source
\[
    F^{\cl}_{D-\bzeta_0(D)}
    =
    \eta\bigl(
        U_{D-\bzeta_0(D)}\otimes U_{D-\bzeta_0(D)}
        +
        R_{D-\bzeta_0(D)}
    \bigr),
\]
or the residual clean source
\[
    F^{\cl}_{D-\bzeta_0(D)}
    =
    \eta\bigl(
        U_{D-\bzeta_0(D)}\otimes U_{D-\bzeta_0(D)}
        +
        R_{D-\bzeta_0(D)}
    \bigr)
    +
    E_{F,D-\bzeta_0(D)}.
\]
Define the selected source family
\[
    \mathcal F_{\Lambda,0}
    :=
    \{F^{\cl}_{D-\bzeta_0(D)}:D\in\mathcal A_{\Lambda}^{\NS}\}
    \subset X_{\src}.
\]
Let
\[
    \mathcal R_{\mathrm{prs}}:X_{\src}\to Y_{\mathrm{prs}},
    \qquad
    \mathcal R_{\mathrm{prs}}(F):=R_iR_j(F_{ij})|_{B_{1/2}},
\]
where \(F\) is extended by zero before applying the Riesz transforms.  The
fixed-window Calderon--Zygmund estimate gives
\[
    \|\mathcal R_{\mathrm{prs}}(F)\|_{Y_{\mathrm{prs}}}
    \le
    C_{\mathrm{CZ}}\|F\|_{X_{\src}}.
\]
The selected clean pressure image is
\[
    \mathcal G_{\Lambda,0}
    :=
    \mathcal R_{\mathrm{prs}}(\mathcal F_{\Lambda,0})
    \subset Y_{\mathrm{prs}}.
\]

Let \(P_N^{\cl}:Y_{\mathrm{prs}}\to Y_{\mathrm{prs}}\) be finite-rank clean
pressure projections satisfying
\[
    P_N^{\cl}g\to g
    \qquad\text{for every }g\in Y_{\mathrm{prs}},
\]
and
\[
    C_P:=\sup_N\|P_N^{\cl}\|_{Y_{\mathrm{prs}}\to Y_{\mathrm{prs}}}<\infty.
\]
Define the uniform clean projection tail by
\[
    \Delta_{\mathrm{proj},N}^{\mathrm{unif}}
    (\mathcal A_{\Lambda}^{\NS})
    :=
    \sup_{F\in\mathcal F_{\Lambda,0}}
    \|(I-P_N^{\cl})\mathcal R_{\mathrm{prs}}(F)\|_{Y_{\mathrm{prs}}}.
\]

\begin{remark}[Boundedness alone is not compactness]
\label{rem:boundedness-not-compactness}
If \(Y\) is infinite-dimensional and \(P_N:Y\to Y\) are finite-rank operators
with \(P_Ng\to g\) strongly for each fixed \(g\), one cannot infer
\[
    \sup_{\|g\|_Y\le1}\|(I-P_N)g\|_Y\to0.
\]
Otherwise \(P_N\to I\) in operator norm, so the identity would be a norm limit
of compact operators and hence compact, impossible on an infinite-dimensional
Banach space.  Thus bounded finite amplitude is not a substitute for compactness
or an effective projection estimate.
\end{remark}

\subsubsection{Compactness implies uniform projection tails}

\begin{theorem}[Compact pressure image gives uniform projection-tail convergence]
\label{thm:compact-pressure-image-tail}
Assume
\[
    \mathcal G_{\Lambda,0}\Subset Y_{\mathrm{prs}}.
\]
Then
\[
    \Delta_{\mathrm{proj},N}^{\mathrm{unif}}
    (\mathcal A_{\Lambda}^{\NS})
    \to0.
\]
\end{theorem}

\begin{proof}
Let \(K=\overline{\mathcal G_{\Lambda,0}}\), compact in \(Y_{\mathrm{prs}}\).
Fix \(\varepsilon>0\).  Choose points \(g_1,\ldots,g_J\in K\) such that
\[
    K\subset
    \bigcup_{j=1}^J
    B_{Y_{\mathrm{prs}}}
    \left(g_j,\frac{\varepsilon}{3(1+C_P)}\right).
\]
For each fixed \(j\), strong convergence gives
\[
    \|(I-P_N^{\cl})g_j\|_{Y_{\mathrm{prs}}}\to0.
\]
After increasing \(N\), this quantity is at most \(\varepsilon/3\) for all
\(1\le j\le J\).  If \(g\in K\), choose \(j\) with
\(\|g-g_j\|_{Y_{\mathrm{prs}}}<\varepsilon/[3(1+C_P)]\).  Then
\[
\begin{aligned}
    \|(I-P_N^{\cl})g\|_{Y_{\mathrm{prs}}}
    &\le
    \|(I-P_N^{\cl})(g-g_j)\|_{Y_{\mathrm{prs}}}
    +
    \|(I-P_N^{\cl})g_j\|_{Y_{\mathrm{prs}}}\\
    &\le
    (1+C_P)\|g-g_j\|_{Y_{\mathrm{prs}}}
    +
    \frac{\varepsilon}{3}
    <
    \varepsilon.
\end{aligned}
\]
Taking the supremum over \(g\in\mathcal G_{\Lambda,0}\subset K\) proves the
claim.
\end{proof}

\begin{theorem}[Source-level compactness criterion]
\label{thm:source-compact-pressure-compact}
If
\[
    \mathcal F_{\Lambda,0}\Subset X_{\src},
\]
then
\[
    \mathcal G_{\Lambda,0}\Subset Y_{\mathrm{prs}},
\]
and hence
\[
    \Delta_{\mathrm{proj},N}^{\mathrm{unif}}
    (\mathcal A_{\Lambda}^{\NS})
    \to0.
\]
\end{theorem}

\begin{proof}
The map \(\mathcal R_{\mathrm{prs}}:X_{\src}\to Y_{\mathrm{prs}}\) is bounded
linear by the fixed Calderon--Zygmund estimate, hence continuous.  The continuous
image of a compact set is compact, so
\(\mathcal R_{\mathrm{prs}}(\overline{\mathcal F_{\Lambda,0}})\) is compact in
\(Y_{\mathrm{prs}}\) and contains \(\mathcal G_{\Lambda,0}\).  The uniform
projection-tail convergence then follows from
\Cref{thm:compact-pressure-image-tail}.
\end{proof}

\subsubsection{Concrete compactness criteria}

\begin{theorem}[Finite-dimensional source model]
\label{thm:finite-dimensional-source-compactness}
Assume there is a finite-dimensional subspace \(S_J\subset X_{\src}\) such that
\[
    F^{\cl}_{D-\bzeta_0(D)}\in S_J
    \qquad
    (D\in\mathcal A_{\Lambda}^{\NS}),
\]
and the corresponding coefficient set is bounded in \(S_J\).  Then
\[
    \mathcal F_{\Lambda,0}\Subset X_{\src},
\]
and consequently
\[
    \Delta_{\mathrm{proj},N}^{\mathrm{unif}}
    (\mathcal A_{\Lambda}^{\NS})\to0.
\]
\end{theorem}

\begin{proof}
On a finite-dimensional normed space, bounded sets are precompact.  Since all
norms on \(S_J\) are equivalent and \(S_J\) is continuously embedded in
\(X_{\src}\), the bounded coefficient set has compact closure in \(X_{\src}\).
Thus \(\mathcal F_{\Lambda,0}\Subset X_{\src}\), and the conclusion follows from
\Cref{thm:source-compact-pressure-compact}.
\end{proof}

\begin{remark}
This theorem applies to reduced finite-window models, Galerkin source families,
numerical package classes, or explicitly truncated clean-source models.  It is
not a statement that arbitrary Navier--Stokes sources are finite-dimensional.
\end{remark}

\begin{theorem}[Strong clean-coordinate compactness]
\label{thm:strong-coordinate-compactness}
Assume the clean source has the form
\[
    F^{\cl}_{D-\bzeta_0(D)}
    =
    \eta\bigl(
        U_{D-\bzeta_0(D)}\otimes U_{D-\bzeta_0(D)}
        +
        R_{D-\bzeta_0(D)}
    \bigr)
    +
    E_{F,D-\bzeta_0(D)}.
\]
Define
\[
    \mathcal U_{\Lambda,0}
    :=
    \{U_{D-\bzeta_0(D)}:D\in\mathcal A_{\Lambda}^{\NS}\}
    \subset L^3(Q_1)^3,
\]
\[
    \mathcal C_{\Lambda,0}^{R}
    :=
    \{R_{D-\bzeta_0(D)}:D\in\mathcal A_{\Lambda}^{\NS}\}
    \subset L^{3/2}(Q_1)^{3\times3},
\]
and
\[
    \mathcal E_{\Lambda,0}^{F}
    :=
    \{E_{F,D-\bzeta_0(D)}:D\in\mathcal A_{\Lambda}^{\NS}\}
    \subset X_{\src}.
\]
If
\[
    \mathcal U_{\Lambda,0}\Subset L^3(Q_1)^3,
    \qquad
    \mathcal C_{\Lambda,0}^{R}\Subset L^{3/2}(Q_1)^{3\times3},
    \qquad
    \mathcal E_{\Lambda,0}^{F}\Subset X_{\src},
\]
then
\[
    \mathcal F_{\Lambda,0}\Subset X_{\src}.
\]
\end{theorem}

\begin{proof}
It is enough to prove continuity of
\[
    (U,R,E)\longmapsto \eta(U\otimes U+R)+E
\]
from \(L^3(Q_1)^3\times L^{3/2}(Q_1)^{3\times3}\times X_{\src}\) into
\(X_{\src}\).  Suppose \(U_n\to U\) in \(L^3\), \(R_n\to R\) in \(L^{3/2}\),
and \(E_n\to E\) in \(X_{\src}\).  Then
\[
\begin{aligned}
    \|U_n\otimes U_n-U\otimes U\|_{L^{3/2}}
    &\le
    \|(U_n-U)\otimes U_n\|_{L^{3/2}}
    +
    \|U\otimes(U_n-U)\|_{L^{3/2}}\\
    &\le
    \|U_n-U\|_{L^3}
    \bigl(\|U_n\|_{L^3}+\|U\|_{L^3}\bigr).
\end{aligned}
\]
Since \(U_n\to U\), the \(L^3\)-norms of \(U_n\) are bounded, and the right-hand
side tends to zero.  Multiplication by the smooth cutoff \(\eta\) is bounded on
\(L^{3/2}\).  Hence the clean sources converge in \(X_{\src}\).  The product of
three compact sets is compact, and the continuous image of that product is
compact.
\end{proof}

\begin{corollary}
\label{cor:strong-coordinate-tail}
Under the hypotheses of \Cref{thm:strong-coordinate-compactness},
\[
    \Delta_{\mathrm{proj},N}^{\mathrm{unif}}
    (\mathcal A_{\Lambda}^{\NS})\to0.
\]
\end{corollary}

\begin{proof}
Combine \Cref{thm:strong-coordinate-compactness} with
\Cref{thm:source-compact-pressure-compact}.
\end{proof}

\begin{theorem}[Sobolev clean-source compactness]
\label{thm:sobolev-source-compactness}
Assume that for some \(s>0\),
\[
    \sup_{F\in\mathcal F_{\Lambda,0}}
    \|F\|_{W^{s,3/2}(Q_1)^{3\times3}}
    <\infty.
\]
Then
\[
    \mathcal F_{\Lambda,0}\Subset X_{\src},
\]
and consequently
\[
    \Delta_{\mathrm{proj},N}^{\mathrm{unif}}
    (\mathcal A_{\Lambda}^{\NS})\to0.
\]
\end{theorem}

\begin{proof}
Since \(Q_1\) is bounded, the Rellich--Kondrachov theorem gives the compact
embedding
\[
    W^{s,3/2}(Q_1)\Subset L^{3/2}(Q_1)
\]
for \(s>0\).  Applying this componentwise to tensor-valued sources proves
\(\mathcal F_{\Lambda,0}\Subset X_{\src}\).  The projection-tail conclusion is
\Cref{thm:source-compact-pressure-compact}.
\end{proof}

\begin{remark}
This is only a sufficient criterion.  No Sobolev or fractional differentiability
bound for arbitrary suitable weak solutions is proved here.
\end{remark}

\begin{theorem}[Kolmogorov--Riesz source compactness]
\label{thm:kolmogorov-source-compactness}
Assume \(\mathcal F_{\Lambda,0}\subset L^{3/2}(Q_1)^{3\times3}\) is uniformly
bounded and, after zero extension, satisfies uniform translation continuity:
\[
    \lim_{|h|\to0}
    \sup_{F\in\mathcal F_{\Lambda,0}}
    \|F(\cdot+h)-F(\cdot)\|_{L^{3/2}(\R^4)^{3\times3}}
    =
    0.
\]
Then
\[
    \mathcal F_{\Lambda,0}\Subset X_{\src}.
\]
Consequently,
\[
    \mathcal G_{\Lambda,0}\Subset Y_{\mathrm{prs}},
    \qquad
    \Delta_{\mathrm{proj},N}^{\mathrm{unif}}
    (\mathcal A_{\Lambda}^{\NS})\to0.
\]
\end{theorem}

\begin{proof}
After zero extension, the family is uniformly bounded in
\(L^{3/2}(\R^4)^{3\times3}\).  Its supports lie in the fixed bounded set
\(\overline{Q_1}\), so tightness is automatic.  The assumed uniform translation
continuity is the remaining Kolmogorov--Riesz hypothesis.  Therefore the family
is precompact in \(L^{3/2}(\R^4)^{3\times3}\), and restriction to \(Q_1\) gives
precompactness in \(X_{\src}\).  The pressure-image and projection-tail
conclusions follow from \Cref{thm:source-compact-pressure-compact}.
\end{proof}

\begin{theorem}[Conditional NS-generated compactness under source translation compactness]
\label{thm:ns-generated-translation-compactness}
Let \(\mathcal A_{\Lambda}^{\NS}\) be generated by a selected local data family
\(\mathcal S_\Lambda\).  Assume the selected clean source family
\(\mathcal F_{\Lambda,0}\) is compatible with the same representative selection,
is uniformly bounded in \(L^{3/2}(Q_1)^{3\times3}\), and satisfies the uniform
translation continuity hypothesis of
\Cref{thm:kolmogorov-source-compactness}.  Then
\[
    \mathcal F_{\Lambda,0}\Subset X_{\src},
    \qquad
    \mathcal G_{\Lambda,0}\Subset Y_{\mathrm{prs}},
\]
and
\[
    \Delta_{\mathrm{proj},N}^{\mathrm{unif}}
    (\mathcal A_{\Lambda}^{\NS})\to0.
\]
\end{theorem}

\begin{proof}
This is \Cref{thm:kolmogorov-source-compactness} applied to the
NS-generated selected clean source family.
\end{proof}

\begin{remark}[Status of the NS-facing criterion]
\label{rem:ns-compactness-status}
\Cref{thm:ns-generated-translation-compactness} identifies a concrete structural
input: uniform \(L^{3/2}\) translation compactness of the selected clean sources.
The theorem does not prove that arbitrary suitable weak solutions satisfy this
input.
\end{remark}

\subsubsection{Effective projection without compactness}

\begin{proposition}[Effective projection replacement]
\label{prop:effective-projection-replacement}
Assume there is a sequence \(\varepsilon_N\downarrow0\) such that
\[
    \Delta_{\mathrm{proj},N}^{\mathrm{unif}}
    (\mathcal A_{\Lambda}^{\NS})
    \le
    \varepsilon_N.
\]
Then any finite-window transfer estimate whose error contains the projection
term
\[
    \alpha_{\mathrm{proj}}
    \Delta_{\mathrm{proj},N}^{\mathrm{unif}}
    (\mathcal A_{\Lambda}^{\NS})
\]
remains valid with this contribution replaced by
\(\alpha_{\mathrm{proj}}\varepsilon_N\).  In particular, the projection
contribution tends to zero as \(N\to\infty\).
\end{proposition}

\begin{proof}
The assumed estimate gives
\[
    \alpha_{\mathrm{proj}}
    \Delta_{\mathrm{proj},N}^{\mathrm{unif}}
    (\mathcal A_{\Lambda}^{\NS})
    \le
    \alpha_{\mathrm{proj}}\varepsilon_N.
\]
This is a direct substitution into the transfer error budget.
\end{proof}

\begin{corollary}[Insertion into finite-window transfer errors]
\label{cor:projection-tail-transfer-insertion}
Suppose a finite-window transfer estimate has an error budget of the form
\[
    \mathcal E_{\Lambda}^{(N)}(D)
    =
    \mathcal E_{\Lambda}^{0}(D)
    +
    \alpha_{\mathrm{proj}}
    \Delta_{\mathrm{proj},N}^{\mathrm{unif}}
    (\mathcal A_{\Lambda}^{\NS}),
\]
where \(\mathcal E_{\Lambda}^{0}\) is independent of \(N\).  If any one of the
compactness criteria in this section applies, or if the effective projection
bound of \Cref{prop:effective-projection-replacement} holds, then for every
\(\varepsilon>0\) there is \(N_\varepsilon\) such that for
\(N\ge N_\varepsilon\),
\[
    \mathcal E_{\Lambda}^{(N)}(D)
    \le
    \mathcal E_{\Lambda}^{0}(D)+\varepsilon.
\]
\end{corollary}

\begin{proof}
Under the compactness criteria, the projection tail tends to zero by the
preceding theorems.  Under the effective projection assumption, it tends to zero
by \Cref{prop:effective-projection-replacement}.  Choose \(N_\varepsilon\) so
that
\[
    \alpha_{\mathrm{proj}}
    \Delta_{\mathrm{proj},N}^{\mathrm{unif}}
    (\mathcal A_{\Lambda}^{\NS})
    \le
    \varepsilon
\]
for all \(N\ge N_\varepsilon\), and substitute into the displayed error budget.
\end{proof}

\begin{remark}[Status]
\label{rem:clean-compactness-status}
This section gives compactness and effective-projection criteria for selected
clean sources.  It does not prove that all suitable weak solutions generate
compact clean-source families.  It also does not prove pressure/tax
kernel-freeness, baseline visibility, component-to-baseline comparison from
Navier--Stokes, scale-uniformity, regularity, or singularity exclusion.
\end{remark}

\subsection{Pressure/Tax Kernel-Freeness in Reduced Packages}
\label{sec:pressure-tax-kernel}

The previous sections close residual and projection-tail inputs under explicit
finite-window hypotheses.  The present section addresses a different structural
question: the common zero set of the active observation and tax channels.  The
goal is not residual absorption.  It is zero-set rigidity: a package with no
observable defect and no active tax should be an admissible gauge direction.
The results below are fixed finite-window and reduced-model criteria only.

\subsubsection{Normalized pressure/tax detector}

\begin{definition}[Normalized pressure/tax detector]
\label{def:tax-detector}
Let \(\mathcal A_{\Lambda}^{\mathrm{tax}}\) be a finite-window tax-admissible
package class.  A normalized pressure/tax detector is a nonnegative functional
\[
\begin{aligned}
    M_{\Lambda}^{\mathrm{tax}}(D)
    &:=
    \|O^0_{\Lambda}D\|_{\mathcal O}
    +
    \beta_{\mathrm{prs}}\Tax_{\mathrm{prs}}(D)
    +
    \beta_{\mathrm{loc}}\Tax_{\mathrm{loc}}(D)\\
    &\quad+
    \beta_{\mathrm{rep}}\Tax_{\mathrm{rep}}(D)
    +
    \beta_{\mathrm{gs}}\Tax_{\mathrm{gs}}(D)
    +
    \beta_{\detc}\Tax_{\detc}(D),
\end{aligned}
\]
where all tax channels are nonnegative and all active weights satisfy
\[
    \beta_{\mathrm{prs}},
    \beta_{\mathrm{loc}},
    \beta_{\mathrm{rep}},
    \beta_{\mathrm{gs}},
    \beta_{\detc}
    >
    0.
\]
The term \(O^0_{\Lambda}D\) is the older-baseline observable.  The five tax
channels measure pressure-source cost, localization leakage, reproduction
drift, gate/slack violation, and detector-comparison mismatch.
\end{definition}

\begin{remark}[Active weights]
If one of the displayed weights is zero, its channel is not part of the
zero-set detector unless a separate zero-set containment hypothesis is imposed.
The positivity assumption ensures that
\[
    M_{\Lambda}^{\mathrm{tax}}(D)=0
\]
forces the observable and every active tax channel to vanish.
\end{remark}

\begin{definition}[Pressure/tax kernel]
\label{def:tax-kernel}
The pressure/tax kernel is
\[
    \mathcal K_{\Lambda}^{\mathrm{tax}}
    :=
    \{D\in\mathcal A_{\Lambda}^{\mathrm{tax}}:
    M_{\Lambda}^{\mathrm{tax}}(D)=0\}.
\]
Equivalently, under the positivity convention in \Cref{def:tax-detector},
\(D\in\mathcal K_{\Lambda}^{\mathrm{tax}}\) if and only if
\[
    O^0_{\Lambda}D=0,
\]
\[
    \Tax_{\mathrm{prs}}(D)=
    \Tax_{\mathrm{loc}}(D)=
    \Tax_{\mathrm{rep}}(D)=
    \Tax_{\mathrm{gs}}(D)=
    \Tax_{\detc}(D)=0.
\]
The kernel-free condition is
\[
    \mathcal K_{\Lambda}^{\mathrm{tax}}
    \subset
    \Gamma_{\Lambda,\adm}^{\mathrm{int}}.
\]
\end{definition}

\begin{lemma}[No coercivity without kernel-freeness]
\label{lem:no-coercivity-without-kernel}
Assume there is \(H\in\mathcal A_{\Lambda}^{\mathrm{tax}}\) such that
\[
    H\notin\Gamma_{\Lambda,\adm}^{\mathrm{int}},
    \qquad
    \Dist_{\loc,\mathrm{int},0}
    (H,\Gamma_{\Lambda,\adm}^{\mathrm{int}})>0,
\]
and all active observation and tax channels vanish on the ray
\(\{\lambda H:\lambda\ge0\}\).  Assume also that this ray is contained in
\(\mathcal A_{\Lambda}^{\mathrm{tax}}\) and the baseline quotient distance is
homogeneous on it:
\[
    \Dist_{\loc,\mathrm{int},0}
    (\lambda H,\Gamma_{\Lambda,\adm}^{\mathrm{int}})
    =
    \lambda
    \Dist_{\loc,\mathrm{int},0}
    (H,\Gamma_{\Lambda,\adm}^{\mathrm{int}})
    \qquad(\lambda\ge0).
\]
Then no estimate of the form
\[
    M_{\Lambda}^{\mathrm{tax}}(D)
    \ge
    \mu_{\Lambda}
    \Dist_{\loc,\mathrm{int},0}
    (D,\Gamma_{\Lambda,\adm}^{\mathrm{int}})
\]
can hold on this class with \(\mu_{\Lambda}>0\).
\end{lemma}

\begin{proof}
For every \(\lambda>0\), the channel-vanishing assumption gives
\[
    M_{\Lambda}^{\mathrm{tax}}(\lambda H)=0.
\]
The alleged coercive estimate would give
\[
    0
    \ge
    \mu_{\Lambda}
    \Dist_{\loc,\mathrm{int},0}
    (\lambda H,\Gamma_{\Lambda,\adm}^{\mathrm{int}})
    =
    \mu_{\Lambda}\lambda
    \Dist_{\loc,\mathrm{int},0}
    (H,\Gamma_{\Lambda,\adm}^{\mathrm{int}})
    >
    0,
\]
a contradiction.
\end{proof}

\subsubsection{Reduced finite-dimensional criterion}

\begin{definition}[Reduced finite-window package model]
\label{def:reduced-package-model}
A reduced finite-window package model consists of a finite-dimensional coordinate
space \(X_N\), a coordinate map
\[
    x:\mathcal A_{\Lambda,N}^{\mathrm{red}}\to X_N,
\]
a gauge subspace \(G_N\subset X_N\) representing
\(\Gamma_{\Lambda,\adm}^{\mathrm{int}}\), and a finite-dimensional combined
detector/tax map
\[
    T_{\Lambda}:X_N\to Y_N,
\]
where
\[
    T_{\Lambda}x
    =
    \left(
        O^0_{\Lambda}x,
        T_{\mathrm{prs}}x,
        T_{\mathrm{loc}}x,
        T_{\mathrm{rep}}x,
        T_{\mathrm{gs}}x,
        T_{\detc}x
    \right).
\]
The model is compatible with the tax kernel if
\[
    M_{\Lambda}^{\mathrm{tax}}(D)=0
    \quad\Longleftrightarrow\quad
    T_{\Lambda}x(D)=0.
\]
\end{definition}

\begin{theorem}[Finite-dimensional matrix kernel criterion]
\label{thm:matrix-kernel-criterion}
Let \(\mathcal A_{\Lambda,N}^{\mathrm{red}}\) be a reduced finite-dimensional
package class in the sense of \Cref{def:reduced-package-model}.  Assume
\[
    \ker T_{\Lambda}\subset G_N.
\]
Then
\[
    \mathcal K_{\Lambda}^{\mathrm{tax}}
    \cap
    \mathcal A_{\Lambda,N}^{\mathrm{red}}
    \subset
    \Gamma_{\Lambda,\adm}^{\mathrm{int}}.
\]
If additionally \(G_N\subset\ker T_{\Lambda}\), then the reduced kernel-free
condition is equivalent to
\[
    \ker T_{\Lambda}=G_N.
\]
\end{theorem}

\begin{proof}
Let \(D\in\mathcal K_{\Lambda}^{\mathrm{tax}}\cap
\mathcal A_{\Lambda,N}^{\mathrm{red}}\).  Compatibility gives
\[
    T_{\Lambda}x(D)=0,
\]
so \(x(D)\in\ker T_{\Lambda}\).  By hypothesis, \(x(D)\in G_N\), which means
that \(D\) is an admissible gauge direction in the reduced model.  This proves
kernel-freeness.  If \(G_N\subset\ker T_{\Lambda}\), the inclusion
\(\ker T_{\Lambda}\subset G_N\) is exactly equivalent to equality.
\end{proof}

\subsubsection{Zero-set rigidity}

\begin{definition}[Channel zero sets]
\label{def:channel-zero-sets}
Define
\[
    Z_{\mathrm{obs}}:=\{D:O^0_{\Lambda}D=0\},
\]
\[
    Z_{\mathrm{prs}}:=\{D:\Tax_{\mathrm{prs}}(D)=0\},
    \qquad
    Z_{\mathrm{loc}}:=\{D:\Tax_{\mathrm{loc}}(D)=0\},
\]
\[
    Z_{\mathrm{rep}}:=\{D:\Tax_{\mathrm{rep}}(D)=0\},
\]
\[
    Z_{\mathrm{gs}}:=\{D:\Tax_{\mathrm{gs}}(D)=0\},
\]
\[
    Z_{\detc}:=\{D:\Tax_{\detc}(D)=0\}.
\]
\end{definition}

\begin{theorem}[Zero-set rigidity criterion]
\label{thm:zero-set-rigidity}
If
\[
    Z_{\mathrm{obs}}
    \cap
    Z_{\mathrm{prs}}
    \cap
    Z_{\mathrm{loc}}
    \cap
    Z_{\mathrm{rep}}
    \cap
    Z_{\mathrm{gs}}
    \cap
    Z_{\detc}
    \subset
    \Gamma_{\Lambda,\adm}^{\mathrm{int}},
\]
then
\[
    \mathcal K_{\Lambda}^{\mathrm{tax}}
    \subset
    \Gamma_{\Lambda,\adm}^{\mathrm{int}}.
\]
\end{theorem}

\begin{proof}
If \(D\in\mathcal K_{\Lambda}^{\mathrm{tax}}\), then by
\Cref{def:tax-kernel} the observable and all active tax channels vanish.  Hence
\[
    D\in
    Z_{\mathrm{obs}}
    \cap
    Z_{\mathrm{prs}}
    \cap
    Z_{\mathrm{loc}}
    \cap
    Z_{\mathrm{rep}}
    \cap
    Z_{\mathrm{gs}}
    \cap
    Z_{\detc}.
\]
The assumed zero-set intersection containment gives
\(D\in\Gamma_{\Lambda,\adm}^{\mathrm{int}}\).
\end{proof}

\subsubsection{Channelwise zero-set implications}

\begin{proposition}[Pressure-source zero-set implication]
\label{prop:prs-zero-set}
Assume there is \(c_{\mathrm{prs}}>0\) such that
\[
    \Tax_{\mathrm{prs}}(D)
    \ge
    c_{\mathrm{prs}}
    \left(
        \|C(D)\|
        +
        \|E_F(D)\|_{X_{\src}}
        +
        \|F^{\mathrm{act}}_D-F^{\mathrm{mod}}_D\|_{X_{\src}}
        +
        \|p_D^{\mathrm{act}}-p_D^{\mathrm{mod}}\|_{Y_{\mathrm{prs}}}
    \right).
\]
If \(\Tax_{\mathrm{prs}}(D)=0\), then
\[
    C(D)=0,\qquad
    E_F(D)=0,\qquad
    F^{\mathrm{act}}_D=F^{\mathrm{mod}}_D,\qquad
    p_D^{\mathrm{act}}=p_D^{\mathrm{mod}}.
\]
\end{proposition}

\begin{proof}
The right-hand side is a sum of nonnegative terms multiplied by a positive
constant.  If the left-hand side is zero, the sum is zero, and every term in the
sum must vanish.
\end{proof}

\begin{proposition}[Localization zero-set implication]
\label{prop:loc-zero-set}
Assume there is \(c_{\mathrm{loc}}>0\) such that
\[
    \Tax_{\mathrm{loc}}(D)
    \ge
    c_{\mathrm{loc}}
    \left(
        \Leak_{\nabla u}(D)
        +
        \Leak_u(D)
        +
        \Leak_p(D)
    \right).
\]
If \(\Tax_{\mathrm{loc}}(D)=0\), then
\[
    \Leak_{\nabla u}(D)=\Leak_u(D)=\Leak_p(D)=0.
\]
\end{proposition}

\begin{proof}
The displayed lower bound is a positive multiple of a sum of nonnegative
leakage coordinates.  Vanishing of the tax forces each leakage coordinate to
vanish.
\end{proof}

\begin{proposition}[Reproduction zero-set implication]
\label{prop:rep-zero-set}
Assume that along a finite chain the reproduction tax dominates all active
coordinate drifts:
\[
    \Tax_{\mathrm{rep}}(\calD)
    \ge
    c_{\mathrm{rep}}
    \sum_{k=0}^{K-1}
    \|D_{k+1}-\mathcal R_{k\to k+1}D_k\|_{\mathrm{act}}
\]
for some \(c_{\mathrm{rep}}>0\).  If \(\Tax_{\mathrm{rep}}(\calD)=0\), then
\[
    D_{k+1}=\mathcal R_{k\to k+1}D_k
    \qquad(0\le k<K)
\]
in all active coordinates.
\end{proposition}

\begin{proof}
The proof is the same nonnegative-sum argument.  If the tax vanishes, every
active drift norm in the finite sum vanishes.
\end{proof}

\begin{proposition}[Gate/slack zero-set implication]
\label{prop:gs-zero-set}
Assume the gate/slack tax dominates
\[
    \sum_{a\in\mathfrak A}
    \left(
        (B_a(D)-\tau_a(D))_+
        +
        |B_a(D)+s_a(D)-\tau_a(D)|
    \right)
\]
with a positive constant.  If \(\Tax_{\mathrm{gs}}(D)=0\), then for every
\(a\in\mathfrak A\),
\[
    B_a(D)\le\tau_a(D),
    \qquad
    B_a(D)+s_a(D)=\tau_a(D).
\]
\end{proposition}

\begin{proof}
Vanishing of the tax forces both nonnegative terms for each gate channel to
vanish.  Thus \((B_a-\tau_a)_+=0\), which is \(B_a\le\tau_a\), and
\(|B_a+s_a-\tau_a|=0\), which is the displayed identity.
\end{proof}

\begin{proposition}[Detector zero-set implication]
\label{prop:det-zero-set}
Assume the detector-comparison tax dominates the detector mismatch:
\[
    \Tax_{\detc}(D)\ge c_{\detc}^{0}\Err_{\detc}(D;\bzeta_0)
\]
for some \(c_{\detc}^{0}>0\).  If \(\Tax_{\detc}(D)=0\) and
\(O^0_{\Lambda}D=0\), then
\[
    \Err_{\detc}(D;\bzeta_0)=0.
\]
In any detector model where \(O^0_{\Lambda}D=0\) and
\(\Err_{\detc}(D;\bzeta_0)=0\) imply that the clean and localized detector
coordinates have no active defect, the clean and localized detectors see no
active defect.
\end{proposition}

\begin{proof}
The displayed domination gives \(\Err_{\detc}(D;\bzeta_0)=0\) when
\(\Tax_{\detc}(D)=0\).  The final statement is exactly the additional detector
model implication stated in the proposition.
\end{proof}

\begin{theorem}[Component zero-set rigidity implies kernel-freeness]
\label{thm:component-zero-set-kernel}
Assume all active detector/tax weights are positive.  Assume the zero-tax
conditions imply the channel constraints in
\Cref{prop:prs-zero-set,prop:loc-zero-set,prop:rep-zero-set,prop:gs-zero-set,prop:det-zero-set}.
Finally assume the combined constraints
\[
    O^0_{\Lambda}D=0,\qquad
    C(D)=0,\qquad
    E_F(D)=0,\qquad
    F_D^{\mathrm{act}}=F_D^{\mathrm{mod}},
\]
\[
    \Leak_{\nabla u}(D)=\Leak_u(D)=\Leak_p(D)=0,
\]
\[
    D_{k+1}=\mathcal R_{k\to k+1}D_k,
    \qquad
    B_a(D)\le\tau_a(D),
    \qquad
    B_a(D)+s_a(D)=\tau_a(D),
\]
and
\[
    \Err_{\detc}(D;\bzeta_0)=0
\]
jointly imply
\[
    D\in\Gamma_{\Lambda,\adm}^{\mathrm{int}}.
\]
Then
\[
    \mathcal K_{\Lambda}^{\mathrm{tax}}
    \subset
    \Gamma_{\Lambda,\adm}^{\mathrm{int}}.
\]
\end{theorem}

\begin{proof}
Let \(D\in\mathcal K_{\Lambda}^{\mathrm{tax}}\).  Since all active weights are
positive, \(O^0_{\Lambda}D=0\) and every tax channel vanishes.  The channelwise
zero-set implications give the displayed pressure, localization, reproduction,
gate/slack, and detector constraints.  The combined zero-set rigidity assumption
then gives \(D\in\Gamma_{\Lambda,\adm}^{\mathrm{int}}\).
\end{proof}

\subsubsection{Compact quotient coercivity}

\begin{theorem}[Conditional compact quotient pressure/tax gap]
\label{thm:compact-tax-gap}
Assume:
\begin{enumerate}[label=\textup{(\roman*)},leftmargin=*]
    \item the tax class is stable under quotient normalization: if
    \(r=\Dist_{\loc,\mathrm{int},0}(D,\Gamma_{\Lambda,\adm}^{\mathrm{int}})>0\),
    then a normalized representative \(D/r\) belongs to the class and has
    distance \(1\);
    \item the unit quotient sphere
    \[
        S_{\Lambda,0}
        :=
        \{D\in\mathcal A_{\Lambda}^{\mathrm{tax}}:
        \Dist_{\loc,\mathrm{int},0}
        (D,\Gamma_{\Lambda,\adm}^{\mathrm{int}})=1\}
    \]
    is compact modulo gauge;
    \item \(M_{\Lambda}^{\mathrm{tax}}\) is lower semicontinuous on this
    quotient;
    \item \(M_{\Lambda}^{\mathrm{tax}}\) is positively homogeneous under the
    quotient normalization;
    \item kernel-freeness holds:
    \[
        \mathcal K_{\Lambda}^{\mathrm{tax}}
        \subset
        \Gamma_{\Lambda,\adm}^{\mathrm{int}}.
    \]
\end{enumerate}
Then
\[
    \mu_{\Lambda}^{\mathrm{tax}}
    :=
    \inf_{D\in S_{\Lambda,0}}
    M_{\Lambda}^{\mathrm{tax}}(D)
    >
    0,
\]
and for every \(D\in\mathcal A_{\Lambda}^{\mathrm{tax}}\),
\[
    M_{\Lambda}^{\mathrm{tax}}(D)
    \ge
    \mu_{\Lambda}^{\mathrm{tax}}
    \Dist_{\loc,\mathrm{int},0}
    (D,\Gamma_{\Lambda,\adm}^{\mathrm{int}}).
\]
\end{theorem}

\begin{proof}
If the infimum were zero, compactness of \(S_{\Lambda,0}\) modulo gauge and
lower semicontinuity would give \(D_*\in S_{\Lambda,0}\) with
\[
    M_{\Lambda}^{\mathrm{tax}}(D_*)=0.
\]
Thus \(D_*\in\mathcal K_{\Lambda}^{\mathrm{tax}}\).  Kernel-freeness gives
\[
    D_*\in\Gamma_{\Lambda,\adm}^{\mathrm{int}},
\]
which contradicts
\[
    \Dist_{\loc,\mathrm{int},0}
    (D_*,\Gamma_{\Lambda,\adm}^{\mathrm{int}})=1.
\]
Hence \(\mu_{\Lambda}^{\mathrm{tax}}>0\).  If \(D\) has quotient distance
\(r=0\), the asserted inequality is immediate.  If \(r>0\), normalize to
\(\widehat D=D/r\in S_{\Lambda,0}\).  Homogeneity gives
\[
    M_{\Lambda}^{\mathrm{tax}}(D)
    =
    rM_{\Lambda}^{\mathrm{tax}}(\widehat D)
    \ge
    r\mu_{\Lambda}^{\mathrm{tax}},
\]
which is the desired estimate.
\end{proof}

\begin{theorem}[Additive-error tax coercivity]
\label{thm:additive-tax-coercivity}
Assume an ideal detector \(\widetilde M_{\Lambda}^{\mathrm{tax}}\) satisfies
\[
    \widetilde M_{\Lambda}^{\mathrm{tax}}(D)
    \ge
    \widetilde\mu_{\Lambda}^{\mathrm{tax}}
    \Dist_{\loc,\mathrm{int},0}
    (D,\Gamma_{\Lambda,\adm}^{\mathrm{int}})
\]
and the realized detector satisfies
\[
    M_{\Lambda}^{\mathrm{tax}}(D)
    +
    \Delta_{\mathrm{model}}
    \ge
    \widetilde M_{\Lambda}^{\mathrm{tax}}(D).
\]
Then
\[
    M_{\Lambda}^{\mathrm{tax}}(D)
    \ge
    \widetilde\mu_{\Lambda}^{\mathrm{tax}}
    \Dist_{\loc,\mathrm{int},0}
    (D,\Gamma_{\Lambda,\adm}^{\mathrm{int}})
    -
    \Delta_{\mathrm{model}}.
\]
\end{theorem}

\begin{proof}
Rearrange the model comparison and insert the ideal coercivity estimate:
\[
    M_{\Lambda}^{\mathrm{tax}}(D)
    \ge
    \widetilde M_{\Lambda}^{\mathrm{tax}}(D)
    -
    \Delta_{\mathrm{model}}
    \ge
    \widetilde\mu_{\Lambda}^{\mathrm{tax}}
    \Dist_{\loc,\mathrm{int},0}
    (D,\Gamma_{\Lambda,\adm}^{\mathrm{int}})
    -
    \Delta_{\mathrm{model}}.
\]
\end{proof}

\begin{corollary}[Conditional tax-route detection]
\label{cor:tax-route-transfer}
Assume the additive-error tax coercivity estimate of
\Cref{thm:additive-tax-coercivity}.  If the localized detector dominates the tax
detector in the sense that
\[
    M_{\Lambda}^{\loc}(D)
    +
    \Delta_{\mathrm{det/tax}}
    \ge
    a_{\Lambda}M_{\Lambda}^{\mathrm{tax}}(D)
\]
with \(a_{\Lambda}>0\), then
\[
    M_{\Lambda}^{\loc}(D)
    \ge
    a_{\Lambda}\widetilde\mu_{\Lambda}^{\mathrm{tax}}
    \Dist_{\loc,\mathrm{int},0}
    (D,\Gamma_{\Lambda,\adm}^{\mathrm{int}})
    -
    a_{\Lambda}\Delta_{\mathrm{model}}
    -
    \Delta_{\mathrm{det/tax}}.
\]
\end{corollary}

\begin{proof}
From the detector domination,
\[
    M_{\Lambda}^{\loc}(D)
    \ge
    a_{\Lambda}M_{\Lambda}^{\mathrm{tax}}(D)
    -
    \Delta_{\mathrm{det/tax}}.
\]
Insert \Cref{thm:additive-tax-coercivity} and distribute the factor
\(a_{\Lambda}\).
\end{proof}

\begin{remark}[Status]
\label{rem:tax-kernel-status}
This section proves fixed finite-window and reduced-model criteria for
pressure/tax kernel-freeness and the resulting compact quotient gap.  It does
not prove kernel-freeness for all suitable weak solutions, scale-uniform
pressure/tax coercivity, regularity, singularity exclusion, or a Clay-problem
conclusion.
\end{remark}

\subsection{Scope of the finite-window theorem}
\label{sec:scope-finite-window}

The results assembled above close the fixed-window package-realizability,
clean-source compactness, effective projection, reduced pressure/tax
kernel-freeness, and compact quotient-coercivity steps needed by the conditional
local-to-clean transfer theorem.  The conclusion remains a finite-window
detection statement.  The following issues are structural inputs or directions
outside the conclusion proved here.
\begin{enumerate}[label=(\roman*),leftmargin=*]
    \item A concrete detector model must verify the observable and tax-channel
    assumptions \Cref{ass:detector-lipschitz,ass:detector-residual-control}.
    \item The clean gap, chart visibility inequality, and
    component-to-baseline comparison are imported hypotheses for the transfer
    theorem; they are not derived from the package-realizability theorem.
    \item The compactness and effective-projection criteria are sufficient
    finite-window criteria.  The paper does not assert that arbitrary suitable
    weak solutions automatically generate compact clean-source families.
    \item The reduced pressure/tax criteria give usable finite-window
    kernel-free tests.  The paper does not assert global pressure/tax
    kernel-freeness for all localized Navier--Stokes packages.
    \item No scale-uniform propagation, infinite-chain iteration, regularity
    theorem, singularity-exclusion theorem, or Clay-problem conclusion follows
    from the finite-window result alone.
\end{enumerate}

\subsection{Summary of auxiliary proof modules}
\label{sec:proof-inventory}

The detailed proof above establishes the following auxiliary modules used in the finite-window assembly.
\begin{enumerate}[label=\textbf{Step \arabic*.},leftmargin=*]
    \item The finite-window detector comparison theorem is proved from the two
    detector-intertwining assumptions.
    \item The weighted detector-comparison corollary and the channelwise
    sufficiency criterion for detector mismatch control are proved.
    \item Unassigned detector channels are isolated in
    \(\Delta_{\detc}^{\mathrm{rem}}\).
    \item The conditional finite-window local-to-clean transfer theorem is
    proved by combining detector comparison with the clean gap, chart
    visibility, and weighted component comparison.
    \item The finite-window detection threshold corollary is proved.
    \item The package-realizability theorem shows that local
    pressure-admissible Navier--Stokes data generate the package coordinates.
    \item Clean pressure-image compactness is proved to imply uniform projection
    tail convergence.
    \item Source compactness, finite-dimensional source models, strong
    clean-coordinate compactness, Sobolev compactness, and Kolmogorov--Riesz
    translation compactness are proved to be sufficient compactness criteria.
    \item Effective projection bounds are shown to replace compactness in the
    finite-window transfer error.
    \item Reduced finite-window pressure/tax kernel-freeness criteria are proved
    through matrix and zero-set rigidity conditions.
    \item Compact quotient pressure/tax coercivity is proved under
    kernel-freeness, quotient compactness, lower semicontinuity, and
    homogeneity.
\end{enumerate}

\end{document}